\documentclass[reqno,11pt]{amsart}
\usepackage{tikz,genyoungtabtikz,enumitem,rotating,latexsym,bm,stmaryrd}
\usetikzlibrary{positioning,intersections}
\usepackage{amsmath,amsthm,amsfonts,amssymb,hyperref
  }
\usepackage{mathrsfs}
\usepackage{caption,cases}
\usepackage{lipsum,wasysym}
\usepackage{mathtools}
\usetikzlibrary{arrows,calc,through,backgrounds,matrix,decorations.pathmorphing}
\usepackage[a4paper,margin=2.3cm]{geometry}
\usepackage{cleveref}
\usepackage{xcolor, comment}
\usepackage{scalefnt}
\usepackage{xspace} %For Theorem A[optional space], etc.

\allowdisplaybreaks

\newcommand{\concat}{\circledast}

\newcommand{\DQ}{{\mathsf{DQ}}}
\DeclareMathOperator{\PSSYT}{PSSYT}
\newcommand{\PP}[3]{\mathscr{P}\!\mathscr{P}^#2_#1(#3)}
\renewcommand{\boxtimes}{\otimes}
\renewcommand{\epsilon}{\varepsilon}

\newtheorem{thm}{Theorem}[section]
\newtheorem{cor}[thm]{Corollary}
\newtheorem{lem}[thm]{Lemma}
\newtheorem{prop}[thm]{Proposition}
\newtheorem*{thmA}{Theorem~A}
\newtheorem*{thmB}{\TheoremB}
\newtheorem*{thmC}{\TheoremC}
\newtheorem*{thmD}{\TheoremD}

\theoremstyle{remark}
\newtheorem{rmk}[thm]{Remark}

\theoremstyle{definition}
\newtheorem{defn}[thm]{Definition}
\newtheorem{eg}[thm]{Example}

\crefname{defn}{Definition}{Definitions}
\crefname{thm}{Theorem}{Theorems}
\crefname{prop}{Proposition}{Propositions}
\crefname{lem}{Lemma}{Lemmas}
\crefname{cor}{Corollary}{Corollaries}
\crefname{conj}{Conjecture}{Conjectures}
\crefname{section}{Section}{Sections}
\crefname{subsection}{Subsection}{Subsections}
\crefname{eg}{Example}{Examples}
\crefname{figure}{Figure}{Figures}
\crefname{rem}{Remark}{Remarks}
\crefname{rmk}{Remark}{Remarks}
\crefname{equation}{equation}{equation}

\newcommand{\ten}{10}
\newcommand{\eleven}{11}
\newcommand{\elevena}{12}
\newcommand{\elevenb}{13}
\newcommand{\elevenc}{14}
\newcommand{\elevend}{15}
\newcommand{\Cee}{{p}}
\newcommand{\Dee}{{q}}

\numberwithin{equation}{section}

\newcommand{\pgen}{{\sf p}}
\newcommand{\egen}{e}
\newcommand{\sgen}{s}

\newcommand{\stt}{\mathsf{t}}  %standard index
\newcommand{\W}{{\mathfrak{S}}}

\newcommand{\Hom}{\operatorname{Hom}}
\newcommand{\End}{\operatorname{End}}

\newcommand{\Sym}{\W}

\newcommand{\sgn}{\operatorname{sgn}}

\newcommand{\ZZ}{{\mathbb Z}}
\newcommand{\NN}{{{\mathbb N}}}

\newcommand{\la}{{\lambda}}

\renewcommand{\le}{\leqslant}
\renewcommand{\ge}{\geqslant}
\newcommand{\GL}{\mathrm{GL}}

\newcommand{\Id}{\mathrm{Id}}
 
\newcommand{\rightspecht}[1]{{{\bf S}^{#1}}}
\newcommand{\leftspecht}[1]{{{\bf S}_{#1}}}

\renewcommand{\P}{\mathrm{P}}
\newcommand{\C }{\mathbb{C}}

\newcommand{\N}{\mathrm{N}}

\newcommand{\Stab}{\operatorname{Stab}}

\renewcommand{\C}{\mathbb{C}}

\newcommand{\ParSet}{\mathscr{P}}
\newcommand{\DMParSet}{\mathscr{M\!P}^\star}
\newcommand{\MParSet}{\mathscr{M\!P}}
\newcommand{\Pdelta}{P_r(\delta_{\rm in}  \delta_{\rm out} )}

\newcommand{\Fdelta}{\Delta_r( \alpha^\beta)}

\usetikzlibrary{decorations.pathreplacing}

\newcounter{thmlistcnt}
\newenvironment{thmlist}%
	{\setcounter{thmlistcnt}{0}%
	\begin{list}{\emph{(\roman{thmlistcnt})}}{%
		\usecounter{thmlistcnt}%
		\setlength{\topsep}{0pt}%
		\setlength{\leftmargin}{0pt}%
		\setlength{\itemsep}{0pt}%
		\setlength{\labelwidth}{17pt}
		\setlength{\itemindent}{30pt}}%
	}%
	{\end{list}}%

\DeclareMathOperator{\Inf}{Inf} 

\newcommand{\Ind}{\bigl\uparrow}
\newcommand{\ind}{{\uparrow}}
\newcommand{\Res}{\bigl\downarrow}
\newcommand{\res}{{\downarrow}}

\newcommand{\ram}{R}

\newcommand{\Std}{\mathrm{Std }}

\let\originalleft\left
\let\originalright\right
\def\left#1{\mathopen{}\originalleft#1}
\def\right#1{\originalright#1\mathclose{}}

\renewcommand{\ge}{\geqslant}

\renewcommand{\ge}{\geqslant}
\renewcommand{\geq}{\geqslant}
\renewcommand{\le}{\leqslant}
\renewcommand{\leq}{\leqslant}

\renewcommand{\preceq}{\preccurlyeq}

\newcommand{\one}{{$\bullet$}}
\newcommand{\two}{{$\bullet$}}
\newcommand{\three}{{$\bullet$}}
\newcommand{\four}{{$\bullet$}}
\newcommand{\five}{{$\bullet$}}
\newcommand{\six}{{$\bullet$}}

\newcounter{i}
\newcounter{j}

\newcommand{\pyoung}[3]{
	\begin{tikzpicture}[line width=0.6pt, x=#1,y=-#2]
		\pyoungInner{#3}
	\end{tikzpicture}
}

\newcommand{\pyoungInner}[1]{
	\setcounter{i}{0}
	\setcounter{j}{0}
	\foreach \tRow in #1 {
		\addtocounter{i}{1}
		\setcounter{j}{0}
		\foreach \tEntry in \tRow {
			\addtocounter{j}{1}
			\node at (\value{j}+0.5, \value{i}+0.5) {\tEntry};
			\draw (\value{j},\value{i})--(\value{j},\value{i}+1);
		}
		\draw (\value{j}+1,\value{i})--(\value{j}+1,\value{i}+1);
		\draw (1,\value{i})--(\value{j}+1,\value{i});
		\draw (1,\value{i}+1)--(\value{j}+1,\value{i}+1);
	}
}

\DeclareMathOperator{\IndInfRes}{IndInfRes}

\newcommand{\barone}{{$\bullet$}}
\newcommand{\bartwo}{{$\bullet$}}
\newcommand{\barthree}{{$\bullet$}}
\newcommand{\barfour}{{$\bullet$}}
\newcommand{\barfive}{{$\bullet$}}
\newcommand{\barsix}{{$\bullet$}}
\newcommand{\barseven}{{$\bullet$}}
\newcommand{\bareight}{{$\bullet$}}

\def\ignore#1{\relax}

\def\ignore#1{\relax}

\newcommand{\rc}{{\overline{rc}}}
 \newcommand{\Tmnkr}[4]{{T}{\bigl((#2-#3,#3),#1\bigr)_{#4}}}

\newcommand\encircle[1]{%
  \tikz[baseline=(X.base)] 
    \node (X) [draw, shape=circle, inner sep=0] {\strut #1};}

\newcommand{\ForPrmn}{_{P_r(mn)}}
\newcommand{\ForPrmnLow}{{\raisebox{-6pt}{$\scriptstyle P_r(mn)$}}}
\newcommand{\deltaP}{\epsilon}

\newcommand{\lambdaP}{\gamma}
\newcommand{\muP}{\pi}

\newcommand{\TheoremA}{Theorem~\hyperlink{thmA}{A}\xspace}
\newcommand{\TheoremB}{Theorem~\hyperlink{thmB}{B}\xspace}
\newcommand{\TheoremC}{Theorem~\hyperlink{thmC}{C}\xspace}
\newcommand{\TheoremD}{Theorem~\hyperlink{thmD}{D}\xspace}

\renewcommand{\underline}{\relax} %bold is sufficient for composition c

\begin{document}
\title[The partition algebra and the plethysm coefficients, II] {The partition algebra and the plethysm coefficients  II: Ramified plethysm } 
\author{Chris Bowman, Rowena Paget and Mark Wildon}
\address{Department of Mathematics,  University of York, Heslington, York,  UK}
\email{Chris.Bowman-Scargill@york.ac.uk}
\address{School of Mathematics, Statistics and Actuarial Science, University of Kent, CT2 7NF, UK}
\email{R.E.Paget@kent.ac.uk}
\address{Heilbronn Institute for Mathematical Research, School of Mathematics, University of Bristol, Woodland Road, Bristol BS8 1UG, UK}
%Department of Mathematics, Royal Holloway University of London, Egham TW20 0EX, UK}
\email{mark.wildon@bristol.ac.uk}
%\email{mark.wildon@rhul.ac.uk} 
\maketitle
 
\thispagestyle{empty}

\begin{abstract}
The plethysm coefficient $p(\nu, \mu, \lambda)$ is the multiplicity of the Schur function $s_\lambda$ in the plethysm product $s_\nu \circ s_\mu$. In this paper we use Schur--Weyl duality between wreath products of symmetric groups and the ramified partition algebra to interpret an arbitrary plethysm coefficient as the multiplicity of an appropriate composition factor in the restriction of a  module for the ramified partition algebra to the partition algebra. This result implies new stability phenomenon for plethysm coefficients when the first parts of $\nu$, $\mu$ and $\lambda$ are all large. In particular, it gives the first positive formula  in the case when $\nu$ and $\lambda$ are arbitrary and $\mu$ has one part. Corollaries include new explicit positive formulae and combinatorial interpretations for the plethysm coefficients $p((n-b,b), (m), (mn-r,r))$, and $p((n-b,1^b), (m), (mn-r,r))$ when $m$ and $n$ are large.
\end{abstract}
%MkW Oct: 1^\ell -> 1^b in abstract
 
\section{Introduction}

Understanding the {\sf plethysm coefficients} is a fundamental problem in the representation theories of symmetric and general linear groups.
It was identified by R. Stanley as one of the most important open problems in algebraic combinatorics \cite{MR1754784}.    
Beyond pure mathematics, the plethysm coefficients 
arise in quantum information theory \cite{MR2421478,MR2745569} and 
are central objects  in geometric complexity theory (GCT), 
an approach that seeks to settle the $\mathrm{P} \neq \mathrm{NP}$ problem. 
The importance of  plethysm coefficients  in GCT derives from their frequent appearance in   formulas for multiplicities in coordinate rings: the orbit of the product of variables \cite[Section 9]{MR3729273}, the permanent polynomial \cite[Equation (5.5.2)]{MR2861717}, 
the power sum polynomial  and  the unit tensor \cite[Introduction]{MR4171310}.  
Thus  the problems of both {\em calculating} and  {\em bounding}   plethysm  coefficients are of fundamental importance  in  GCT. % (see 
% \cite{MR4171310} for more details).  

One of the most effective ways to study  plethysm  coefficients is through  their stability phenomena 
%(see, for example, \cite{} where the authors use such stabilities to seismic effect in GCT).  
 (for example, these stabilities were used to spectacular effect in \cite{MR3868002}). 
 % to disprove  famous  conjectures in GCT).  
%(see, for example  \cite{MR3868002,MR3695867}). 
  Throughout this paper  we set  $\alpha{[d]} = (d-|\alpha|, \alpha_1, \alpha_2, \dots, \alpha_{\ell(\alpha)})$
and we write  $p(\beta{[n]}, \alpha{[m]},  \kappa{[mn]})$ for the plethysm coefficient
$\langle s_{\beta{[n]}} \circ s_{\alpha{[m]}}\, , \,s_{\kappa{[mn]}} \rangle$ defined using the plethysm
product of Schur functions.
  Our original motivation  was to study the {\em stable plethysm coefficients}, 
  %$\overline{p} (\beta,\varnothing,\la)$, which are    
defined by  
\[
%\overline{p} (\beta,\varnothing ,\la)= 
\lim_{m,n\to \infty} p(\beta{[n]}, (m) ,  \kappa{[mn]}).
\]
(This stability is proven in  \cite{MR1190119,MR1037395}.)   
In the very special case
when $\beta = \varnothing$,
the stable coefficients  % $\overline{p} (\varnothing,\varnothing,\la)$ 
 are amongst the most celebrated and well-understood  of all plethysm coefficients:
 they satisfy  the stable version of  Foulkes'  conjecture \cite[Theorem 4.3.1]{MR1651092} %Manivel
  and 
 they 
 were an important stepping stone in the resolution of Weintraub's conjecture \cite{MR1651092,MR2745569}. 
 % Moreover
Moreover,
%Manivel proved a positive combinatorial formula  for their calculation 
% \cite[Theorem 4.1.1]{MR1651092}, a result reproved by
% the first two authors using the partition algebra \cite[Theorem B]{MR4756467}.
there is a positive combinatorial formula  for their calculation 
\cite[Theorem 4.1.1]{MR1651092} (see also \cite[Theorem B]{MR4756467}). 
Our first main result is a vast generalisation of this formula %vastly -> vast; death to adverbs
to the case when the partition $\beta$ is arbitrary.

\begin{thmA}\label{thmA}\hypertarget{thmA}
Let $\beta\vdash b$ and $\kappa \vdash r$ be partitions such that 
$\beta[n]$ and $\kappa[mn]$ are also partitions.
The plethysm coefficient 
 $p(\beta[n], (m), \kappa[mn])$ is constant for
$m \ge r - |\beta| + [\beta \not= \varnothing]$ and $n \ge r + \beta_1$
and its value is
\[  \sum_{\begin{subarray}c
\Cee,\,\Dee\, : \, \Cee+\Dee \,=\, r  \\
\gamma=  
   ( \Cee^{c_\Cee}, \dots ,  1^{c_1} ) \vdash \Cee,\, \ell(\gamma) = b 
%Note \ell(\gamma) = b is correct in \alpha = \varnothing case
%I'll leave this here as a reminder of the can of worms one opens by combining a/=0 and a=0 cases
%\ell( (a^b) + \gamma) )= b
 \\
 \beta^i \vdash c_i \text{ for }  1\leq i \leq p 
   \\
   \deltaP \in \ParSet_{>1}(\Dee)
\end{subarray}}
c^\beta_{\beta^\Cee, \ldots, \beta^1}
\Bigl[ \Bigl(\bigl( \hskip0.5pt \bigotimes_{i=1}^{\Cee}
\Inf_{\W_{c_i}}^{\W_i \wr \W_{c_i}} 
\rightspecht{\beta^i} \bigr) \Ind^{\W_\Cee} \, \otimes\, \C_{{\rm Stab}(\deltaP)}\Ind^{\W_\Dee} \Bigr)\Ind^{\W_r} :
\rightspecht{\kappa}
\Big]_{\W_r}.
\]
\end{thmA}
Unavoidably, this formula is heavy on notation:
\begin{itemize}[leftmargin=12pt]
\item $[\beta \not=\varnothing]$ is an Iverson bracket, equal to $1$ when $\beta \not=\varnothing$ 
and $0$ when  $\beta = \varnothing$;
\item $\ell(\gamma)$ is the length of the partition $\gamma$: thus $c_\Cee + \cdots + c_1 = b$;
%MkW: I think this has to be emphasised, for instance if r < b then one needs this to see that the sum is empty.
%(Otherwise we leave it implicit that the sum is zero because the Littlewood--Richardson coefficients all vanish...
%but some people might say that e.g. c^{(3)}_{(1),(1)} is not well-defined in the first place. 
\item $\ParSet_{> 1}(q)$ is the set of
partitions of $q$ having no singleton parts;
%for instance if $\epsilon = (2,2)$
%then taking the set-partition to be $\bigl\{ \{1,2\}, \{3,4\} \bigr\}$,
%we have $\Stab(\epsilon) = \langle (1,2), (3,4), (1,3)(2,4) \rangle \cong \W_2
%\wr \W_2 \le \W_4$
\item \smash{$c^\beta_{\beta^\Cee, \ldots, \beta^1 }$} is a generalized Littlewood--Richardson coefficient,
as defined in~equation~\eqref{eq:genLR}; 
\item $\rightspecht{\kappa}$ denotes the right Specht module canonically labelled
by the partition $\kappa$;
\item $\Stab(\epsilon)$ is the stabiliser in the symmetric group $\W_\Dee$
of a set-partition of $\{1,\ldots, \Dee\}$
into parts of sizes $\epsilon_1,  \ldots, \epsilon_{\ell(\epsilon)}$, as defined
in Definition~\ref{defn:stabiliserSubgroup};
\item
{\color{red}\color{black} \smash{$\Inf_{\W_{c_i}}^{\W_i \wr \W_{c_i}} $}
 is the usual inflation functor along the canonical surjection $ \W_i \wr \W_{c_i} \to \W_{c_i}$.}
\end{itemize}

%Above written without bullet points.
% Here $[\beta \not= \varnothing]$ is an Iverson bracket, equal to $1$ when $\beta \not=\varnothing$ 
%and $0$ when  $\beta = \varnothing$. In the sum
% $\ParSet_{> 1}(q)$ 
%and $\gamma$ varies over all partitions of $\Cee$; the notation indicates
%that $c_i$ is the multiplicity of $i$ as a part of $\gamma$. 
%By $\Stab(\epsilon)$ we mean the stabiliser of a set-partition of $\{1,\ldots, \Dee\}$
%into parts of size $\epsilon_1,  \ldots, \epsilon_{\ell(\epsilon)}$, as defined
%in Definition~\ref{defn:stabiliserSubgroup}; for instance if $\epsilon = (2,2)$
%then taking the set-partition to be $\bigl\{ \{1,2\}, \{3,4\} \bigr\}$,
%we have $\Stab(\epsilon) = \langle (1,2), (3,4), (1,3)(2,4) \rangle \cong \W_2
%\wr \W_2 \le \W_4$. Finally
%$\rightspecht{\kappa}$ denotes the Specht module canonically labelled
%by the partition $\kappa$ and
%$c^\beta_{\beta^\Cee, \ldots, \beta^1, \beta^0}$ is a generalized Littlewood--Richardson coefficient,
%as defined in~\eqref{eq:genLR}.
 
The entire proof of \TheoremA is outlined in Section~\ref{subsec:structure} at the end of this introduction.
%We  give a more conceptual restatement of \TheoremA in functorial language  in \cref{prop:IndInfRes}
%and restate it in the language of symmetric functions
%in Proposition~\ref{prop:symmetricFunctions}. 

 \subsection{Recasting plethysm in terms of diagram algebras}
We deduce  \TheoremA  as
a corollary of  a number of more powerful theorems for calculating and bounding  plethysm coefficients proved in this paper.  
The key to our approach is the see-saw pair {\color{red}\color{black}(in the sense of \cite{MR1606831})} below.

\[ \smallskip
\begin{minipage}{6.75cm}
\begin{tikzpicture}

\path(-2.5,1) node {$\W_{mn}$};
\path(-2.5,-1) node {$\W_{m}\wr \W_{n} $};

%\path(-2.5,0) node [rotate=90] {$\hookxrightarrow{ \ \ }$};

\path(-2.5,-0.5) --++(-45:0.1) coordinate (X)--++(45:0.1) coordinate (XX);
\draw[<-](-2.5,0.5) -- (-2.5,-0.5) to [out =-90, in =180]  (X)
 to [out =0, in =-90]  (XX)--++(90:0.05);

\begin{scope}[xshift=0.5cm]
\path(2.5,0.5) --++(45:0.1) coordinate (X)--++(-45:0.1) coordinate (XX);
\draw[<-](2.5,-0.5) -- (2.5,0.5) to [out =90, in =180]  (X)
 to [out =0, in =90]  (XX)--++(90:0.05);

\path(2.5,1) node {$\; P_r(mn) $};
\path(2.5,-1) node {$\; R_r(m,n) $};

\draw[->] (1.75,0.85) --(0.55,0.255);
\draw[->] (1.75,-0.85) --(0.55,-0.255);
\end{scope}

\path(0.25,0) node {$(\mathbb C^{mn})^{\otimes r}$};

\draw[->] (-1.75,0.85) --(-0.55,0.255);
\draw[->] (-1.75,-0.85) --(-0.55,-0.255);

\end{tikzpicture}
\end{minipage}\smallskip
\]
This shows the {\em partition algebra} $P_r(mn)$  in its  well-known  Schur--Weyl duality with the symmetric group 
$\W_{mn}$: see \cref{sec2,SWsec}. Restricting  to the subgroups  
$\W_{m}\wr \W_{n} \leq \W_{mn}$,  
we obtain a larger   algebra  of invariants on the other side of Schur--Weyl duality, the % lesser-known 
{\em ramified partition algebras} $R_r(m,n)$: see \cref{Ramified partition algebras,SWsec}.  
By~\eqref{eq:plethysmBranchingCoefficient}, the 
plethysm coefficients are the branching coefficients for 
the subgroups $\W_{m}\wr \W_{n} \leq \W_{mn}$ 
and thus, by 
the general theory of see-saw pairs \cite{MR1606831}, 
we can recast these coefficients as branching coefficients for restriction to  the  
subalgebra $P_r(mn)$ of the ramified partition algebra $R_r(m,n)$.  
 In \cref{The ramified Schur functor} we match up the labels of simple modules under 
 these Schur--Weyl dualities and interpret plethysm coefficients as composition multiplicities
 in the partition algebra.

\begin{thmB}\label{thmB}\hypertarget{thmB}
Let $\alpha, \beta, \kappa$ be partitions such that $\alpha{[m]}$, $\beta{[n]}$ and $\kappa{[mn]}$ are partitions. Suppose that $r \ge |\kappa|$.
The plethysm coefficient $p(\beta{[n]}, \alpha{[m]}, \kappa{[mn]})$ is equal
to the composition multiplicity of the simple module $L_r(\kappa)$ for the
partition algebra~$P_r(mn)$ in the relevant restricted simple module for the ramified
partition algebra specified on the right-hand side below:
\[
p(\beta{[n]}, \alpha{[m]}, \kappa{[mn]})=
\begin{cases} %\tag{$\dagger$}
\big[ L_r(\varnothing^{\beta}) \Res\ForPrmn^{\ram _r(m,n)} 
: \,L_r(\kappa)\big]\ForPrmn
& \textrm{if } \alpha=\varnothing,  r\ge |\beta|,\\[3pt]
\big[ L_r(\alpha^{\beta[n]})\Res\ForPrmn^{\ram _r(m,n)} : \,L_r(\kappa)\big]\ForPrmn
& \textrm{if } \alpha \neq \varnothing,  r\ge n|\alpha|. 
\end{cases}\]
\end{thmB}

 A key observation in the proof of \TheoremB, used at the start
 of each case in the proof of  \cref{SW-simples}, is that the 
ramified partition algebras are quasi-hereditary and hence their simple modules 
 $L_r(\alpha^\beta)$, which are in general   difficult to construct, arise as the 
 simple heads of  corresponding standard modules $\Delta_r(\alpha^\beta)$, which in turn
can easily be constructed using Young symmetrizers: see ~\eqref{actual-d-basis}.

This provides a new  dichotomy between the role played by the inner partition $\mu$
and the outer partition $\nu$ in a plethysm product $s_\nu \circ s_\mu$;
another such dichotomy is obtained in  \cite[Theorem 3.1 versus Proposition 3.3]{MR4171310}.   
{\color{red}\color{black}This is discussed in  detail in \cref{subsec:alphaEmptyVersusNonEmpty}.}
%{\color{red}\color{black}In more detail, if $\alpha=\varnothing$ (and we place no assumptions on $\beta$) then we get a huge reduction in the difficulty of our problem (as we  can immediately reduce the rank of {\em both} partitions). 
%If $\beta=\varnothing$ (and we place no assumptions on $\alpha$) then the reduction is only on the rank of 
%one of the partitions).
%This dichotomy will be seen mostly keenly    in Theorem C below, where we  obtain {\em sharp} bounds  for plethysms 
% with trivial inner part (but not for plethysms with trivial outer part).}

\subsection{Stability  and ramified branching coefficients}
 
The simple modules $L_r(\alpha^\beta)$ for the ramified partition algebras are quotients of the 
standard modules  $\Delta_r(\alpha^\beta)$.  We may therefore bound the plethysm coefficients appearing in \TheoremB by the multiplicity of a simple module in the restriction of a standard module to the partition algebra, 
as in the first part of \TheoremC below. 
%
%Moreover the ramified partition algebras are semisimple for suitably large values of the parameters $m,n \in \NN$ and thus  
%the standard modules are  simple for large  $m,n \in \NN$.  
%In this situation we have equality and not just an upper bound.
% 
%MkW: I've taken this out as I think it is misleading, okay the algebras become semisimple, but we are not 
%free to vary m and n in this way. Instead I prefer the more detailed discussion I added later in Section 
%1.5 (contrasting the cases ...)
%MkW query to Rowena and Chris: are you okay with this?
In the case when $\alpha = \varnothing$ we improve on this bound. In Section 7 we prove that the standard module 
$\Delta_r(\varnothing^{\beta})$ for the ramified partition algebra
is simple whenever both
%the (weaker) conditions; % weaker than what? Please can we shorten?
$m \ge r-|\beta|+[\beta \not= \varnothing]$ and $n \ge r + \beta_1$,
proving the equality in the second part of the theorem below.

\begin{thmC}\label{thmC}\hypertarget{thmC}
 Let $\alpha, \beta, \kappa$ be partitions such that $\alpha{[m]}, \beta{[n]}, \kappa{[mn]}$ are also  partitions. Suppose that $r \ge |\kappa|$. Then
\[
p( \beta{[n]}, \alpha{[m]},  \kappa{[mn]}) \leq  
\begin{cases}
\big[ \Delta_r(\varnothing^{\beta}) \Res\ForPrmn^{\ram _r(m,n)} 
: \,L_r(\kappa)\big]\ForPrmn
& \textrm{if } \alpha=\varnothing,  r\ge |\beta|,\\[3pt]
\big[ \Delta_r(\alpha^{\beta[n]})\Res\ForPrmn^{\ram _r(m,n)} : \,L_r(\kappa)\big]\ForPrmn
& \textrm{if } \alpha \neq \varnothing,  r\ge n|\alpha|. 
\end{cases}\]
Moreover, if $m \ge r-|\beta|+[\beta \not= \varnothing]$ and $n \ge r + \beta_1$, then
\[
p(\beta{[n]}, (m),  \kappa{[mn]}) = \bigl[\Delta_r( \varnothing^{\beta})\Res^{\ram_r(m,n)}_{P_r(mn)} :
L_r(\kappa)\bigr]\ForPrmn. \]
 %Here $[\beta \not= \varnothing]$ is an Iverson bracket, equal to $1$ when $|\beta| \ge 1$, and $0$ when   $|\beta|=0$.
\end{thmC}

%It's already mentioned ... MkW. 
%As already mentioned, both calculating and bounding the values  for plethysm coefficients is of central importance in GCT and algebraic combinatorics. 

\subsection{Combinatorial formulas for ramified branching coefficients }
In light of \TheoremB, we seek to calculate the \textsf{ramified branching coefficients}
$\bigl[\Delta_r(\alpha^\beta)\res^{\ram_r(m,n)}_{P_r(mn)} : L_r(\kappa)\bigr]\ForPrmn$,
and hence calculate and bound the plethysm coefficients.
The partition algebras and  ramified partition algebras arise as towers of recollement in the sense of \cite{MR2236606}; roughly speaking this means that 
they arise as sequences of algebras and are
 equipped with idempotents which allow us to work by induction on the rank. 
In \cref{whynamed,explicitlystated}, we utilise this theory to construct quotient $P_r(mn)$-modules 
\begin{equation}\label{eq:DQintro}
\Delta_r(\alpha^\beta)\Res^{R_r(m,n)}_{P_r(mn)} \xrightarrow{ \ \ \ }  {\sf DQ}\bigl(\Delta_{r}(\alpha^\beta)\bigr).
\end{equation}
These quotient modules  possess beautiful planar diagram bases (see Section~\ref{subsec:elementaryDiagrams}) which are 
amenable  to computation. By decomposing these quotient modules  we are able to calculate the ramified branching coefficients explicitly, as follows.

 \begin{thmD}\label{thmD}\hypertarget{thmD}
 The ramified branching coefficient $[\Delta_r(\alpha^\beta)\res^{\ram_r(m,n)}_{P_r(mn)} : L_r(\kappa)]$ 
 for $\alpha\vdash a, \beta\vdash b  $ and $\kappa \vdash r$ is equal to the multiplicity of 
$\rightspecht {\kappa}$ in the following $\C \W_r$-module: 
%
% \ParSet_{>1}(\Dee )$
%  decomposes as follows 
\[
\bigoplus
_{
\begin{subarray}c
\Cee,\,\Dee\, : \, \Cee+\Dee \,=\, r-ab  \\
\gamma=  
   ( \Cee^{c_\Cee}, \dots ,  1^{c_1}, 0^{c_0} )
\in  \ParSet _{\!(a^b)}(\Cee) 
%  \\
%\ell( (a^b) + \gamma) )= b
   \\[-3pt]
    \beta^i \vdash c_i \text{ for } 0 \leq i \leq \Cee 
\\
   \deltaP \in \ParSet_{>1}(\Dee)
\end{subarray}
} 
\!\!\!\!\!\!\!\!\!\!\! c^\beta_{\beta^\Cee, \ldots, \beta^1, \beta^0} \Bigl(\bigotimes  _{i=0}^\Cee
\Bigl(
 (
\rightspecht {\alpha} \, \boxtimes\, \C  )
\Ind
_{\W_a\times \W_{ i }}
^{\W_{a+i}}
)
\,\oslash\, \rightspecht{\beta^i}
 \Bigr)
 \,\boxtimes\, \C _{{\rm Stab}(\deltaP)}\Bigr)
 \Ind_{{\rm Stab}((a^b)+\gamma)\times {\rm Stab}(\deltaP)} ^{\W_r}
\]
 where 
$\ParSet_{>1}(\Dee)$
 is the set of partitions of~$\Dee$ with no singleton parts, 
 $(a^b)$ is the empty partition if $a=0$ and otherwise the partition with $b$ parts all of size~$a$,
and the set $\ParSet_{(a^b)}(\Cee)$ is as defined in Definition~\ref{defn:ParSet}.
{\color{red}\color{black}The 
wreath product of modules, denoted $\oslash$, is defined in 
\cref{subsec:wreathProducts}.}
\end{thmD}

By Definition~\ref{defn:ParSet}, 
the elements $\gamma \in \ParSet_{(a^b)}(\Cee)$ 
are partitions having exactly $b$ parts \emph{including~$c_0$ distinguished zero parts}. Thus
$c_\Cee  + \cdots + c_1 + c_0 = b$ and the Littlewood--Richardson coefficient
$c^\beta_{\beta^\Cee, \ldots, \beta^1, \beta^0}$ is well-defined. When $a=0$,  the partitions 
$\gamma \in \ParSet_{(a^b)}(\Cee)$
have exactly~$b$ non-zero parts and $c_0 = 0$, and the sum is the same as in \TheoremA.

\TheoremC provides a  combinatorial  formula for computing arbitrary ramified branching coefficients in terms of {\em  much smaller} plethysm products and Littlewood--Richardson coefficients.
We prove \TheoremD in Section~\ref{sec9} and then immediately deduce 
\TheoremA as a corollary of \TheoremC and \TheoremD.

Using symmetric functions we prove in Section 11 that the bounds on $m$
and $n$ in \TheoremA cannot be weakened in infinitely many cases.
This is notable because, unlike the usual direction in this paper, we successfully
apply symmetric functions to deduce a result about the ramified partition algebra.

\subsection{Examples and applications}\label{subsec:introExamplesApplications}
To make this paper accessible, particularly to a non-diagram\-matic-algebra audience, 
we give plenty of examples throughout the paper. After the proof of \TheoremA is complete,
we begin  Section~\ref{egs} by giving two
substantial examples showing how the decomposition of the
depth quotient (see Definition~\ref{dq}) determines  plethysm coefficients.
We then give a more conceptual restatement of \TheoremA using a functor on
symmetric group modules in Proposition~\ref{prop:IndInfRes},
before proving Proposition~\ref{prop:symmetricFunctions} which restates \TheoremA and Proposition~\ref{prop:IndInfRes} in the 
language of symmetric functions.
As applications we find explicitly positive formulae for the stable limits of the plethysm
coefficients $p\bigl((n-b,b), (m), (mn-r,r)\bigr)$, 
 $p\bigl((n-b,1^b), (m), (mn-r,r)\bigr)$
and $p\bigl((n-b,1^b), (m), (mn-r,1^r)\bigr)$.
%We end in Section~\ref{sec:tight} by using symmetric
%functions and the plethystic
%semistandard tableaux model (as developed in \cite{deBoeckPagetWildon})
%to show that the bounds on $m$ and $n$ in \TheoremA are tight in infinitely many cases.

\subsection{Contrasting the cases $\alpha = \varnothing$ and $\alpha \not= \varnothing$}
\label{subsec:alphaEmptyVersusNonEmpty}
Since it illuminates two key points in the proof, we remark on the qualitative difference in our results
on plethysm coefficients
in these two cases. The reader may prefer to return to this remark after reading Section 6.

\subsubsection*{The ramified Schur functor}
The outer partition in the standard module in \TheoremC is $\beta$ when $\alpha =\varnothing$
and $\beta[n]$ when $\alpha\not=\varnothing$. The difference arises
 from the behaviour  of the ramified Schur functor 
$\Hom_{\W_m \wr \W_n}\bigl( -, (\mathbb{C}^{mn})^{\otimes r} \bigr)$ 
mapping left $\C\W_m \wr \W_n$-modules to right $R_r(m,n)$-modules. 
The module $\leftspecht{(m)} \oslash \leftspecht{\beta[n]} = \Inf_{\W_n}^{\W_m \wr \W_n} \leftspecht{\beta[n]}$
is acted on trivially by the base group $\W_m \times \cdots \times \W_m$ in the wreath
product $\W_m \wr \W_n$ and embeds in the tensor space $(\mathbb{C}^{mn})^{\otimes r}$
whenever $r \ge b$; its image under the Schur functor
is then the simple module $L_r(\varnothing^\beta)$ for the ramified partition
algebra $\ram_r(m,n)$. 
 In contrast, 
when $\alpha\not=\varnothing$, the base group acts non-trivially on $\leftspecht{\alpha}
\oslash \leftspecht{\beta[n]}$ and it is necessary to 
take $r \ge n |\alpha|$ to get an embedding. We then obtain the simple
module $L_r(\alpha^{\beta[n]})$. The distinction can be seen by comparing equations~\eqref{eq:rowenaz}
and~\eqref{eq:rowenazNonEmpty}. This is the first and most critical point where the two cases
diverge.

\subsubsection*{Standard modules versus simple modules and semisimplicity}
In the case when $\alpha = \varnothing$, the standard $R_r(m,n)$-module $\Delta_r(\varnothing^\beta)$ is isomorphic
to the simple module $L_r(\varnothing^\beta)$ by Theorem~\ref{this!!}
for suitably large values of $m$ and $n$,
and so the plethysm coefficient $p(\beta[n], (m), \kappa[mn])$ is equal to the ramified branching coefficient $\bigl[ \Delta_r(\varnothing^\beta)\res^{R_r(m,n)}_{P_r(mn)} : L_r(\kappa) \bigr]_{P_r(mn)}$.
Putting aside bounds for now, this leads to the formula in \TheoremA for the joint limit of the plethysm coefficient
when $m$ and $n$ independently become large.
%, obtained using the depth quotient and the diagramatic basis of the cell module. 
When $\alpha \not= \varnothing$,
the equality of \TheoremB requires $r \ge n |\alpha|$, but $R_r(m,n)$ is never semisimple when this condition holds.
%the cell module $\Delta_r(\alpha^{\beta[n]})$ is not in general simple (see Corollary~\ref{cor:onlyIfProperQuotient}),
In this case we obtain only an inequality bounding $p(\beta[n], \alpha[m], \kappa[mn])$ by
the ramified branching coefficient  $\bigl[ \Delta_r(\alpha^{\beta[n]})\res^{R_r(m,n)}_{P_r(mn)} : L_r(\kappa) \bigr]_{P_r(mn)}$. 
This is the second point of divergence and explains why our results on plethysm
coefficients are sharp only when $\alpha=\varnothing$.

\subsection{Analogies and motivation from Kronecker coefficients}
It is worth emphasising that all the ideas of this paper have analogues in the context of the Kronecker coefficients. 

By \TheoremB, the   plethysm  and  ramified branching coefficients can be  interpreted as restriction multiplicities
 of simple and standard modules from the ramified partition algebra to the partition algebra.  
By \cite[Equation 3.1.3]{BDO15} 
the   Kronecker coefficients and stable  Kronecker coefficients can be interpreted as the 
  restriction multiplicities of simple and standard modules from the partition algebra to a certain Young subalgebra.
The  formula  for  calculating ramified branching coefficients in   \TheoremD of this paper  has an exact analogue for the stable Kronecker coefficients, which is proven in the partition algebra context in \cite[Theorem~4.3]{BDO15}.

 \TheoremC of this paper says that the plethysm coefficients are bounded above by their ramified branching coefficient analogues and examines when this bound is sharp.  
The Kronecker coefficients are bounded above by their stable analogues and  
analogous bounds were found by Brion in \cite{BrionStability} and reproved in the context of the partition algebra \cite[Corollary 3.6]{BDO15}.

\subsection{Structure of the paper}\label{subsec:structure}
We intend that this paper will be found readable both by people primarily interested in plethysms
of symmetric functions and by people primarily interested in diagram algebras. We therefore take
some care to collect all the necessary background. 
\begin{itemize}[leftmargin=18pt]
\item In Section 2 we give background on symmetric functions and 
modules for symmetric groups and wreath products. 
\item In Section~3 we give a self-contained introduction to the partition algebra, showing 
how to use diagrams to compute its action on its standard modules in 
Example~\ref{egnocross}.
\item In Section~4 we give a similar self-contained introduction to the ramified partition algebra.
\item In Section~5 we finish the background material with results on Schur--Weyl duality.
\end{itemize}
The proofs of the main theorems occupy Sections~6 to 9, following the outline above: Section~6 proves \TheoremB;
Section 7 proves \TheoremC, showing in particular that 
the simple module $L_r(\varnothing^\beta)$ for the ramified partition algebra
is equal to the standard module $\Delta_r(\varnothing^\beta)$ provided~$m$ and $n$
satisfy the inequalities in \TheoremA; Sections 8 and 9 study the restricted
module $\Delta_r(\varnothing^\beta)\res^{R_r(m,n)}_{P_r(mn)}$ and hence prove \TheoremD and deduce \TheoremA.
Thus, when presented as a series of equations, the proof of \TheoremA is 
\begin{align*}
p&(\beta[n], (m), \kappa[mn]) \\
\quad &= \langle s_{\beta[n]} \circ s_\mu, s_{\kappa[mn]} \rangle  \tag{a} \\
&= \bigl[ (\leftspecht{(m)} \oslash \leftspecht{\beta[n]} ) \Ind_{\W_m \wr \W_n}^{\W_{mn}} : \leftspecht{\kappa[mn]}
\bigr]_{\W_{mn}} \tag{b} \\
&= \bigl[ \mathrm{Hom}_{\mathbb{C}\W_{mn}}\bigl( (\leftspecht{(m)} \oslash \leftspecht{\beta[n]})\Ind_{\W_m \wr \W_n}^{\W_{mn}},
(\mathbb{C}^{mn})^{\otimes r}\bigr) : \mathrm{Hom}_{\mathbb{C}\W_{mn}} \bigl( \leftspecht{\kappa[mn]}, 
(\mathbb{C}^{mn})^{\otimes r} \bigr)
 \bigr]_{P_r(mn)} \tag{c}
\\
&= \bigl[ \mathrm{Hom}_{\mathbb{C}\W_m \wr \W_n} \bigl( \leftspecht{(m)} \oslash \leftspecht{\beta[n]}, 
(\mathbb{C}^{mn})^{\otimes r}
\bigr) \Res_{\W_m \wr \W_n} : L_r(\kappa) \bigr]_{P_r(mn)} \tag{d}\\
% go straight to Delta as mn >= 2r so semisimple
&= \bigl[ L_r(\varnothing^\beta)\Res^{R_r(m,n)}_{P_r(mn)} : L_r(\kappa) \bigr]_{P_r(mn)} \tag{e}\\
&=\bigl[ \Delta_r(\varnothing^\beta)\Res^{R_r(m,n)}_{P_r(mn)} : L_r(\kappa) \bigr]_{P_r(mn)}  \tag{f}\\
&= \bigl[ \DQ\bigl( \Delta_r(\varnothing^\beta) \bigr) : L_r(\kappa) \bigr]_{P_r(mn)} \tag{g} \\
&= \sum  c^\beta_{\beta^\Cee, \ldots, \beta^1}  \Bigl[
\Bigl( \bigl( \bigotimes
\Inf_{\W_{c_i}}^{\W_i \wr \W_{c_i}} 
\rightspecht{\beta^i} \bigr) \Ind^{\W_\Cee} \, \otimes\, \C_{{\rm Stab}(\deltaP)}\Ind^{\W_\Dee} \Bigr)\Ind^{\W_r}_{\W_\Cee \times \W_\Dee} :
\rightspecht{\kappa} \Bigr]_{\W_r} \tag{h} 
\end{align*}
\noindent where the outline argument for each step is as follows:
\begin{itemize}[leftmargin=18pt] %MkW: I really want this to be no further left than the left margin
\item[(a)] Definition of plethysm coefficients; 
\item[(b)] Apply~\eqref{eq:plethysmBranchingCoefficient} to get the equivalent restatement
using left Specht modules;
\item[(c)] Apply the Schur functor for the partition algebra in Corollary~\ref{cor:partitionAlgebraSchurFunctor} 
to get right modules for $P_r(mn)$; the module action is defined
using the action of $P_r(mn)$ on $(\mathbb{C}^{mn})^{\otimes r}$;
\item[(d)] Apply Frobenius reciprocity on the left-hand
module and use Corollary~\ref{cor:partitionAlgebraSchurFunctor} to identify 
the $P_r(mn)$-module $\mathrm{Hom}_{\mathbb{C}\W_{mn}} \bigl( \leftspecht{\kappa[mn]}, (\mathbb{C}^{mn})^{\otimes r} \bigr)$ with
the simple module $L_r(\kappa)$;
\item[(e)] Apply the case of $\alpha = \varnothing$ of Proposition~\ref{SW-simples}
(with the conclusion that the right-hand side equals $p(\beta[n],(m),\alpha)$
being a special case of \TheoremA);
\item[(f)] Apply Theorem~\ref{this!!} %yes, that really is the label
using the hypotheses $m \ge r-|\beta| + [\beta \not=\varnothing]$ and  $n \ge r + \beta_1$ 
(with the conclusion that the right-hand side equals $p(\beta[n],(m),\alpha)$
being a special case of \TheoremC);
\item[(g)] Apply 
Corollary~\ref{explicitlystated};
\item[(h)] Apply Theorem~\ref{thm:DQdecomp}
decomposing the depth quotient into right Specht modules (from which we deduce \TheoremD).
\end{itemize}
%One feature that can be seen from this outline is
%that the proof begins with left Specht modules and finishes, after the passage through
%Schur--Weyl duality to the partition algebra and ramified partition algebra, with right Specht modules.
%BUT NOT Lemma~\ref{semisimple2} equating $L_r(\kappa)$ and $\Delta_r(\kappa)$ (using that $n \ge r + \kappa_1$)
% since we don't have this bound on n, only the weaker $n \ge r + \beta_1$
We end in Sections 10 and 11 with the examples and applications already outlined.

\subsection{Other diagram algebras} %I think they are diagram not diagrammatic
In recent work Orellana, Saliola, Schilling and
Zabrocki have recast the plethysm coefficients  in the context of the {\em party algebra}, a subalgebra of the partition algebra \cite{saliola}.
There does not appear to be  any overlap in our results, but 
the ideas do have a similar diagrammatic flavour.

\subsection*{Acknowledgements}
%We would like to thank: then why not do so? (Littlewood.)
We thank  Christian Ikenmeyer, Stacey Law, Rosa Orellana, Anne Schil\-ling, and Mike Zabrocki for interesting and informative conversations.
We thank  the Oberwolfach workshop 
`Character Theory and Categorification' and the Institute for Mathematical Sciences workshop `Representation Theory, Combinatorics and Geometry' held at the University of Singapore 
and  the Heilbronn Institute for the award of a HIMR Focused Research Grant to run the workshop `New approaches to plethysm: computation and complexity', held at the University of York in September 2023 
---
 all three of which  provided excellent environments for collaboration.  
The first author is grateful to financial support from EPSRC grant EP/V00090X/1.
The third author is grateful to the Okinawa Institute of Science and Technology  for
sponsoring a visit in November 2022 during which some of this work was done.
We would like to thank Eoghan McDowell, Maud De Visscher, and both of the anonymous referees for   their helpful comments on a previous version of this paper.

  \section{ Symmetric groups, wreath products  and symemtric functions }\label{sec1}
  
 \subsection{Symmetric groups}
  We let $\W_{n}$ denote the symmetric group on $n$ letters
 generated by the Coxeter generators $s_i=(i,i+1)$ for $1\leq i <n$.   
     The combinatorics underlying the representation theory of  symmetric      groups, and also of partition algebras,  is based on integer %compositions and 
       partitions.  
%A {\sf composition}    $\lambda $ of $n$, denoted $\lambda \vDash n$, is defined to be a   sequence  $(\lambda_1, \lambda_2, \dots, \lambda_\ell)$ of non-negative integers which  sum to $n$.  
 A  {\sf partition}  of $n$ is defined to be a   sequence  $(\lambda_1, \lambda_2, \dots, \lambda_\ell)$ of strictly positive  {\color{red}\color{black} weakly decreasing integers} which  sum to $n$. We write  $\lambda \vdash n$ and say that the \textsf{length} of $\lambda$,
 denoted $\ell(\lambda)$, is $\ell$.  
There is a unique partition of zero, namely the empty partition $\varnothing$.
 We let 
  $\ParSet(n)$ denote the set of all partitions of $n$ and we let   $\ParSet_{>1}(n)$ denote the subset of those partitions with no singleton parts. (That is, all parts are strictly greater than~1.)
%Let $\succeq$ denote the lexicographic ordering on partitions.
%This was only used once and I prefer <. Changed where used. MkW

  Given $\lambda\in \ParSet(n)$, we define a {\sf tableau} of shape $\lambda$ to be a filling of the nodes  of
the Young diagram of  $\la $ with the numbers
$\{1,\dots , n\}$.
We define a {\sf  standard tableau} to be a tableau  in which    the entries increase along the rows
read left to right and the columns read top to bottom.  % of each component.
 We let $\Std(\lambda)$ denote the set of all standard tableaux of shape $\lambda\in\ParSet(n)$. For $\lambda \vdash n$, let $\stt^\lambda$ denote the  $\lambda$-tableau with the numbers $1, 2, \ldots, n$ entered in increasing order along the rows from left to right and then from top to bottom. For example, 
\[
\stt^{(3,2,1)}=\begin{minipage}{2cm}\gyoung(1;2;3,4;5,6)\, \end{minipage}\hspace*{-12pt}.
\]
We denote by $C(\stt^\lambda)$ the subgroup of $\W_n$ that    preserves the set of entries in each column of $\stt^\lambda$ and by $R(\stt^\lambda)$ the subgroup that preserves the set of entries in each row. Recall (for example from \cite[Chapter 4]{fultonharris}) that the \textsf{Young symmetrizer} $c_\lambda$ is defined by
\begin{equation}\label{eq:youngSymmetrizer} c_\lambda= \Big( \sum_{\rho\in R(\stt^\lambda)} \rho\Big)
\Big(\sum_{\pi \in C(\stt^\lambda)} \sgn(\pi)\pi\Big).\end{equation}
It is well-known that  $c_\lambda$ is a quasi-idempotent (that is $c_\lambda^2$ is equal to $c_\lambda$
up to a non-zero scalar), and, for  partitions  $\lambda, \nu$ of $n$, that
\[ c_\lambda (\C  \W_n) c_\nu = \begin{cases}
\C  c_\lambda & \textrm{if } \lambda=\nu,\\ 0 & \textrm{otherwise.}
\end{cases}\]
The left \textsf{Specht module} labelled by $\lambda$ is the $\C  \W_n$-module 
$\C  \W_n  c_\lambda = \leftspecht {\lambda}$.
Later, we shall  use right modules for symmetric groups to  construct right modules for partition algebras. We 
define  the right Specht  $\C  \W_n$-module labelled by $\lambda$ to be
$\rightspecht{\la}= c_\lambda^\ast \C  \W_n  $ where
\begin{equation}c_\lambda^\ast= 
\Big(\sum_{\pi \in C(\stt^\lambda)} \sgn(\pi)\pi\Big) \Big( \sum_{\rho\in R(\stt^\lambda)} \rho\Big).
\label{eq:dualYoungSymmetrizer}\end{equation}
Observe that $c_\lambda^\ast$ is the image of  $c_\lambda$ under the anti-involution on $\C \W_n$ that sends each group element to its inverse. The Specht modules are a full set of non-isomorphic irreducible
$\C \W_n$-modules.

\subsection{The Littlewood--Richardson rule} %what is so classical about it? I'd rather not. MkW
The Littlewood--Richardson rule is a combinatorial rule
for the restriction of a Specht module to a Young subgroup of the symmetric group.  
Through Schur--Weyl duality, the  rule also computes the decomposition of a tensor product of 
simple representations of $\GL_n(\mathbb{C})$.  
%MkW: I don't think this adds much
%The Littlewood--Richardson rule is the most famous algorithm for decomposing tensor products and has been generalised %in several directions. 

\begin{thm}[The Littlewood--Richardson Rule]\label{thm:LR} 
%removed n_1, n_2 and changed for consistency with sym funcs version MkW
 For $\lambda\vdash m+n$, $\mu\vdash m$ and $\nu \vdash n$,
\begin{align*}
\leftspecht {\lambda} \Res_{\W_{m} \times \W_{n}}^{\W_{m+n}} \cong \bigoplus_{ 
\mu \vdash m, \nu\vdash n}c^\lambda_{\mu,\nu} (\leftspecht{\mu} \boxtimes \leftspecht{\nu})
\end{align*}
where the $c^\lambda_{\mu, \nu}$ are the Littlewood--Richardson coefficients 
as defined in \cite[Section 2.8.13]{jk}.
%(defined below).
\end{thm}

Given a partition $\beta$ and $c_1, \ldots, c_r \in \N_0$ such that  $c_r + \cdots + c_1 = |\beta|$,
we have
\begin{equation}
\label{eq:genLR} \leftspecht{\beta} \Res_{\W_{c_r} \times \cdots \times \W_{c_1}} %\prod_{i=1}^r \W_{c_i}} 
\cong
\bigoplus_{\beta^i \vdash c_i} c^\beta_{\beta^r,\ldots, \beta^1}
\bigl( \leftspecht{\beta^r} \otimes \cdots \otimes \leftspecht{\beta^1}\bigr)
% \bigotimes_{i=1}^r \leftspecht{\beta^i} 
%MkW: writing out since the product has to be taken in a non-standard order, r down to 1
\end{equation}
for some coefficients $c^\beta_{\beta^r,\ldots, \beta^1} \in \NN_0$.
(As a standing convention
$\beta^i \vdash c_i$ in a sum indicates that the sum is over all relevant
sequences of partitions.)
We call these coefficients \textsf{generalized Littlewood--Richardson coefficients};
they may be computed by iterative applications of Theorem~\ref{thm:LR}.
In our examples it works best to order these sequences by \emph{decreasing} index,
hence the order $\beta^r, \ldots, \beta^1$ above; since $c^\lambda_{\mu, \nu} = c^\lambda_{\nu, \mu}$,
this is just a matter of notation.
Later, for example in the proof of Theorem~\ref{thm:DQdecomp},
we use
the equivalent formulation of the rule for right Specht modules. The coefficients are of course the same.
 %MkW: well really it's worse than that because we need \beta^i to be 
%attached to i for use in a wreath product with S_i as the base group but all the diagrams etc have 
%the singleton southerns and then the no-propagating parts at the right, so we have to do it this way

  \subsection{Wreath products and their modules}\label{subsec:wreathProducts}
Let $m,n\in \NN$.  % be such that $mn=r$.   
Following~\cite[Section 4.1]{jk}, we consider
\begin{equation} \W_m \wr \W_n =\bigl\{ (\sigma_1,\sigma_2,\ldots, \sigma_n ; \pi) \mid \sigma_i \in  \Sym _m, i=1,\ldots, n, \pi \in \Sym _n\bigr\}, \label{eq:wreathProductElement} \end{equation}
which we identify with a subgroup of $ \W _{mn}$ via the embedding
\begin{equation}\label{eqn:wreath_embedding}
(\sigma_1,\sigma_2,\ldots, \sigma_n ; \pi ) \mapsto \Bigl( \begin{array}{c} (j-1)m+i \\ (\pi(j)-1)m+\sigma_{\pi(j)}(i) \end{array}\Bigr)_{i=1,\ldots, m \atop j=1\ldots, n}.
\end{equation}
  The  representation theory of $\W_m \wr \W_n$ is well-developed (see \cite[chapter 4]{jk} or \cite{ChuangTan}).  If $\mu \vdash m$ and $\nu \vdash n$ then we can use the irreducible $\C \W_m$-module $\leftspecht{\mu}$ and the irreducible $\C \W_n$-module $\leftspecht{\nu}$ to construct an irreducible $\C \W_m \wr \W_n$-module
\[\leftspecht{\mu} \oslash \leftspecht{\nu} = ( \leftspecht{\mu})^{\otimes n} \otimes \Inf_{S_n}^{S_m \wr S_n} \leftspecht{\nu},\]
  where elements of the distinguished top group $\W_n$ in the wreath
  product $\W_m \wr \W_n$ act on $(\leftspecht{\mu})^{\otimes n}$ by place permutation.
   (The symbol $\oslash$ was introduced  in \cite{ChuangTan}.)
A complete set of irreducible left $\C \W_m \wr \W_n$-modules 
is obtained by inducing suitable tensor products of these modules.

\subsection{Young symmetrizers for wreath product modules}
We shall need 
an alternative construction of a complete set of simple    $\C \W_m \wr \W_n$-modules using Young symmetrizers
that generalises the theory for symmetric groups 
outlined  above 
to wreath  products of symmetric groups. 
Generalizing \cite[Chapter 4]{fultonharris} we  define Young symmetrisers 
for the wreath product $\W_m \wr \W_n$. 
Let $L=|\ParSet(m)|$ and list the $L$ partitions of $m$  in {\color{red}\color{black}decreasing}  lexicographic order as
\[\color{red}\color{black} \mu^1 > \mu^2 > \ldots > \mu^L. \]
An    {\sf $L$-partition}  $\bm\nu=(\nu^{(1)},\dots,\nu^{(L)})$ of $n$  is an $L$-tuple of     partitions  such that
 $|\nu^{(1)}|+\dots+ |\nu^{(L)}|=n$. 
We define  $\ParSet(m,n)$  to be  the set of all 
 $L$-partitions of $n$.
  (This is the labelling set for the simple $\C \W_m \wr \W_n$-modules.)
We shall write elements of $\ParSet(m,n)$ as
\[{\bm \mu} ^{\bm \nu} = (\mu^1,  \mu^2 , \dots, \mu^L)^{ (\nu^1,  \nu^2 , \dots, \nu^L)},\]
remembering that $\bm \mu$ is fixed and is the tuple of all partitions of $m$ in lexicographic order.
For example 
\[  \bigl( (3),  (2,1) , (1^3) \bigr)^{ ((2), \varnothing, (3,1))} \in \ParSet(3,6).\] 
We shall frequently be interested in the case where there is a unique non-empty entry $\nu^j$ of~$\bm \nu$;
in this case, we shall write $(\mu^j)  ^{ \nu^j} $ in place of ${\bm \mu} ^{\bm \nu}$. For example, we write  $(2,1)^{(3,2,1)}$ in place of
 $ \bigl((3),  (2,1) , (1^3)\bigr)^{ ( \varnothing, (3,2,1), \varnothing )}$; this labels the simple 
 $\C \W_3 \wr \W_6$-module $\leftspecht{(2,1)} \oslash \leftspecht{(3,2,1)}$. 

Given $\color{red}\color{black}\bm \nu  \in \ParSet (m,n)$ we 
define
\[
\stt^{\bm \nu}=(\stt^{\nu_1}, \stt^{\nu_2}, \dots, \stt^{\nu_L})
\]
by placing the entries $\{1,\dots, n\}$ 
 in the tableaux $\stt^{\nu_i}$ in 
 increasing order along the rows from left to right, finishing with the bottom row in each tableau,
 working in order of increasing $i$. For example
\[
\stt^{((3,2,1),(2^2,1),(3,1))}=\scalefont{0.9}
\left(\;
\begin{minipage}{1.5cm}
\gyoung(1;2;3,4;5,6) 
\end{minipage}\!, \;
\begin{minipage}{1.15cm}\gyoung(7;8,9;\ten,\eleven) 
\end{minipage}, \;
\begin{minipage}{1.6cm}\gyoung(\elevena;\elevenb;\elevenc,\elevend)\end{minipage}\!
\right).
\] 
We extend the notion of row and column stabilisers in the obvious fashion.  
 For  $\bm \mu ^{\bm \nu}\in \ParSet(m,n)$,  define
\[ c_{\bm \mu ^{\bm \nu}}=
 \displaystyle\sum_{\begin{subarray}c
 \pi \in C(\stt^{\bm \nu}) \\
  \rho \in R(\stt^{\bm \nu}) 
 \end{subarray}
 } \sgn(\pi)
\big(\underbrace{c_{\mu^1}, \ldots, c_{\mu^1}}_{|\nu^1|}  , \ldots, , \underbrace{c_{\mu^L}, \ldots, c_{\mu^L}}_{|\nu^L|} \, ; \,  \rho \pi \big). \]
Here each $\color{red}\color{black}c_{\mu^i}$ should be interpreted as the appropriate sum over elements of $\W_m \wr \W_n$
using the notation of~\eqref{eq:wreathProductElement}. %MkW: added notation note: NOV.
  In particular, if $\mu\vdash m$ and $\nu \vdash n$ then 
\[c_{ \mu ^{ \nu}}=
 \displaystyle\sum_{\begin{subarray}c
 \pi \in C(\stt^{ \nu}) \\
  \rho \in R(\stt^{ \nu}) 
 \end{subarray}
 } \sgn(\pi)
\big(\underbrace{c_{\mu}, \ldots, c_{\mu}}_{n}  \, ; \,  \rho \pi \big).\]
 The following proposition is well-known to experts, but we do not believe it has
  appeared before in print.  
   The proof is technical but follows precisely the   method of \cite[Chapter~4]{fultonharris}.  

\begin{prop} \label{Young}
 For  $\bm \mu ^{\bm \nu}$,
%\bm \alpha ^{\bm \beta}
$\bm \mu ^{\bm \lambda}
\in \ParSet(m,n)$, 
\[ c_{\bm \mu ^{\bm \nu}}(\C  \W_m \wr \W_n)c_{\bm \mu ^{\bm \lambda}}=\begin{cases}
\C  c_{ \bm \mu ^{\bm \lambda}} & \textrm{if } \bm \mu ^{\bm \lambda}
=\bm \mu ^{\bm \nu},\\ 0 & \textrm{otherwise.}
\end{cases}\]
In particular, $c_{\bm \mu ^{\bm \nu}}$ is a quasi-idempotent and  $c_{\bm \mu ^{\bm \nu}} c_{\bm \mu ^{\bm \lambda}}=0$ for $\bm \mu ^{\bm \lambda} \neq  \bm \mu ^{\bm \nu}$.
The set 
\begin{align}\label{ohgood!} 
\bigl\{  \leftspecht{\bm \mu ^{\bm \nu}} =
{  \color{red}\color{black}(\C \W_m \wr \W_n) }c_{\bm \mu ^{\bm \nu}}
\mid\bm \mu ^{\bm \nu}\in \ParSet(m,n)\bigr\} 
\end{align}
provides a full set of pairwise non-isomorphic  {\color{red}\color{black}simple}  left
 $\C  \W_m \wr \W_n$-modules and 
 \[
{ \color{red}\color{black}(\C  \W_m \wr \W_n)} c_{\bm \mu ^{\bm \nu}}\cong  
 \big[ (\leftspecht{\mu^1} \oslash \leftspecht{\nu^1}) \boxtimes \cdots \boxtimes  
 (\leftspecht{\mu^L} \oslash \leftspecht{\nu^L})\big]\Ind_{\W_m \wr \W_{|\bm \nu|}}^{\W_{m} \wr \W_{n}}.\]
\end{prop}

  In particular, if $\mu\vdash m$ and $\nu \vdash n$ then 
${\color{red}\color{black}(\C  \W_m \wr \W_n)}c_{\mu ^{\nu}}\cong  \leftspecht {\mu}  \oslash  \leftspecht{\nu}.$ 
Analogous statements hold for the right  modules $ \rightspecht{\bm \mu ^{\bm \nu}} =
 c_{\bm \mu ^{\bm \nu}}^\ast {\color{red}\color{black}( \C  \W_m \wr \W_n)}$. Here $ c_{\bm \mu ^{\bm \nu}}^\ast$ is obtained from  $c_{\bm \mu ^{\bm \nu}}$ by applying the anti-involution of $\W_m \wr \W_n$ that inverts group elements.

% These are the modules we use  to define plethysm coefficients.

\subsection{Symmetric functions and plethysm}\label{subsec:symmetricFunctionsAndPlethysm}
For background on symmetric functions we refer the reader to \cite[Ch.~7]{StanleyII} or \cite{MacDonald}.
Here we recall that the ring $\Lambda$ of symmetric functions has as an orthonormal $\mathbb{Z}$-basis
the Schur functions $s_\lambda$ indexed by partitions $\lambda$ and is graded by degree: $s_\lambda$
has degree $|\lambda|$.
The plethysm product $s_\nu \circ s_\mu$ may be defined by substituting the monomials
in $s_\mu$ (taken with multiplicities) for the variables in $s_\nu$;
see \cite[A2.6 Definition]{StanleyII} or \cite[Ch.~I (8.2)]{MacDonald}. 
As a small example,
$s_{(2)}(x_1,x_2, x_3, \ldots) = x_1^2 + x_2^2 + x_3^2 +
 \cdots + x_1x_2 + x_1x_3 +x_2x_3 + \cdots$ is the complete homogeneous symmetric function of
 degree $2$, and
 so working with two variables, we have $s_{(2)}(x_1,x_2) = x_1^2 + x_2^2 + x_1x_2$ and
 \begin{align*}
 (s_{(2)} \,\circ\,& s_{(2)})(x_1,x_2)\\ &= s_{(2)}(x_1^2,x_2^2,x_1x_2) \\ &=
 (x_1^2)(x_1^2) + (x_2^2)(x_2^2) + (x_1x_2)(x_1x_2) + (x_1^2)(x_2^2) + (x_1^2)(x_1x_2) + (x_2^2)(x_1x_2)
 \\ &= x_1^4 + x_1^3x_2 + 2x_1^2x_2^2 + x_1x_2^3 + x_2^4 \\ &= s_{(4)}(x_1,x_2) + s_{(2,2)}(x_1,x_2).
 \end{align*}

For our purposes we need two key results
on the characteristic isometry from the character ring of $\W_{d}$
to the degree $d$ component of $\Lambda$.  

\begin{lem}\label{lemma:characteristicIsometry}
For all partitions $\nu$ of $n$ and $\mu$ of $m$ we have 
\begin{thmlist}
\item[\emph{(a)}] the induced product $\bigl( \leftspecht{\mu} \otimes \leftspecht{\nu}\bigr)
 \ind_{\W_{m} 
\times \W_{n}}^{\W_{m+n}}$ corresponds to the ordinary product
$s_\mu s_\nu${\color{red}\color{black};}
\item[\emph{(b)}] 
the $\W_{mn}$-module $\bigl( \leftspecht{\mu} \oslash \leftspecht{\nu} \bigr)
\ind_{\W_m \wr \W_n}^{\W_{mn}}$ 
corresponds to the plethysm product $s_\nu \circ s_\mu$. 
\end{thmlist} 
\end{lem}
\begin{proof}
See \cite[Proposition 7.18.2, A2.2.3]{StanleyII} or 
\cite[Ch.~I (7.3), Appendix A (6.2)]{MacDonald}.
\end{proof}

In Section 10 we also use the analogous versions of (a) and (b) for right Specht modules.

\subsection{Plethysm coefficients}
We are now ready to introduce the combinatorial objects which motivate this paper: the plethysm coefficients.  
\begin{defn}\label{defn:plethysmCoefficient}
Given $\nu \vdash n$ and $\mu \vdash m$ we
define the {\sf plethysm coefficient} $p (\nu, \mu, \lambda)$ for $\lambda \vdash mn$
by
 \[ s_\nu \circ s_\mu = \sum_{\lambda \vdash mn} p(\nu, \mu, \lambda)
 s_\lambda. \]
\end{defn}

Equivalently by~Lemma~\ref{lemma:characteristicIsometry}(b), the plethysm coefficients may be defined
by 
\begin{equation}\label{eq:plethysmBranchingCoefficient}
(\rightspecht{\mu} \oslash \rightspecht{\nu})\Ind_{\W_m\wr \W_n}^{\W_{mn}}
 \cong \bigoplus _{\lambda\vdash mn} p (\nu, \mu, \lambda) \rightspecht{\lambda}.\end{equation}
We invite the  reader to use~\eqref{eq:plethysmBranchingCoefficient}
to show that $p\bigl((2),(2),\lambda\bigr) \not= 0$ only in the two cases $\lambda = (4)$
or $\lambda = (2,2)$ identified in the previous subsection
by considering the $3$-dimensional symmetric group module
$\bigl( \rightspecht{(2)} \oslash \rightspecht{(2)} \bigr) 
\ind_{\W_2 \wr \W_2}^{\W_4}$.

\subsection{Stabiliser subgroups and induction}\label{subsec:stabiliserSubgroups}
The following subgroups are used in \TheoremA and \TheoremD.

\begin{defn}\label{defn:stabiliserSubgroup}
Given a partition $\epsilon$ of $\Dee$
 we define
$\Stab(\deltaP)$ to be the stabiliser in $\W_\Dee$ of a set-partition of $\{1,\ldots, \Dee\}$
into parts of sizes $\epsilon_1, \ldots, \epsilon_{\ell(\epsilon)}$.
\end{defn}

Equivalently, if $\deltaP$ has exactly $e_j$ parts of size $j$ then
\[ \Stab(\deltaP) \cong \W_{1} \wr \W_{e_1} \times \W_2 \wr \W_{e_2}
 \cdots
\times \W_{\Dee} \wr \W_{e_\Dee}.
\]
We note that $\Stab(\deltaP) \le \W_{e_1} \times \W_{2e_2} \times
\cdots \times \W_{\Dee e_\Dee}$ and 
that (by transitivity of induction) the induced module  
\smash{$\mathbb{C}\ind_{\Stab(\deltaP)}^{S_\Dee}$} decomposes as a direct sum of 
Specht modules with coefficients   equal to  products of Littlewood--Richardson and plethysm coefficients. 
For example the induced module $\C \ind_{\W_2 \wr \W_2}^{\W_4} 
= \bigl( \rightspecht{(2)} \oslash \rightspecht{(2)} \bigr) 
\ind_{\W_2 \wr \W_2}^{\W_4}$
is the permutation module of
 $\W_4$ acting on the cosets of $\W_2 \wr
\W_2$, and so it may be written as
 $\mathbb{C}_{\mathrm{Stab}((2,2))}\ind_{\W_2 \wr \W_2}^{\W_4}$.
   
%The symmetric function $G_\deltaP$ defined later in 
%Section~\ref{subsec:symFuncs} corresponds to this induced module under the characteristic
%isometry.
    
  \section{Partition algebras  }\label{sec2}
  The partition   algebra  was originally defined by
   Martin   
    in \cite{marbook}. 
   In this section we recall the definition and basic properties of this algebra, which can be found in
   \cite{mar1}.
     
   \subsection{Set-partitions}\label{convetnions}
  For $r,s\in\mathbb{N}$, we consider the set $\{1,2,\ldots, r, \bar{1},\bar{2}, \ldots, \bar{s}\}$ with the total ordering 
 $$
 1<2<\dots <r <\bar{1}<\bar{2}<\dots <\overline{s}.
 $$  
We refer to a set-partition  of $\{1,2,\ldots, r, \bar{1},\bar{2}, \ldots, \bar{s}\}$ 
 as an $(r,s)$-\textsf{set-partition}.
A subset appearing in a set-partition is called a {\sf block}.
For example,
\begin{equation}\label{set-par}
\Lambda=\bigl\{\{1, 2, 4, \bar{2}, \bar{5}\}, \{3\}, \{5, 6, 7,   \bar{4}, \bar{6}, \bar{7}, \bar{8}\}, \{8, \bar{3}\}, \{\bar{1}\}\bigr\},
\end{equation}
is an $(8,8)$-set-partition   with five  blocks. 

\begin{rmk}\label{usefulconvention}
Let $\Lambda$ be an $(r,s)$-set-partition.  We order the  subsets in $\Lambda = \{\Lambda_1,
\dots, \Lambda_l\}$ by \emph{increasing minima}, so that
\[
 1 =   \min \Lambda_1 < \min \Lambda_2 < \cdots < \min \Lambda_{l-1} <
   \min \Lambda_l
  \leq  \overline{s}.
\]

\end{rmk}

An  $(r,s)$-set-partition, $\Lambda$,
% $\Lambda$,
 can be represented  
 by an $(r,s)$-{\sf partition diagram}, $d_\Lambda$, 
 consisting of  $r$     {\sf northern} and $s$ {\sf southern vertices}. %MkW points are vertices
 We number the northern vertices from left to right by $1,2,\ldots, r$ and the southern vertices
 from left to right by $\bar{1},\bar{2},\ldots, \bar{s}$ and connect two   vertices by an edge
  if they belong to the same block and 
   are 
 adjacent in the total ordering 
  given by restriction of the above ordering to the given block. %; if a block contains a northern and a southern vertex, we connect the minimal such vertices. 
In this manner, we  pick a unique representative from the equivalence class of all diagrams having the same connected components. 
 For example, the diagram $d_\Lambda$ of the $(8,8)$-set-partition $\Lambda$ in~\eqref{set-par}
 is shown in \cref{diargram1}.
\begin{figure}[ht!]
 $$\scalefont{0.8}
 \begin{tikzpicture}[scale=0.45]
    \draw  (0,3.5) arc (180:360:1 and 0);
      \draw  (2,3.5) arc (180:360:2 and .5);
       \draw  (8,3.5) arc (180:360:1 and .5);
              \draw  (10,3.5) arc (180:360:1 and .5);
        \draw  (2,0) arc (180:360:3 and -1.5);
                \draw  (4,0) arc (180:360:1 and -0.5);
        \draw  (6,0) arc (180:360:2 and -1);      
                \draw  (10,0) arc (180:360:1 and -1);

%%%
    \draw  (2,0)to [out=90,in=-90] (6,3.5);   
                                
           \draw  (14,0) to [out=90,in=-90] (12,3.5);                            
   \draw  (4,0) to [out=60,in=-120] (14,3.5);     
                                 
  \foreach \x in {0,2,...,16}
        \fill[white](\x,3.5) circle (7pt);   
      \draw (4,3.5) node {$ \bullet  $};   
                  \draw (2,3.5) node{$ \bullet  $};   
                           \draw (0,3.5) node {$ \bullet  $};   
         \draw (6,3.5) node {$ \bullet  $};   
                  \draw (8,3.5) node {$ \bullet  $};   
                           \draw (10,3.5) node{$ \bullet  $};   
         \draw (12,3.5) node{$ \bullet  $};         \draw (14,3.5) node {$ \bullet  $};   
  \foreach \x in {0,2,...,16}
        \fill[white](\x,0) circle (7pt);   
      \draw (4,0) node {{\color{black} \barthree}}; 
            \draw (2,0) node {{\color{black} \bartwo}}; 
       \draw (0,0) node {\color{white}  \encircle{\color{black} \barone}};      
             \draw (6,0) node {{\color{black} \barfour}}; 
            \draw (8,0) node {{\color{black} \barfive}}; 
       \draw (10,0) node {{\color{black} \barsix}};    
             \draw (12,0) node {{\color{black} \barseven}}; 
            \draw (14,0) node {{\color{black} \bareight}};    
          \end{tikzpicture} 
$$
\caption{The diagram $d_\Lambda$ for $\Lambda$
%=\{\{1, 2, 4, \bar{2}, \bar{5}\}, \{3\}, \{5, 6, 7,   \bar{4}, \bar{6}, \bar{7}, \bar{8}\}, \{8, \bar{3}\}, \{\bar{1}\}\}$.
as in \eqref{set-par}.}
\label{diargram1}
\end{figure}
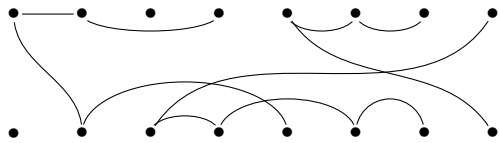

 A block of a set-partition  is called a  {\sf propagating block}  if it contains at least one  northern and at least one southern vertex; a block consisting of either all southern or all northern vertices is called
a {\sf non-propagating block}. In \eqref{set-par}, $\Lambda$ has three propagating blocks and two non-propagating blocks;  both non-propagating blocks are singleton blocks, consisting of a single vertex.

To each $(r,s)$-set-partition diagram $d_\Lambda$, we have an associated
 permutation $\pi_\Lambda$ given 
by 
% first deleting any node 
%  which does not belong to a propagating block, then 
%deleting any node which is not minimal in the   intersection of its block with 
%  $ \{1,\dots,r\}$ or $\{\overline1,\dots,\overline s\}$, 
%search
{\color{red}\color{black}   deleting all nodes which are not incidental to an edge connecting a northern vertex to the southern vertex}, and then interpreting
  the diagram as a permutation. For instance,
  for the diagram $d_\Lambda$ above, we have 
  $\pi_\Lambda=(2,3) \in \W_3$.

   We now consider a parameter  $\delta \in \C $.  We define the product 
   $d_{\Lambda}  d_{\Gamma}$ of two $(r,r)$-partition diagrams $d_{\Lambda}$
and $d_{\Gamma}$ by concatenating $d_{\Lambda}$ above $d_{\Gamma}$, and identifying
the southern vertices of $d_{\Lambda}$ with the northern vertices of $d_{\Gamma}$.   
If there are $t$ connected components consisting only of  middle vertices, 
then the product is $ \delta ^t$ times the 
$(r,r)$-partition diagram equivalent to the 
diagram with the middle components removed. 

For example, take $d_{\Lambda}$ to be the diagram of \cref{diargram1} and  $d_{\Gamma}$  to be the diagram of 
\[ \Gamma= \{\{1\}, \{2, \bar{1}, \bar{2}\}, \{3,\bar{4}\}, \{4,\bar{3}\}, \{5,\bar{5},\bar{6}\}, \{6\}, \{7,8,\bar{7}, \bar{8}\}\}.\] Then
$d_{\Lambda}d_{\Gamma}$ equals $\delta$ times the diagram of 
$\{\{1,2,4 ,\bar{1}, \bar{2}, \bar{5}, \bar{6}\}, \{3\}, \{5,6,7,8,\bar{3},\bar{4},\bar{7}, 
\bar{8}\}\}$, as shown in \cref{multiplication}.

 \begin{figure}[ht!]
\[\scalefont{0.8}
\begin{minipage}{6.5cm}
 \begin{tikzpicture}[scale=0.4]
    \draw  (0,3.5) arc (180:360:1 and 0.5);
      \draw  (2,3.5) arc (180:360:2 and .5);
       \draw  (8,3.5) arc (180:360:1 and .5);
              \draw  (10,3.5) arc (180:360:1 and .5);
        \draw  (2,0) arc (180:360:3 and -1.5);
                \draw  (4,0) arc (180:360:1 and -0.5);
        \draw  (6,0) arc (180:360:2 and -1);      
                \draw  (10,0) arc (180:360:1 and -1);  
%                 \draw  (2,0)to [out=90,in=-90] (0,3.5);   
%
%
%                                
%           \draw  (14,0) to [out=120,in=-60] (8,3.5);                            
%   \draw  (4,0) to [out=60,in=-120] (14,3.5);                                   
    \draw  (2,0)to [out=90,in=-90] (6,3.5);                                 
           \draw  (14,0) to [out=90,in=-90] (12,3.5);                            
   \draw  (4,0) to [out=60,in=-120] (14,3.5);     
     \foreach \x in {0,2,...,16}
        \fill[white](\x,3.5) circle (7pt);   
      \draw (4,3.5) node {$ \bullet  $};   
                  \draw (2,3.5) node{$ \bullet  $};   
                           \draw (0,3.5) node {$ \bullet  $};   
         \draw (6,3.5) node {$ \bullet  $};   
                  \draw (8,3.5) node {$ \bullet  $};   
                           \draw (10,3.5) node{$ \bullet  $};   
         \draw (12,3.5) node{$ \bullet  $};         \draw (14,3.5) node {$ \bullet  $};   
  \foreach \x in {0,2,...,16}
        \fill[white](\x,0) circle (7pt);   
      \draw (4,0) node {{\color{black} \barthree}}; 
            \draw (2,0) node {{\color{black} \bartwo}}; 
       \draw (0,0) node {\color{white}  \encircle{\color{black} \barone}};      
             \draw (6,0) node {{\color{black} \barfour}}; 
            \draw (8,0) node {{\color{black} \barfive}}; 
       \draw (10,0) node {{\color{black} \barsix}};    
             \draw (12,0) node {{\color{black} \barseven}}; 
            \draw (14,0) node {{\color{black} \bareight}};

       \draw  (12,0) arc (180:360:1 and .5);

%        \draw  (2,-3.75) arc (180:360:3 and -1.5);
%                \draw  (4,-3.75) arc (180:360:1 and -0.5);
%        \draw  (6,-3.75) arc (180:360:2 and -1);      
                \draw  (8,-3.75) arc (180:360:1 and -0.5);  
                \draw  (0,-3.75) arc (180:360:1 and -0.5);  
                \draw  (12,-3.75) arc (180:360:1 and -0.5);  
%                 \draw  (2,-3.75)to [out=90,in=-90] (0,0);   
                                 
           \draw  (0,-3.75) to [out=90,in=-90] (2,0);                            
%   \draw  (4,-3.75) to [out=60,in=-120] (14,0);                                   

           \draw  (4,-3.75) -- (6,0);                            
           \draw  (6,-3.75) -- (4,0);                            

           \draw  (12,-3.75) -- (12,0);           
           \draw  (8,-3.75) -- (8,0);                            

  \foreach \x in {0,2,...,16}
        \fill[white](\x,0) circle (7pt);   
      \draw (4,0) node {$ \bullet  $};   
                  \draw (2,0) node{$ \bullet  $};   
                           \draw (0,0) node {$ \bullet  $};   
         \draw (6,0) node {$ \bullet  $};   
                  \draw (8,0) node {$ \bullet  $};   
                           \draw (10,0) node{$ \bullet  $};   
         \draw (12,0) node{$ \bullet  $};         \draw (14,0) node {$ \bullet  $};   
  \foreach \x in {0,2,...,16}
        \fill[white](\x,-3.75) circle (7pt);   
      \draw (4,-3.75) node {{\color{black} \barthree}}; 
            \draw (2,-3.75) node {{\color{black} \bartwo}}; 
       \draw (0,-3.75) node {\color{white}  {\color{black} \barone}};      
             \draw (6,-3.75) node {{\color{black} \barfour}}; 
            \draw (8,-3.75) node {{\color{black} \barfive}}; 
       \draw (10,-3.75) node {{\color{black} \barsix}};    
             \draw (12,-3.75) node {{\color{black} \barseven}}; 
            \draw (14,-3.75) node {{\color{black} \bareight}};    
          \end{tikzpicture} 
          \end{minipage}          
          = \; \; \delta \;
          \begin{minipage}{6.5cm} \begin{tikzpicture}[scale=0.4]
    \draw  (0,3.5) arc (180:360:1 and 0.5);
      \draw  (2,3.5) arc (180:360:2 and .5);
       \draw  (8,3.5) arc (180:360:1 and .5);
              \draw  (10,3.5) arc (180:360:1 and .5);
              \draw  (12,3.5) arc (180:360:1 and .5);              
              
        \draw  (2,0) arc (180:360:3 and -1.5);
                \draw  (4,0) arc (180:360:1 and -0.5);
                 \draw  (6,0) arc (180:360:3 and -1);                  \draw  (12,0) arc (180:360:1 and -0.5);  
%                 \draw  (0,0)to [out=90,in=-90] (0,3.5);   
%            \draw  (4,0) to [out=90,in=-90] (8,3.5);                            

                 \draw  (0,0)to [out=90,in=-90] (6,3.5);   
            \draw  (4,0) to [out=70,in=-100] (14,3.5);   

     \draw  (8,0) arc (180:360:1 and -0.5);  
                 \draw  (0,0) arc (180:360:1 and -1);

  \foreach \x in {0,2,...,16}
        \fill[white](\x,3.5) circle (7pt);   
      \draw (4,3.5) node {$ \bullet  $};   
                  \draw (2,3.5) node{$ \bullet  $};   
                           \draw (0,3.5) node {$ \bullet  $};   
         \draw (6,3.5) node {$ \bullet  $};   
                  \draw (8,3.5) node {$ \bullet  $};   
                           \draw (10,3.5) node{$ \bullet  $};   
         \draw (12,3.5) node{$ \bullet  $};         \draw (14,3.5) node {$ \bullet  $};   
  \foreach \x in {0,2,...,16}
        \fill[white](\x,0) circle (7pt);   
      \draw (4,0) node {{\color{black} \barthree}}; 
            \draw (2,0) node {{\color{black} \bartwo}}; 
       \draw (0,0) node {\color{white}  {\color{black} \barone}};      
             \draw (6,0) node {{\color{black} \barfour}}; 
            \draw (8,0) node {{\color{black} \barfive}}; 
       \draw (10,0) node {{\color{black} \barsix}};    
             \draw (12,0) node {{\color{black} \barseven}}; 
            \draw (14,0) node {{\color{black} \bareight}};    
          \end{tikzpicture} \end{minipage}
\]
\caption{An example of a product  in $P_8(\delta)$. }
\label{multiplication}
\end{figure}
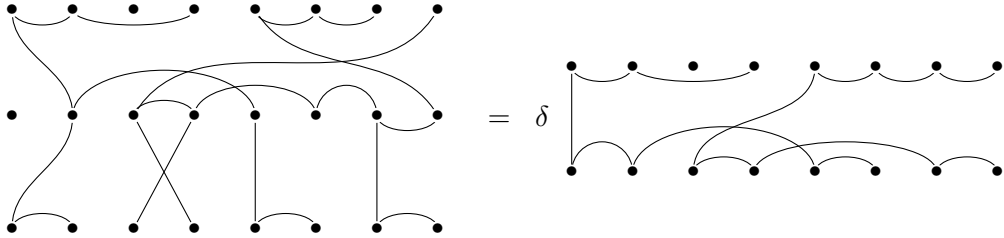

We let $P_r(\delta)$ denote the complex vector space with basis given by  all set-partitions of  $\{1,2,\ldots, r, $ $\bar{1},\bar{2}, \ldots, \bar{r}\} $  
and with multiplication given by  linearly extending  the multiplication of diagrams.   
Then $P_r(\delta)$ is an associative $\C $-algebra, known as the  {\sf partition  algebra}.
  The $(r,r)$-set-partitions with exactly $r$ propagating blocks are just the permutations, 
  hence
  $\C \W_r$ is a subalgebra of $P_r(\delta) $.

The partition algebra is generated by the usual Coxeter generators of $\W_r$ together with the two diagrams ${\sf p}_1=d_\Lambda$ for $\Lambda=
\{\{1\},\{\overline{1}\}\} \cup \{\{k,\overline{k}\}
\mid k>1 \}$ and  ${\sf p}_{1,2}=d_\Lambda$ for
$\Lambda=\{\{1,2,\overline{1},\overline{2}\}\} \,\cup\, \{\{k,\overline{k}\}
\mid 2 < k \le r \}$ depicted  in 
\cref{generators}.
Set
${\sf p}_i = s_{i-1}   \dots s_2 s_1{\sf p}_1 s_1 s_2 \dots  s_{i-1}$ for $i\geq 2$.

\begin{figure}[h!]
  $$ 
 \scalefont{0.8}
\begin{minipage}{60mm}\begin{tikzpicture}[scale=0.45]
   \draw  (2,3.5) --(2,0);  \draw  (4,3.5) --(4,0);
    \draw  (8,3.5) --(8,0); 
  \foreach \x in {0,2,4,8}
        \fill[white](\x,3.5) circle (10pt);   
      \draw (4,3.5) node {$ \color{white}\encircle{\color{black} \three}  $};   
                  \draw (2,3.5) node {$\color{white}\encircle{\color{black} \two}$}; 
                           \draw (0,3.5) node {$\color{white}\encircle{\color{black} \one}$}; 
     \draw (6,0) node {$\dots $};          \draw (6,3.5) node {$\dots $};                             \draw (8,3.5) node {$\color{white}\encircle{\color{black} $\bullet$}$};    
   \foreach \x in {0,2,4,8}
        \fill[white](\x,0) circle (10pt);   
%                    \draw (14,0) node {\color{white}\encircle{\color{black} $\bareight$}};    
      \draw (4,0) node {  \color{white}\encircle{\color{black} \barthree}};   
                 \draw (2,0) node { \color{white}\encircle{\color{black} \bartwo}}; 
                           \draw (0,0) node {\color{white}\encircle{\color{black} \barone}}; 
                            \draw (8,0) node {\color{white}\encircle{\color{black} $\bullet$}};    
      \end{tikzpicture}
      \end{minipage}
    \begin{minipage}{50mm}  \begin{tikzpicture}[scale=0.45,yscale=-1]
   \draw  (2,3.5) --(2,0);  \draw  (4,3.5) --(4,0);
    \draw  (8,3.5) --(8,0); 
%%%%%%
%%%%%
        \draw  (-2,3.5) --(-0,0); 
        \draw  (-2,3.5) arc (180:360:1 and 0.5);
                \draw  (-2,0) arc (180:360:1 and -0.5);
  \foreach \x in {-2,0,2,4,8}
        \fill[white](\x,3.5) circle (10pt);   
      \draw (4,3.5) node {$ \color{white}\encircle{\color{black} \four}  $};   
                  \draw (2,3.5) node {$\color{white}\encircle{\color{black} \three}$}; 
                           \draw (-2,3.5) node {$\color{white}\encircle{\color{black} \one}$};                            \draw (0,3.5) node {$\color{white}\encircle{\color{black} \two}$}; 
     \draw (6,0) node {$\dots $};          \draw (6,3.5) node {$\dots $};                             \draw (8,3.5) node {$\color{white}\encircle{\color{black} $\bullet$}$};    
   \foreach \x in {-2,0,2,4,8}
        \fill[white](\x,0) circle (10pt);   
      \draw (4,0) node {\color{white}\encircle{\color{black} \barfour  }};   
                 \draw (2,0) node {\color{white}\encircle{\color{black} \barthree}}; 
                           \draw (-2,0) node {\color{white}\encircle{\color{black} \barone}}; 
                           \draw (0,0) node {\color{white}\encircle{\color{black} \bartwo}}; 
                            \draw (8,0) node {\color{white}\encircle{\color{black}  $\bullet$}};    
      \end{tikzpicture}      \end{minipage}
 $$
\caption{The non-Coxeter generators ${\sf p}_1$ and  ${\sf p}_{1,2}$  of $P_r(\delta)$}
\label{generators}
\end{figure}
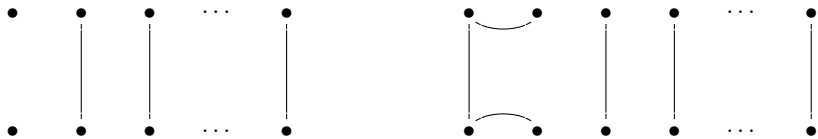
   
 \subsection{Horizontal concatenation}\label{hoz}
 Given  an $(r_1,s_1)$-set-partition diagram 
$d_{\Lambda_1}$ and an $(r_2,s_2)$-set-partition diagram  $d_{\Lambda_2}$ we define their 
\textsf{horizontal concatenation}, $d_{\Lambda_1}\circledast d_{\Lambda_2}$,  to be the 
  $(r_1+r_2,s_1+s_2)$-set-partition obtained by  placing the diagram $d_{\Lambda_1}$ to the left of $d_{\Lambda_2}$ and relabelling the vertices (that is, the $k$th northern vertex in $d_{\Lambda_2}$ is relabelled by $r_1+k$ and 
 the $k$th southern vertex in $d_{\Lambda_2}$ is relabelled by $s_1+k$).  This is illustrated in \cref{concate-hoz}.

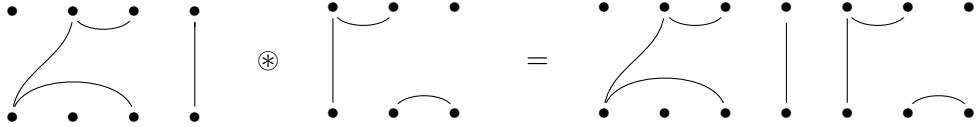
\begin{figure}[ht!]
\[ \scalefont{0.8}
 \begin{minipage}{3cm}
 \begin{tikzpicture}[scale=0.4]
    \draw  (2,3.5) to [out=-90,in=-90] (4,3.5);  
   \draw  (0,0) to [out=90,in=-90] (4,3.5);   
   \draw  (0,0) to [out=90,in=90] (4,0);  
      \draw  (6,0) to [out=90,in=90] (6,3);  
  \foreach \x in {0,2,4,6}
     {   \fill[white](\x,3.5) circle (10pt);   \fill[white](\x,0) circle (10pt);   
      \draw (\x,3.5) node {$ \bullet  $};   
                  \draw (\x,0) node{$ \bullet  $};   
}               
     \end{tikzpicture}\end{minipage} \;\;\; \circledast \; \; \;\;
      \begin{minipage}{2.3cm}
 \begin{tikzpicture}[scale=0.4] 
   \draw  (0,3.5) to [out=-90,in=-90] (2,3.5);  
   \draw  (0,0) to [out=90,in=-90] (2,3.5);   
   \draw  (2,0) to [out=90,in=90] (4,0);  
 
  \foreach \x in {0,2,4}
     {   \fill[white](\x,3.5) circle (10pt);   \fill[white](\x,0) circle (10pt);   
      \draw (\x,3.5) node {$ \bullet  $};   
                  \draw (\x,0) node{$ \bullet  $};   
}                      
     \end{tikzpicture}\end{minipage} \;\;\; =\; \; \;\;
      \begin{minipage}{6.2cm}
 \begin{tikzpicture}[scale=0.4]
  \draw  (2,3.5) to [out=-90,in=-90] (4,3.5);  
   \draw  (0,0) to [out=90,in=-90] (4,3.5);   
   \draw  (0,0) to [out=90,in=90] (4,0);  
   
      \draw  (6,0) -- (6,3);  
   
    \draw  (8+0,3.5) to [out=-90,in=-90] (8+2,3.5);  
   \draw  (8+0,0) to [out=90,in=-90] (10,3.5);   
   \draw  (8+2,0) to [out=90,in=90] (8+4,0);  

  \foreach \x in {0,2,4,6,8,10,12}
     {   \fill[white](\x,3.5) circle (10pt);   \fill[white](\x,0) circle (10pt);   
      \draw (\x,3.5) node {$ \bullet  $};   
                  \draw (\x,0) node{$ \bullet  $};   
}              
       \end{tikzpicture}\end{minipage} \]
     
\caption{A horizontal concatenation.}
\label{concate-hoz}
\end{figure}
 
\subsection{A filtration of the partition algebra}\label{subsec:filtrationPartitionAlgebra}
Recall that a block of a set-partition 
is propagating if the block contains both northern and southern vertices. %We've already defined this
%For example  $$\Lambda=\{\{1, 2, 4, \bar{2}, \bar{5}\}, \{3\}, \{5, 6, 7, \bar{3}, \bar{4}, \bar{6}, \bar{7}\}, \{8, \bar{8}\}, \{\bar{1}\}\} $$ has three propagating blocks.
 %  
It is clear that multiplication in $P_r(\delta)$ cannot increase the number of propagating blocks.  
   This leads to a filtration of the algebra $P_r(\delta)$ by the number of propagating blocks.
  Supposing that  $\delta  \ne 0$, there are idempotents %pseudo-idempotents
   ${e}_l= {\color{red}\color{black} \delta ^{ l-r}}
   \mathsf{p}_r\mathsf{p}_{r-1}\ldots \mathsf{p}_{l+1}$ for $0\leq l \leq r$, as depicted in \cref{ghgh221}.

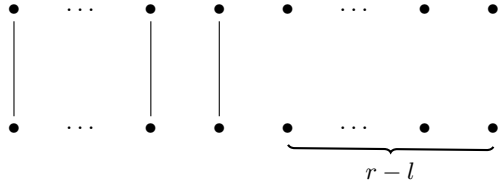
\begin{figure}[h!]
 \[ 
 \scalefont{0.8}
 \begin{tikzpicture}[yscale=0.45,xscale=-0.45]
   \draw  (2,3.5) --(2,0);  \draw  (4,3.5) --(4,0);
    \draw  (8,3.5) --(8,0); 
  \foreach \x in {0,2,4,8}
        \fill[white](\x,3.5) circle (10pt);   
      \draw (4,3.5) node{$ \bullet  $};   %node {$ \color{white}\encircle{\color{black} \three}  $};   
                  \draw (2,3.5) node{$ \bullet  $};   %node {$\color{white}\encircle{\color{black} \two}$}; 
                           \draw (0,3.5) node{$ \bullet  $};   %node {$\color{white}\encircle{\color{black} \one}$}; 
     \draw (6,0) node {$\dots $};          \draw (6,3.5) node {$\dots $};                             \draw (8,3.5) node{$ \bullet  $};   %node {$\color{white}\encircle{\color{black} $\bullet$}$};    
   \foreach \x in {0,2,4,8}
        \fill[white](\x,0) circle (10pt);   
%                    \draw (14,0) node {\color{white}\encircle{\color{black} $\bareight$}};    
      \draw (4,0) node{$ \bullet  $};   %node {  \color{white}{\color{black} \barthree}};   
                 \draw (2,0) node{$ \bullet  $};   %node { \color{white}{\color{black} \bartwo}}; 
                           \draw (0,0) node{$ \bullet  $};   %node {\color{white}{\color{black} \barone}}; 
                            \draw (8,0) node{$ \bullet  $};   %node {\color{white}{\color{black} $\bullet$}};    

                       \draw (-4,0) node{$ \bullet  $};  
                       \draw (-4,3.5) node{$ \bullet  $};  

                       \draw (-6,0) node{$ \bullet  $};  
                       \draw (-6,3.5) node{$ \bullet  $};  

     \draw (-2,0) node {$\dots $};          \draw (-2,3.5) node {$\dots $};     

\draw [
    thick,
    decoration={
        brace,
        mirror,
        raise=-0.5cm
    },
    decorate
]   (0,-1.6) --(-6,-1.66) ;
%node [pos=0.5,anchor=north]
\draw(-3,-1.3) node {$r-l$};

      \end{tikzpicture}
\]  
     \vspace{-0.6cm}
\caption{The idempotent $e_l$ is defined to be $\delta ^{{\color{red}\color{black}l-r}}$ times the diagram above having $\ell$ propagating
blocks.}
\label{ghgh221}
\end{figure}

We have
\begin{equation}\label{eq:Pfiltration} \{0\} \subset
 P_r(\delta) {e}_0 P_r(\delta) \subset P_r(\delta) {e}_{1}P_r(\delta) \subset \ldots \subset  P_r(\delta){e}_{r-1}P_r(\delta) \subset P_r(\delta).		\end{equation}
Set $J_l = P_r(\delta) {e}_{l}P_r(\delta)$  for $0\leq l \leq r$.  
 The ideal $J_l$ is spanned by all $(r,r)$-partition-diagrams having at most $l$ propagating blocks.
It is easy to see that
\begin{equation}\label{bob}
{e}_{r-1}P_r(\delta) {e}_{r-1} \cong  P_{r-1}(\delta),  \end{equation}
and that this generalises to $P_{l}(\delta)\cong {e}_{l}P_r(\delta) {e}_{l}$ for $0\leq l\leq r$. 
Moreover, since $P_r(\delta){e}_{r-1}P_r(\delta)$ is the span of all $(r,r)$-partition diagrams with at most $r-1$ propagating blocks, 
\begin{equation}\label{bob2} \frac{P_r(\delta)}{P_r(\delta) {e}_{r-1}P_r(\delta)}\cong \mathbb{C}\W_r\end{equation}
where the left-hand side is $J_r/J_{r-1}$.
%Why is it called bob2? 
 %We now give an explicit description of the cell modules which follows directly from (\ref{cell}). 

\subsection{Standard and simple modules for the partition algebra}
 \label{subsec:standardSimpleModulesPartitionAlgebra}
We use this filtration to construct the standard modules for the partition algebra.
 Since we later use the
 commuting left action of the symmetric group and right action of the partition algebra on tensor space (see \cref{SWsec}), we require right $P_r(\delta)$-modules.
   
 We set $V_r(k)= e_k (J_k / J_{k-1})$.  Observe that $V_r(k)$ has a basis given by all $(r,r)$-partition diagrams with exactly $k$ propagating blocks such that $\{j\}$ is a singleton part for all $j\geq k+1$.  
Thus the corresponding diagrams have no edges from the north vertices $k+1,\ldots, r$.
We identify such diagrams with the  $(k,r)$-partition diagrams having precisely $k$ propagating blocks. 
For instance, two different examples of $(3,5)$-partition diagrams with 3 propagating blocks 
appear as bottom halves of the concatenated diagrams in \cref{delta21}.

Since $e_kP_r(\delta)e_k \cong P_k(\delta)$ and $P_k(\delta)$ has $\C\W_k$ as a quotient by~\eqref{bob2}, 
%MkW: note it's the group algebra that is a quotient, not the group
 $V_r(k)$ has the  structure of a $(\C\W_{k}, P_r(\delta))$-bimodule.
We remark that
% Moreover,
%\begin{align}\label{otimes-isom-P} 
$V_r(k) \otimes _{\W_k} 
 {\C \W_k} 
\otimes _{\W_k} 			V_r(k) \cong J_k / J_{k-1}$
%\end{align}
where the isomorphism is defined 
 on diagrams 
 by $u \otimes \sigma \otimes v \mapsto u^\ast \sigma v$,
  where $u^\ast$ denotes the $(r,k)$-partition diagram obtained from the $(k,r)$-partition diagram $u$ by horizontal reflection; this makes concrete the filtration in~\eqref{eq:Pfiltration}.
From~\eqref{bob2}, we see that any right 
$\mathbb{C}\W_k$-module can be inflated %MkW: removed \sf as I don't think this counts as a definition
 to a $P_k(\delta)$-module. 
The simple right $\mathbb{C}\W_k$-modules are the right Specht modules $\rightspecht{\kappa}$
for $\kappa \vdash k$.
 Thus by induction using~\eqref{bob} and~\eqref{bob2}
 we find that the simple $P_r(\delta)$-modules are indexed by $\ParSet(\leq r) =\bigcup_{0\leq i\leq r}\ParSet(i)$.
For any $\kappa \vdash k$, we define the {\sf standard} (right) $P_r(\delta)$-module, $\Delta_{r     }(\kappa)$, by
\begin{equation}\label{eq:DeltaPartitionAlgebra} \Delta_{r  
}(\kappa) \cong \rightspecht {\kappa} \otimes_{\W_{k}} V_r(k) \end{equation}
where    
the action of $(r,r)$-diagrams $d \in P_r(\delta)$ is given as follows. Let $v$ be a $(k,r)$-partition diagram in $V_r(k)$ and let
 $x\in \rightspecht {\kappa}$. Concatenate  $v$ above $d$ to get $\delta ^t v'$ for some $(k,r)$-partition diagram~$v'$ and some non-negative integer $t$. If $v'$ has fewer than $k$ propagating blocks then we set
$(x\otimes v)d=0$.
Otherwise we set $(x\otimes v)d=\delta ^t x\otimes v'$. 

By \cite[8.4]{james}, the 
$\C \W_k$-Specht module $ \rightspecht{\kappa}$ has basis $\{  c_{\kappa}^\ast \sigma \mid   t_\kappa \sigma \in \Std(\kappa)\}$, where $c_\kappa^\ast$ is the dual Young symmetrizer defined 
in~\eqref{eq:dualYoungSymmetrizer}.
By~\eqref{eq:DeltaPartitionAlgebra}, we have $v \otimes \tau d = v \tau \otimes d$ for any $\tau \in \W_k$,
and so we need to multiply the basis elements of $\rightspecht{\kappa}$ only by diagrams $d$ such 
that~$\pi_d$ is the identity permutation.
After this reduction, we obtain the basis
\begin{align}\label{actual-d-basis}
 \Bigl\{  c_{\kappa}^\ast \sigma d_{\Lambda }
\, \Bigl|  \,
 \begin{minipage}{3.2in}$\stt^\kappa \sigma \in \Std(\kappa)$, 
 $\Lambda$ is a $(k,r)$-set-partition with $k$ propagating blocks, 
  $\pi _\Lambda=1 \in \W_k$\end{minipage} \Bigr\} 
 \end{align}  
of $\Delta_r(\kappa)$, with the action of $P_r(\delta)$ as specified after~\eqref{eq:DeltaPartitionAlgebra}.
%MkW added details for why permutation is 1.
  
In particular, taking $k=r$, we have
\begin{equation}\label{annihilate2}
\Delta_{r %, \delta
}(\kappa)\cong \rightspecht{\kappa}\otimes_{\W_r}V_r( r) \cong %MkW: I think this is \cong not =; also moved
%after (3.6).
 \rightspecht{\kappa},
\end{equation}
where the right-hand side is viewed as a $P_r(\delta)$-module by inflation using  (\ref{bob2}).

\begin{eg}\label{egnocross} 
 Three distinct basis elements of $\Delta_5\bigl((2,1)\bigr)$ are depicted in \cref{delta21}.
The middle diagram shows ${\color{red}\color{black}c_{(2,1)}^*} (2,3) d_\Lambda$ where $\Lambda = \bigl\{ \{1, \overline{1}, \overline{3} \},
\{ 2, \overline{2}, \overline{5} \}, \{3, \overline{4} \bigr\}$. 
Consider the two diagrams~$d_\Gamma$ and $d_{\Gamma'}$ shown in \cref{delta21gammas}.
The product
${\color{red}\color{black}c_{(2,1)}^*} (2,3) d_\Lambda d_\Gamma$ is non-zero as $d_\Lambda d_\Gamma$ has~$3$ propagating blocks; it
is computed in \cref{delta21product}
 using the action described after~\eqref{eq:DeltaPartitionAlgebra}, with a further `untwisting' step to
 obtain a canonical basis element from~\eqref{actual-d-basis}.
 On the other hand since $d_\Lambda d_{\Gamma'}$ has only~$2$ 
 propagating blocks, we have ${\color{red}\color{black}c_{(2,1)}^*} (2,3) d_\Lambda d_{\Gamma'} = 0$.

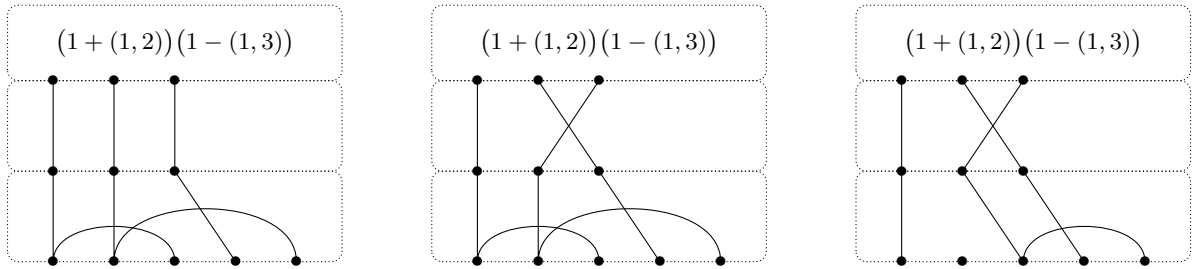
\begin{figure}[ht!]$$ 
 \begin{tikzpicture}[xscale=0.4,yscale=-0.4]
  \scalefont{0.8}
       \foreach \x in {0,1,2,3,4,5}
     {  
      \path(\x*2,0) coordinate (down\x);  
   }

       \foreach \x in {0,1,2,3,4,5,6}
    {   \path(\x*2,3) coordinate (up\x);  
    }

        \foreach \x in {0,2,4,6,8}
     {   \fill[white](\x,6) circle (4pt);   
                        \draw (\x,6) node {$\bullet$}; 
                  }

        \foreach \x in {0,2,4}%,6,8}
     {   \fill[white](\x,3) circle (4pt);   
                        \draw (\x,3) node {$\bullet$}; 
                  }
                
      \foreach \x in {0,2,4}%,6,8}
     {   \fill[white](\x,0) circle (4pt);   
                        \draw (\x,0) node {$\bullet$}; 
                  }

\draw[rounded corners ,densely dotted ] (-1.5,-0)  rectangle( 9.5,-2.5)   node[midway] 
%{$\bigl(1+(1,2)\bigr)\bigl(1-(1,3)\bigr)$}; ;  %{$(1-(12))(1+(23))$};  
{$\color{red}\color{black} \bigl(1 - (1,3) \bigr)\bigl( 1 + (1,2) \bigr)$};

\draw[rounded corners ,densely dotted ] (-1.5,-0)  rectangle( 9.5,3);

\draw[rounded corners ,densely dotted ] (-1.5,6)  rectangle( 9.5,3);

%\draw(8,0)--(8,3)         ;      \draw(6,0)--(6,3)         ;      
                        
\draw(0,0)--(0,3)         ;      
\draw(2,0)--(2,3)         ;         \draw(4,0)--(4,3)         ;    
\draw(0,3)--(0,6) to [out=-90,in=-90] (4,6)        ;               
\draw(2,3)--(2,6) to [out=-90,in=-90] (8,6)        ;                             
\draw(4,3)--(6,6) ;
         \end{tikzpicture}  \ \ 
          \begin{tikzpicture}[xscale=0.4,yscale=-0.4]
   \scalefont{0.8}
       \foreach \x in {0,1,2,3,4,5}
     {  
      \path(\x*2,0) coordinate (down\x);  
   }

       \foreach \x in {0,1,2,3,4,5,6}
    {   \path(\x*2,3) coordinate (up\x);  
    }

        \foreach \x in {0,2,4,6,8}
     {   \fill[white](\x,6) circle (4pt);   
                        \draw (\x,6) node {$\bullet$}; 
                  }

        \foreach \x in {0,2,4}%,6,8}
     {   \fill[white](\x,3) circle (4pt);   
                        \draw (\x,3) node {$\bullet$}; 
                  }
                
      \foreach \x in {0,2,4}%,6,8}
     {   \fill[white](\x,0) circle (4pt);   
                        \draw (\x,0) node {$\bullet$}; 
                  }

\draw[rounded corners ,densely dotted ] (-1.5,-0)  rectangle( 9.5,-2.5)   node[midway] 
%{$\bigl(1+(1,2)\bigr)\bigl(1-(1,3)\bigr)$};  
%{$(1-(12))(1+(23))$};  
{$\color{red}\color{black} \bigl(1 - (1,3) \bigr)\bigl( 1 + (1,2) \bigr)$};

\draw[rounded corners ,densely dotted ] (-1.5,-0)  rectangle( 9.5,3);

\draw[rounded corners ,densely dotted ] (-1.5,6)  rectangle( 9.5,3);

%\draw(8,0)--(8,3)         ;      \draw(6,0)--(6,3)         ;      
                        
\draw(0,0)--(0,3)         ;      
\draw(2,0)--(4,3)         ;         \draw(4,0)--(2,3)         ;    
\draw(0,3)--(0,6) to [out=-90,in=-90] (4,6)        ;               
\draw(2,3)--(2,6) to [out=-90,in=-90] (8,6)        ;                             
\draw(4,3)--(6,6) ;
         \end{tikzpicture} 
\ \ 
          \begin{tikzpicture}[xscale=0.4,yscale=-0.4]
   \scalefont{0.8}
       \foreach \x in {0,1,2,3,4,5}
     {  
      \path(\x*2,0) coordinate (down\x);  
   }

       \foreach \x in {0,1,2,3,4,5,6}
    {   \path(\x*2,3) coordinate (up\x);  
    }

        \foreach \x in {0,2,4,6,8}
     {   \fill[white](\x,6) circle (4pt);   
                        \draw (\x,6) node {$\bullet$}; 
                  }

        \foreach \x in {0,2,4}%,6,8}
     {   \fill[white](\x,3) circle (4pt);   
                        \draw (\x,3) node {$\bullet$}; 
                  }
                
      \foreach \x in {0,2,4}%,6,8}
     {   \fill[white](\x,0) circle (4pt);   
                        \draw (\x,0) node {$\bullet$}; 
                  }

\draw[rounded corners ,densely dotted ] (-1.5,-0)  rectangle( 9.5,-2.5)   node[midway] 
%{$\bigl(1+(1,2)\bigr)\bigl(1-(1,3)\bigr)$};   %{$(1-(12))(1+(23))$};  
{$\color{red}\color{black} \bigl(1 - (1,3) \bigr)\bigl( 1 + (1,2) \bigr)$};

\draw[rounded corners ,densely dotted ] (-1.5,-0)  rectangle( 9.5,3);

\draw[rounded corners ,densely dotted ] (-1.5,6)  rectangle( 9.5,3);

%\draw(8,0)--(8,3)         ;      \draw(6,0)--(6,3)         ;      
                        
\draw(0,0)--(0,3)         ;      
\draw(2,0)--(4,3)         ;         \draw(4,0)--(2,3)         ;    
\draw(0,3)--(0,6) ; %to [out=-90,in=-90] (2,6)        ;               
\draw(2,3)--(4,6) to [out=-90,in=-90] (8,6)        ;                             
\draw(4,3)--(6,6) ;
         \end{tikzpicture}   
$$
 \caption{Three  elements shown in the form $ {\color{red}\color{black}c_{(2,1)}^*} \sigma  d_{\Lambda}$  in the basis of 
 $\Delta_5\bigl((2,1)\bigr)$ from \cref{actual-d-basis}. The bottom halves are $(3,5)$-diagrams 
lying in the basis of $V_5(3)$, regarding these halves as $(5,5)$-diagrams using 
the identification made at the start of 
 Section~\ref{subsec:standardSimpleModulesPartitionAlgebra}.}
 \label{delta21}
 \end{figure}

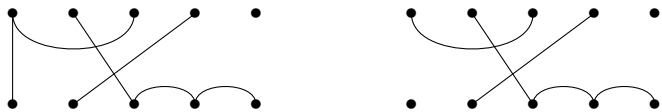
\begin{figure}[ht!]
          \begin{tikzpicture}[xscale=0.4,yscale=-0.4]
   \scalefont{0.8}
       \foreach \x in {0,1,2,3,4,5}
     {  
      \path(\x*2,0) coordinate (down\x);  
   }

       \foreach \x in {0,1,2,3,4,5,6}
    {   \path(\x*2,3) coordinate (up\x);  
    }

        \foreach \x in {0,2,4,6,8}
     {   \fill[white](\x,0) circle (4pt);   
                        \draw (\x,0) node {$\bullet$}; 
                  }

        \foreach \x in {0,2,4,6,8}%,6,8}
     {   \fill[white](\x,3) circle (4pt);   
                        \draw (\x,3) node {$\bullet$}; 
                  }
                
%\draw (4,0)--(0,3);
%\draw (2,0)--(4,3); \draw (4,3) to [out=-90,in=-90] (6,3); \draw (6,3) to [out=-90,in=-90] (8,3);
%\draw (4,0) to [out=90,in=90] (0,0);
%\draw (6,0)--(2,3);
        
\draw (4,0) to [out=90,in=-90](0,3);
\draw (2,0)--(4,3); 
\draw (4,3) to [out=-90,in=-90] (6,3);
 \draw (6,3) to [out=-90,in=-90] (8,3);
\draw (4,0) to [out=120,in=0] (2,1) to [out=180,in=60] (0,0);
\draw (6,0)--(2,3);
                    
\end{tikzpicture}\ \
\begin{tikzpicture}[xscale=0.4,yscale=-0.4]
   \scalefont{0.8}
       \foreach \x in {0,1,2,3,4,5}
     {  
      \path(\x*2,0) coordinate (down\x);  
   }

       \foreach \x in {0,1,2,3,4,5,6}
    {   \path(\x*2,3) coordinate (up\x);  
    }

        \foreach \x in {0,2,4,6,8}
     {   \fill[white](\x,0) circle (4pt);   
                        \draw (\x,0) node {$\bullet$}; 
                  }

        \foreach \x in {0,2,4,6,8}%,6,8}
     {   \fill[white](\x,3) circle (4pt);   
                        \draw (\x,3) node {$\bullet$}; 
                  }

\draw (2,0)--(4,3); \draw (4,3) to [out=-90,in=-90] (6,3); \draw (6,3) to [out=-90,in=-90] (8,3);
\draw (4,0) to [out=90,in=90] (0,0);
\draw (6,0)--(2,3);
                    
\end{tikzpicture}

\caption{The diagrams $d_\Gamma$ and $d_{\Gamma'}$ of two $(5,5)$-set-partitions.}
\label{delta21gammas}
\end{figure}

\begin{figure}[ht!]
%\hspace*{-0.5cm}
\begin{minipage}{8in}
   \begin{tikzpicture}[xscale=0.4,yscale=-0.4]
   \scalefont{0.8}
       \foreach \x in {0,1,2,3,4,5}
     {  
      \path(\x*2,0) coordinate (down\x);  
   }

       \foreach \x in {0,1,2,3,4,5,6}
    {   \path(\x*2,3) coordinate (up\x);  
    }

        \foreach \x in {0,2,4,6,8}
     {   \fill[white](\x,6) circle (4pt);   
                        \draw (\x,6) node {$\bullet$}; 
                  }

        \foreach \x in {0,2,4}%,6,8}
     {   \fill[white](\x,3) circle (4pt);   
                        \draw (\x,3) node {$\bullet$}; 
                  }
                
      \foreach \x in {0,2,4}%,6,8}
     {   \fill[white](\x,0) circle (4pt);   
                        \draw (\x,0) node {$\bullet$}; 
                  }

\draw[rounded corners ,densely dotted ] (-1.5,-0)  rectangle( 9.5,-2.5)   node[midway] 
%{$\bigl(1+(1,2)\bigr)\bigl(1-(1,3)\bigr)$};  
{$\color{red}\color{black} \bigl(1 - (1,3) \bigr)\bigl( 1 + (1,2) \bigr)$};
%{$(1-(12))(1+(23))$};  

\draw[rounded corners ,densely dotted ] (-1.5,-0)  rectangle( 9.5,3);

\draw[rounded corners ,densely dotted ] (-1.5,6)  rectangle( 9.5,3);

%\draw(8,0)--(8,3)         ;      \draw(6,0)--(6,3)         ;      
                        
\draw(0,0)--(0,3)         ;      
\draw(2,0)--(4,3)         ;         \draw(4,0)--(2,3)         ;    
\draw(0,3)--(0,6) to [out=-90,in=-90] (4,6)        ;               
\draw(2,3)--(2,6) to [out=-90,in=-90] (8,6)        ;                             
\draw(4,3)--(6,6) ;

\begin{scope}[yshift=7cm] %this is copied from previous figure MkW
\draw[rounded corners ,densely dotted ] (-1.5,0)  rectangle( 9.5,3);
       \foreach \x in {0,1,2,3,4,5}
     {  
      \path(\x*2,0) coordinate (down\x);  
   }

       \foreach \x in {0,1,2,3,4,5,6}
    {   \path(\x*2,3) coordinate (up\x);  
    }

        \foreach \x in {0,2,4,6,8}
     {   \fill[white](\x,0) circle (4pt);   
                        \draw (\x,0) node {$\bullet$}; 
                  }

        \foreach \x in {0,2,4,6,8}%,6,8}
     {   \fill[white](\x,3) circle (4pt);   
                        \draw (\x,3) node {$\bullet$}; 
                  }
                
\draw (4,0) to [out=90,in=-90](0,3);
\draw (2,0)--(4,3); 
\draw (4,3) to [out=-90,in=-90] (6,3);
 \draw (6,3) to [out=-90,in=-90] (8,3);
\draw (4,0) to [out=120,in=0] (2,1) to [out=180,in=60] (0,0);
\draw (6,0)--(2,3);
   \end{scope}
         \end{tikzpicture} 
\!\!\!\!\!\!\!\!\!\!\!\!\!\!\!\! \raisebox{2.4cm}{$=$} \ 
\raisebox{0.7cm}{\begin{tikzpicture}[xscale=0.4,yscale=-0.4]
   \scalefont{0.8}
       \foreach \x in {0,1,2,3,4,5}
     {  
      \path(\x*2,0) coordinate (down\x);  
   }

       \foreach \x in {0,1,2,3,4,5,6}
    {   \path(\x*2,3) coordinate (up\x);  
    }

        \foreach \x in {0,2,4,6,8}
     {   \fill[white](\x,6) circle (4pt);   
                        \draw (\x,6) node {$\bullet$}; 
                  }

        \foreach \x in {0,2,4}%,6,8}
     {   \fill[white](\x,3) circle (4pt);   
                        \draw (\x,3) node {$\bullet$}; 
                  }
                
      \foreach \x in {0,2,4}%,6,8}
     {   \fill[white](\x,0) circle (4pt);   
                        \draw (\x,0) node {$\bullet$}; 
                  }

\draw[rounded corners ,densely dotted ] (-1.5,-0)  rectangle( 9.5,-2.5)   node[midway] 
%{$\bigl(1+(1,2)\bigr)\bigl(1-(1,3)\bigr)$};  
{$\color{red}\color{black} \bigl(1 - (1,3) \bigr)\bigl( 1 + (1,2) \bigr)$};
%{$(1-(12))(1+(23))$};  

\draw[rounded corners ,densely dotted ] (-1.5,-0)  rectangle( 9.5,3);

\draw[rounded corners ,densely dotted ] (-1.5,6)  rectangle( 9.5,3);

%\draw(8,0)--(8,3)         ;      \draw(6,0)--(6,3)         ;      

\draw(0,0)--(0,3)         ;      
\draw(2,0)--(4,3)         ;         \draw(4,0)--(2,3)         ; 
\draw (2,3)--(4,6); \draw (4,6) to [out=-90,in=-90] (6,6); \draw (6,6) to [out=-90,in=-90] (8,6);
\draw (4,3)--(2,6);
\draw (0,3)--(0,6);

         \end{tikzpicture} }
\!\!\!\!\!\!\!\!\!\!\!\!\!\!\!\! \raisebox{2.4cm}{$=$} \ 
\raisebox{0.7cm}{\begin{tikzpicture}[xscale=0.4,yscale=-0.4]
   \scalefont{0.8}
       \foreach \x in {0,1,2,3,4,5}
     {  
      \path(\x*2,0) coordinate (down\x);  
   }

       \foreach \x in {0,1,2,3,4,5,6}
    {   \path(\x*2,3) coordinate (up\x);  
    }

        \foreach \x in {0,2,4,6,8}
     {   \fill[white](\x,6) circle (4pt);   
                        \draw (\x,6) node {$\bullet$}; 
                  }

        \foreach \x in {0,2,4}%,6,8}
     {   \fill[white](\x,3) circle (4pt);   
                        \draw (\x,3) node {$\bullet$}; 
                  }
                
      \foreach \x in {0,2,4}%,6,8}
     {   \fill[white](\x,0) circle (4pt);   
                        \draw (\x,0) node {$\bullet$}; 
                  }

\draw[rounded corners ,densely dotted ] (-1.5,-0)  rectangle( 9.5,-2.5)   node[midway] 
%{$\bigl(1+(1,2)\bigr)\bigl(1-(1,3)\bigr)$};  
{$\color{red}\color{black} \bigl(1 - (1,3) \bigr)\bigl( 1 + (1,2) \bigr)$};
%{$(1-(12))(1+(23))$};  

\draw[rounded corners ,densely dotted ] (-1.5,-0)  rectangle( 9.5,3);

\draw[rounded corners ,densely dotted ] (-1.5,6)  rectangle( 9.5,3);

%\draw(8,0)--(8,3)         ;      \draw(6,0)--(6,3)         ;      

\draw(0,0)--(0,3)         ;      
\draw(2,0)--(2,3)         ;         \draw(4,0)--(4,3)         ; 
\draw (4,3)--(4,6); \draw (4,6) to [out=-90,in=-90] (6,6); \draw (6,6) to [out=-90,in=-90] (8,6);
\draw (2,3)--(2,6);
\draw (0,3)--(0,6);

         \end{tikzpicture} }         
\end{minipage}
\caption{The product ${\color{red}\color{black}c_{(2,1)}^*} (2,3) d_\Lambda d_\Gamma$ shown first in non-canonical
form as $\color{red}\color{black} \bigl(1 - (1,3) \bigr)\bigl( 1 + (1,2) \bigr)(2,3)d$ where $\pi_d = (2,3)$ and then
as a canonical basis element from~\eqref{actual-d-basis} as
$\color{red}\color{black} \bigl(1 - (1,3) \bigr)\bigl( 1 + (1,2) \bigr)d'$ where $\pi_{d'} = 1$.}
\label{delta21product}

\end{figure}
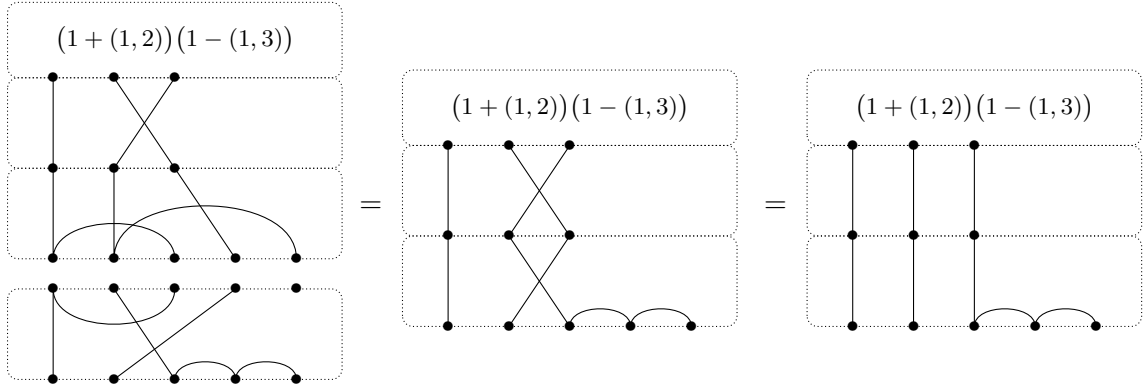
\end{eg}

\begin{thm}{\cite[Proposition 3, Proposition 9]{mar1}}
\label{semisimple}
 The partition algebra $P_r(\delta)$ is semisimple if and only if $\delta \not\in \{0,1,\ldots, 2r-2\}$ and, in this case, the set $\{\Delta_{r %, \delta
}(\kappa):\kappa \in  \ParSet(\leq r) \}$ is a complete set of non-isomorphic simple right $P_r(\delta)$-modules. 
 More generally, provided $\delta \ne 0$,  the standard module $ \Delta_{r%, \delta
 }(\kappa)$ has a simple head, which we denote $L_{r}(\kappa)$, and $\{L_{r}(\kappa):\kappa \in \ParSet(\leq r) \}$ is a complete set of non-isomorphic simple right $P_r(\delta)$-modules.  
 \end{thm}

Martin showed in  \cite{mar1} that   $P_r(\delta)$ is a quasi-hereditary algebra provided $\delta \neq 0$. The partition algebra $P_r(\delta)$ is also a cellular algebra {\color{red}\color{black}\cite{MR1711582}}, for any value of $\delta$, and the standard modules 
$\Delta_{r }(\kappa)$ for $\kappa \in  \ParSet(\leq r)$ are the {\sf cell modules}. The following proposition 
%(a consequence of Martin's proof of the previous result) 
%MkW: it has a numbered proposition reference from Martin which seems more authoritative to me 
%will be used later and 
tells us that certain standard modules are  simple.

\begin{lem}{\cite[Proposition 23]{mar1}}
\label{semisimple2}
Let ${\color{red}\color{black}\delta= n\in \mathbb Z_{>0}}$, suppose that $\kappa$ is a partition such that  $\kappa{[n]}$ is a partition of $n$ 
(i.e.~$n-|\kappa| \ge \kappa_1$). Then the $P_r(n)$-standard  module $\Delta_r(\kappa)=L_r(\kappa)$   if and only if $\delta\geq r+\kappa_1$. Moreover, in this case, the module belongs to a simple block.  
 \end{lem}
 
  \subsection{The orbit basis of Benkart--Halverson}
 
The diagram basis is the most natural basis for the partition algebra.  In particular, we were able to define a multiplication with respect to this basis with ease.  There is another basis of the partition algebra, the orbit basis, which was studied  by Benkart--Halverson in \cite{MR3969570}.  The advantage of this basis is that it is more intimately connected to the semisimple quotient of the partition algebra that acts {\em faithfully} on tensor space.

 We note that the set of  {$(r,r)$-set-partitions} is a lattice (a partially ordered set in which each pair of elements admits an upper and lower bound) under the partial order
\begin{align}\label{refine1}
 \Lambda \leq \Lambda' \quad \text{if every block of $\Lambda$ is contained in a block of $\Lambda'$}.
\end{align}
If $\Lambda \leq \Lambda'$ we say that  $\Lambda'$ is {\sf coarser} 
than $\Lambda$; equivalently $\Lambda'$ is coarser than $\Lambda$ if
 $\Lambda'$ may be formed from $\Lambda$ by merging some of its blocks together.

The {\sf orbit basis} of the partition algebra consists 
of the elements $x_{\Lambda }$  indexed by set-partitions~${\Lambda}$  
defined by the coarsening relation as follows 
\begin{align}\label{refine2}
d_{{\Lambda }}= \sum_{\Lambda \leq \Lambda' } x_{\Lambda'}.
\end{align}
In other words,  the diagram basis element $d_\Lambda$ is the sum of all orbit basis elements $x_{\Lambda'}$
 for which $\Lambda'$ is coarser
 than $\Lambda$.  Conversely, the elements $x_{\Lambda}$ can be written  as a sum of diagram basis elements by way of the M\"obius function for the coarsening partial order; we refer to \cite[Section 4.3]{MR3969570} for more details.   The orbit basis was so-named by Benkart and Halverson because, by \cite[Remark 4.7]{MR3969570}, 
 the action (see~equation \eqref{eq:Ptn-action} below) 
 of  $x_{\Lambda}$  on tensor space  corresponds to an $\W_n$-orbit on simple tensors. 

In this paper, we shall only need to know that the elements $x_{\Lambda}$ form a basis and to know how these basis elements act on tensor space (shown later in \cref{eq:Ptn-action}). Knowing this set forms a basis of the partition algebra, and using \eqref{actual-d-basis}, we deduce  immediately that
the standard module $\Delta_r(\kappa)$ has an {\sf orbit basis} given by
\begin{align}\label{rowenaB}
 \Bigl\{  c_{\kappa}^\ast \sigma x_{\Lambda} 
\, \Bigl|  \,
 \begin{minipage}{3.2in}$\stt^\kappa \sigma \in \Std(\kappa)$, 
 $\Lambda$ is a $(k,r)$-set-partition with $k$ propagating blocks, 
  $\pi _\Lambda=1 \in \W_k$\end{minipage} \Bigr\}.
 \end{align}
 %It's already been defined: MkW
% We refer to this as the {\sf orbit basis} of $\Delta_r(\kappa)$.

\section{Ramified partition algebras}\label{Ramified partition algebras}
  The   ramified partition algebra  was originally defined by
  Martin--Elgamal  
    in \cite{MR2073453} and later rediscovered by Kennedy   \cite{MR2287557},
    who referred to it as the class partition algebra.

\subsection{Ramified set-partitions}
 We define a {\sf ramified $(r,s)$-set-partition} to be an ordered pair  $( {\Lambda},{\Lambda'})$ such that $\Lambda, \Lambda'$ are  set-partitions of $\{1,2,\ldots, r, \bar{1},\bar{2}, \ldots, \bar{s}\}$ and $\Lambda'$ is coarser than $ \Lambda$ in the sense of \cref{refine1}. 
  We refer to $ {\Lambda}$ as the {\sf inner set-partition} and $ {\Lambda'}$ as the {\sf outer set-partition}. 
% That the outer partition $\Lambda'$ is coarser than $\Lambda$ means that it may be formed from $\Lambda$ by merging some parts of the inner set-partition together. 
  Diagrammatically we represent these  ordered pairs 
  by drawing the partition diagram of the inner set-partition as usual and then merging 
  parts from~$d_{\Lambda}$ to form~$d_{\Lambda'}$.  
  We continue to draw the (inner and outer) set-partitions with respect to the conventions of \cref{usefulconvention}.  
%   Several examples are depicted below 
Examples are depicted in \cref{draw,draw2}.

 \begin{figure}[ht!]
 $$ 
   \scalefont{0.8}
 \begin{minipage}{2.2cm} \begin{tikzpicture}[scale=0.45]
  
       \foreach \x in {0,2}
     {   \path(\x,3) coordinate (up\x);  
      \path(\x,0) coordinate (down\x);  
   }

     \foreach \x in {0,2}
     {   \path(up\x) --++ (135:0.4) coordinate(up1\x);  
      \path(up\x) --++ (45:0.4) coordinate(up2\x);
            \path(up\x) --++ (-45:0.4) coordinate(up3\x);  
                        \path(up\x) --++ (-135:0.4) coordinate(up4\x);  
   }

    \foreach \x in {0,2}
     {   \path(down\x) --++ (135:0.4) coordinate(down1\x);  
      \path(down\x) --++ (45:0.4) coordinate(down2\x);
            \path(down\x) --++ (-45:0.4) coordinate(down3\x);  
                        \path(down\x) --++ (-135:0.4) coordinate(down4\x);  
   }

  \draw [fill=white] plot [smooth cycle]
%  coordinates {(up10) (up22) (down32)   (down40)  };
   coordinates {(up12) (up32) (down30) (down10)  };
      \draw(2,3)--(0,0);

   \draw [fill=white] plot [smooth cycle]
%  coordinates {(up10) (up22) (down32)   (down40)  };
   coordinates {(up40) (up20) (down22) (down42)  };
 \draw(0,3)--(2,0);

     \foreach \x in {0,2}
     {   \fill[white](\x,3) circle (4pt);   
                \fill[white](\x,0) circle (4pt);   }
                        \draw (2,3) node {$\bullet$}; 
                           \draw (0,3) node {$\bullet$}; 
                    \draw (2,0) node {$\bullet$}; 
                           \draw (0,0) node {$\bullet$};

         \end{tikzpicture}\end{minipage} 
 \quad\quad 
  \begin{minipage}{2.2cm}\begin{tikzpicture}[scale=0.45]
  
       \foreach \x in {0,2}
     {   \path(\x,3) coordinate (up\x);  
      \path(\x,0) coordinate (down\x);  
   }

     \foreach \x in {0,2}
     {   \path(up\x) --++ (135:0.4) coordinate(up1\x);  
      \path(up\x) --++ (45:0.4) coordinate(up2\x);
            \path(up\x) --++ (-45:0.4) coordinate(up3\x);  
                        \path(up\x) --++ (-135:0.4) coordinate(up4\x);  
   }

    \foreach \x in {0,2}
     {   \path(down\x) --++ (135:0.4) coordinate(down1\x);  
      \path(down\x) --++ (45:0.4) coordinate(down2\x);
            \path(down\x) --++ (-45:0.4) coordinate(down3\x);  
                        \path(down\x) --++ (-135:0.4) coordinate(down4\x);  
   }

   \draw [fill=white] plot [smooth cycle]
  coordinates {(up10) (up22) (down32)   (down40)  };

 \draw(0,3)--(2,0);                       
    \draw(2,3)--(0,0);

     \foreach \x in {0,2}
     {   \fill[white](\x,3) circle (4pt);   
                \fill[white](\x,0) circle (4pt);   }
                        \draw (2,3) node {$\bullet$}; 
                           \draw (0,3) node {${\bullet}$}; 
                    \draw (2,0) node {$\bullet$}; 
                           \draw (0,0) node {${\bullet}$};

         \end{tikzpicture}  \end{minipage} 
 \quad\quad   
     \begin{minipage}{2.2cm}      \begin{tikzpicture}[scale=0.45]
  
       \foreach \x in {0,2}
     {   \path(\x,3) coordinate (up\x);  
      \path(\x,0) coordinate (down\x);  
   }

     \foreach \x in {0,2}
     {   \path(up\x) --++ (135:0.4) coordinate(up1\x);  
      \path(up\x) --++ (45:0.4) coordinate(up2\x);
            \path(up\x) --++ (-45:0.4) coordinate(up3\x);  
                        \path(up\x) --++ (-135:0.4) coordinate(up4\x);  
   }

    \foreach \x in {0,2}
     {   \path(down\x) --++ (135:0.4) coordinate(down1\x);  
      \path(down\x) --++ (45:0.4) coordinate(down2\x);
            \path(down\x) --++ (-45:0.4) coordinate(down3\x);  
                        \path(down\x) --++ (-135:0.4) coordinate(down4\x);  
   }

   \draw [fill=white] plot [smooth cycle]
  coordinates {(up10) (up20) (down30)   (down40)  };

   \draw [fill=white] plot [smooth cycle]
  coordinates {(up12) (up22) (down32)   (down42)  };

    \draw(2,3)--(2,0);

     \foreach \x in {0,2}
     {   \fill[white](\x,3) circle (4pt);   
                \fill[white](\x,0) circle (4pt);   }
                        \draw (2,3) node {$\bullet$}; 
                           \draw (0,3) node {${\bullet}$}; 
                    \draw (2,0) node {$\bullet$}; 
                           \draw (0,0) node {${\bullet}$};

         \end{tikzpicture} \end{minipage} 
         \quad\quad    
     \begin{minipage}{2.2cm}      \begin{tikzpicture}[scale=0.45]
  
       \foreach \x in {0,2}
     {   \path(\x,3) coordinate (up\x);  
      \path(\x,0) coordinate (down\x);  
   }

     \foreach \x in {0,2}
     {   \path(up\x) --++ (135:0.4) coordinate(up1\x);  
      \path(up\x) --++ (45:0.4) coordinate(up2\x);
            \path(up\x) --++ (-45:0.4) coordinate(up3\x);  
                        \path(up\x) --++ (-135:0.4) coordinate(up4\x);  
   }

    \foreach \x in {0,2}
     {   \path(down\x) --++ (135:0.4) coordinate(down1\x);  
      \path(down\x) --++ (45:0.4) coordinate(down2\x);
            \path(down\x) --++ (-45:0.4) coordinate(down3\x);  
                        \path(down\x) --++ (-135:0.4) coordinate(down4\x);  
   }

   \draw [fill=white] plot [smooth cycle]
  coordinates {(up12) (up22) (down32)   (down42)  };

                           \draw (0,0) node {$ \encircle{\color{black} \one}$}; 

                           \draw (0,3) node {$ \encircle{\color{black} \one}$};

    \draw(2,3)--(2,0);

     \foreach \x in {0,2}
     {   \fill[white](\x,3) circle (4pt);   
                \fill[white](\x,0) circle (4pt);   }
                        \draw (2,3) node {$\bullet$}; 
                           \draw (0,3) node {${\bullet}$}; 
                    \draw (2,0) node {$\bullet$}; 
                           \draw (0,0) node {${\bullet}$};

         \end{tikzpicture} \end{minipage} 
$$
\smallskip
$$    \scalefont{0.8}
         \hspace*{1.1in}
      \begin{minipage}{2.2cm}      \begin{tikzpicture}[scale=0.45]
  
       \foreach \x in {0,2}
     {   \path(\x,3) coordinate (up\x);  
      \path(\x,0) coordinate (down\x);  
   }

     \foreach \x in {0,2}
     {   \path(up\x) --++ (135:0.4) coordinate(up1\x);  
      \path(up\x) --++ (45:0.4) coordinate(up2\x);
            \path(up\x) --++ (-45:0.4) coordinate(up3\x);  
                        \path(up\x) --++ (-135:0.4) coordinate(up4\x);  
   }

    \foreach \x in {0,2}
     {   \path(down\x) --++ (135:0.4) coordinate(down1\x);  
      \path(down\x) --++ (45:0.4) coordinate(down2\x);
            \path(down\x) --++ (-45:0.4) coordinate(down3\x);  
                        \path(down\x) --++ (-135:0.4) coordinate(down4\x);  
   }

   \draw [fill=white] plot [smooth cycle]
  coordinates {(up10) (up20) (down30)   (down40)  };

   \draw [fill=white] plot [smooth cycle]
  coordinates {(up12) (up22) (down32)   (down42)  };
   
     \foreach \x in {0,2}
     {   \fill[white](\x,3) circle (4pt);   
                \fill[white](\x,0) circle (4pt);   }
                        \draw (2,3) node {$\bullet$}; 
                           \draw (0,3) node {${\bullet}$}; 
                    \draw (2,0) node {$\bullet$}; 
                           \draw (0,0) node {${\bullet}$};

         \end{tikzpicture} \end{minipage} 
         \quad\quad  
          \begin{minipage}{2.2cm}
 \begin{tikzpicture}[scale=0.45]
  
       \foreach \x in {0,2}
     {   \path(\x,3) coordinate (up\x);  
      \path(\x,0) coordinate (down\x);  
   }

     \foreach \x in {0,2}
     {   \path(up\x) --++ (135:0.4) coordinate(up1\x);  
      \path(up\x) --++ (45:0.4) coordinate(up2\x);
            \path(up\x) --++ (-45:0.4) coordinate(up3\x);  
                        \path(up\x) --++ (-135:0.4) coordinate(up4\x);  
   }

    \foreach \x in {0,2}
     {   \path(down\x) --++ (135:0.4) coordinate(down1\x);  
      \path(down\x) --++ (45:0.4) coordinate(down2\x);
            \path(down\x) --++ (-45:0.4) coordinate(down3\x);  
                        \path(down\x) --++ (-135:0.4) coordinate(down4\x);  
   }

 \path(up0) --++ (165:0.4) coordinate (up10);

   \draw [fill=white] plot [smooth cycle]
  coordinates {(up10) (up22) (down32)   (down42)   };

  \draw  (0,3) to [out=-30,in=-150] (2,3);

     \foreach \x in {0,2}
     {   \fill[white](\x,3) circle (4pt);   
                \fill[white](\x,0) circle (4pt);   }
                        \draw (2,3) node {$\bullet$}; 
                           \draw (0,3) node {${\bullet}$}; 
                    \draw (2,0) node {$\bullet$}; 
                           \draw (0,0) node {$ \encircle{\color{black} \one}$};

         \end{tikzpicture} 
         \end{minipage} 
 \quad\quad 
   \begin{minipage}{2cm}
 \begin{tikzpicture}[scale=0.45]
  
       \foreach \x in {0,2}
     {   \path(\x,3) coordinate (up\x);  
      \path(\x,0) coordinate (down\x);  
   }

     \foreach \x in {0,2}
     {   \path(up\x) --++ (135:0.4) coordinate(up1\x);  
      \path(up\x) --++ (45:0.4) coordinate(up2\x);
            \path(up\x) --++ (-45:0.4) coordinate(up3\x);  
                        \path(up\x) --++ (-135:0.4) coordinate(up4\x);  
   }

    \foreach \x in {0,2}
     {   \path(down\x) --++ (135:0.4) coordinate(down1\x);  
      \path(down\x) --++ (45:0.4) coordinate(down2\x);
            \path(down\x) --++ (-45:0.4) coordinate(down3\x);  
                        \path(down\x) --++ (-135:0.4) coordinate(down4\x);  
   }
  
   \draw [fill=white] plot [smooth cycle]
  coordinates {(up10) (up22) (up32)   (up40)   };

  \draw  (0,3) to  (2,3);  
    
     \foreach \x in {0,2}
     {   \fill[white](\x,3) circle (4pt);   
                \fill[white](\x,0) circle (4pt);   }
                        \draw (2,3) node {$\bullet$}; 
                           \draw (0,3) node {${\bullet}$}; 
     \draw (2,0) node {$ \encircle{\color{black} \two}$}; 
                           \draw (0,0) node {$ \encircle{\color{black} \one}$};

         \end{tikzpicture} \end{minipage} 
         \quad\quad\!
           \begin{minipage}{2cm}
 \begin{tikzpicture}[scale=0.45]
  
       \foreach \x in {0,2}
     {   \path(\x,3) coordinate (up\x);  
      \path(\x,0) coordinate (down\x);  
   }

     \foreach \x in {0,2}
     {   \path(up\x) --++ (135:0.4) coordinate(up1\x);  
      \path(up\x) --++ (45:0.4) coordinate(up2\x);
            \path(up\x) --++ (-45:0.4) coordinate(up3\x);  
                        \path(up\x) --++ (-135:0.4) coordinate(up4\x);  
   }

   \foreach \x in {0,2}
     {   \path(down\x) --++ (135:0.4) coordinate(down1\x);  
      \path(down\x) --++ (45:0.4) coordinate(down2\x);
            \path(down\x) --++ (-45:0.4) coordinate(down3\x);  
                        \path(down\x) --++ (-135:0.4) coordinate(down4\x);  
   }
 
   \draw [white,fill=white] plot [smooth cycle]
  coordinates {(up10) (up22) (up32)   (up40)   };

  \draw[white]  (0,3) to  (2,3);

     \foreach \x in {0,2}
     {   \fill[white](\x,3) circle (4pt);   
                \fill[white](\x,0) circle (4pt);   }
                        \draw[white] (2,3) node {$\bullet$}; 
                           \draw[white] (0,3) node {${\bullet}$}; 
     \draw[white] (2,0) node {$ \encircle{\color{white} \two}$}; 
                           \draw[white] (0,0) node {$ \encircle{\color{white} \one}$};

         \end{tikzpicture} \end{minipage}\quad\;
$$
\caption{Some examples of ramified $(2,2)$-set-partitions.  
The propagating indices (see Section~\ref{subsec:propagatingIndex})
are $(1,1),(2), (1,0), (1), (0,0), (0),$ and $ \varnothing$ respectively. 
 }
\label{draw2}
\end{figure}
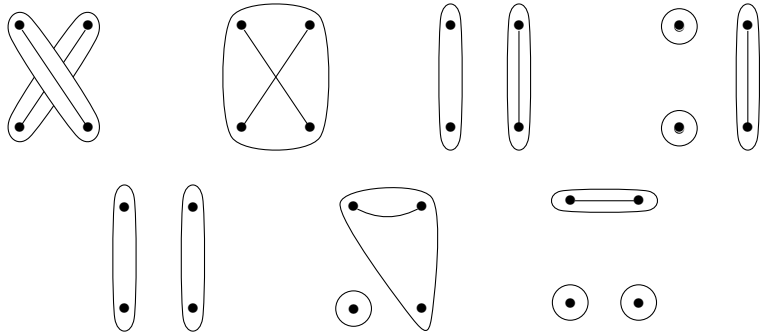

We now consider two parameters $\delta_{\rm in} $ and $\delta _{\rm out}\in \C $. 
The product  of ramified set-partitions 
is  derived from the product in the two partition algebras  $P_r(\delta_{\rm in} )$ and  $P_r( \delta_{\rm out} )$
 in a manner we shall now describe.   
 To distinguish the products in these two algebras let us temporarily denote the product in the partition algebra  $P_r(\delta_{\rm in} )$ by $\cdot_{\delta_{\rm in} }$, and that in $P_r(\delta_{\rm out} )$ by $\cdot_{\delta_{\rm out} }$. 
 Then  we define the product of two ramified  $(r,r)$-set-partition diagrams  as follows:
\begin{align}\label{multi}
(d_{\Lambda}, d_{\Lambda'})(d_{\Gamma}, d_{\Gamma'})=(\delta_{\rm in} )^{t} (\delta_{\rm out} )^{s} (d_{\Delta}, d_{\Delta'}),  
\end{align} where $
 d _{\Lambda } \cdot_{\delta_{\rm in} } d_{\Gamma }=(\delta_{\rm in} )^{s}  d_{\Delta }
 \textrm{ and } d _{\Lambda'} \cdot_{\delta_{\rm out} } d_{\Gamma'}=(\delta_{\rm out} )^{s}  d_{\Delta'}$.
 %MkW: changed from indices 1, 2, 3 to \Lambda, \Gamma, \Delta
In other words, the multiplication of inner set-partition diagrams  yields a parameter $\delta_{\rm in} $ 
and the multiplication of outer set-partition diagrams a parameter~$\delta_{\rm out} $.
{\color{red}\color{black} The fact that   $\Delta'$ is a coarsening of $\Delta$ is proven in \cite[Proposition 2]{MR2073453}.
}
We let $R_r(\delta_{\rm in} , \delta_{\rm out} )$ denote the complex vector space with basis given by  all ramified set-partitions of  $\{1,2,\ldots, r, \bar{1},\bar{2}, \ldots, \bar{r}\} $  
and with multiplication given by  linearly extending  the multiplication of ramified diagrams.   
Then $R_r(\delta_{\rm in} , \delta_{\rm out} )$ is an associative $\C $-algebra, known as the  {\sf ramified partition  algebra}.
 We define an anti-involution $\ast$ on   $R_r(\delta_{\rm in} , \delta_{\rm out} )$ by reflecting a  diagram
 through its horizontal axis.

%  We remark that we can assume that $r=\sum_{i\geq 1}\vartheta

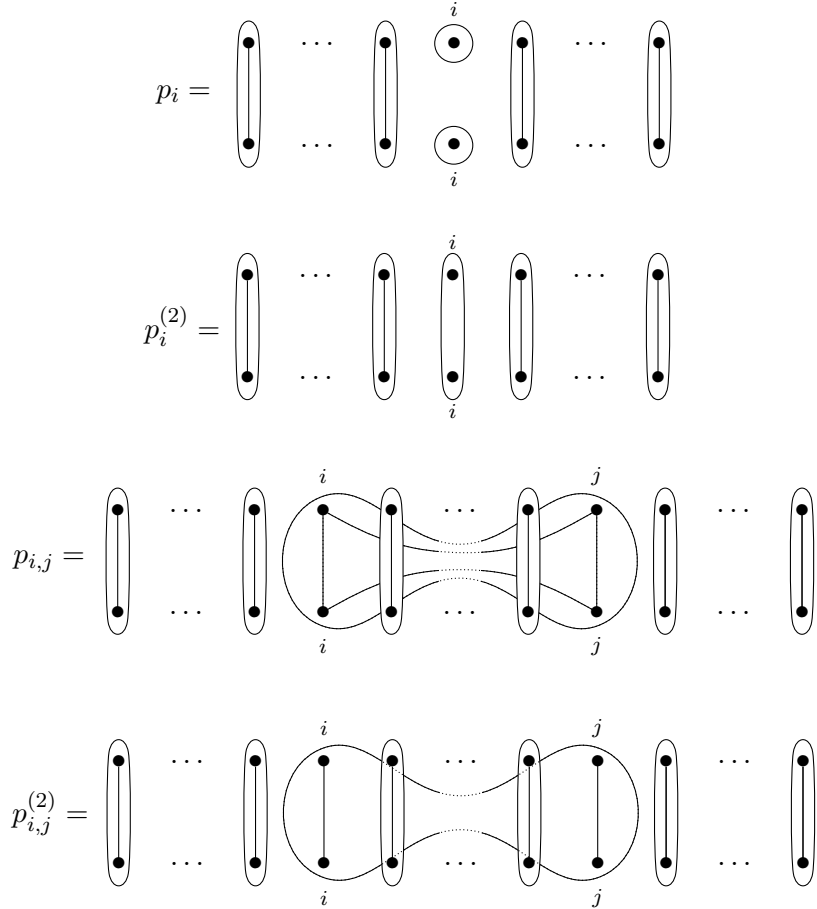
\begin{figure}[ht!]
 $$
 \begin{tikzpicture}[scale=0.45]
  \draw(-1.9,1.5) node {$p_{i}=$};  
       \foreach \x in {0,2,4,6,8,10,12}
     {   \path(\x,3) coordinate (up\x);  
      \path(\x,0) coordinate (down\x);  
   }

  {   \path(down12) --++ (135:0.4) coordinate(down1twelve);  
      \path(down12) --++ (45:0.4) coordinate(down2twelve);
            \path(down12) --++ (-45:0.4) coordinate(down3twelve);  
                        \path(down12) --++ (-135:0.4) coordinate(down4twelve);  
   }

     {   \path(up12) --++ (135:0.4) coordinate(up1twelve);  
      \path(up12) --++ (45:0.4) coordinate(up2twelve);
            \path(up12) --++ (-45:0.4) coordinate(up3twelve);  
                        \path(up12) --++ (-135:0.4) coordinate(up4twelve);  
   }

     \foreach \x in {0,2,4,6,8,10,12}
     {   \path(up\x) --++ (135:0.4) coordinate(up1\x);  
      \path(up\x) --++ (45:0.4) coordinate(up2\x);
            \path(up\x) --++ (-45:0.4) coordinate(up3\x);  
                        \path(up\x) --++ (-135:0.4) coordinate(up4\x);  
   }

    \foreach \x in {0,2,4,6,8,10,12}
     {   \path(down\x) --++ (135:0.4) coordinate(down1\x);  
      \path(down\x) --++ (45:0.4) coordinate(down2\x);
            \path(down\x) --++ (-45:0.4) coordinate(down3\x);  
                        \path(down\x) --++ (-135:0.4) coordinate(down4\x);  
   }

     \draw [fill=white] plot [smooth cycle]
  coordinates {(up18) (up28) (down38)   (down48)  };
         \draw(8,0)--(8,3);

%        
%     \draw [fill=white] plot [smooth cycle]
%  coordinates {(up16) (up26) (down36)   (down46)  };
% 

     \draw [fill=white] plot [smooth cycle]
  coordinates {(up14) (up24) (down34)   (down44)  };
         \draw(4,0)--(4,3);

     \draw [fill=white] plot [smooth cycle]
  coordinates {(up10) (up20) (down30)   (down40)  };
         \draw(0,0)--(0,3);

     \draw [fill=white] plot [smooth cycle]
  coordinates {(up1twelve) (up2twelve) (down3twelve)   (down4twelve)  };
         \draw(12,0)--(12,3);
%    \draw(10,0)--(10,3);

                                    \path (6,0) node {$ \encircle{\color{black} \ }$}; 
                                    \path (6,3) node {$ \encircle{\color{black} \ }$};     
        
     \foreach \x in {0,4,6,8,12}
     {   \fill[white](\x,3) circle (4pt);   
                \fill[white](\x,0) circle (4pt);    
                        \draw (\x,3) node {$\bullet$}; 
                           \draw (\x,0) node {${\bullet}$};  }
                         
 	\draw (2*3,4) node {$\scriptstyle i$};
	\draw (2*3,-1) node {$\scriptstyle \overline{i}$};
\draw (2*1,0) node {$\ldots$};\draw (2*1,3) node {$\ldots$};
\draw (2*5,0) node {$\ldots$};\draw (2*5,3) node {$\ldots$};

         \end{tikzpicture}  
 $$

          \vspace*{-0.5cm}
$$ \hspace*{-0.17cm}\begin{tikzpicture}[scale=0.45]
  \draw(-1.9,1.5) node {$p_{i}^{(2)}=$};  
       \foreach \x in {0,2,4,6,8,10,12}
     {   \path(\x,3) coordinate (up\x);  
      \path(\x,0) coordinate (down\x);  
   }

  {   \path(down12) --++ (135:0.4) coordinate(down1twelve);  
      \path(down12) --++ (45:0.4) coordinate(down2twelve);
            \path(down12) --++ (-45:0.4) coordinate(down3twelve);  
                        \path(down12) --++ (-135:0.4) coordinate(down4twelve);  
   }

     {   \path(up12) --++ (135:0.4) coordinate(up1twelve);  
      \path(up12) --++ (45:0.4) coordinate(up2twelve);
            \path(up12) --++ (-45:0.4) coordinate(up3twelve);  
                        \path(up12) --++ (-135:0.4) coordinate(up4twelve);  
   }

     \foreach \x in {0,2,4,6,8,10,12}
     {   \path(up\x) --++ (135:0.4) coordinate(up1\x);  
      \path(up\x) --++ (45:0.4) coordinate(up2\x);
            \path(up\x) --++ (-45:0.4) coordinate(up3\x);  
                        \path(up\x) --++ (-135:0.4) coordinate(up4\x);  
   }

    \foreach \x in {0,2,4,6,8,10,12}
     {   \path(down\x) --++ (135:0.4) coordinate(down1\x);  
      \path(down\x) --++ (45:0.4) coordinate(down2\x);
            \path(down\x) --++ (-45:0.4) coordinate(down3\x);  
                        \path(down\x) --++ (-135:0.4) coordinate(down4\x);  
   }

     \draw [fill=white] plot [smooth cycle]
  coordinates {(up18) (up28) (down38)   (down48)  };
         \draw(8,0)--(8,3);

     \draw [fill=white] plot [smooth cycle]
  coordinates {(up16) (up26) (down36)   (down46)  };

     \draw [fill=white] plot [smooth cycle]
  coordinates {(up14) (up24) (down34)   (down44)  };
         \draw(4,0)--(4,3);

     \draw [fill=white] plot [smooth cycle]
  coordinates {(up10) (up20) (down30)   (down40)  };
         \draw(0,0)--(0,3);

     \draw [fill=white] plot [smooth cycle]
  coordinates {(up1twelve) (up2twelve) (down3twelve)   (down4twelve)  };
         \draw(12,0)--(12,3);
%    \draw(10,0)--(10,3);

     \foreach \x in {0,4,6,8,12}
     {   \fill[white](\x,3) circle (4pt);   
                \fill[white](\x,0) circle (4pt);    
                        \draw (\x,3) node {$\bullet$}; 
                           \draw (\x,0) node {${\bullet}$};  }
                         
 	\draw (2*3,4) node {$\scriptstyle i$};
	\draw (2*3,-1) node {$\scriptstyle \overline i$};
\draw (2*1,0) node {$\ldots$};\draw (2*1,3) node {$\ldots$};
\draw (2*5,0) node {$\ldots$};\draw (2*5,3) node {$\ldots$};

         \end{tikzpicture}  
 $$
 
          \vspace*{-0.5cm}
 $$
\begin{tikzpicture}[scale=0.45]
  \draw(-2,1.5) node {$p _{i,j}=$};
       \foreach \x in {0,2,4,6,8,10,12,14,16,18,20}
     {   \path(\x,3) coordinate (up\x);  
      \path(\x,0) coordinate (down\x);  
   }

  {   \path(down10) --++ (135:0.4) coordinate(down1ten);  
      \path(down10) --++ (45:0.4) coordinate(down2ten);
            \path(down10) --++ (-45:0.4) coordinate(down3ten);  
                        \path(down10) --++ (-135:0.4) coordinate(down4ten);  
   }

     {   \path(up10) --++ (135:0.4) coordinate(up1ten);  
      \path(up10) --++ (45:0.4) coordinate(up2ten);
            \path(up10) --++ (-45:0.4) coordinate(up3ten);  
                        \path(up10) --++ (-135:0.4) coordinate(up4ten);  
   }

  {   \path(down12) --++ (135:0.4) coordinate(down1twelve);  
      \path(down12) --++ (45:0.4) coordinate(down2twelve);
            \path(down12) --++ (-45:0.4) coordinate(down3twelve);  
                        \path(down12) --++ (-135:0.4) coordinate(down4twelve);  
   }

     {   \path(up12) --++ (135:0.4) coordinate(up1twelve);  
      \path(up12) --++ (45:0.4) coordinate(up2twelve);
            \path(up12) --++ (-45:0.4) coordinate(up3twelve);  
                        \path(up12) --++ (-135:0.4) coordinate(up4twelve);  
   }

  {   \path(down14) --++ (135:0.4) coordinate(down1fourteen);  
      \path(down14) --++ (45:0.4) coordinate(down2fourteen);
            \path(down14) --++ (-45:0.4) coordinate(down3fourteen);  
                        \path(down14) --++ (-135:0.4) coordinate(down4fourteen);  
   }

     {   \path(up14) --++ (135:0.4) coordinate(up1fourteen);  
      \path(up14) --++ (45:0.4) coordinate(up2fourteen);
            \path(up14) --++ (-45:0.4) coordinate(up3fourteen);  
                        \path(up14) --++ (-135:0.4) coordinate(up4fourteen);  
   }

  {   \path(down16) --++ (135:0.4) coordinate(down1sixteen);  
      \path(down16) --++ (45:0.4) coordinate(down2sixteen);
            \path(down16) --++ (-45:0.4) coordinate(down3sixteen);  
                        \path(down16) --++ (-135:0.4) coordinate(down4sixteen);  
   }

     {   \path(up16) --++ (135:0.4) coordinate(up1sixteen);  
      \path(up16) --++ (45:0.4) coordinate(up2sixteen);
            \path(up16) --++ (-45:0.4) coordinate(up3sixteen);  
                        \path(up16) --++ (-135:0.4) coordinate(up4sixteen);  
   }

  {   \path(down18) --++ (135:0.4) coordinate(down1eighteen);  
      \path(down18) --++ (45:0.4) coordinate(down2eighteen);
            \path(down18) --++ (-45:0.4) coordinate(down3eighteen);  
                        \path(down18) --++ (-135:0.4) coordinate(down4eighteen);  
   }

     {   \path(up18) --++ (135:0.4) coordinate(up1eighteen);  
      \path(up18) --++ (45:0.4) coordinate(up2eighteen);
            \path(up18) --++ (-45:0.4) coordinate(up3eighteen);  
                        \path(up18) --++ (-135:0.4) coordinate(up4eighteen);  
   }

  {   \path(down20) --++ (135:0.4) coordinate(down1twenty);  
      \path(down20) --++ (45:0.4) coordinate(down2twenty);
            \path(down20) --++ (-45:0.4) coordinate(down3twenty);  
                        \path(down20) --++ (-135:0.4) coordinate(down4twenty);  
   }

     {   \path(up20) --++ (135:0.4) coordinate(up1twenty);  
      \path(up20) --++ (45:0.4) coordinate(up2twenty);
            \path(up20) --++ (-45:0.4) coordinate(up3twenty);  
                        \path(up20) --++ (-135:0.4) coordinate(up4twenty);  
   }

     \foreach \x in {0,2,4,6,8,10,12}
     {   \path(up\x) --++ (135:0.4) coordinate(up1\x);  
      \path(up\x) --++ (45:0.4) coordinate(up2\x);
            \path(up\x) --++ (-45:0.4) coordinate(up3\x);  
                        \path(up\x) --++ (-135:0.4) coordinate(up4\x);  
   }

    \foreach \x in {0,2,4,6,8,10,12}
     {   \path(down\x) --++ (135:0.4) coordinate(down1\x);  
      \path(down\x) --++ (45:0.4) coordinate(down2\x);
            \path(down\x) --++ (-45:0.4) coordinate(down3\x);  
                        \path(down\x) --++ (-135:0.4) coordinate(down4\x);  
   }

     \draw[smooth ]
(down46)   to [out=150,in=-150]  
(up16)to [out=30,in=180] (10,2) to  [out=0,in=150] (up2fourteen) to [out=-30,in=30] (down3fourteen)
to [out=-150,in=0] (10,1) to  [out=180,in=-30] (down46)
;

 \draw[smooth ]
 (6,0) to [out=30,in=180] (10,1.25) to  [out=0,in=150] (14,0) to  (14,3)
 to [out=-150,in=0] (10.2,1.75) to  [out=180,in=-30] (6,3)  to  (6,0) 
;

\draw[white,fill=white](9.4,2.5) rectangle (10.6,0.5);

     \draw[smooth,densely dotted ]
(down46)   to [out=150,in=-150]  
(up16)to [out=30,in=180] (10,2) to  [out=0,in=150] (up2fourteen) to [out=-30,in=30] (down3fourteen)
to [out=-150,in=0] (10,1) to  [out=180,in=-30] (down46)
;

 \draw[smooth,densely dotted ]
 (6,0) to [out=30,in=180] (10,1.25) to  [out=0,in=150] (14,0) to  (14,3)
 to [out=-150,in=0] (10.2,1.75) to  [out=180,in=-30] (6,3)  to  (6,0) 
;

     \draw [fill=white] plot [smooth cycle]
  coordinates {(up18) (up28) (down38)   (down48)  };
         \draw(8,0)--(8,3);

%        
%     \draw [fill=white] plot [smooth cycle]
%  coordinates {(up16) (up26) (down36)   (down46)  };

     \draw [fill=white] plot [smooth cycle]
  coordinates {(up14) (up24) (down34)   (down44)  };
         \draw(4,0)--(4,3);

     \draw [fill=white] plot [smooth cycle]
  coordinates {(up10) (up20) (down30)   (down40)  };
         \draw(0,0)--(0,3);

     \draw [fill=white] plot [smooth cycle]
  coordinates {(up1twelve) (up2twelve) (down3twelve)   (down4twelve)  };
         \draw(12,0)--(12,3);
%    \draw(10,0)--(10,3);

             \draw(16,0)--(16,3);
     \draw  plot [smooth cycle]
  coordinates {(up1sixteen) (up2sixteen) (down3sixteen)   (down4sixteen)  };
         \draw(16,0)--(16,3);

             \draw(20,0)--(20,3);
       \draw  plot [smooth cycle]
  coordinates {(up1twenty) (up2twenty) (down3twenty)   (down4twenty)  };
         \draw(20,0)--(20,3);

     \foreach \x in {0,4,6,8,12,14,16,20}
     {   \fill[white](\x,3) circle (4pt);   
                \fill[white](\x,0) circle (4pt);    
                        \draw (\x,3) node {$\bullet$}; 
                           \draw (\x,0) node {${\bullet}$};  }
                         
	\draw (2*3,4) node {$\scriptstyle i$};\draw (2*7,4)
        node {$\scriptstyle j$};
	\draw (2*3,-1) node {$\scriptstyle \overline i$};\draw (2*7,-1)
        node {$\scriptstyle \overline j$};
\draw (2*1,3) node {$\ldots$};\draw (2*1,0) node {$\ldots$};
\draw (2*5,3) node {$\ldots$};\draw (2*5,0) node {$\ldots$};
\draw (2*9,3) node {$\ldots$};\draw (2*9,0) node {$\ldots$};    
         \end{tikzpicture}           
           $$ 
           
           \vspace*{-0.3cm}
            $$
\begin{tikzpicture}[scale=0.45]
  \draw(-2,1.5) node {$p _{i,j}^{(2)}=$};
       \foreach \x in {0,2,4,6,8,10,12,14,16,18,20}
     {   \path(\x,3) coordinate (up\x);  
      \path(\x,0) coordinate (down\x);  
   }

  {   \path(down10) --++ (135:0.4) coordinate(down1ten);  
      \path(down10) --++ (45:0.4) coordinate(down2ten);
            \path(down10) --++ (-45:0.4) coordinate(down3ten);  
                        \path(down10) --++ (-135:0.4) coordinate(down4ten);  
   }

     {   \path(up10) --++ (135:0.4) coordinate(up1ten);  
      \path(up10) --++ (45:0.4) coordinate(up2ten);
            \path(up10) --++ (-45:0.4) coordinate(up3ten);  
                        \path(up10) --++ (-135:0.4) coordinate(up4ten);  
   }

  {   \path(down12) --++ (135:0.4) coordinate(down1twelve);  
      \path(down12) --++ (45:0.4) coordinate(down2twelve);
            \path(down12) --++ (-45:0.4) coordinate(down3twelve);  
                        \path(down12) --++ (-135:0.4) coordinate(down4twelve);  
   }

     {   \path(up12) --++ (135:0.4) coordinate(up1twelve);  
      \path(up12) --++ (45:0.4) coordinate(up2twelve);
            \path(up12) --++ (-45:0.4) coordinate(up3twelve);  
                        \path(up12) --++ (-135:0.4) coordinate(up4twelve);  
   }

  {   \path(down14) --++ (135:0.4) coordinate(down1fourteen);  
      \path(down14) --++ (45:0.4) coordinate(down2fourteen);
            \path(down14) --++ (-45:0.4) coordinate(down3fourteen);  
                        \path(down14) --++ (-135:0.4) coordinate(down4fourteen);  
   }

     {   \path(up14) --++ (135:0.4) coordinate(up1fourteen);  
      \path(up14) --++ (45:0.4) coordinate(up2fourteen);
            \path(up14) --++ (-45:0.4) coordinate(up3fourteen);  
                        \path(up14) --++ (-135:0.4) coordinate(up4fourteen);  
   }

  {   \path(down16) --++ (135:0.4) coordinate(down1sixteen);  
      \path(down16) --++ (45:0.4) coordinate(down2sixteen);
            \path(down16) --++ (-45:0.4) coordinate(down3sixteen);  
                        \path(down16) --++ (-135:0.4) coordinate(down4sixteen);  
   }

     {   \path(up16) --++ (135:0.4) coordinate(up1sixteen);  
      \path(up16) --++ (45:0.4) coordinate(up2sixteen);
            \path(up16) --++ (-45:0.4) coordinate(up3sixteen);  
                        \path(up16) --++ (-135:0.4) coordinate(up4sixteen);  
   }

  {   \path(down18) --++ (135:0.4) coordinate(down1eighteen);  
      \path(down18) --++ (45:0.4) coordinate(down2eighteen);
            \path(down18) --++ (-45:0.4) coordinate(down3eighteen);  
                        \path(down18) --++ (-135:0.4) coordinate(down4eighteen);  
   }

     {   \path(up18) --++ (135:0.4) coordinate(up1eighteen);  
      \path(up18) --++ (45:0.4) coordinate(up2eighteen);
            \path(up18) --++ (-45:0.4) coordinate(up3eighteen);  
                        \path(up18) --++ (-135:0.4) coordinate(up4eighteen);  
   }

  {   \path(down20) --++ (135:0.4) coordinate(down1twenty);  
      \path(down20) --++ (45:0.4) coordinate(down2twenty);
            \path(down20) --++ (-45:0.4) coordinate(down3twenty);  
                        \path(down20) --++ (-135:0.4) coordinate(down4twenty);  
   }

     {   \path(up20) --++ (135:0.4) coordinate(up1twenty);  
      \path(up20) --++ (45:0.4) coordinate(up2twenty);
            \path(up20) --++ (-45:0.4) coordinate(up3twenty);  
                        \path(up20) --++ (-135:0.4) coordinate(up4twenty);  
   }

     \foreach \x in {0,2,4,6,8,10,12}
     {   \path(up\x) --++ (135:0.4) coordinate(up1\x);  
      \path(up\x) --++ (45:0.4) coordinate(up2\x);
            \path(up\x) --++ (-45:0.4) coordinate(up3\x);  
                        \path(up\x) --++ (-135:0.4) coordinate(up4\x);  
   }

    \foreach \x in {0,2,4,6,8,10,12}
     {   \path(down\x) --++ (135:0.4) coordinate(down1\x);  
      \path(down\x) --++ (45:0.4) coordinate(down2\x);
            \path(down\x) --++ (-45:0.4) coordinate(down3\x);  
                        \path(down\x) --++ (-135:0.4) coordinate(down4\x);  
   }
  \draw[red] (14,0)--(14,0);
     \draw[smooth ]
(down46)   to [out=150,in=-150]  
(up16)to [out=30,in=180] (10,2) to  [out=0,in=150] (up2fourteen) to [out=-30,in=30] (down3fourteen)
to [out=-150,in=0] (10,1) to  [out=180,in=-30] (down46)
;

     \draw [fill=white] plot [smooth cycle]
  coordinates {(up18) (up28) (down38)   (down48)  };
         \draw(8,0)--(8,3);

%        
%     \draw [fill=white] plot [smooth cycle]
%  coordinates {(up16) (up26) (down36)   (down46)  };

     \draw [fill=white] plot [smooth cycle]
  coordinates {(up14) (up24) (down34)   (down44)  };
         \draw(4,0)--(4,3);

     \draw [fill=white] plot [smooth cycle]
  coordinates {(up10) (up20) (down30)   (down40)  };
         \draw(0,0)--(0,3);

      \draw[red] (14,0)--(14,0);
    
     \draw [fill=white] plot [smooth cycle]
  coordinates {(up1twelve) (up2twelve) (down3twelve)   (down4twelve)  };
         \draw(12,0)--(12,3);
%    \draw(10,0)--(10,3);

             \draw(16,0)--(16,3);
     \draw  plot [smooth cycle]
  coordinates {(up1sixteen) (up2sixteen) (down3sixteen)   (down4sixteen)  };
         \draw(16,0)--(16,3);

             \draw(20,0)--(20,3);
       \draw  plot [smooth cycle]
  coordinates {(up1twenty) (up2twenty) (down3twenty)   (down4twenty)  };
         \draw(20,0)--(20,3);

     \foreach \x in {0,4,6,8,12,14,16,20}
     {   \fill[white](\x,3) circle (4pt);   
                \fill[white](\x,0) circle (4pt);    
                        \draw (\x,3) node {$\bullet$}; 
                           \draw (\x,0) node {${\bullet}$};  }
                         
	\draw (2*3,4) node {$\scriptstyle i$};\draw (2*7,4)
        node {$\scriptstyle j$};
	\draw (2*3,-1) node {$\scriptstyle \overline i$};\draw (2*7,-1)
        node {$\scriptstyle \overline j$};
\draw (2*1,3) node {$\ldots$};\draw (2*1,0) node {$\ldots$};
\draw (2*5,3) node {$\ldots$};\draw (2*5,0) node {$\ldots$};
\draw (2*9,3) node {$\ldots$};\draw (2*9,0) node {$\ldots$};

        \draw  (6,0)--(6,3);
         
          \draw  (14,0)--(14,3);
        
\draw[white,fill=white](9.4,2.5) rectangle (10.6,0.5);

     \draw[smooth,densely dotted ]
(down46)   to [out=150,in=-150]  
(up16)to [out=30,in=180] (10,2) to  [out=0,in=150] (up2fourteen) to [out=-30,in=30] (down3fourteen)
to [out=-150,in=0] (10,1) to  [out=180,in=-30] (down46)
;   
         \end{tikzpicture}  $$
        
\caption{Important diagrams for $1\leq i \neq j\leq r$.  Note that the diagrams $p_i$ and $p_{i,j}$ are simply  the  generators of the partition algebra  embedded 
 as a subalgebra of the ramified partition algebra. \color{red}\color{black} We have drawn two propagating lines for the inner block 
 $\{i, \overline{i}, j,  \overline{j}\}$
  of $p_{i,j}$, going against   our  conventions from \cref{convetnions}, as we think this makes the block structure  clearer to the reader. }
\label{annihilators}
\end{figure}

\begin{eg}
The square of the third diagram 
\cref{draw2} is equal to $\delta_{\rm in}  $ times itself  
and the square of the seventh diagram in \cref{draw2} is equal to 
$\delta_{\rm in}  \delta_{\rm out}$ times itself.  
\end{eg}

  It is apparent that the ramified partition algebra  $ \ram _r(\delta_{\rm in} , \delta_{\rm out} )$ contains a subalgebra isomorphic to the partition algebra  $P_r(\delta_{\rm in}   \delta_{\rm out} )$ whose parameter is the product of the two original parameters: simply take the span of basis elements $(d_{\Lambda} , d_{\Lambda })$
  whose inner and outer partitions are identical.

\begin{eg}
The first, fourth, and seventh diagrams in \cref{draw2}  belong  to the subalgebra
 $P_2(\delta_{\rm in}    \delta_{\rm out} )$ of $\ram _2(\delta_{\rm in} , \delta_{\rm out} )$.
\end{eg}

% \subsection{The generators of the ramified partition algebra}
In addition to the Coxeter generators $s_i$ for $1\leq i <r$ (embedded via the partition algebra embedding, see for example the first diagram of \cref{draw2}), we shall also require the diagrams in  \cref{annihilators}. 
These diagrams, together with the  Coxeter generators,
generate the algebra $\ram _r(\delta_{\rm in} ,\delta_{\rm out} )$, see \cite[Theorem 3.1.3]{MR2287557}.  
   
 We also find subalgebras of $\ram _r(\delta_{\rm in} ,\delta_{\rm out} )$ isomorphic to  group algebras of wreath products. Suppose $a, b$ are positive integers with $ab= r$. Then 
$\C \W_a \wr \W_b$ is a subalgebra of  $\ram _{ab}(\delta_{\rm in} ,\delta_{\rm out} )$. 
Diagrammatically, elements of  $\W_a \wr \W_b$ can be 
visualised as the ramified diagrams with $b$ consecutive
  outer blocks each consisting of $a$ inner propagating pairs.  The element
$(\sigma_1, \dots, \sigma_b; \pi) \in   \W_a \wr \W_b$ is visualised as the ramified diagram with outer blocks
\[ 
\bigl\{(j-1)a+1,(j-1)a+2,\dots , ja,\overline{(\pi(j)-1)a+1}, \overline{(\pi(j)-1)a+2} \dots,\overline{\pi(j)a} 
\bigr\},\] 
for $j=1, \dots b$, and inner blocks
$\bigl\{(j-1)a+i,\overline{(\pi(j)-1)a+\sigma_j(i)}\bigr\}$, for $i=1, \dots a$.
  An example for $a=2$ and $b=3$ is  depicted in \cref{draw}.  

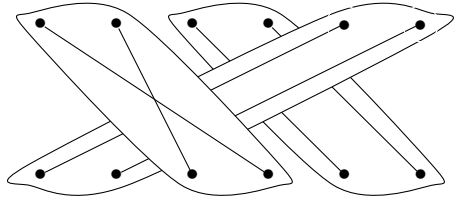
\begin{figure}[ht!] 
 $$ 
   \scalefont{0.8}
%\begin{minipage}{9cm}
 \begin{tikzpicture}[scale=0.5] %MkW Oct this somehow got changed to 14 and pushed all later figures to end! 
  
       \foreach \x in {0,2,4,8,10,6}
     {   \path(\x,4) coordinate (up\x);  
      \path(\x,0) coordinate (down\x);  
   }

  {   \path(down10) --++ (135:0.4) coordinate(down1ten);  
      \path(down10) --++ (45:0.4) coordinate(down2ten);
            \path(down10) --++ (-45:0.4) coordinate(down3ten);  
                        \path(down10) --++ (-135:0.4) coordinate(down4ten);  
   }

     {   \path(up10) --++ (135:0.4) coordinate(up1ten);  
      \path(up10) --++ (45:0.4) coordinate(up2ten);
            \path(up10) --++ (-45:0.4) coordinate(up3ten);  
                        \path(up10) --++ (-135:0.4) coordinate(up4ten);  
   }

     \foreach \x in {0,2,4,8,10,6}
     {   \path(up\x) --++ (135:0.4) coordinate(up1\x);  
      \path(up\x) --++ (45:0.4) coordinate(up2\x);
            \path(up\x) --++ (-45:0.4) coordinate(up3\x);  
                        \path(up\x) --++ (-135:0.4) coordinate(up4\x);  
   }

    \foreach \x in {0,2,4,8,10,6}
     {   \path(down\x) --++ (135:0.4) coordinate(down1\x);  
      \path(down\x) --++ (45:0.4) coordinate(down2\x);
            \path(down\x) --++ (-45:0.4) coordinate(down3\x);  
                        \path(down\x) --++ (-135:0.4) coordinate(down4\x);  
   }
         
       \draw [fill=white] plot [smooth cycle]
  coordinates {(up14) (up26) (down2ten) (down3ten) (down48) (up44)};

          \draw(4,4)--(8,0);                       
    \draw(6,4)--(10,0);                         

     \draw [fill=white] plot [smooth cycle]
  coordinates {(up18) (up2ten) (up3ten)   (down32) (down40) (down10)  };

        \draw(8,4)--(0,0);                       
    \draw(10,4)--(2,0);

   \draw [fill=white] plot [smooth cycle]
  coordinates {(up10) (up22) (down26) (down36) (down44) (up40)};
 \draw(0,4)--(6,0);                       
    \draw(2,4)--(4,0);       

     \foreach \x in {0,2,4,8,10,6}
     {   \fill[white](\x,4) circle (4pt);   
                \fill[white](\x,0) circle (4pt);   }
%      \draw (12,4) node {$ \color{white}\encircle{\color{black} \four}  $};   
      \draw (10,4) node { \color{white}\encircle{\color{black} \six}  };   
                  \draw (8,4) node {\color{white}\encircle{\color{black} \five}}; 

      \draw (6,4) node { \four  };   
      \draw (4,4) node { \three  };   
                  \draw (2,4) node { \two}; 
                           \draw (0,4) node { \one};

      \draw (10,0) node { \six  };   
                  \draw (8,0) node { \five}; 

      \draw (6,0) node { \four  };   
      \draw (4,0) node { \three  };   
                  \draw (2,0) node { \two}; 
                           \draw (0,0) node { \one}; 
  
         \end{tikzpicture}%\end{minipage}
$$

\vspace*{-12pt}
\!\!\!\!\!\!\!
\caption{The visualisation of $\bigl((12),1_{\W_2},1_{\W_2}; (1,2,3)\bigr)\in \W_2 \wr \W_3$ as a ramified $(6,6)$-set-partition.}
\label{draw}
\end{figure}

\subsection{Horizontal concatenation}Given a ramified $(r_1,s_1)$-set-partition 
$(d_{\Lambda_1}, d_{\Lambda_1'})$ and a  ramified $(r_2,s_2)$-set-partition 
$(d_{\Lambda_2}, d_{\Lambda_2'})$, we define their \textsf{horizontal concatenation}
  $(d_{\Lambda_1}, d_{\Lambda_1'})\circledast 
  (d_{\Lambda_2}, d_{\Lambda_2'})
  $  in the analogous fashion to the partition algebra case (see Section~\ref{hoz}).
  % (by horizontally connecting both the inner and outer blocks).  
Note that  the resulting diagram is indeed a ramified $(r_1+r_2,s_1+s_2)$-set-partition. 

\subsection{Propagating indices and a filtration of the ramified partition algebra}
\label{subsec:propagatingIndex}
Following \cite{MR2073453}, we define the {\sf propagating index} of a ramified $(r,s)$-set-partition $(\Lambda, \Lambda')$  %of $\ram _r(\delta_{\rm in} , \delta_{\rm out} )$
%. We set  $$\# (\Lambda,\Lambda')=
to be  $(a_1, a_2, \ldots, a_k)$
 if the outer partition $\Lambda'$ has $k$ propagating blocks and, within the $i\textsuperscript{th}$ such block of $\Lambda'$, the inner partition $\Lambda$ has $a_i$ propagating blocks, for $i=1, \ldots, k$. 
 We arrange the numbers of the propagating index so that $a_1 \ge a_2 \ge \cdots \ge a_k \ge 0$.  An example is depicted in \cref{propindex}.

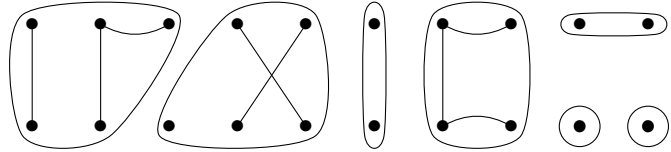
\begin{figure}[ht!]
 $$\begin{tikzpicture}[scale=0.45]
  
       \foreach \x in {0,2,4,6,8,10,12,14,16,18,20}
     {   \path(\x,3) coordinate (up\x);  
      \path(\x,0) coordinate (down\x);  
   }

  {   \path(down10) --++ (135:0.4) coordinate(down1ten);  
      \path(down10) --++ (45:0.4) coordinate(down2ten);
            \path(down10) --++ (-45:0.4) coordinate(down3ten);  
                        \path(down10) --++ (-135:0.4) coordinate(down4ten);  
   }

     {   \path(up10) --++ (135:0.4) coordinate(up1ten);  
      \path(up10) --++ (45:0.4) coordinate(up2ten);
            \path(up10) --++ (-45:0.4) coordinate(up3ten);  
                        \path(up10) --++ (-135:0.4) coordinate(up4ten);  
   }

  {   \path(down12) --++ (135:0.4) coordinate(down1twelve);  
      \path(down12) --++ (45:0.4) coordinate(down2twelve);
            \path(down12) --++ (-45:0.4) coordinate(down3twelve);  
                        \path(down12) --++ (-135:0.4) coordinate(down4twelve);  
   }

     {   \path(up12) --++ (135:0.4) coordinate(up1twelve);  
      \path(up12) --++ (45:0.4) coordinate(up2twelve);
            \path(up12) --++ (-45:0.4) coordinate(up3twelve);  
                        \path(up12) --++ (-135:0.4) coordinate(up4twelve);  
   }

  {   \path(down14) --++ (135:0.4) coordinate(down1fourteen);  
      \path(down14) --++ (45:0.4) coordinate(down2fourteen);
            \path(down14) --++ (-45:0.4) coordinate(down3fourteen);  
                        \path(down14) --++ (-135:0.4) coordinate(down4fourteen);  
   }

     {   \path(up14) --++ (135:0.4) coordinate(up1fourteen);  
      \path(up14) --++ (45:0.4) coordinate(up2fourteen);
            \path(up14) --++ (-45:0.4) coordinate(up3fourteen);  
                        \path(up14) --++ (-135:0.4) coordinate(up4fourteen);  
   }

  {   \path(down16) --++ (135:0.4) coordinate(down1sixteen);  
      \path(down16) --++ (45:0.4) coordinate(down2sixteen);
            \path(down16) --++ (-45:0.4) coordinate(down3sixteen);  
                        \path(down16) --++ (-135:0.4) coordinate(down4sixteen);  
   }

     {   \path(up16) --++ (135:0.4) coordinate(up1sixteen);  
      \path(up16) --++ (45:0.4) coordinate(up2sixteen);
            \path(up16) --++ (-45:0.4) coordinate(up3sixteen);  
                        \path(up16) --++ (-135:0.4) coordinate(up4sixteen);  
   }

  {   \path(down18) --++ (135:0.4) coordinate(down1eighteen);  
      \path(down18) --++ (45:0.4) coordinate(down2eighteen);
            \path(down18) --++ (-45:0.4) coordinate(down3eighteen);  
                        \path(down18) --++ (-135:0.4) coordinate(down4eighteen);  
   }

     {   \path(up18) --++ (135:0.4) coordinate(up1eighteen);  
      \path(up18) --++ (45:0.4) coordinate(up2eighteen);
            \path(up18) --++ (-45:0.4) coordinate(up3eighteen);  
                        \path(up18) --++ (-135:0.4) coordinate(up4eighteen);  
   }

  {   \path(down16) --++ (135:0.4) coordinate(down1twenty);  
      \path(down16) --++ (45:0.4) coordinate(down2twenty);
            \path(down16) --++ (-45:0.4) coordinate(down3twenty);  
                        \path(down16) --++ (-135:0.4) coordinate(down4twenty);  
   }

     {   \path(up16) --++ (135:0.4) coordinate(up1twenty);  
      \path(up16) --++ (45:0.4) coordinate(up2twenty);
            \path(up16) --++ (-45:0.4) coordinate(up3twenty);  
                        \path(up16) --++ (-135:0.4) coordinate(up4twenty);  
   }

     \foreach \x in {0,2,4,6,8,10,12,14,16,18,20}
     {   \path(up\x) --++ (135:0.4) coordinate(up1\x);  
      \path(up\x) --++ (45:0.4) coordinate(up2\x);
            \path(up\x) --++ (-45:0.4) coordinate(up3\x);  
                        \path(up\x) --++ (-135:0.4) coordinate(up4\x);  
   }

    \foreach \x in {0,2,4,6,8,10,12,14,16,18,20}
     {   \path(down\x) --++ (135:0.4) coordinate(down1\x);  
      \path(down\x) --++ (45:0.4) coordinate(down2\x);
            \path(down\x) --++ (-45:0.4) coordinate(down3\x);  
                        \path(down\x) --++ (-135:0.4) coordinate(down4\x);  
   }

   \draw [fill=white] plot [smooth cycle]
  coordinates {(up10) (up24) (down32)   (down40)  };

%  coordinates {(up10) (up22) (down22) (down32) (down40) (up40)};

 \draw(0,3)--(0,0);                       
    \draw(2,3) to [out=-30,in=-150] (4,3);
    \draw(4,3) --(2,0);

      \draw [fill=white] plot [smooth cycle]
  coordinates {(up16) (up28) (down38)   (down44)  };
 
    \draw(6,0)--(8,3);
        \draw(8,0)--(6,3);

     \draw [fill=white] plot [smooth cycle]
  coordinates {(up1ten) (up2ten) (down3ten)   (down4ten)  };
 
%    \draw(10,0)--(10,3);

     \draw [fill=white] plot [smooth cycle]
  coordinates {(up1twelve) (up2fourteen) (down3fourteen)   (down4twelve)  };

    \draw(12,0) to [out=30,in=150] (14,0);
    \draw(14,3) to [out=-150,in=-30] (12,3); 
    \draw(12,0) to   (14,3); 
    
    \draw  (16,0) node {\encircle{$\bullet$}};

    \draw  (18,0) node {\encircle{$\bullet$}};

     \draw [fill=white] plot [smooth cycle]
  coordinates {(up1sixteen) (up2eighteen) (up3eighteen)   (up4sixteen)  };

     \foreach \x in {0,2,4,6,8,10,12,14,16,18}
     {   \fill[white](\x,3) circle (4pt);   
                \fill[white](\x,0) circle (4pt);    
                        \draw (\x,3) node {$\bullet$}; 
                           \draw (\x,0) node {${\bullet}$};  }

         \end{tikzpicture}    $$

\caption{The diagram of a ramified $(10,10)$-set-partition with  % four propagating blocks.  
%The 
propagating index   $(2,2,1,0)$.  Note that the unique zero entry
in the propagating index records that there is a unique outer propagating block containing 
no inner propagating blocks. No information about non-propagating outer
blocks is recorded in the propagating index.
}
\label{propindex}
\end{figure}

%Note that for $(a_1, a_2, \ldots, a_k)$ to be a propagating index of a diagram of   $\ram _r(\delta_{\rm in} , \delta_{\rm out} )$, we require 
%$\sum_{a_i \ne 0} a_i + | \{ i \, : \, a_i=0\}| \le r.$
We let $\Theta_r$ denote the set of  all possible propagating indices for  $\ram _r(\delta_{\rm in} , \delta_{\rm out} )$. 
For example, if $r=2$ then $ \Theta_2=\{ (1,1),(2), (1,0), (1), (0,0), (0), \varnothing \}$; an example of a ramified diagram with each propagating index is  depicted in \cref{draw2}.  
%We denote by $\ram _r[\vartheta]$ the span  of all basis elements $(\Lambda,\Lambda') \in \ram _r(\delta_{\rm in} , \delta_{\rm out} )$ with $\#(\Lambda,\Lambda') = \vartheta$. For example, $$   \C (\W_m\wr \W_n)\subseteq \ram _{mn}[(m^n )] \subseteq  \ram _{mn}(\delta_{\rm in} ,\delta_{\rm out} ), $$ an example of this embedding for $m=2$ and $n=3$ is depicted in \cref{draw}.    
We write $\vartheta'<\vartheta $   if  $\vartheta'$ is obtained   by subtracting~$1$ from a single entry from $\vartheta $,  or
 if  $\vartheta'$ is obtained from $\vartheta$  by merging two parts into a single part, or finally if $\vartheta'=\varnothing$ and $\vartheta=(0 ) $.  
Abusing notation, we  
let $<$ denote the transitive closure of the above relation. 
Then $ \Theta_r$ is   partially ordered by $\le$.    
The poset $\Theta_3$ is illustrated in the \cref{poset} below; for example,  $(1,0,0)<(1,1,0)$ because we have subtracted 1 from the second entry of $(1,1,0)$ and $(2,0)<(1,1,0)$ because we have merged the first two parts of 
$(1,1,0)$.
Martin and Elgamal \cite[Proposition 6]{MR2073453} show that multiplication in  $\ram _r(\delta_{\rm in} , \delta_{\rm out} )$ preserves or decreases the 
propagating index under $\leq$. 

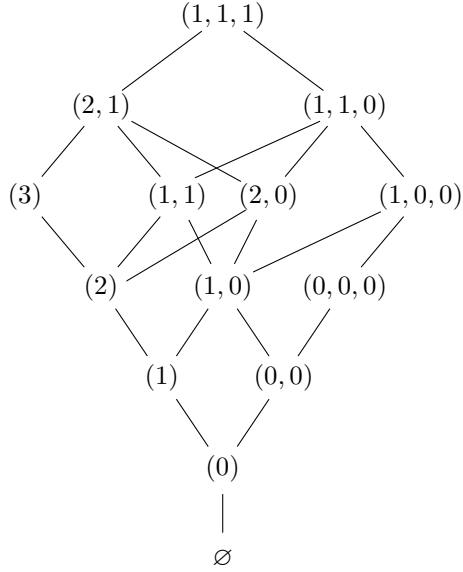
\begin{figure}[ht!]
 $$\scalefont{0.9}
 \begin{tikzpicture}[scale=0.8]

 \draw(0,0)-- (-2,-1.5);  \draw(0,0)--  (2,-1.5); 
 \draw(-2,-1.5)--(-0.75,-3)
 (-2,-1.5)--(-3.25,-3);

 \draw(-2,-4.5)--(-0.75,-3)
 (-2,-4.5)--(-3.25,-3);

\draw(3.25,-3)--(2,-1.5) ;
\draw(2,-1.5)-- (0.75,-3);

\draw(3.25,-3)--(2,-4.5) ;
 \draw(0.5,-4.25)-- (2.6,-3.25); %MkW minor tweaks 

\draw(0,-4.5)--(0.75,-3);
\draw(0,-4.5)--(-0.75,-3);

\draw(0,-4.5)--(1,-6);\draw(0,-4.5)--(-1,-6);
\draw(-2,-4.5)--(-1,-6);
\draw(2,-4.5)--(1,-6);

\draw(0,-7.5)--(1,-6);\draw(0,-7.5)--(-1,-6);

\draw(0,-9)--(0,-7.5);

\draw(1.6,-1.75)--(-0.7,-2.8);

  \draw(0.3,-2.75)--(-1.5,-1.75);
    \draw(0.75,-3)--(-1.75,-4.5);

 \draw[white,fill=white](0,0) circle (12pt) ; % {$(1,1,1)$};
 \draw[white,fill=white](-2,-1.5) circle (12pt) ; % {$(2,1)$};
  \draw[white,fill=white](2,-1.5) circle (12pt) ; % {$(1,1,0)$};

%  \draw(0.75,-3)--(-1.75,-1.5);

 \draw[white,fill=white](-0.75,-3) circle (12pt) ; % {$(1,1)$};
  \draw[white,fill=white](0.75,-3) circle (12pt) ; % {$(2,0)$};
 \draw[white,fill=white](-3.25,-3) circle (12pt) ; % {$(3)$};
  \draw[white,fill=white](3.25,-3) circle (12pt) ; % {$(1,0,0)$};

 \draw[white,fill=white](0,-4.5) circle (12pt) ; % {$(1,0)$};
  \draw[white,fill=white](2,-4.5) circle (12pt) ; % {$(0,0,0)$};
 \draw[white,fill=white](-2,-4.5) circle (12pt) ; % {$(1,0)$};
 
 \draw[white,fill=white](0,-6) circle (12pt) ; % {$(0,0,0)$};
 
  \draw[white,fill=white](0,-7.5) circle (12pt) ; % {$\varnothing$};

  \draw[white,fill=white](1,-6) circle (12pt) ; 
    \draw[white,fill=white](-1,-6) circle (12pt) ; 
 
    \draw[white,fill=white](0,-9) circle (12pt) ;

 \draw(0,0) node {$(1,1,1)$};
 \draw(-2,-1.5) node {$(2,1)$};
  \draw(2,-1.5) node {$(1,1,0)$};

 \draw(-0.75,-3) node {$(1,1)$};
  \draw(0.75,-3) node {$(2,0)$};
 \draw(-3.25,-3) node {$(3)$};
  \draw(3.25,-3) node {$(1,0,0)$};

 \draw(0,-4.5) node {$(1,0)$};
  \draw(2,-4.5) node {$(0,0,0)$};
 \draw(-2,-4.5) node {$(2)$};
 
 \draw(-1,-6) node {$(1)$};
  \draw(1,-6) node {$(0,0)$};
 
  \draw(0,-7.5) node {$(0)$};
  
    \draw(0,-9) node {$\varnothing$};

 \end{tikzpicture} $$
 
 \vspace*{-9pt}
\caption{The Hasse diagram for the poset $\Theta_3$}\label{poset}
\end{figure}

 \renewcommand{\Id}{e}

To each element $\vartheta =(a_1, a_2, \ldots a_k)\in \Theta_r$,  we have a canonically associated basis element, $\Id_\vartheta\in \ram _r(\delta_{\rm in} ,\delta_{\rm out} )$ constructed as follows (an example is depicted in \cref{thetaidemp}).
% Given $\vartheta  \in \Theta_r$ we define a basis element
%$I_\vartheta$ in the following way. 
The first outer block of $\Id_\vartheta$ consists of the leftmost $a_1$ (or $1$ if $a_1=0$) nodes from both top and bottom rows, the second outer block consists of the next $a_2$ nodes (or $1$ if $a_2=0$)  from both top and bottom rows, and so on until the $k\textsuperscript{th}$ outer block consists of the next $a_k$ nodes (or  $1$ if $a_k=0$) from top and bottom rows, and any remaining nodes form singleton outer blocks (necessarily containing a singleton inner block). Inside the $j\textsuperscript{th}$ outer propagating block, each top row node is joined to the bottom row node immediately below it, unless $a_j=0$ when the inner blocks are singletons. 
The basis element $\Id_\vartheta$ can be scaled to an idempotent
 (provided  $\delta_{\rm in} , \delta_{\rm out} \neq 0$).

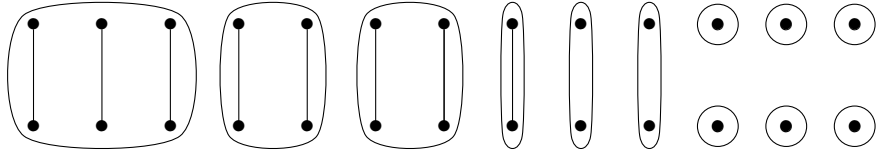
\begin{figure}[ht!]
 $$\begin{tikzpicture}[scale=0.45]
  
       \foreach \x in {0,2,4,6,8,10,12,14,16,18,20}
     {   \path(\x,3) coordinate (up\x);  
      \path(\x,0) coordinate (down\x);  
   }

  {   \path(down10) --++ (135:0.4) coordinate(down1ten);  
      \path(down10) --++ (45:0.4) coordinate(down2ten);
            \path(down10) --++ (-45:0.4) coordinate(down3ten);  
                        \path(down10) --++ (-135:0.4) coordinate(down4ten);  
   }

     {   \path(up10) --++ (135:0.4) coordinate(up1ten);  
      \path(up10) --++ (45:0.4) coordinate(up2ten);
            \path(up10) --++ (-45:0.4) coordinate(up3ten);  
                        \path(up10) --++ (-135:0.4) coordinate(up4ten);  
   }

  {   \path(down12) --++ (135:0.4) coordinate(down1twelve);  
      \path(down12) --++ (45:0.4) coordinate(down2twelve);
            \path(down12) --++ (-45:0.4) coordinate(down3twelve);  
                        \path(down12) --++ (-135:0.4) coordinate(down4twelve);  
   }

     {   \path(up12) --++ (135:0.4) coordinate(up1twelve);  
      \path(up12) --++ (45:0.4) coordinate(up2twelve);
            \path(up12) --++ (-45:0.4) coordinate(up3twelve);  
                        \path(up12) --++ (-135:0.4) coordinate(up4twelve);  
   }

  {   \path(down14) --++ (135:0.4) coordinate(down1fourteen);  
      \path(down14) --++ (45:0.4) coordinate(down2fourteen);
            \path(down14) --++ (-45:0.4) coordinate(down3fourteen);  
                        \path(down14) --++ (-135:0.4) coordinate(down4fourteen);  
   }

     {   \path(up14) --++ (135:0.4) coordinate(up1fourteen);  
      \path(up14) --++ (45:0.4) coordinate(up2fourteen);
            \path(up14) --++ (-45:0.4) coordinate(up3fourteen);  
                        \path(up14) --++ (-135:0.4) coordinate(up4fourteen);  
   }

  {   \path(down16) --++ (135:0.4) coordinate(down1sixteen);  
      \path(down16) --++ (45:0.4) coordinate(down2sixteen);
            \path(down16) --++ (-45:0.4) coordinate(down3sixteen);  
                        \path(down16) --++ (-135:0.4) coordinate(down4sixteen);  
   }

     {   \path(up16) --++ (135:0.4) coordinate(up1sixteen);  
      \path(up16) --++ (45:0.4) coordinate(up2sixteen);
            \path(up16) --++ (-45:0.4) coordinate(up3sixteen);  
                        \path(up16) --++ (-135:0.4) coordinate(up4sixteen);  
   }

  {   \path(down18) --++ (135:0.4) coordinate(down1eighteen);  
      \path(down18) --++ (45:0.4) coordinate(down2eighteen);
            \path(down18) --++ (-45:0.4) coordinate(down3eighteen);  
                        \path(down18) --++ (-135:0.4) coordinate(down4eighteen);  
   }

     {   \path(up18) --++ (135:0.4) coordinate(up1eighteen);  
      \path(up18) --++ (45:0.4) coordinate(up2eighteen);
            \path(up18) --++ (-45:0.4) coordinate(up3eighteen);  
                        \path(up18) --++ (-135:0.4) coordinate(up4eighteen);  
   }

  {   \path(down16) --++ (135:0.4) coordinate(down1twenty);  
      \path(down16) --++ (45:0.4) coordinate(down2twenty);
            \path(down16) --++ (-45:0.4) coordinate(down3twenty);  
                        \path(down16) --++ (-135:0.4) coordinate(down4twenty);  
   }

     {   \path(up16) --++ (135:0.4) coordinate(up1twenty);  
      \path(up16) --++ (45:0.4) coordinate(up2twenty);
            \path(up16) --++ (-45:0.4) coordinate(up3twenty);  
                        \path(up16) --++ (-135:0.4) coordinate(up4twenty);  
   }

     \foreach \x in {0,2,4,6,8,10,12,14,16,18,20}
     {   \path(up\x) --++ (135:0.4) coordinate(up1\x);  
      \path(up\x) --++ (45:0.4) coordinate(up2\x);
            \path(up\x) --++ (-45:0.4) coordinate(up3\x);  
                        \path(up\x) --++ (-135:0.4) coordinate(up4\x);  
   }

    \foreach \x in {0,2,4,6,8,10,12,14,16,18,20}
     {   \path(down\x) --++ (135:0.4) coordinate(down1\x);  
      \path(down\x) --++ (45:0.4) coordinate(down2\x);
            \path(down\x) --++ (-45:0.4) coordinate(down3\x);  
                        \path(down\x) --++ (-135:0.4) coordinate(down4\x);  
   }

   \draw [fill=white] plot [smooth cycle]
  coordinates {(up10) (up24) (down34)   (down40)  };

%  coordinates {(up10) (up22) (down22) (down32) (down40) (up40)};

 \draw(0,3)--(0,0);                       
 \draw(2,3)--(2,0);       
 \draw(4,3)--(4,0);       
    
      \draw [fill=white] plot [smooth cycle]
  coordinates {(up16) (up28) (down38)   (down46)  };
 
    \draw(6,0)--(6,3);
        \draw(8,0)--(8,3);

      \draw [fill=white] plot [smooth cycle]
  coordinates {(up1ten) (up2twelve) (down3twelve)   (down4ten)  };

     \draw [fill=white] plot [smooth cycle]
  coordinates {(up1fourteen) (up2fourteen) (down3fourteen)   (down4fourteen)  };

    \draw(6+4,0)--(6+4,3);
        \draw(8+4,0)--(8+4,3);
  \draw(6+4+2,0)--(6+4+2,3);
        \draw(8+4+2,0)--(8+4+2,3);

%\draw(2+6+4+2,0)--(2+6+4+2,3);
%        \draw(2+8+4+2,0)--(2+8+4+2,3);
%\draw(4+6+4+2,0)--(4+6+4+2,3);
%        \draw(4+8+4+2,0)--(4+8+4+2,3);

       \draw [fill=white] plot [smooth cycle]
  coordinates {(up1sixteen) (up2sixteen) (down3sixteen)   (down4sixteen)  };

      \draw [fill=white] plot [smooth cycle]
  coordinates {(up1eighteen) (up2eighteen) (down3eighteen)   (down4eighteen)  };

    \draw  (20,0) node {\encircle{$\bullet$}};

    \draw  (22,0) node {\encircle{$\bullet$}};
    
    \draw  (20,3) node {\encircle{$\bullet$}};

    \draw  (22,3) node {\encircle{$\bullet$}};

    \draw  (24,0) node {\encircle{$\bullet$}};

    \draw  (24,3) node {\encircle{$\bullet$}};

     \foreach \x in {0,2,4,6,8,10,12,14,16,18,20,22,24}
     {   \fill[white](\x,3) circle (4pt);   
                \fill[white](\x,0) circle (4pt);    
                        \draw (\x,3) node {$\bullet$}; 
                           \draw (\x,0) node {${\bullet}$};  }

         \end{tikzpicture}    $$

\vspace*{-3pt}
\caption{ The quasi-idempotent   $e_\vartheta$ for $\vartheta=(3,2^2,1,0^2)$ in $\ram_{13}(\delta_{\rm in} ,\delta_{\rm out} )$. Note that there are two outer-propagating blocks that are not inner-propagating from 
the two zero entries. Since $1. 3+ 2.2+1.1+2.1 = 10$, there are $3$ remaining northern and southern nodes forming
singleton outer blocks.
}
\label{thetaidemp}
\end{figure}

We fix $\prec$ to be  any total refinement of the ordering $\leq $ on $\Theta_r$.  
We set $$J_{ \preceq\vartheta}= \bigcup _{\vartheta'\preceq \vartheta}
 \ram _r(\delta_{\rm in} , \delta_{\rm out} ) \Id_{\color{red}\color{black}\vartheta '} \ram _r(\delta_{\rm in} , \delta_{\rm out} ) $$
and we define $J_{\prec \vartheta}$ in the obvious fashion. 
Similarly, we define 
$J_{\leq  \vartheta}$, $J_{<  \vartheta}$
in an analogous fashion in terms of the partial ordering $\leq$.  
Denote the list of all elements of $\Theta_r$ in  order as 
$$ \varnothing= \vartheta_1 \prec  \vartheta_2\dots \prec \vartheta_N=(1^r).$$  
Then we obtain the chain of ideals
\begin{equation}\label{eqn:chain}
0 \subset J_{\preceq\vartheta_1} \subset J_{\preceq\vartheta_2} \subset \cdots \subset  J_{\preceq\vartheta_N}= \ram _r(\delta_{\rm in} , \delta_{\rm out} ).
\end{equation} 
analogous to~\eqref{eq:Pfiltration}.
The quotient  $ J_{\preceq\vartheta_i} / J_{\preceq\vartheta_{i-1}} = J_{\preceq\vartheta_i} / J_{\prec\vartheta_i} $ has basis 
consisting of those ramified diagrams with propagating index $\vartheta_i$.
%$\ram_r [\vartheta_i]$.  
We observe, for 
  $\vartheta  =(a_1^{b_1}, \dots, a_\ell^{b_\ell}) \in\Theta_r$,   that 
\begin{equation}\label{eqn:inflation2}
 e_{\vartheta  }(J_{\preceq\vartheta } / J_{\prec \vartheta })e_{\vartheta }\cong   \C{\rm Stab}(\vartheta)={\C}  \prod_{i=1}^\ell \W_{a_i}\wr \W_{b_i}
\end{equation}
simply using the identification of group elements with ramified diagrams  illustrated in \cref{draw} for each $1\leq i \leq  \ell$.  
We set $V_r({\vartheta })$ to be the 
%$(R_r(\delta_{\rm in} ,\delta_{\rm out} ), {\rm Stab}(\vartheta ))$-bimodule  
$( {\rm Stab}(\vartheta ), R_r(\delta_{\rm in} ,\delta_{\rm out} ))$-bimodule 
\[V_r({\vartheta })= 
 e_{\vartheta } 
(J_{\preceq  {\vartheta }} / J_{\prec  {\vartheta }}  )   \]
 with basis given by all ramified  diagrams in
$ e_{\vartheta }R_r(\delta_{\rm in} ,\delta_{\rm out} ) $  with propagating index $\vartheta $.   
 For $\vartheta =(a_1^{b_1}, \dots, a_\ell^{b_\ell})$, we have that 
\begin{equation}\label{eqn:inflation}
 \frac{J_{\preceq\vartheta }}{ J_{\prec\vartheta }} \cong V_r{(\vartheta }) \otimes_{\C{\rm Stab}(\vartheta)}    
\C {\rm Stab}(\vartheta) \otimes_{\C{\rm Stab}(\vartheta)} V_r{(\vartheta )}.
\end{equation}
 The isomorphism  in \cref{eqn:inflation} is the key observation needed for the following theorem.

 \begin{thm} [{\cite[Proposition 11]{MR2073453}, \cite{davidthesis}}]
\label{quasi-hereditary}
The algebra $\ram _r(\delta_{\rm in} ,\delta_{\rm out} )$  is 
an iterated inflation  of the algebras $\C {\rm Stab}(\vartheta)$ for  
$\vartheta \in   \Theta_r$ and thus is a cellular algebra.    If $\delta_{\rm in} , \delta_{\rm out} \not=0$,
then
$\ram _r(\delta_{\rm in} ,\delta_{\rm out} )$  is quasi-hereditary. 
 \end{thm}

\begin{rmk}
The conditions for an iterated inflation in \cite[Theorem 1]{GreenCellular}  
are checked in \cite{davidthesis}.
 The group algebra 
$\C ( \prod_{i=1}^\ell \W_{a_i}\wr \W_{b_i})$ is cellular by the work of    \cite{GG} or
  \cite[Proposition 1, Theorem 4]{GreenCellular}.
 Thus $\ram _r(\delta_{\rm in} ,\delta_{\rm out} )$ is an iterated inflation of cellular algebras and hence a cellular algebra. 
If $\delta_{\rm in} ,\delta_{\rm out} \not =0$, then  $J_{\leq  \vartheta}$ (for $\vartheta\in \Theta_r$) is a heredity ideal containing the quasi-idempotent 
 $\Id_\vartheta$ which can be rescaled.   Hence the algebra is quasi-hereditary (see also \cite[Proposition 11]{MR2073453}).  
 \end{rmk}

\subsection{Standard and simple modules for the ramified partition algebra}
\label{subsec:standardSimpleModulesRamifiedPartitionAlgebra}
By Theorem~\ref{quasi-hereditary}, the ramified partition algebra has distinguished
cell modules. We shall
assume $\delta_{\rm in} ,\delta_{\rm out} \neq 0$ and therefore the cell modules are standard modules.
We describe these standard modules in generality here, 
although we require only a good understanding of the most elementary case when $\vartheta =(a^b)$.
 %We first recall the construction of the simple modules of 
Fix $\vartheta=(a_1^{b_1}, \dots, a_\ell^{b_\ell}) \in \Theta_r$.  
The simple right modules for $\C {\rm Stab}(\vartheta) $ of \cref{eqn:inflation} are  the outer tensor products of the (right module analogues of) modules from \cref{ohgood!} as follows: 
\begin{equation} \rightspecht{\bm \alpha ^{\bm \beta}} = \rightspecht{{ \bm \alpha_1 ^{\bm \beta_1} } }\otimes  \rightspecht{{ \bm \alpha _2^{\bm \beta_2} } }
\otimes  \dots  \otimes  \rightspecht{{ \bm \alpha_\ell ^{\bm \beta_\ell} } },
\label{eq:wreathModule} \end{equation}
where ${\bm \alpha ^{\bm \beta}}
%= (\bm\alpha_1 ^{\bm \beta_1}, \dots , \bm\alpha_\ell ^{\bm \beta_\ell}		)
$ 
{\color{red}\color{black}is}
an $\ell$-tuple of $\bm\alpha_i ^{\bm \beta_i} \in \ParSet(a_i,b_i)$.  
 We define the (right) {\sf standard}   $R_r(\delta_{\rm in}, \delta_{\rm out})$-module, 
$\Delta_{r     }(\bm \alpha ^{\bm \beta})$, by
\begin{equation}\label{ramcell}
\Delta_{r  
}(\bm \alpha ^{\bm \beta}) \cong \rightspecht{\bm \alpha ^{\bm \beta}} \otimes_{{\rm Stab}(\vartheta)} V_r(\vartheta) 
\end{equation}
 where    
the action of $\ram_r(\delta_{\rm in}, \delta_{\rm out})$ is given as follows. 
%%%%%%%%%%
Let $v$ be a ramified partition diagram in $V_r(\vartheta)$, 
$x\in \color{red}\color{black} \rightspecht  { \bm \alpha ^{\bm \beta}} $ and 
$d$ be an {\color{red}\color{black}ramified $(r,r)$-set-partition diagram}. 
Concatenate  $v$ above $d$ to get $\delta_{\rm in} ^ {s}
\delta_{\rm out} ^{t}
 v'$ for some ramified partition diagram $v'$ and    $s,t \in \ZZ_{\geq0}$.
  If
   %\#(v')\prec 
the propagating index  of $v'$ is not equal to  
  $\vartheta $  (and so it has smaller index in the order $\prec$) then we set
$(x\otimes v)d=0$.
Otherwise we set $(x\otimes v)d=\delta_{\rm in} ^ {s}
\delta_{\rm out} ^{t} x\otimes v'$. 
% Note that if $\vartheta  =(a^b)$ with $ab=r$, then we have
%\begin{equation}\label{annihilate2*}
%\Delta_{r %, \delta
%}(\bm \alpha ^{\bm \beta})\cong \specht (\bm \alpha ^{\bm \beta})\otimes_{\W_a \wr 
%\W_b}V_r( a^b) = \specht (\bm \alpha ^{\bm \beta}),
%\end{equation}
% viewed as a $R_r(\delta_{\rm in}, \delta_{\rm out})$-module via trivial inflation. 
 As $\ram_r(\delta_{\rm in}, \delta_{\rm out})$ is a cellular algebra, the simple $\ram_r(\delta_{\rm in}, \delta_{\rm out})$-modules are quotients of the standard modules. If $\delta_{\rm in}, \delta_{\rm out} \neq 0$, then each standard module $\Delta_{r  
}(\bm \alpha ^{\bm \beta})$ has a simple quotient which we shall denote by $L_{r  
}(\bm \alpha ^{\bm \beta})$, and these are a complete set of non-isomorphic simple modules.
 
  We now restrict our attention to the case 
  $\vartheta =(a^b)$ (with $ab\leq r$ or $b \leq r$ in the case $a=0$). 
Ramified diagrams in $V_r(a^b)=  e_{(a^b) } 
(J_{\preceq  {(a^b) }} / J_{\prec  {(a^b) }}  )   $ will be identified with {\color{red}\color{black}ramified $(ab, r)$-set-partition diagrams} provided $a\ne 0$ (respectively {\color{red}\color{black}ramified $(b,r)$-set-partition diagrams} if $a=0$) of propagating index $(a^b)$. We simply delete the additional $r-ab$ (respectively $r-b$) northern outer singleton vertices. An example of a ramified diagram from $V_5(2^2)$ is shown in \cref{thetaidemp2}. We let $\alpha \vdash a$ and $\beta\vdash b$ and use the right $\C \W_a \wr \W_b$-module   $\rightspecht{\alpha} \oslash \rightspecht{\beta}$  to construct 
   $$\Delta_r(  \alpha ^{  \beta}) = (\rightspecht{\alpha} \oslash \rightspecht{\beta}  )
    \otimes_{\W_a \wr \W_b} V_r(a^b)
    = c_{\alpha^\beta}^\ast  \C \W_a \wr \W_b
    \otimes_{\W_a \wr \W_b} V_r(a^b).$$ 
   % 
%  
%   {\bf\noindent \red RP: Do we need this basis of the ramified cell module in general? I think we only need it when a=0 and then it is much simpler. I have edited this as it was written in terms of tableaux that we hadn't defined.}
%  \\ 
Let $\mathcal{S}({\alpha^\beta})$ be a set of elements of $\color{red}\color{black}\W_a \wr \W_b$ chosen so that 
$\{  c_{\alpha^\beta}^\ast \sigma \mid \sigma \in \mathcal{S}({\alpha^\beta}) \}$ is a basis of 
   $c_{\alpha^\beta}^\ast {\color{red}\color{black}(\C\W_a \wr \W_b)}= \rightspecht{\alpha} \oslash \rightspecht{\beta}$. This set has cardinality $|\Std(\alpha)|^b |\Std(\beta)|$. 
    We therefore obtain a basis of  $\Delta_r(  \alpha ^{  \beta})$: 
\begin{align}\label{ramified basis}
 \Bigg\{  c_{\alpha^\beta}^\ast \sigma d_{(\Lambda,\Lambda')}
\,\Bigg|\hskip0.5pt
 \begin{array}{l}
 %\stt_{\alpha^\beta} \sigma \in \Std(\alpha^\beta), 
 \sigma \in \mathcal{S}(\alpha^\beta)\\ %MkW removed commas
 (\Lambda,\Lambda') \textrm{ is a {\color{red}\color{black}ramified $(ab,r)$-set-partition} of propagating index } (a^b)\\
  \pi_{\Lambda'}=1_{ \W_b },
    \pi_{\Lambda'_j   \,\cap\, \Lambda }=1_{ \W_a }, 1\leq j \leq b
    \end{array}
   \Bigg\}.
 \end{align} 
In the case $\alpha =\varnothing$, we need to replace {\color{red}\color{black}ramified  $(ab,r)$-set-partitions} with 
{\color{red}\color{black}ramified 
$(b,r)$-set-partitions} in the above.
%In other words, 
%if we  draw  $d_{(\Lambda,\Lambda')}$ with the conventions of \cref{useful-convention} (for both inner and outer parts) 
%then 
% any crossing of  inner propagating lines can only occur between {\em distinct} blocks. 
An example is depicted in \cref{thetaidemp2}. 
   These standard modules $\Delta_r(  \alpha ^{  \beta})$, for $\alpha \vdash a$ and $\beta \vdash b$ 
   (with $ab\leq r$ or $b \leq r$ in the case $a=0$), and their simple quotients $L_r(  \alpha ^{  \beta})$,  are the only ramified partition algebra modules which will be of importance to the plethysm question. %The general case follows by  careful iterations of the following construction (this is left as an exercise for the enthusiastic reader).  

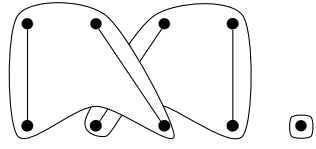
\begin{figure}[ht!]
\vspace*{-3pt}
 $$\begin{minipage}{4cm}\begin{tikzpicture}[scale=0.45]
  
       \foreach \x in {0,2,4,6,8,10,12,14,16,18,20}
     {   \path(\x,3) coordinate (up\x);  
      \path(\x,0) coordinate (down\x);  
   }

  {   \path(down10) --++ (135:0.4) coordinate(down1ten);  
      \path(down10) --++ (45:0.4) coordinate(down2ten);
            \path(down10) --++ (-45:0.4) coordinate(down3ten);  
                        \path(down10) --++ (-135:0.4) coordinate(down4ten);  
   }

     {   \path(up10) --++ (135:0.4) coordinate(up1ten);  
      \path(up10) --++ (45:0.4) coordinate(up2ten);
            \path(up10) --++ (-45:0.4) coordinate(up3ten);  
                        \path(up10) --++ (-135:0.4) coordinate(up4ten);  
   }

  {   \path(down12) --++ (135:0.4) coordinate(down1twelve);  
      \path(down12) --++ (45:0.4) coordinate(down2twelve);
            \path(down12) --++ (-45:0.4) coordinate(down3twelve);  
                        \path(down12) --++ (-135:0.4) coordinate(down4twelve);  
   }

     {   \path(up12) --++ (135:0.4) coordinate(up1twelve);  
      \path(up12) --++ (45:0.4) coordinate(up2twelve);
            \path(up12) --++ (-45:0.4) coordinate(up3twelve);  
                        \path(up12) --++ (-135:0.4) coordinate(up4twelve);  
   }

  {   \path(down14) --++ (135:0.4) coordinate(down1fourteen);  
      \path(down14) --++ (45:0.4) coordinate(down2fourteen);
            \path(down14) --++ (-45:0.4) coordinate(down3fourteen);  
                        \path(down14) --++ (-135:0.4) coordinate(down4fourteen);  
   }

     {   \path(up14) --++ (135:0.4) coordinate(up1fourteen);  
      \path(up14) --++ (45:0.4) coordinate(up2fourteen);
            \path(up14) --++ (-45:0.4) coordinate(up3fourteen);  
                        \path(up14) --++ (-135:0.4) coordinate(up4fourteen);  
   }

  {   \path(down16) --++ (135:0.4) coordinate(down1sixteen);  
      \path(down16) --++ (45:0.4) coordinate(down2sixteen);
            \path(down16) --++ (-45:0.4) coordinate(down3sixteen);  
                        \path(down16) --++ (-135:0.4) coordinate(down4sixteen);  
   }

     {   \path(up16) --++ (135:0.4) coordinate(up1sixteen);  
      \path(up16) --++ (45:0.4) coordinate(up2sixteen);
            \path(up16) --++ (-45:0.4) coordinate(up3sixteen);  
                        \path(up16) --++ (-135:0.4) coordinate(up4sixteen);  
   }

  {   \path(down18) --++ (135:0.4) coordinate(down1eighteen);  
      \path(down18) --++ (45:0.4) coordinate(down2eighteen);
            \path(down18) --++ (-45:0.4) coordinate(down3eighteen);  
                        \path(down18) --++ (-135:0.4) coordinate(down4eighteen);  
   }

     {   \path(up18) --++ (135:0.4) coordinate(up1eighteen);  
      \path(up18) --++ (45:0.4) coordinate(up2eighteen);
            \path(up18) --++ (-45:0.4) coordinate(up3eighteen);  
                        \path(up18) --++ (-135:0.4) coordinate(up4eighteen);  
   }

  {   \path(down16) --++ (135:0.4) coordinate(down1twenty);  
      \path(down16) --++ (45:0.4) coordinate(down2twenty);
            \path(down16) --++ (-45:0.4) coordinate(down3twenty);  
                        \path(down16) --++ (-135:0.4) coordinate(down4twenty);  
   }

     {   \path(up16) --++ (135:0.4) coordinate(up1twenty);  
      \path(up16) --++ (45:0.4) coordinate(up2twenty);
            \path(up16) --++ (-45:0.4) coordinate(up3twenty);  
                        \path(up16) --++ (-135:0.4) coordinate(up4twenty);  
   }

     \foreach \x in {0,2,4,6,8,10,12,14,16,18,20}
     {   \path(up\x) --++ (135:0.4) coordinate(up1\x);  
      \path(up\x) --++ (45:0.4) coordinate(up2\x);
            \path(up\x) --++ (-45:0.4) coordinate(up3\x);  
                        \path(up\x) --++ (-135:0.4) coordinate(up4\x);  
   }

    \foreach \x in {0,2,4,6,8,10,12,14,16,18,20}
     {   \path(down\x) --++ (135:0.4) coordinate(down1\x);  
      \path(down\x) --++ (45:0.4) coordinate(down2\x);
            \path(down\x) --++ (-45:0.4) coordinate(down3\x);  
                        \path(down\x) --++ (-135:0.4) coordinate(down4\x);  
   }

        \draw [fill=white] plot [smooth cycle]
  coordinates {(up14) (up26) (down36) (4 ,0.6) (3 ,0.6)  (down32) (down12)  };

 \draw(4,3)--(2,0);

    \draw(6,0)--(6,3);

   \draw [fill=white] plot [smooth cycle]
  coordinates {(up10) (up22) (down34) (2,0.6)  (down40)  };

 \draw(0,3)--(0,0);                       
 \draw(2,3)--(4,0);

   \draw [fill=white] plot [smooth cycle]
  coordinates {(down18) (down28) (down38)  (down48)  };

     \foreach \x in {0,2,4,6,6}
     {   \fill[white](\x,3) circle (4pt);   
                \fill[white](\x,0) circle (4pt);    
                        \draw (\x,3) node {$\bullet$}; 
                           \draw (\x,0) node {${\bullet}$};  }

     \foreach \x in {0,2,4,6,8 }
     {   
                           \draw (\x,0) node {${\bullet}$};  }
      
         \end{tikzpicture}   \end{minipage} $$

\vspace*{-6pt}
\caption{  An element $d_{(\Lambda,\Lambda')}$ appearing in 
the basis in \cref{ramified basis}.  Note that $\pi_{\Lambda}=(2,3)$ but $\pi_{\Lambda'_1\cap \Lambda}={\rm id}_{\W_2}=\pi_{\Lambda'_2\cap \Lambda}$, as required. 
}
\label{thetaidemp2}
\end{figure}

\begin{rmk}
The orbit-basis  of the ramified partition algebras has not yet been studied in the literature.  We do not  require this basis for our purposes, but we posit that it should be worthy of study.  In particular, it is natural 
in light of Benkart--Halverson's work \cite{MR3969570}
to expect that such orbit bases should exist and that they split into natural bases of the kernel and image of the ramified partition algebras acting on tensor space.
 \end{rmk}

\newcommand{\vf}{\mathsf{v}}\newcommand{\sff}{\mathsf{s}}
 
\section{Partition algebras, ramified partition algebras and Schur--Weyl duality}\label{SWsec}

In this section we review results from the literature regarding the Schur--Weyl dualities between    group algebras of symmetric groups and partition algebras, and between group algebras of wreath products of symmetric groups and ramified partition algebras.  
 
\subsection{Symmetric groups and partition  algebras}  Let $I(d,r) = \{1, \dots, d\}^r$ be the set of multi-indices. For a
given multi-index $i = (i_1, \dots, i_r) \in I(d,r)$, we put $e^i =
e^{i_1} \otimes \cdots \otimes e^{i_r}$. Then $\{ e^i \mid i \in I(d,r)
\}$ is a basis of tensor space $(\C  ^{d})^{\otimes r}$ over $\C  $.
% The action of the Weyl group $\W_{d}$  on $(\C  ^{d})^{\otimes r}$ is simply the restriction of the
%diagonal action of $\GL_{d}(\C  )$.
We define the diagonal  action $\Phi :    \W_{d}  \rightarrow \End((\C  ^{d})^{\otimes r})$ by 
\begin{equation} \label{eqn:tensor_action}
 \Phi(\sigma )  (e^{i_1} \otimes \cdots e^{i_r}) = e^{\sigma(i_1)} \otimes \cdots e^{\sigma(i_r)}
\end{equation} 
for any $\sigma \in \W_{d}$ and $i = (i_1, \dots, i_r) \in I(d,r)$. (The reason for using upper indices
will be seen in Section 6.)
Now  we consider the partition algebra $P_r(d)$ with parameter $d$ and define a right action of this algebra on tensor space. Let $\Lambda$ be a $(r,r)$-set-partition. 
Following \cite{MR3969570}, we define  $\Psi :  P_r(d) \rightarrow \End((\C  ^{d})^{\otimes r})$ on the orbit basis of \cref{refine2} by 
first setting   
\begin{equation}\label{eq:Phi-coeffs-orbit}
 (x_\Lambda)_{{i_{\overline{1}},\dots, i_{\overline{r}}}}^{ {i_1,\dots, i_r}}   = 
  \begin{cases}  1  & \quad \text{if $i_a = i_b$ if and only if %MkW: removed emphasis
   $a$ and $b$ are in the same block of $\Lambda$},  \\
 0  & \quad \text{otherwise},
  \end{cases} 
  \end{equation}  
  where $(i_{\overline{1}}, \dots,     i_{\overline{r}}) \in I(d,r)$ and $a,b$ run over $\{1,\dots, r, \overline{1}, \dots, \overline{r}\}$.
   We then set
   \begin{equation}\label{eq:Ptn-action}
(e^{i_1} \otimes \cdots \otimes e^{i_r})   \Psi  (x_\Lambda) = 
\sum_{(i_{\overline{1}}, \dots,     i_{\overline{r}}) \in I(d,r)} 
(x_\Lambda)^{i_1, \dots, i_r}_{i_{\overline{1}}, \dots,    i_{\overline{r}}} \,
 (e^{i_{\overline{1}}} \otimes \cdots \otimes e^{i_{\overline{r}}}).
\end{equation}
  Since the diagram basis   is related to the orbit basis  by 
 the refinement relation of \cref{refine1} and \cref{refine2}, we have as an immediate consequence, 
that, setting 
 \begin{equation}\label{eq:Phi-coeffs-diagram}
 (d_\Lambda)_{{i_{\overline{1}},\dots, i_{\overline{r}}}}^{ {i_1,\dots, i_r}}   =
  \begin{cases}  1  & \quad \text{if $i_a = i_b$ when %MkW: removed emphasis, I don't think the reader will
  %see what subtlety is being emphasised
  $a$ and $b$ are in the same block of $\Lambda$,}  \\
 0  & \quad \text{otherwise},
  \end{cases} 
  \end{equation} 
the  diagram basis element $d_\Lambda$ acts  on the right, by
the rule
\begin{equation}\label{eq:Ptn-action2}
(e^{i_1} \otimes \cdots \otimes e^{i_r})   \Psi  (d_\Lambda) = \sum_{(i_{\overline{1}}, \dots,
    i_{\overline{r}}) \in I(d,r)} (d_\Lambda)^{i_1, \dots, i_r}_{i_{\overline{1}}, \dots,
    i_{\overline{r}}} \, (e^{i_{\overline{1}}} \otimes \cdots \otimes e^{i_{\overline{r}}}).
\end{equation}
 We have constructed actions of the symmetric group and   partition algebra on tensor space as follows:
\begin{equation}\label{eq:theSituationAbove}
  \C   \W_{d}   \xrightarrow{\ \Phi  \ }
   \End_\C  \bigl((\C  ^{d})^{\otimes r}\bigr)
\xleftarrow{\ \Psi \ }  P_r(d).
\end{equation}

\begin{thm}[\cite{jones, marbook}]\label{mrjones}    In the situation 
of~\eqref{eq:theSituationAbove}, the image of each representation is
  equal to the full centraliser algebra for the other action. That is,
  \[
  \Phi (\C\W_{d}) = \End_{P_r(d)}\bigl((\C  ^{d})^{\otimes r}\bigr), \qquad 
  \Psi(P_r(d)) = \End_{\C\W_{d}}\bigl((\C  ^{d})^{\otimes r}\bigr).
  \]
Moreover %paragraph breaks in theorems make me feel ill. I have to change this to a thmlist. MkW
\begin{thmlist}
\item
  As a $(\C  \W_{d},P_r(d))$-bimodule, the tensor space decomposes as 
 \[
  (\C^{d})^{\otimes r} \cong \bigoplus  \leftspecht{\kappa[d]} \otimes L_{r }(\kappa) %no subscript for [d]
  \]
  where the sum is over all partitions $\kappa \in \ParSet(\leq r)$ with $\kappa_1 \le d-|\kappa|$.
  %of $d$  such that $|\lambda_{>1}|\leq r$.  
  \item
 For $d \geq 2r$, the   partition algebra $P_r(d)$ is isomorphic to 
 $\mathrm{End}_{\C\W_d}\bigl((\C^{d})^{\otimes r}\bigr)$ and
  acts faithfully on tensor space.
   Thus $\color{red}\color{black}P_r(d)$ 
  is a semisimple $\C$-algebra and the modules
$\{\Delta_r(\kappa ) \mid  \kappa \in \ParSet(\leq r) 
\} $
provide a complete set of non-isomorphic simple  $P_r(d)$-modules.
\end{thmlist}
\end{thm}

As an immediate corollary of the decomposition of tensor space we obtain a  Schur functor
from left symmetric group modules to right modules for the partition algebra.
This step is very familiar to experts but we give full details as we need 
an analogous argument as part of the proof of Proposition~\ref{SW-simples}
on the ramified partition algebra.

\begin{cor}\label{cor:partitionAlgebraSchurFunctor}
%Let $\kappa$ be a partition of $r$ such that $\kappa[d]$ is a partition. The functor
Let $\color{red}\color{black}\kappa\in \ParSet({\le r})$  be  such that $\kappa[d]$ is a partition. The functor
\[ \Hom_{ \C  \W_{d}}({-}, 
(\C^{d})^{\otimes r}): \C (\W_d)
{\rm -mod}\to  {\rm  mod-}P_r(d) \] 
sends the left Specht module $\leftspecht{\kappa[d]}$ to the simple module $L_r(\kappa)$
for the partition algebra $P_r(d)$.
\end{cor}

\begin{proof}
Let $\rho \in \ParSet({\le r})$ be such that $\rho[d]$ is a partition. 
Considered as a left $\C\W_d$-module, $\leftspecht{\rho[d]} \otimes L_r(\rho)$ is a direct
sum of $\dim L_r(\rho)$ copies of $\leftspecht{\rho[d]}$. Therefore
$\Hom_{\C\W_d}\bigl( \leftspecht{\kappa[d]}, \leftspecht{\rho[d]} \otimes L_r(\rho) \bigr) = 0$
unless $\rho = \kappa$; in the remaining case then, considering the right $P_r(d)$-action on the bimodule,
we have 
$\Hom_{\C\W_d}\bigl( \leftspecht{\kappa[d]}, \leftspecht{\kappa[d]} \otimes L_r(\kappa) \bigr) 
\cong L_r(\kappa)$ as $P_r(d)$-modules. The corollary now follows from the decomposition of
tensor space in Theorem~\ref{mrjones}(i).
\end{proof}

For subsequent work with tensor space it will be convenient
to associate a set-partition of $\{1,2, \ldots, r\}$ to each pure tensor. 
Given a set partition $P$, we write $i \sim_P i'$ if $i$ and $i'$ are in the same part of $P$. 
 
\begin{defn} \label{valuetype}
  We say that the basis vector
  $$e = e^{j_1} 
  \otimes 
  e^{j_2} 
  \otimes 
  \dots
  \otimes 
  e^{j_r} 
  \in (\C ^{d})^{\otimes r}$$ has
{\sf  value-type} 
% $  \Gamma' $ 
$S$ if 
%    $k \sim_{\Gamma'}l$ if and only if $j_k=j_l$.
   $k \sim_{S}l$ if and only if $j_k=j_l$.  
 %   We write  ${\rm val}(e)= \Gamma' $.
    We write  ${\rm val}(e)= S$.
    \end{defn}

For example,  $ e =e^1 \otimes e^1 \otimes e^1 \otimes e^2 \otimes e^3 \otimes e^2 \otimes e^3$ 
has ${\rm val}(v) =  \{\{1,2,3\}, \{4,6\}, \{5,7\}\}$.

\subsection{Wreath product groups and  ramified partition  algebras}
Recall from Section~\ref{subsec:wreathProducts} that, following~\cite[Section 4.1]{jk}, we have defined
\[ \W_m \wr \W_n =\bigl\{ 
(\sigma_1,\sigma_2,\ldots, \sigma_n ; \pi) \mid \sigma_i \in  \Sym _m, i=1,\ldots, n, \pi \in \Sym _n\bigr\},
\]
which we identify with a subgroup of $ \Sym _{mn}$ via the embedding \cref{eqn:wreath_embedding}.
%\begin{equation}\label{eqn:wreath_embedding}
%(\sigma_1,\sigma_2,\ldots, \sigma_n ; \pi ) \mapsto \left( \begin{array}{c} (j-1)m+i \\ (\pi(j)-1)m+\sigma_{\pi(j)}(i) \end{array}\right)_{i=1,\ldots, m, j=1\ldots, n}.
%\end{equation}
   The diagonal 
   action~(\ref{eqn:tensor_action}) of $\W_{mn}$ on tensor space $(\C  ^{mn})^{\otimes r}$ restricts to an action $\Phi : \W_m \wr \W_n \rightarrow \End\bigl( (\mathbb{C}^{mn})^{\otimes r} \bigr)$ 
   of $\W_m \wr \W_n$. %using the  identification~(\ref{eqn:wreath_embedding})).     
 Having chosen our wreath product subgroup in the fashion above, we let this guide our choice of a new labelling set for the basis of tensor space   as follows.  For  $1\leq i \leq m$ and $1\leq j \leq n$, we set
   $$v^j_i = e^{(j-1)m+i} ,$$ 
 and we note that    
 \begin{equation}
 \label{eqn:action}
(\sigma_1,\sigma_2,\dots, \sigma_n ; \pi)    (v_{i_1}^{j_1}\otimes     v_{i_2}^{j_2} \otimes \dots \otimes     v_{i_r}^{j_r})
=
 v^{\pi(j_1)}_{\sigma_{\pi(j_1)}(i_1)}
 \otimes 
  v^{\pi(j_2)}_{\sigma_{\pi(j_2)}(i_2)}
  \otimes 
  \dots
  \otimes
   v^{\pi(j_r)}_{\sigma_{\pi(j_r)}(i_r)}.  
  \end{equation} 
On the other hand, there is an action $\color{red}\color{black}\Psi: R_r(m,n)\rightarrow \End\bigl( (\mathbb{C}^{mn})^{\otimes r} \bigr)$ 
of the ramified partition algebra $R_r(m,n)$ with parameters $ \delta_{\rm in} =m$ and $ \delta_{\rm out}  =n$ on the right on tensor space. This is given by
identifying a ramified diagram
$d_{(\Lambda,\Lambda')} \in P_r(m,n)$ with the pair of 
elements $d_\Lambda\in P_r(m)$ and $d_{\Lambda'}\in P_r(n)$ and then acting by these  elements
as in \cref{eq:Phi-coeffs-diagram} 
on the subscripts and superscripts, respectively:
\begin{equation}\label{eq:RPAaction}
(v^{j_1}_{i_1} \otimes \cdots \otimes v^{j_r}_{i_r})   \Psi  (d_{(\Lambda, \Lambda')}) = 
\sum_{
\begin{subarray}c
(i_{\overline{1}}, \dots,
    i_{\overline{r}}) \in I(m,r)
   \\
(j_{\overline{1}}, \dots,
    j_{\overline{r}}) \in I(n,r)  
\end{subarray}
} 
 (d_\Lambda)^{i_1, \dots, i_r}_{i_{\overline{1}}, \dots,
    i_{\overline{r}}}
    ({\color{red}\color{black}d_{\Lambda'}})^{j_1, \dots, j_r}_{j_{\overline{1}}, \dots,
    j_{\overline{r}}}
     \, (v^{j_{\overline{1}}}_{i_{\overline{1}}} \otimes \cdots \otimes v^{j_{\overline{r}}}_{j_{\overline{1}}}).
\end{equation}

%For example, the generators $\mathsf{p}_{1}^{(2)}$ and $\mathsf{p}_{1,2}^{(2)}$ (as in \cref{annihilators}) act on a pure tensor as follows:
%$$(v_{i_1}^{j_1}\otimes     v_{i_2}^{j_2} \otimes \dots \otimes     v_{i_r}^{j_r})\mathsf{p}_{1}^{(2)} =\sum_{i=1}^m  v_{i}^{j_1}\otimes     v_{i_2}^{j_2} \otimes \dots \otimes     v_{i_r}^{j_r},$$
%$$(v_{i_1}^{j_1}\otimes     v_{i_2}^{j_2} \otimes \dots \otimes     v_{i_r}^{j_r})\mathsf{p}_{1,2}^{(2)} = \begin{cases} 0 & \textrm{if } j_1 \ne j_2\\
% v_{i_1}^{j_1}\otimes     v_{i_2}^{j_2} \otimes \dots \otimes     v_{i_r}^{j_r} &\textrm{if } j_1 = j_2.\end{cases}$$
The left action $\Phi$ of $\W_m \wr \W_n$ and the right action $\Psi$ of $R_r(m,m)$ on
$(\C  ^{mn})^{\otimes r}$ commute and we have the following analogue of 
Theorem~\ref{mrjones}.

 \begin{thm} [{\cite[Corollary 3.3.3]{MR2287557}}]\label{semisimpleee}
% We set $
% \delta_{\rm in} =m,
% \delta_{\rm out}  =n
% \in 
% \NN$. 
   In the situation outlined above, the image of each representation is
  equal to the full centraliser algebra for the other action. That is,
  \[
  \Phi \bigl( \C   \W_{m} \wr \W_n \bigr) = \End_{R_r(m,n)}\bigl(
  (\C  ^{mn})^{\otimes r}\bigr), \qquad 
  \Psi\bigl(\ram_r(m,n)\bigr) = \End_{\C  \W_{m} \wr \W_n}
  \bigl((\C  ^{mn})^{\otimes r}\bigr).
  \]
 For parameters $m \geq 2r$ and	 $n\geq 2r$, the ramified  partition algebra $\ram _r(m,n)$
 is isomorphic to $\End_{\C  ( \W_{m} \wr \W_n)}\bigl((\C  ^{mn})^{\otimes r}\bigr)$ and
  acts faithfully on tensor space.  Therefore  $\ram _r(m,n )$  is a semisimple $\C $-algebra and the modules
\[ \bigl\{\Delta_r(\bm \alpha ^{\bm \beta} ) \mid 
% \bm \alpha ^{\bm \beta} \in 
\text{$\vartheta=(a_1^{b_1}, \dots, a_\ell^{b_\ell})\in \Theta_r$ and $\bm \alpha ^{\bm \beta}$ an $\ell$-tuple of $\bm\alpha_i ^{\bm \beta_i} \in \ParSet(a_i,b_i)$ for 
$1\leq i \leq \ell$}  
\bigr\} \]
provide a complete set of non-isomorphic simple  $\ram _r(m,n )$-modules.
 \end{thm}
 
We now delve a little deeper into the combinatorics of tensor space.  % as this will be helpful in what follows. 
In the following definition we associate a ramified set-partition of $\{1,2, \ldots, r\}$ to each pure tensor.
%Given a set partition $P$, we write $i \sim_P i'$ if $i$ and $i'$ are in the same part of $P$.
%MkW added definition of \sim

\begin{defn} \label{Rvaluetype}
  We say that the pure tensor
  $$v = v^{j_1}_{i_1}
  \otimes 
  v^{j_2}_{i_2}
  \otimes 
  \dots
  \otimes 
  v^{j_r}_{i_r}
  \in (\C ^{mn})^{\otimes r}$$ has
{\sf ramified value-type}  $(R,S)$ if
%$(\Gamma, \Gamma')$ if 
%    $k \sim_{\Gamma'}l$ if and only if $j_k=j_l$ and 
%   $k \sim_{\Gamma}l$ if and only if $j_k=j_l$ and $i_k=i_l$. We write  ${\rm ramval}(v)=(\Gamma, \Gamma')$.  Note that $\Gamma < \Gamma'$.
     $k \sim_{S}l$ if and only if $j_k=j_l$ and 
   $k \sim_{R}l$ if and only if $j_k=j_l$ and $i_k=i_l$. We write  ${\rm ramval}(v)=(R, S)$.  Note that $R \leq S$.
    \end{defn}

   \begin{eg}  For example, the pure tensor
$ v =v_2^1 \otimes v_1^1 \otimes v_1^1 \otimes v_3^2 \otimes v_2^3 \otimes v_3^2 \otimes v_3^3$ 
has 
\[{\rm ramval}(v) =(R,S)= \bigl( \bigl\{\{1\}, \{2,3\}, \{4,6\},\{5\},\{7\}\}, \{\{1,2,3\}, \{4,6\}, \{5,7\}\bigr\}\bigr).
\]
To obtain $S= \{\{1,2,3\}, \{4,6\}, \{5,7\}\}$  note that 
the superscripts match in positions 1,2,3 and they match in positions 4 and 6 and they also match in positions 5 and 7.
 Although the subscripts match in positions $1$ and $5$,  the superscripts do not match and so $1\nsim_{R} 5$.  
  \end{eg}
 
%\begin{eg} \label{ex:4,5}Let $m=4$ and $n=r=5$ and $$ {(\Gamma,\Gamma')}= 
%(  \{\{1,2,4\},\{3\},\{5\}\}, \{\{1,2,3,4\},\{5\}\} ).$$  Then
% $$
% v_ {(\Gamma,\Gamma')}= \sum_{
% \begin{subarray}c
% 1\leq i_1, i_2,i_3 \leq 4\\
%   1\leq j_1, j_2 \leq 5\\
% i_1 \neq i_2, j_1\neq j_2 \\  
% \end{subarray}
% }
% v_{i_1}^{j_1}\otimes  v_{i_1}^{j_1}\otimes  v_{i_2}^{j_1}\otimes  v_{i_1}^{j_1}
% \otimes  v_{i_3}^{j_2}.
% $$
%\end{eg}
% \color{black}

  \begin{defn}\label{minimal vector}
Fix a ramified value-type $(R,S)$.
 %$(\Gamma,\Gamma')$.
%Let 
% $  j_1,\dots,j_r $
% be a tuple such that $j_k=j_l$ if and only if $k \sim_{\Gamma'}  l$. 
  We define the associated {\sf minimal %$\Gamma$
  $R$ value-type tuple} 
%  $$
%  v^{j_1}_{i_1^\ast}\otimes   v^{j_2}_{i_2^\ast}
%  \otimes \dots   v^{j_r}_{i_r^\ast}
%  $$
%   to be given by 
%%allowing the entries $j_1,\dots j_r$ to be generic, but 
%specifying 
$(i_1^\ast,\dots, i_r^\ast)$ by specifying the numbers 
from left to right  so that each is the minimal possible value such that 
 $$
  v^{j_1}_{i_1^\ast}\otimes   v^{j_2}_{i_2^\ast}
  \otimes \dots   v^{j_r}_{i_r^\ast}
  $$
 has value-type $(R,S)$
% $(\Gamma, \Gamma')$
  for any $  j_1,\dots,j_r $  such that $j_k=j_l$ if and only if $k \sim_{S}  l$.
  %$k \sim_{\Gamma'}  l$.
  \end{defn}

%
%  \begin{defn}\label{minimal vector}
%Fix our  ramified value type $(\Gamma,\Gamma')$.
%Let 
% $  j_1,\dots,j_r $
% be a tuple such that $j_k=j_l$ if and only if $k \sim_{\Gamma'}  l$. 
%  We define the associated {\sf minimal $\Gamma$ value type vector} 
%  $$
%  v^{j_1}_{i_1^\ast}\otimes   v^{j_2}_{i_2^\ast}
%  \otimes \dots   v^{j_r}_{i_r^\ast}
%  $$
%   to be given by 
%%allowing the entries $j_1,\dots j_r$ to be generic, but 
%specifying 
%the value of each $i_1^\ast,\dots, i_r^\ast$ in order to be minimal value such  that the resulting vector has value type $(\Gamma, \Gamma')$.
%  \end{defn}

  This definition is best understood via an example. 

  \begin{eg}
Fix %$\Gamma' 
$(R,S)= \bigl(\bigl\{\{1\}, \{2,3\}, \{4,6\},\{5\},\{7\}\bigr\}, \bigl\{\{1,2,3\}, \{4,6\}, \{5,7\}\bigr\} \bigr)$. %, fix distinct values $j_1, j_2, j_3$.  
The     
minimal 
%$\Gamma$
$R$ value-type tuple  %for $\Gamma=  \{\{1\}, \{2,3\}, \{4,6\},\{5\},\{7\}\}$ 
is given by
%$$
%v^{j_1}_{1}
%\otimes 
%v^{j_1}_{2}
%\otimes 
%v^{j_1}_{2}
%\otimes 
%v^{j_2}_{1}
%\otimes 
%v^{j_3}_{1}
%\otimes 
%v^{j_2}_{1}
%\otimes 
%v^{j_3}_{2}
% $$
%and
  $i^\ast =(1,2,2,1,1,1,2)$.  In particular, note that 
  $$
v^{j_1}_{1}
\otimes 
v^{j_1}_{2}
\otimes 
v^{j_1}_{2}
\otimes 
v^{j_2}_{1}
\otimes 
v^{j_3}_{1}
\otimes 
v^{j_2}_{1}
\otimes 
v^{j_3}_{2}
 $$
 has value-type $(R,S)$,
 % $(\Gamma, \Gamma')$,
   for any distinct values $j_1, j_2, j_3$.  

   \end{eg}

\color{black}

 \color{black}
\section{The ramified Schur functor}
\label{The ramified Schur functor}

To prove \TheoremA, we shall use the Schur--Weyl duality seen in \cref{mrjones}
between $\C  \W_{d}$ and the partition algebra $P_r(d)$, together with the Schur--Weyl duality seen in \cref{semisimpleee}
between $\C  \W_{m} \wr  \W_{n}$ and the  ramified partition algebra $\ram_r(m,n)$.
The former is well understood: as seen in the proof of Corollary~\ref{cor:partitionAlgebraSchurFunctor},
the bimodule decomposition of tensor space in \cref{mrjones} 
immediately provides the correspondence between simple left modules for $\C S_d$ and simple right
modules for $P_r(d)$.
In the latter case, however, we need to establish the correspondence between simple modules. 

We shall need this correspondence only for simple $\C  \W_{m} \wr  \W_{n}$-modules of the special form 
$\leftspecht {\mu}  \oslash \leftspecht{\nu} $, with $\mu$ a partition of $m$ and $\nu$ a partition of $n$.
We remark that the case where $\mu = (m)$ and $\nu = (n)$ was 
studied in \cite[Theorem 6.1]{MR4756467} entirely in the language of  classical partition algebras:
$\Delta_r(\varnothing^\varnothing) \res^{R_r(m,n)}\ForPrmn$ is the stable Foulkes module, denoted 
${\mathbb F}^r ({m,n})$ in \cite{MR4756467}.
Our proof includes this in its first case.

\begin{prop}\label{SW-simples}
The ramified Schur functor 
\[ \Hom_{ \C  \W_m \wr \W_n}({-}, 
\bigl( \C   ^{mn})^{\otimes r}\bigr): \C \W_m \wr \W_n
{\rm -mod}\to  {\rm  mod-}\ram_r(m,n) \] 
satisfies
\[\leftspecht{\alpha{[m]}}  \oslash \leftspecht{\beta{[n]}}  \mapsto
\begin{cases}  
% L_r(\alpha^\beta) & \textrm{if } \alpha= \varnothing, \beta= \varnothing\\
 L_r(\varnothing^\beta) & \textrm{if } \alpha=\varnothing, %\beta\neq \varnothing, 
 r\ge |\beta|,\\
 L_r(\alpha^{\beta{[n]}}) & \textrm{if } \alpha\neq \varnothing,  r\ge n|\alpha|,\\
0 &\textrm{if either } \alpha=\varnothing, \beta\neq \varnothing, r< |\beta| \textrm{ or }
\alpha\ne \varnothing,  r< n|\alpha|.
\end{cases} \]
\end{prop}

\begin{proof}The proof splits into two parts considering the top two  cases  separately.
Each part ends by showing that the image is zero when the condition for the third case holds.
%Moved \beta empty case to remark above.

\subsubsection*{The case $\alpha= \varnothing$} % and $\beta \ne \varnothing$.}
Suppose that $ \beta \vdash b$.   Assume first that $r \ge b.$
We must show that 
$$  \Hom_{ \C  \W_m \wr \W_n}\bigl(  \leftspecht{(m)}  \oslash \leftspecht{\beta{[n]}}   , (\C   ^{mn})^{\otimes r}
\bigr)  \cong L_r(\varnothing^{ \beta  }).$$
The simple module on the left-hand side is 
\[ \Hom_{ \C  \W_m \wr \W_n}\bigl(    \C ( \W_m \wr \W_n )c_{(m) ^{\beta{[n]}}}   , 
(\C   ^{mn})^{\otimes r}\bigr) \cong   c_{(m) ^{\beta{[n]}}} (\C   ^{mn})^{\otimes r}.\]
Therefore, since the standard module has a simple head,  it suffices to construct a 
 non-zero homomorphism of right $\ram_r(m,n)$-modules from the standard module
$$\Delta_r(\varnothing^{ \beta  }) \to c_{(m) ^{\beta{[n]}}}  (\C   ^{mn})^{\otimes r}.$$
As $\ram_r(m,n)$-modules, we have that 
\[ \Delta_r(\varnothing^{ \beta  }) \cong  c^\ast _{\varnothing^\beta 	} e_{(0^b)}
(J_{\preceq  (0^b)}/J_{\prec(0^b)}).
   \]
where $e_{(0^b)} \in \ram_r(m,n)$ is the quasi-idempotent constructed in Section~\ref{subsec:propagatingIndex}
and exemplified in Figure~\ref{thetaidemp}.
 Observe here that $e_{(0^b)}$ and $c^\ast _{\varnothing^\beta 	} $ commute, and that $e_{(0^b)}
(J_{\preceq  (0^b)}/J_{\prec(0^b)})$ is spanned by all ramified diagrams whose propagating index
in the sense of Section~\ref{subsec:propagatingIndex} is $(0^b)$.
Thus the north vertices $b+1, \dots, r$ are outer singletons; if $b=0$ then
$e_{(0^0)} = e_\varnothing$ is the diagram in which all $r$ northern and southern vertices are outer singletons.
Our aim is to define a non-zero homomorphism of right $\ram_r(m,n)$-modules, 
\[ c^\ast _{\varnothing^\beta 	} e_{(0^b)} (J_{\preceq  (0^b)}/J_{\prec(0^b)})  \to  c_{(m) ^{\beta{[n]}}} (\C   ^{mn})^{\otimes r}.\]
Firstly, we define 
\[ \chi: e_{(0^b)}  (J_{\preceq  (0^b)}/J_{\prec(0^b)}) \to   c_{(m) ^{\beta{[n]}}} (\C^{mn})^{\otimes r},\] 
by setting
$\chi( e_{(0^b)})=   c_{(m) ^{\beta{[n]}}} z$, where
\begin{align}\label{eq:rowenaz}
z= \sum_{\substack{1 \le i_1, \ldots, i_r\le m \\ 1 \le  j_{b+1}, \ldots, j_{r} \le n }}
(v_{i_1}^{n-b+1} \otimes v_{i_2}^{n-b+2} \otimes \cdots \otimes v_{i_{b-1}}^{n-1}  \otimes v^{n}_{i_{b}})\otimes
 (v^{j_{b+1}}_{i_{b+1}} \otimes \cdots \otimes   v^{j_{r}}_{i_{r}}) .
 \end{align}
Note that we have assumed that $r \ge b$. In the first $b$ places, the superscripts are distinct and equal the $b$ entries lying outside the first row of the tableau $\stt^{\beta{[n]}}$.

We must  show that $\chi$ is well-defined.
In general, given an idempotent $e$ in an algebra $A$, there is a well-defined homomorphism $eA \to U$, to 
a right $A$-module $U$, with $e \mapsto u$ provided $ue=u$. If $I$ is an ideal then we obtain a well-defined map from the quotient $e A/I \to U$ provided, in addition, $ueI=0$, or equivalently, $uI=0$.

Using the  diagrammatic right action of $\ram_r(m,n)$, we have
$z e_{(0^b)}= m^b \times ( m^{r-b} n^{r-b}) z =m^rn^{r-b}z,$  so {\color{red}\color{black}(after rescaling)} the first condition is certainly satisfied. We 
now check the second. Thus  we verify that  \smash{$c_{(m) ^{\beta{[n]}}} z J_{\prec(0^b)}= c_{(m) ^{\beta{[n]}}} z e_{(0^b)} J_{\prec(0^b)}=  0$}. Now,  $e_{(0^b)} J_{\prec(0^b)}$ is generated by two types of ramified diagrams:
(a) those obtained by merging two outer propagating blocks of $e_{(0^b)}$, and (b) those obtained by replacing an outer propagating block of $e_{(0^b)}$ with two non-propagating outer parts, i.e.,~those which factor through $e_{(0^{b-1})}$.  (Thus if $b=0$ there are no outer propagating blocks and there is nothing to prove;
correspondingly $(0^0) = \varnothing$ appears at the bottom of the Hasse diagram in Figure~\ref{poset}.) 
%MkW b=0: this seems to be the main point, please check --- we discussed on the phone and I added the pointer
%to the Hasse diagram, but nothing more, since it really doesn't seem to need much explanation.
\begin{itemize}[leftmargin=18pt]
\item[(a)] This case is easy: if $d_{(\Lambda_1,\Lambda_1')}$ is obtained by merging the $i$\textsuperscript{th} and $j$\textsuperscript{th} outer propagating blocks of $e_{(0^b)}$ (with $1 \le i \ne j \le b$), then  $ c_{(m) ^{\beta{[n]}}} z d_{(\Lambda_1,\Lambda_1')} = 0$ because the vectors in the $i$\textsuperscript{th} and $j$\textsuperscript{th} tensor places  differ in their superscripts and therefore are killed by a ramified partition with $i$ and $j$ in the same outer block. 
\item[(b)] This case requires a little more work. We shall show that $c_{(m) ^{\beta{[n]}}} z e_{(0^{b-1})}=0$.
Let $k$ be the length of the final row of $\beta$. Then entries $k$ and $n$ appear in the same column of $\stt^{\beta{[n]}}$  and the transposition $(k,n)$ lies in   $C(\stt^{\beta{[n]}})$. We may take cosets and write
\begin{align*} \hspace*{0.25in}c_{(m)^{\beta{[n]}}}  z &=\tilde{c}\displaystyle\sum  \bigl(
(v_{i_1}^{n-b+1} \otimes \cdots \otimes v_{i_{b-1}}^{n-1} \otimes  v_{i_b}^{n})  \otimes(  v^{j_{b+1}}_{i_{b+1}} \otimes \cdots \otimes   v^{j_{r}}_{i_{r}})
  \\
& \qquad\ \, - (v_{i_1}^{n-b+1} \otimes \cdots \otimes v_{n-1}^{n-1} \otimes v_{i_b}^{k})  \otimes  (v^{j_{b+1}}_{i_{b+1}} \otimes \cdots \otimes   v^{j_{r}}_{i_{r}}  ) \bigr)\end{align*}
for some  $\color{red}\color{black}\tilde{c} \in \mathbb C C(\stt^{\beta{[n]}})$. The vectors appearing in all tensor positions except the $b$\textsuperscript{th} are equal, and, when we act from the right by the ramified diagram $e_{(0^{b-1})}$, all terms cancel.
\end{itemize}

Having shown that the map $\chi$ is well-defined, we may now 
define the homomorphism of right $\ram_r(m,n)$-modules that we require by restriction: we define
\[\widetilde{\chi}:  c_{\varnothing ^{ \beta  }}^\ast e_{(0^b)}(J_{\preceq  (0^b)}/J_{\prec(0^b)}) \to c_{(m) ^{\beta{[n]} }} (\C   ^{mn})^{\otimes r},$$
by setting
$$\widetilde{\chi}(c_{\varnothing ^{ \beta  }}^\ast  e_{(0^b)}  ) =  c_{(m) ^{\beta{[n]} }}  z   c_{\varnothing ^{ \beta  }}^\ast 
 .\]
Here the left action  on $z$ is that of the wreath product and the right action is the diagrammatic action on tensor space.

Finally, we must show that $\widetilde{\chi}$ is non-zero. To do this we show there is a strictly positive coefficient of the basis vector
\[ \mathbf{v}_0=(v_{1}^{n-b+1} \otimes v_{1}^{n-b+2} \otimes\cdots \otimes v_{1}^{n}) \otimes
( v^1_{1} \otimes \cdots \otimes    v^1_{1})
\] in 
$\widetilde{\chi}(c_{\varnothing ^{ \beta  }}^\ast  e_{(0^b)}  )  = c_{(m) ^{\beta{[n]} }}  z   c_{\varnothing ^{ \beta  }}^\ast $.
To see this, first note that, expressed as a sum of diagrams,
\[ c_{\varnothing ^{ \beta  }}^\ast  e_{(0^b)}  = \displaystyle\sum_{\begin{subarray}c
 \sigma\in C(\stt^{ \beta  })  \\
  \tau \in R(\stt^{ \beta  })
 \end{subarray}
 } 
 { \sgn(\sigma)} (e_{(0^b)}	\sigma \tau).   
\]
Here $e_{(0^b)}	\sigma \tau$ is the ramified diagram with $1,\dots, b$ outer propagating and those outer propagating blocks permuted according to the permutation $\sigma \tau$.
Hence, $c_{(m) ^{\beta{[n]} }}  z   c_{\varnothing ^{ \beta  }}^\ast $ equals
\[
%\displaystyle %\!\!\!\!\!\!\!\! %MkW sum in text
\sum
%\sum_{\begin{subarray}c
%\pi \in C(\stt^{\beta{[n]}})\\
%\rho \in R(\stt^{\beta{[n]}})\\
% \sigma\in C(\stt^{ \beta  })  \\
%  \tau \in R(\stt^{ \beta  })\\
%1 \le i_1, \ldots, i_r\le m \\ 
%1 \le  j_{b+1}, \ldots, j_{r} \le n 
% \end{subarray}}
%  \!\!\!\!\!\!\!\!
\sgn(\pi ) \sgn(\sigma) (c_{(m)}, \dots, c_{(m)};  \rho \pi) 
(v_{i_1}^{n-b+1} \otimes v_{i_2}^{n-b+2} \otimes\cdots \otimes v_{i_b}^{n} \otimes
 v^{j_{b+1}}_{i_{b+1}} \otimes \cdots \otimes   v^{j_{r}}_{i_{r}} )  
 ( e_{(0^b)}	\sigma \tau)   \]
where the sum is over all
$\pi \in C(\stt^{\beta{[n]}})$,
$\rho \in R(\stt^{\beta{[n]}})$,
$\sigma\in C(\stt^{ \beta  })$,
$\tau \in R(\stt^{ \beta  })$ and indices $1 \le i_1, \ldots, i_r\le m$ and $1 \le j_{b+1}, \ldots, j_{r} \le n$.

Taking  $\pi, \rho, \sigma, \tau$  all to be identity permutations and all $i_k=j_k=1$, we have a contribution of $+1$ towards the coefficient of $\mathbf{v}_0$.
Now suppose $\pi, \rho, \sigma, \tau$ contribute to the coefficient of $\mathbf{v}_0$.  Then let us first see that $\pi $ must preserve the set $\{n-b+1, \dots , n \}$. 
If not then some $k \in \{n-b+1, \dots , n \}$ has $\pi(k)\leq n-b$  and (as  $\rho \in R(\stt^{\beta{[n]}}) \leq \W_{n-b} \times \W_b$)  there is a vector with superscript at most $n-b$ among the first $b$ tensors.  But the diagrammatic action of $ e_{(0^b)}	\sigma \tau$ then only changes its position, not its value, and therefore there is no contribution to the coefficient of $\mathbf{v}_0$. 

Assume now that $\pi \in C(\stt^{\beta{[n]}}) $ preserves the set $\{n-b+1, \dots , n\}$; i.e.~$\pi$ fixes the first row of $\stt^{\beta{[n]}}$. By assumption,
the left action by $\pi$ gives a pure tensor  of vectors having superscripts $\{n-b+1,\dots , n\}$  in the  first $b$ places, in some order, and to obtain  $\mathbf{v}_0$ the right action must permute them into increasing order. To do this we require $\sigma(k)=\pi(n-b+k)$ for all $k=1 \dots, b$. Therefore $\sgn(\sigma)=\sgn(\pi)$ and the contribution is strictly positive.

To complete the part of the proof where $\alpha = \varnothing$, we must show that if $r<b$ then
\[\Hom_{ \C  \W_m \wr \W_n}\bigl(    \leftspecht{(m)}  \oslash \leftspecht{\beta{[n]}}  , (\C   ^{mn})^{\otimes r}
\bigr) \cong  c_{ (m) ^{\beta{[n]}}}  (\C   ^{mn})^{\otimes r} =0.\]
Taking any pure tensor $v_{i_1}^{j_1} \otimes \cdots v_{i_r}^{j_r}$, the condition on $r$ ensures that there exist  $x\neq y \in \{1, \ldots, n\}$ such that the  entries $x,y$  lie in the same column  of the
 $\beta{[n]}$-tableau $\stt^{\beta{[n]}}$ but  neither $v^x_i$ nor $v^y_i$ appear in the pure tensor for any $i \in \{1, \ldots, m\}$.
  Taking cosets of the subgroup generated by the transposition  
  $(x,y)$ in $C(\stt^{\beta{[n]}})$ as a subgroup of the top group of  $ \W_m \wr \W_n$,  we may factorise $c_{(m)^{\beta{[n]}} }$ to see that
$c_{(m)^{\beta{[n]}}}(v_{i_1}^{j_1} \otimes \cdots v_{i_r}^{j_r})=0$.

\subsubsection*{The case $\alpha \ne \varnothing$.} Suppose that $r\ge |\alpha|n$. 
We set $a=|\alpha|$.  
We follow the same method as above and construct a 
 non-zero homomorphism of right $\ram_r(m,n)$-modules %from the standard module
\[ \Delta( \alpha ^{\beta{[n]}}) \to c_{ {\alpha{[m]}}^{\beta{[n]}}} (\C   ^{mn})^{\otimes r}.\]
As $\ram_r(m,n)$-modules, we have that 
$$\Delta( \alpha ^{\beta{[n]}}) \cong 
 c_{ \alpha ^{\beta{[n]}}}^\ast e_{(a^n)}(J_{\preceq  (a^n)}/J_{ \prec (a^n)}) ,$$
so we must define a non-zero homomorphism of right $\ram_r(m,n)$-modules 
$$ c_{ \alpha^{\beta{[n]}}}^\ast  e_{(a^n)}(J_{\preceq  (a^n)}/J_{ \prec (a^n)}) 
\to c_{{\alpha{[m]}}^{\beta{[n]}}} 
 (\C^{mn})^{\otimes r}.$$
 Firstly, define 
 $$\chi:
e_{ ( a^n)}    (J_{\preceq  (a^n)}/J_{ \prec (a^n)})  \to c_{{\alpha{[m]}}^{\beta{[n]}}} (\C   ^{mn})^{\otimes r},$$  by setting
$\chi( e_{(a^n)})= 
 c_{{\alpha{[m]}}^{\beta{[n]}}} z$, where
\begin{equation}\label{eq:rowenazNonEmpty} 
\begin{split} z= \sum_{\substack{1 \le i_k\le m \\ 1 \le  j_{k} \le n }}
(v_{m-a+1}^{1} \otimes v_{m-a+2}^{1} \otimes \dots \otimes v_{m}^{1}) \otimes \dots
\otimes (v_{m-a+1}^{n} \otimes &v_{m-a+2}^{n} \otimes \dots \otimes v_{m}^{n}) \\[-18pt] &\ \ \otimes 
 v^{j_{an+1}}_{i_{an+1}} \otimes\dots \otimes   v^{j_{r}}_{i_{r}}.\end{split} \end{equation}
(In the first $a$ tensors 
the superscripts are  1 and the subscripts are distinct, and in the next~$a$ tensors 
the superscripts are 2 and the subscripts are distinct, and the pattern continues until the subscript is $n$; 
this is possible as $r\geq an$.)

To check that $\chi$ is well-defined, again we verify routinely that 
%MkW added 'routinely' to emphasise we're not saying anything more
 multiplication by the quasi-idempotent $e_{(a^n)}$ scales $z$ and then we show that acting on $c_{{\alpha{[m]}}^{\beta{[n]}}} z$ by  ramified diagrams of the  following three types all give zero: 
 (a) ramified diagrams $d_{(\Lambda,\Lambda') }$ obtained by merging two outer propagating blocks of $e_{(a^n)}$; (b) ramified diagrams $d_{(\Lambda,\Lambda') }$ obtained by merging two inner propagating blocks of $e_{(a^n)}$; (c) ramified diagrams $d_{(\Lambda,\Lambda') }$ which replace an inner propagating block of $e_{(a^n)}$ with two inner singleton parts.
 %I really don't want the indices. They are nothing to do with i/j indices. Can't we just rely on the
 %different cases to make it clear \Lambda, \Lambda' changes. MkW
%%%%%%
\begin{itemize}[leftmargin=18pt]
\item[(a)]We see that  $ c_{{\alpha{[m]}}^{\beta{[n]}}} z d_{(\Lambda,\Lambda') } = 0$ because the vectors in the  corresponding tensor positions differ in their superscripts.  
\item[(b)] Similarly, $ c_{{\alpha{[m]}}^{\beta{[n]}}} z  		d_{(\Lambda,\Lambda') } = 0$ because the vectors in the  corresponding tensor positions differ in their subscripts. 
\item[(c)] This case again requires more work. The ramified diagram factors via  $e_{(a^n)} p_1^{(2)}$ so it suffices to consider \smash{$d_{(\Lambda,\Lambda') }= e_{(a^n)} p_1^{(2)}$}.
%, that is  where we replace the leftmost  inner propagating block (of the form $\{1, 1'\}$) of $e_{(a^b)} $ 
% with two singleton parts, $\{1\}$ and $\{1'\}$. 
   As the transposition  $  (1, m-a+1) $ lies in  $C(\stt^{\alpha{[m]}})$, we may take cosets and write
\[ c_{{\alpha{[m]}}^{\beta{[n]}}}=\tilde{c} \,\bigl( 1_{\W_m \wr \W_n}
-  ( (1, m-a+1), 1_{\W_m}, \dots, 1_{\W_m}; 1_{\W_n}) \bigr)\]
%There was an abuse of notation here where we put an algebra element into the group notation for
%elements of the wreath product MkW. I think we could avoid it without it looking too much worse.
   so that $c_{{\alpha{[m]}}^{\beta{[n]}}}  z$ equals
\begin{align*} \qquad\qquad \tilde{c}  
\bigl( & (v_{m-a+1}^{1} \otimes v_{m-a+2}^{1} \otimes  \dots \otimes v_{m}^{1}) \otimes \dots
\otimes (v_{m-a+1}^{n} \otimes \dots \otimes v_{m}^{n}) \otimes 
 v^{j_{an+1}}_{i_{an+1}} \otimes\dots \otimes   v^{j_{r}}_{i_{r}} 
\\ &\quad -   (v_{1}^{1} \otimes v_{m-a+2}^{1} \otimes \dots \otimes v_{m}^{1}) \otimes \dots
\otimes (v_{m-a+1}^{n} \otimes \dots \otimes v_{m}^{n}) \otimes 
 v^{j_{an+1}}_{i_{an+1}} \otimes\dots \otimes   v^{j_{r}}_{i_{r}}.
  \bigr) \end{align*}
  The vectors appearing in all tensor positions except the first are equal, and, when we act from the right by the diagram $d_{(\Lambda,\Lambda') }= e_{(a^n)} p_1^{(2)}$, all  terms cancel.
\end{itemize}

Having shown that the map $\chi$ is well-defined; 
restriction  provides a homomorphism of right $\ram_r(m,n)$-modules:
\[ \widetilde{\chi}:  c _{\alpha ^{\beta{[n]}}}^\ast e_{(a^n)}
 J_{\preceq (a^n)}/J_{\prec(a^n)}  \to c_{{\alpha{[m]}}^{\beta{[n]}}} (\C   ^{mn})^{\otimes r}, 
\widetilde{\chi}( c_{\alpha^{\beta{[n]}}}^\ast e_{(a^n)} ) =  c_{{\alpha{[m]}}^{\beta{[n]}}}  z   c_{\alpha ^{\beta{[n]}}}^\ast
 .\]
It remains to  show that $\widetilde{\chi}$ is non-zero, which we do by considering the  coefficient of 
$$ \mathbf{v}_0=(v_{m-a+1}^{1} \otimes \dots \otimes v_{m}^{1})\otimes(v_{m-a+1}^{2} \otimes \dots \otimes v_{m}^{2}) \otimes  \dots
\otimes (v_{m-a+1}^{n} \otimes \dots \otimes v_{m}^{n}) \otimes 
 v^{1}_{1} \otimes\dots \otimes   v^{1}_{1}  $$ 
 in
 $\widetilde{\chi}(  c_{\alpha^{\beta{[n]}}e_{(a^n)}}^\ast )=  c_{{\alpha{[m]}}^{\beta{[n]}}}  z  c_{\alpha ^{\beta{[n]}}}^\ast $.
Acting on the right  of $z$ is $ c_{\alpha ^{\beta{[n]}}}^\ast $, a sum of ramified diagrams  which will permute the places of the first $an$ tensors.  On the left of $z$  is 
\[  c_{{\alpha{[m]}}^{\beta{[n]}}}=
\sum\sgn(\pi)
(c_{{\alpha{[m]}} }, \ldots, c_{{\alpha{[m]}} } \, ; \, \rho \pi  ) \]
where the sum is over all $\pi \in C(\stt^{\beta{[n]}})$ and $\rho \in R(\stt^{\beta{[n]}})$.
(We emphasise that this is a linear combination  of permutation matrices acting diagonally on tensor
space.)
Each $c_{\alpha{[m]}}$ is itself a Young symmetrizer in $\W_m$, but we observe that,
since the right action only permutes the tensor places, the only contributions to the coefficient of $\mathbf{v}_0$ come from those terms in $\W_{\{m-a+1, \ldots, m\}}$ and (up to translation by $m-a$) we look only at $c_{\alpha ^{\beta{[n]}}} $.  For every permutation summand in   $ c_{{\alpha{[m]}}^{\beta{[n]}}}$ that contributes to the coefficient of $\mathbf{v}_0$,   the inverse permutation appears in  $c_{\alpha ^{\beta{[n]}}} ^\ast$ and undoes its effect.  The signs   of those permutations on the left and right agree and the coefficient of $\mathbf{v}_0$ is strictly positive.  %element -> permutation MkW

To complete this part of the proof where $\alpha\neq \varnothing$, 
we now suppose that $r<a n$. In this case we claim that 
\[ \Hom_{ \C  \W_m \wr \W_n}\bigl( \leftspecht{\alpha{[m]}} \oslash \leftspecht{\beta{[n]}}   , (\C   ^{mn})^{\otimes r}\bigr) \cong  c_{ {\alpha{[m]}} ^{{\beta{[n]}}}}  (\C   ^{mn})^{\otimes r} = 0.\]
Take any pure tensor $v_{i_1}^{j_1} \otimes \cdots v_{i_r}^{j_r}$. The condition on $r$ ensures that for some $j \in \{1, \ldots, n\}$ there exist entries $x\neq y \in \{1, \ldots, m\}$  lying in the same column  of the standard ${\alpha{[m]}}$-tableau $\stt^{\alpha{[m]}}$ but such that neither of $v^j_x$ nor $v^j_y$ appear in the pure tensor. Then taking cosets of the subgroup generated by $(x,y)\in C(\stt^{\alpha{[m]}})$ in the $j^\textsuperscript{th}$ copy of $\W_m$ inside $\W_m \wr \W_n$ we may factorise $c_{{\alpha{[m]}}^{\beta{[n]}}}$ to see that
$c_{{\alpha{[m]}}^{\beta{[n]}}}(v_{i_1}^{j_1} \otimes \cdots v_{i_r}^{j_r})=0$.  
\end{proof}

\TheoremB now follows easily by adapting steps (a) to (d) in the outline proof of \TheoremA 
in Section~\ref{subsec:structure} as follows.
 For $ \kappa \in \ParSet(\leq r)$, we choose $r$ sufficiently large (see the final equality) so that  
\begin{align*}
p\bigl(\beta{[n]},\alpha{[m]},\kappa{[mn]}\bigr) 
&= \langle s_{\beta{[n]}} \circ s_{\alpha{[m]}}, s_{\kappa[mn]} \rangle
\\ & = \left[ (\leftspecht{\alpha{[m]}} \oslash \leftspecht{\beta{[n]}} ) \Ind_{\W_m \wr \W_n}^{\W_{mn}} \, : \, \leftspecht{\kappa{[mn]}}
\right]\raisebox{-6pt}{$\scriptstyle \C  \W_{mn}$} \\
 &= \left[  \Hom_{\C  \W_{mn}}\left(  (\leftspecht{\alpha{[m]}} \oslash \leftspecht{\beta{[n]}}) \Ind_{\W_m \wr \W_n}^{\W_{mn}} , (\C   ^{mn})^{\otimes r} \right) \,
 : \,L_r(\kappa)
\right]\ForPrmnLow\\
 & = \left[  \Hom_{\C  \W_m \wr \W_n}\left(  \leftspecht{\alpha{[m]}} \oslash \leftspecht{\beta{[n]}}  , (\C   ^{mn})^{\otimes r} \Res_{\W_m \wr \W_n}^{\W_{mn}} \right) \,
 : \,L_r(\kappa)
\right]\ForPrmnLow\\
 & =\left[  \Hom_{\C  \W_m \wr \W_n}\left(  \leftspecht{\alpha{[m]}} \oslash \leftspecht{\beta{[n]}}  , (\C   ^{mn})^{\otimes r} \right)\Res\ForPrmn^{\ram_r(m,n)}  \,
 : \,L_r(\kappa)
\right]\ForPrmnLow\\& =
\begin{cases} 
% \left[ L_r(\varnothing^\varnothing)\Res\ForPrmn^{\ram_r(m,n)} : \,L_r(\kappa)\right]\ForPrmnLow %& 
%\textrm{if } \alpha=\varnothing, \beta=\varnothing\\[9pt]
\left[ L_r(\varnothing^{\beta}) \Res\ForPrmn^{\ram_r(m,n)} : \,L_r(\kappa)\right]\ForPrmnLow & \textrm{if }  \alpha=\varnothing, % \beta \neq \varnothing, 
 r\ge |\beta|,\\[9pt]
\left[ L_r(\alpha^{\beta{[n]}})\Res\ForPrmn^{\ram_r(m,n)} : \,L_r(\kappa)\right]\ForPrmnLow & \textrm{if } \alpha \neq \varnothing,  r\ge n|\alpha|.
\end{cases}
\end{align*}
Here, the third equality is obtained by the Schur--Weyl duality between the group algebra of the symmetric group and the partition algebra using Corollary~\ref{cor:partitionAlgebraSchurFunctor}, 
and the final step
comes from the Schur--Weyl duality between the group algebra of the wreath product of symmetric groups and the ramified partition algebra using Theorem~\ref{semisimpleee} and the corollary analogous
to Corollary~\ref{cor:partitionAlgebraSchurFunctor}.
    
This proves \TheoremB. and we also obtain an upper bound on  plethysm coefficients. 
By choosing $r$ sufficiently large (i.e.,~if $\alpha=\varnothing$ then choose $ r\ge |\beta|$, 
or if $\alpha \neq \varnothing$ then choose $r\ge n|\alpha|$), we have 
\begin{align*}
p\bigl(\beta{[n]},\alpha{[m]},\kappa{[mn]}\bigr) &= \begin{cases} 
\left[ L_r(\varnothing^{\beta}) \Res\ForPrmn^{\ram_r(m,n)} : \,L_r(\kappa)\right] 
\leq  \left[ \Delta_r(\varnothing^{\beta}) \Res\ForPrmn^{\ram_r(m,n)} : \,L_r(\kappa)\right]
& \!\!\!\textrm{if }  \alpha=\varnothing, \\
\left[ L_r(\alpha^{\beta{[n]}})\Res\ForPrmn^{\ram_r(m,n)} : \,L_r(\kappa)\right]  \leq
\left[ \Delta_r(\alpha^{\beta{[n]}})\Res\ForPrmn^{\ram_r(m,n)} : \,L_r(\kappa)\right] & \!\!\!\textrm{if } \alpha \neq \varnothing .
\end{cases}
\end{align*}
 In the next section we shall show that this is an equality when $\alpha = \varnothing$.

\section{Stability phenomena when the inner partition is trivial}

The $\C\W_m \wr \W_n$-modules
of the form $\leftspecht{ (m)} \oslash \leftspecht{\nu}$ are obtained by inflation from 
the $\C\W_n$-modules~$\leftspecht{\nu}$.   Thus 
  certain questions about these  modules can be 
 simplified to questions for only the `outer' symmetric group structure, for example 
 $\dim(\leftspecht{ (m)} \oslash \leftspecht{\nu} )=|\Std(\nu)|.$
In this section, we start by restricting our focus to the relationship between this `outer' symmetric group  and the `outer' partition algebra, via Schur--Weyl duality.  
Our aim is to prove the following theorem.

\begin{thm}\label{this!!} 
Let $\beta  \vdash b \leq r $ and suppose that $n\geq r+\beta _1$ and $m\geq r-b+\left[b \not= 0 \right]$. Then, as a $ (\C  \W_m \wr \W_n,R_r(m,n))$-bimodule, the tensor space $(\C ^{mn})^{\otimes r}$ has a direct summand isomorphic to 
   $$
(\leftspecht{ (m)} \oslash   \leftspecht{\beta[n]} )
\otimes \Delta_{r }(\varnothing^{\beta }
).
  $$
  In particular, under these assumptions, $\Delta_{r }(\varnothing^{\beta } ) = L_{r }(\varnothing^{\beta }) $.
  % if and only if $n\geq r+\beta _1$ and $m\geq r-b+1$.    
   \end{thm}
   
%{\bf \color{red}\color{black} To add in another proposition with the `only if' direction}
%No, we'lll get this indirectly out of symmetric functions --- MkW

We reiterate  that the analogous question for the % Foulkes modules 
trivial module $ \leftspecht{ (m)} \oslash \leftspecht{ ( n)}$  was already considered in \cite{MR4756467} and so \cref{this!!}  completes our understanding for the case where the inner partition is trivial.

%Before proving the theorem, we need a good understanding of the orbit basis (\cref{ramified basis}) 
Before proving the theorem, we need to understand  a basis of   the  $P_r(n)$-module $ \Delta_{r }( {\beta })$ and a basis 
of the $\ram_r(m,n)$-module  $\Delta_{r }(\varnothing^{\beta })$.   

We begin with the partition algebra, by delving into the orbit basis of  the  $P_r(n)$-module $ \Delta_{r }( {\beta })$ from \cref{rowenaB}.  %  Say $\beta  \vdash b$. %Already said
 %We let $P,Q$ be such that 
 Let $P\cup Q$ be a set-partition of   $\{\overline 1, \overline2, \dots,\overline r
 \}$ with $|P|=b$; {\color{red}\color{black} we denote the set of all such pairs $(P,Q)$ by $\mathcal{B}_b$}. In other words, we pick a set of $b$ distinguished blocks of the  set-partition $P \cup Q$ of   $\{\overline 1, \overline2, \dots,\overline r
 \}$. We now define a $(b,r)$-set-partition $
\Lambda ({P,Q})$ where the $b$ distinguished blocks from $P$ provide the propagating blocks.
   We set $\Lambda ({P,Q}) $ to be the $(b,r)$-set-partition 
\[
\Lambda ({P,Q})= \bigl\{ \{i\} \cup  P_i  \mid 1\leq i \leq b\bigr\} \cup Q
\]
where  $P$ is written according to the conventions of \cref{usefulconvention}.
It follows that there are no crossings between the $b$ propagating strands of 
$d_{\Lambda ({P,Q})}$, that is $\pi_P  = 1_{\W_b}$.  
%  We note that the $(b,r)$-set-partition $\Lambda ({P,Q})$ is not uniquely determined by the union $P\cup Q$, but rather by the pair $(P,Q)$.   
%   We set 
%\[
%\mathcal{B}_b = 	\bigl\{ (P,Q ) \mid 
%% \bigl(\hskip1pt\bigcup P\bigr)\cap \bigl(\hskip1pt\bigcup Q\bigr)=\varnothing, \bigl(\hskip1pt\bigcup P\bigr)\cup \bigl(\hskip1pt\bigcup Q\bigr)=\{\overline 1, \overline2, \dots,\overline r  \},
%(P,Q) \text{ is a 
%  |P|=b
%% d_{ \Lambda(P\cup Q) } \in V_r(0^b), 
%%     \pi_{P }=1_{ \W_b } 
%      \bigr\}.  \]
    The orbit basis of the $P_r(n )$-module  $\Delta_r( \beta)$  can now be rewritten as follows 
\begin{align}\label{basis2}
 \{  c_{ \beta}^\ast \sigma x_{ \Lambda(P, Q) }
 \mid 
  \stt^{ \beta} \sigma \in \Std( \beta),  
 (P,Q ) \in \mathcal{B}_b  \}.
 \end{align} 
 
We now turn our attention to modules of the ramified partition algebra with the aim of describing a diagram basis of the $\ram_r(m,n)$-module  $\Delta_{r }(\varnothing^{\beta })$.
Suppose $R$ is  a set-partition of  $\{\overline{1}, \overline{2}, \dots,\overline{r}
 \}$  that is a  refinement of $P\cup Q$, i.e.,~$R \leq P \cup Q$.    
  We let 
 $$\Lambda(R)
= \bigl\{ \{i\}  \mid 1\leq i \leq b\bigr\} \cup R ,
 $$
so $\Lambda(R)$ has no propagating blocks. Then $\Lambda(R) \leq \Lambda(P,Q)$ and the ramified diagram
\smash{$d_{(\Lambda(R),\Lambda(P, Q))}$} has its inner set-partition specified by $R$, with no inner propagating blocks, and its outer set-partition $\Lambda(P,Q)$ which has $b$ propagating blocks specified by $P$.
 
Now we can write down a diagram basis of  the $R_r(m,n)$-module
  $\Delta_r(\varnothing^\beta)
    = c_{\varnothing^\beta}^\ast  \C \W_b
    \otimes_{ \W_b} V_r(0^b)$ as follows:
% 
%   Now, the basis of the $R_r(m,n)$-module  $\Delta_r(\varnothing^\beta)$ from \cref{ramified basis} can now be rewritten as follows 
\begin{align}\label{ramified basis2}
 \bigl\{  c_{\varnothing^\beta}^\ast \sigma d_{(\Lambda(R),\Lambda(P, Q))}
 \mid 
  %\stt_{\varnothing^\beta} \sigma \in \Std(\varnothing^\beta),
  \stt^{\beta} \sigma \in \Std(\beta),  
 (P,Q ) \in \mathcal{B}_b, R \leq P\cup Q  \bigr\}.
 \end{align}

 \begin{proof}[Proof of \cref{this!!}.]

We suppose that $\beta _{[n]}$ is a partition of $n$   such that 
$n\geq r+\beta _1$. Then,  by \cref{semisimple2},  the $P_r(n)$-module  $L_r(\beta )=\Delta_r(\beta )$ is alone in its block.
 Thus by \cref{mrjones} we have that 
\begin{align}\label{step1hom}
\Hom_{\C  \W_n}\bigl(\leftspecht{\beta _{[n]}}, (\C^n) ^{\otimes r}\bigr)
\cong c_{\beta _{[n]}}(\C^n)^{\otimes r} \cong L_r(\beta ) 
= \Delta_r(\beta ) 
\end{align}
as right $P_r(n)$-modules. 

The remainder of this proof will consist of three steps. Firstly, we  write down a basis of  $\Hom_{\C  \W_n}(\leftspecht{\beta _{[n]}}, (\C^n) ^{\otimes r})$. Secondly, we use the previously identified basis to construct $\dim(\Delta_r(\varnothing^\beta ) )$  elements of $\Hom_{\C  \W_m \wr  \W_n}(\leftspecht{(m)}\oslash \leftspecht{\beta _{[n]}}, (\C^n) ^{\otimes r}) \cong L_r(\varnothing^\beta)$ (where the isomorphism is provided by \cref{SW-simples}). Thirdly, we prove that these $\dim(\Delta_r(\varnothing^\beta ) )$ distinct  $\C  \W_m \wr  \W_n$-homomorphisms are linearly independent. As  $L_r(\varnothing^\beta)$ is a quotient of $\Delta_r(\varnothing^\beta)$, this will show that  $L_r(\varnothing^\beta)=\Delta_r(\varnothing^\beta)$ as required.

We start by considering $ \Delta_r(\beta )$. A generator of the standard module  is  $c_\beta^\ast e_b$, and we let $Z \in c_{\beta _{[n]}}(\C^n) ^{\otimes r} $ be the image of this generator under the isomorphism of \cref{step1hom}. Now we have a basis of $\Delta_r(\beta)$ from \cref{basis2}, and the image of this orbit basis under the first isomorphism of (\ref{step1hom}) provides us with a basis of  
 $\Hom_{\C  \W_n}\bigl( \leftspecht{\beta _{[n]}}, (\C^n) ^{\otimes r}\bigr)$:
\[\{ \vartheta  _{\sigma,P,Q} \mid  \stt^{ \beta} \sigma \in \Std( \beta),  
 (P,Q ) \in \mathcal{B}_b  \},\]
 where the $\C  \W_n$-homomorphism $\vartheta  _{\sigma,P,Q}$ is defined on the generator $c_{{\beta _{[n]}}} \in \leftspecht{\beta _{[n]}}$ by
\begin{equation}\label{labelthat}
\vartheta  _{\sigma,P,Q}(c_{{\beta _{[n]}}})
= Z \Psi(\sigma x_{\Lambda(P\cup Q)}) =
\sum  \alpha^{j_1,\dots, j_r}
_{\sigma,P,Q}(v^{j_1}\otimes \dots \otimes v^{j_r})
\end{equation} for some coefficients $\alpha^{j_1,\dots, j_r}
_{\sigma,P,Q} \in \C $. 
Observe that, by the action of the orbit basis on tensor space specified 
in \eqref{eq:Ptn-action} the inequality 
$\alpha^{j_1,\dots, j_r}
_{\sigma,P,Q}\neq 0$ implies that $(j_1,\dots, j_r)$ has value-type $  P\cup Q$ (as in \cref{Rvaluetype}).

The second step is to  consider inflations of these homomorphisms to construct elements of $\Hom_{\C  \W_n}(\leftspecht{\beta _{[n]}},(\C^n)^{\otimes r})$. We now assume that $m\geq r-b+\left[ b\geq 1 \right]$.  
Given $R$ any refinement of the set-partition $P\cup Q$, we define a map
 $\vartheta _{\sigma, P,Q,R}
:  \leftspecht{ (m)} \oslash   \leftspecht{\beta _{[n]}  }
\to (\C^{mn})^{\otimes r}
$ given by 
\begin{align}\label{homo}
\vartheta _{\sigma, P,Q,R}\big(c_{(m)^{\beta _{[n]}}}\big)
 =
\sum_{
\begin{subarray}c
j_1,\dots,j_r \in \{1,\dots, n\}
\\
i_1,\dots i_r \in \{1,\dots,m\}
\end{subarray}
}
  \alpha^{j_1,\dots, j_r}
_{\sigma,P,Q}(v^{j_1}_{i_1}\otimes \dots \otimes v^{j_r}_{i_r})
\end{align}
where the coefficients $\alpha^{j_1,\dots, j_r}$ are those \cref{labelthat} and the
 sum is over {\em all indices} $i_1,\dots i_r \in \{1,\dots,m\}$, $ j_1,\dots,j_r \in \{1,\dots, n\}$ satisfying:
if $j_x = j_y$ then $i_x = i_y$ if and only if $x,y$ belong to the same part of $R$.
 Observe that if the tensor 
 $ 
v^{j_1}_{i_1}
  \otimes 
  v^{j_2}_{i_2}
  \otimes 
  \dots
  \otimes 
  v^{j_r}_{i_r} $ 
 appears with non-zero coefficient in \cref{homo}, then its ramified value-type is $(R,P\cup Q)$. 

We claim that $\vartheta _{\sigma, R,P,Q}$ is an $\W_m \wr \W_n$-homomorphism. It is clear from \cref{labelthat}  $\vartheta _{\sigma, R,P,Q}$ is a homomorphism for modules of the distinguished top group $\W_n$ in $\W_m \wr \W_n$. Therefore, we need only check  the action of the base group  fixes the right-hand side of \cref{homo}. The action of the base group is as follows:
$$
(\sigma_1,\dots, \sigma_n;1_{\W_n})
(v^{j_1}_{i_1}\otimes \dots \otimes v^{j_r}_{i_r})
=
(v^{j_1}_{\sigma_{j_1}(i_1)}\otimes \dots \otimes v^{j_r}_{\sigma_{j_r}(i_r)}).
$$
 Only the subscripts have changed and this action preserves ramified value-type as we are applying the same permutation to the subscripts where the superscripts agree.  The base group acts trivially and $\vartheta _{\sigma, R,P,Q}$ is an $\W_m \wr \W_n$-homomorphism.

 We have constructed $\dim \Delta_r(\varnothing^\beta)$ homomorphisms and the third step is to prove that 
\begin{align}\label{lindien}
\{ \vartheta _{\sigma, P,Q,R}
  \mid 
  \stt^{ \beta} \sigma \in \Std( \beta),  
 (P,Q ) \in B , R \leq P\cup Q \} 
 \end{align}
is  a linearly independent set.
  Assume, for a contradiction, that the set of \cref{lindien} is   linearly  dependent  and, in particular, that 
\begin{align}
\sum_{\sigma,P,Q,R} (\beta _{\sigma,P,Q,R} ) 
(\vartheta  _{\sigma,P,Q,R} ) =0, 
\end{align}
for some coefficients $\beta  _{\sigma,P,Q,R}  \in \C $ which are not all zero.  This implies that 
\begin{align}
\label{fixitsoon}
\sum_{\sigma,P,Q,R}
 \beta _{\sigma,P,Q,R}  
\sum  \alpha^{j_1,\dots, j_r}
_{\sigma,P,Q}(v^{j_1}_{i_1}\otimes \dots \otimes v^{j_r}_{i_r})
=0 \end{align}
where the second summation is over the same indexing set as that of \cref{homo}.

 We now fix a ramified value-type $(R, S)$ and restrict our attention to the tensor summand of \cref{fixitsoon} whose vectors   have this fixed ramified value-type. We consider those $P,Q$ with $P \cup Q=S$.
  Because the summation is over {\em all indices} of the form in \cref{homo}, we are able to
restrict  further and only consider the tensor summand (for a given ramified value-type $(R,S)$)  in the image of the projection onto the  minimal $R$ value-type  tuple $i^\ast$
  (as defined in  \cref{minimal vector}). 
This is possible due to our assumption that $m\geq r-b+\left[ b\geq 1 \right]$:  either $b=0$ and we have $m \geq r$, so there  are $r$ distinct subscripts available, or $b \ge 1$ and there are $b$ parts in the set-partition $P$, so the maximal number of parts of $R$ in any part of $P \cup Q$ is $r-(b-1)$ (obtained when $Q=\varnothing$, $P$ has $b-1$ singleton parts and one part of $r-(b-1)$ vertices, and $R$ is all singletons).  It then follows from \cref{fixitsoon} that the following holds for some coefficients $\beta  _{\sigma,P,Q,R}$ that are not all zero:
%
% It suffices to show that 
  \begin{align}
\label{fixitsoon22}
\sum_{
\begin{subarray}c
\sigma,P,Q,R \\
P\cup Q = S , 
\end{subarray}
}\!\!\!\!
 \beta _{\sigma,P,Q,R}  \!\!\!\!
\sum_{j_1,\dots,j_r \in \{1,\dots, n\}}
\!\!\!\!\!  \alpha^{j_1,\dots, j_r}
_{\sigma,P,Q}(v^{j_1}_{i_1^*}\otimes \dots \otimes v^{j_r}_{i_r^*})
=0. \end{align}

%To this end, assume \cref{fixitsoon22} holds for some  coefficients $\beta  _{\sigma,P,Q,R} $ which are not all zero. 
 This implies that 
\begin{align}
\sum_{
\begin{subarray}c
\sigma,P,Q  \\
P\cup Q = S 
 \end{subarray}
}\!\!\!\!
 \beta _{\sigma,P,Q,R}  \!\!\!\!
\sum_{j_1,\dots,j_r \in \{1,\dots, n\}}
\!\!\!\!\!  \alpha^{j_1,\dots, j_r}
_{\sigma,P,Q}(v^{j_1} \otimes \dots \otimes v^{j_r} )
 =0 \end{align} for the same coefficients  $\beta  _{\sigma,P,Q,R} $. 
Since the  $\C \W_{n}$-homomorphisms in \cref{labelthat} are linearly independent, this is a contradiction.  Thus the  set of  $\dim \Delta_r(\varnothing^\beta)$ homomorphisms in \cref{lindien} is linearly independent, and 
\[ \dim  \Hom_{\C  \W_m \wr \W_n  }\bigl( \leftspecht{ (m)} \oslash   \leftspecht{\beta _{[n]}  }
, (\C^{mn})^{\otimes r} \bigr) \ge \dim \Delta_r ( {\varnothing^{\beta }}). \]  
Since 
\[
\Hom_{\C  \W_m \wr \W_n  }\bigl(\leftspecht{ (m)} \oslash   \leftspecht{\beta _{[n]}  }
, (\C^{mn})^{\otimes r}\bigr) \cong L_r (  {\varnothing^{\beta }} ),
\]
and $L_r(\varnothing^\beta)$ is
a quotient of $\Delta_r ( {\varnothing^{\beta }} )$, we conclude that \cref{lindien} specifies a basis 
of $\Delta_r ( {\varnothing^{\beta }} )$ and $\Delta_r(\varnothing^\beta)\cong L_r(\varnothing^\beta) $.
 \end{proof}

We show later in Corollary~\ref{cor:onlyIfProperQuotient} that the bounds on $m$ and $n$ are tight.
We now complete the proof of \TheoremB.
%, taking the
%opportunity to remind the reader of the logical structure of the argument. 
% Not really true any more.

\begin{cor} \label{cor:stable1} 
Provided  $n\geq  r+\beta _1$ and $m \geq r-|\beta|+\left[\beta\not=\varnothing\right]$
the plethysm coefficient $p(\beta[n], (m),\kappa[mn])$ satisfies
\[p(\beta[n], (m),\kappa[mn]) = \bigl[\Delta_r( {\varnothing^\beta} ){\Res}_{P_r(mn)}^{\ram _r(m,n)} : L_r(\kappa)\bigr]\ForPrmn. \]
\end{cor}

\begin{proof}
By Theorem~\ref{this!!} %I'm not going to change this label, but I hate it MkW
we have $\Delta_r(\varnothing^\beta) = L_r(\varnothing^\beta)$. The result now follows from
the final displayed equation at the end of  Section 6, which states in the case $\alpha = \varnothing$ that
$p(\beta[n],\varnothing,\kappa[mn]) = \bigl[ L_r(\varnothing^\beta)\Res^{R_r(m,n)}_{P_r(mn)} :
L_r(\kappa) \bigr]\ForPrmn$.
\end{proof}

% MkW: I cannot bear an equals sign as the verb in a theorem statement.
In \S\ref{sec9} we derive a combinatorial formula for the calculation of  the right-hand side which shows that it is independent of $m$ and $n$ for sufficiently large $r$, and therefore \cref{cor:stable1} provides the new stability of plethysm coefficients stated in \TheoremA.
 
\section{Restricting our attention to a layer of fixed depth}

%MkW added remarks below, indicated MkW as usual
%{\bf \color{red}\color{black}  Does this all work for $\Delta_r(\varnothing^\varnothing)$ as well? If so include at start that $b$ can be zero too.}

In this section, we
consider
% determine  the decomposition of
 the  restriction of the  ramified partition algebra  modules 
 $\Delta_r(  \alpha ^{  \beta})	$ to the partition algebra.   %  obtained under  the image under the Schur functor.   
%We  determine the decomposition of the restriction of 
%these modules from the ramified partition algebra to the partition algebra.
We show the restriction to the partition algebra has a standard module filtration with well-defined filtration multiplicities.
We prove that these  multiplicities 
provide upper bounds for 
plethysm coefficients.  
In the case that the inner partition is trivial, 
we obtain new closed formulas for plethysm coefficients by way of \cref{this!!}.  
%  We do not explicitly consider   the decomposition of $\Delta_r(\varnothing^\varnothing)$ here, as it was the subject of the prequel to this paper \cite{MR4756467}.

\begin{defn}\label{defn:ramifiedBranchingCoefficient}
Assume that $\delta_{\rm in} ,\delta_{\rm out} \neq0$.
% are such that $\ram_r(\delta_{\rm in} ,\delta_{\rm out} )$ is semisimple (such as the condition of \cref{semisimpleee}).  
Given partitions $\alpha$, $\beta$ and $\lambda $,
  we define the {\sf  ramified branching coefficient} 
  of the standard  module $\color{red}\color{black}\Delta_r(\lambda )$ for $P_r(\delta_{\rm in}\delta_{\rm out})$ in the standard module $\Delta_r(\alpha^\beta)$ 
for $R_r(\delta_{\rm in}, \delta_{\rm out})$ to be the filtration multiplicity
%MkW: this was bold, but we don't need multipartitions anywhere in this paper. I.e. \alpha should just
% be a partition, and we're not defining ramified branching coefficients for general R_r(m,n)-modules.
 \[	 \bigl[\Delta_r( \alpha^\beta)
\Res_{P_r(\delta_{\rm in} \delta_{\rm out} )}^{\ram _r(\delta_{\rm in} ,\delta_{\rm out} )}\, :
 {\color{red}\color{black}
 \Delta_r(\lambda )  }
\bigr]_{P_r(\delta_{\rm in} \delta_{\rm out} )}. 
 \]
 \end{defn} 
We shall 
see that the filtration multiplicities are well-defined in \cref{cor:compfactors+DR} regardless of the non-zero values of
$\delta_{\rm in} ,\delta_{\rm out}$.   We have already seen that 
 $\ram_r(\delta_{\rm in} ,\delta_{\rm out} )$ is semisimple  for sufficiently large parameters in \cref{semisimpleee}, and so the reader may prefer to consider 
 only  the semisimple case where these  filtrations  are direct sums.  
%MkW: query to Rowena: is this enough to understand the entire proof? If not I don't think we should
%say it.

\subsection{The action of the partition algebra by restriction} 
The right module~$\rightspecht{\alpha^\beta}$ for the wreath product
$\Sym_a \wr \Sym_b$ was defined in~\eqref{eq:wreathModule} to be $\rightspecht{\alpha} \oslash
\rightspecht{\beta}$.
We shall describe the action of the generators of the partition algebra $\P_r(\delta_{\rm in} \delta_{\rm out} )$ on a basis of $\Delta_{r}(  \alpha ^{  \beta}) = \rightspecht{\alpha ^{  \beta}}\otimes_{\Sym_a \wr \Sym_b} V_r(a^b)$ for $\alpha \vdash a$ and $\beta\vdash b$,
with $ ab \leq r$ or, in the case $a=0$, with $b \leq r$. (If $b=0$
then $(0^0) = \varnothing$ and $V_r(\varnothing)$ is generated by the idempotent $e_\varnothing$;
recall that this is the diagram in which all $r$ northern and southern vertices are outer singletons.)
%MkW: this is the only remark I added. I will look again but I think everything just goes through.

Consider first the right  $\P_r(\delta_{\rm in} \delta_{\rm out} )$ action on the usual diagram basis of $V_r(a^b)$.
Let $(\Lambda, \Lambda') \in V _r( a^b) $ be a ramified diagram
with inner blocks $ \Lambda=\{S_1, S_2, \dots , S_p\}$ and 
outer blocks $\Lambda' = \{ \Sigma_1,\Sigma_2,\dots, \Sigma_q \}$  %MkW this was a tuple, but should be set
written using the convention of Remark~\ref{usefulconvention}
so that the blocks are ordered by increasing minima. %MkW took out p >= q, I don't see what it adds.
We set 
\begin{equation}  \label{p12}
d_{(\Lambda,\Lambda')}{\sf p}_{1,2}
=
\begin{cases}  
d_{  (\{S_1, S_2, S_3, \dots S_p\},\{\Sigma_1,\Sigma_2,\dots, \Sigma_q\}) }&\text{if }1,2\in S_1 \subseteq  \Sigma_1
\\  
d_{(\{S_1\cup S_2, S_3, \dots S_p\},\{\Sigma_1\cup\Sigma_2,\dots, \Sigma_q\}) }&\text{if }1\in S_1 \subseteq  \Sigma_1, 2\in S_2 \subseteq \Sigma_2
\\
d_{  (\{S_1\cup S_2, S_3, \dots S_p\},\{\Sigma_1,\Sigma_2,\dots, \Sigma_q\}) }&\text{if }1\in S_1 \subseteq  \Sigma_1, 2\in S_2 \subseteq \Sigma_1,
 \end{cases}\end{equation}
providing the resulting diagram
belongs to $V _r(a^b) $ and we leave $d_{(\Lambda,\Lambda')}{\sf p}_{1,2}$ undefined otherwise.  
(The diagram does not belong to $V _r(a^b) $
if taking the  product with ${\sf p}_{1,2}$ decreases the number of propagating outer-blocks or the number of propagating inner-blocks.)  
We  also set 
\begin{equation}  \label{p1}
d_{ (\Lambda,\Lambda')}{\sf p}_{1}
=
\begin{cases}
\delta_{\rm in}  \delta_{\rm out}    
d_{ (\Lambda,\Lambda')  }&\text{if } \{1\}=S_1 = \Sigma_1
\\
 d_{  (\{\{1\}, S_1-\{1\}, S_2, S_3, \dots S_p\},\{\{1\}, \Sigma_1-\{1\},\Sigma_2,\dots, \Sigma_q\})} &\text{if }\{1\}\subset S_1 \subseteq   \Sigma_1
\\
\delta_{\rm in}     d_{  (\{\{1\},   S_2, S_3, \dots S_p\},\{\{1\}, \Sigma_1-\{1\},\Sigma_2,\dots, \Sigma_q\}) }&\text{if }\{1\}= S_1 \subset    \Sigma_1,
 \end{cases}
 \end{equation}
 providing the resulting diagram
belongs to $V _r( a^b) $ and we leave $d_{ (\Lambda,\Lambda')}{\sf p}_{1}$ undefined otherwise.   
It is worth noting that the elements on the right of \cref{p1} 
are not necessarily written in the form specified by \cref{usefulconvention}.
 For   ${\sf p} = {\sf p}_1$ or ${\sf p} = {\sf p}_{1,2}$
  %\{{\sf p}_{1},{\sf p}_{1,2}\}$ 
  and $x\in \rightspecht{ \alpha ^{  \beta}}$, we  observe that 
\[ (x \otimes_{ \W_a \wr \W_b}  d_{(\Lambda,\Lambda')}){\sf p} 
 =  
 \begin{cases}
 x \otimes_{ \W_a \wr \W_b} (d_{(\Lambda,\Lambda')}
 {\sf p}	) & \text{ if  $d_{(\Lambda,\Lambda') }{\sf p} \in  V _r( a^b) $ is defined}\\
 0		&\text{otherwise}.
 \end{cases}
\]
 The generators ${\sf s}_{i,i+1}$ for $1\leq i <r$ act in the usual fashion by  permuting $\{1,2,\ldots,r\}$.  
 For ease of notation, we do not write these actions out explicitly.

\subsection{The depth quotient }  
In this section we identify a quotient of the standard module $\Delta_r(\alpha^\beta)$ that 
contains all simple modules $L_r(\kappa)$ in which the partition $\kappa$ has the maximum
possible size $r$.
%MkW: I really don't understand 'stability result' in this context and have rewritten to reflect my
%understanding of what we're doing.

%is a stability result which 
%states that  it is enough to only consider 
%only those partitions $\la \in \ParSet_{ r}\subseteq  \ParSet ({\leq r})$ for each 
%$r\geq0$ in turn. 

%This might seem obvious to the ``diagrammatic algebra readership" of the paper, but the key point here is that this provides  us with   much more  malleable representations.   

\begin{defn}\label{depthradical}Let $r\in \mathbb{N}$ and $(a^b) \in \Theta_r$.  
We define the  {\sf depth-radical}   of $ V_r(a^b)$ to be the subspace $ W_r(a^b)\subseteq V_r(a^b)$ spanned by the ramified diagrams
 $d_{(\Lambda,\Lambda')}$ satisfying either of the following two conditions:
\begin{itemize}[leftmargin=18pt]
\item[(i)] the
{\color{red}\color{black}inner} 
 set-partition $\Lambda$ contains two southern vertices in the same  block; 
\item[(ii)] the 
{\color{red}\color{black}outer} 
set-partition $\Lambda'$ contains a singleton southern block.   %Please no italics for (i)
\end{itemize}
  We define 
 the  {\sf depth-radical}   of $\Fdelta$ to be the subspace  
\[ 
\begin{minipage}{2.04cm}
${\sf DR}\bigl( \Fdelta \bigr)$
\end{minipage}
=
 {\rm S}(  \alpha ^{  \beta} )\otimes_{\W_a \wr \W_b} W_r{(a^b)} \subseteq  
{\rm S}(  \alpha ^{  \beta} )\otimes_{\W_a \wr \W_b} V_r{(a^b)} =\Delta_r(  \alpha ^{  \beta} ).
\]

\end{defn}
 
This construction will allow us to study the  smallest possible modules  
in which we can see the ramified branching coefficients.

\begin{prop}
Given $r \in \mathbb{N}$, the depth radical ${\sf DR}\bigl( \Fdelta \bigr)$  is a $\Pdelta$-submodule of 
\smash{$\Fdelta \res_{P_r(\delta_{\rm in} \delta_{\rm out} )} ^{\ram _r(\delta_{\rm in} }$}.  
\end{prop}

\begin{proof}
 It is clear that
 the generators $\sgen _{i,i+1}$ for $1\leq i <r$ preserve the space ${\sf DR}\bigl( \Fdelta \bigr)$
 as both conditions of \cref{depthradical} are invariant under
% place permutation.
the permutation action.
By \cref{p1}, the generator   $\pgen _1$ acts on a given  diagram $d_{( \Lambda,\Lambda')}$  either 
by scalar multiplication, or by removing an edge from $\Lambda$ {\em at the expense} of introducing a singleton into $\Lambda'$.  Therefore the generator~$\pgen _1$ preserves ${\sf DR}\bigl(\Fdelta\bigr)$ 
by \cref{depthradical}(ii).  
By \cref{p12} the generator   $\pgen _{1,2}$ acts on a given diagram $d_{ (\Lambda,\Lambda')}$    either 
trivially or by introducing an edge in $\Lambda$.   Therefore the generator  $\pgen _{1,2}$ preserves ${\sf DR}\bigl(\Fdelta\bigr)$ by \cref{depthradical}(i).    
\end{proof}

\begin{defn}\label{dq}
Define the {\sf depth quotient}  %$ \Delta^0_r(\alpha^\beta) $
${\sf DQ}(\Fdelta)$
   to be the quotient $P_r(\delta_{\rm in} \delta_{\rm out} )$-module 
%$$  \Delta^0_r(\alpha^\beta)  = \Fdelta/{\sf DR}(\Fdelta) = 
%{rm S}(\alpha^\beta) \otimes V^0_r(a^b) $$ 
%where $V^0_r(a^b)$ consists of 
$$\Delta_r(  \alpha ^{  \beta} )=
\rightspecht{  \alpha ^{  \beta} }\otimes_{\W_a \wr \W_b} V_r{(a^b)}
\twoheadrightarrow 
\rightspecht{  \alpha ^{  \beta} }\otimes_{\W_a \wr \W_b} V^0_r{(a^b)}=
{\sf DQ}\bigl(\Fdelta\bigr),
%\Delta^0_r(  \alpha ^{  \beta} )  
$$
where  $    V^0_r{(a^b)}$ is the quotient vector space $    V_r{(a^b)}/     W_r{(a^b)}$, that is  the span of those ramified diagrams in $ V^0_r{(a^b)}$ that do not lie in $    W_r{(a^b)}$.
%=\C \{ d_{\Lambda,\Lambda'} \in V_r(a^b) \mid d_{\Lambda,\Lambda'} \not \in W_r(a^b)\}$.  
  \end{defn}

\begin{figure}[ht!]
$$\qquad\qquad  \begin{tikzpicture}[xscale=0.45,yscale=-0.45]
  
       \foreach \x in {0,1,2,3,4,5,6,7,8,9,10,11,12,13,14}
     {  
      \path(\x*2,0) coordinate (down\x);  
     \path(\x*2,3) coordinate (up\x);  
    }

     \foreach \x in {0,1,2,3,4,5,6,7,8,9,10,11,12,13,14}
     {   \path(up\x) --++ (135:0.6) coordinate(up1\x);  
      \path(up\x) --++ (45:0.6) coordinate(up2\x);
            \path(up\x) --++ (-45:0.6) coordinate(up3\x);  
                        \path(up\x) --++ (-135:0.6) coordinate(up4\x);  
   }

    \foreach \x in {0,1,2,3,4,5,6,7,8,9,10}
     {   \path(down\x) --++ (135:0.6) coordinate(down1\x);  
      \path(down\x) --++ (45:0.6) coordinate(down2\x);
            \path(down\x) --++ (-45:0.6) coordinate(down3\x);  
                        \path(down\x) --++ (-135:0.6) coordinate(down4\x);  
   }

   \draw [fill=white] plot [smooth cycle]
  coordinates {(up10) (up23) (down32)   (down40)  };

   \draw [fill=white] plot [smooth cycle]
  coordinates {(up14) (up27) (down36)   (down44)  };

   \draw [fill=white] plot [smooth cycle]
  coordinates {(up18) (up29) (down39)   (down48)  };

     \path(10*2,3) coordinate (up10);  
          \path(12*2,3) coordinate (up12);  

      {   \path(up10) --++ (135:0.6) coordinate(Xup1);  
      \path(up10) --++ (45:0.6) coordinate(Xup2);
            \path(up10) --++ (-45:0.6) coordinate(Xup3);  
                        \path(up10) --++ (-135:0.6) coordinate(Xup4);  
   }

   {   \path(up12) --++ (135:0.6) coordinate(XupA1);  
      \path(up12) --++ (45:0.6) coordinate(XupA2);
            \path(up12) --++ (-45:0.6) coordinate(XupA3);  
                        \path(up12) --++ (-135:0.6) coordinate(XupA4);  
   }
   
 \draw(8,0)--(12,3);
    \draw(10,0)--(8,3);

 \draw(16,0)--(18,3);
    \draw(18,0)--(16,3);
 
 \draw(0,3)--(0,0);                       
    \draw(2,0)--(6,3);

      \foreach \x in {0,2,4,6,8,10,12,14,16,18,20,22,24,26,26}
     {   \fill[white](\x,3) circle (4pt);   
                        \draw (\x,3) node {$\bullet$}; 
                  }
                
      \foreach \x in {0,2,8,10,16,18}
     {   \fill[white](\x,0) circle (4pt);   
                        \draw (\x,0) node {$\bullet$}; 
                  }

       \foreach \x in {0,1,2,3,4,5}
     {  
      \path(\x*2+20,0) coordinate (down\x);  
   }

       \foreach \x in {0,1,2,3,4,5,6}
    {   \path(\x*2+20,3) coordinate (up\x);  
    }
  
      \foreach \x in {0,1,2,3,4,5,6}
     {   \path(up\x) --++ (135:0.6) coordinate(up1\x);  
      \path(up\x) --++ (45:0.6) coordinate(up2\x);
            \path(up\x) --++ (-45:0.6) coordinate(up3\x);  
                        \path(up\x) --++ (-135:0.6) coordinate(up4\x);  
   }

    \foreach \x in {0,1,2,3,4,5}
     {   \path(down\x) --++ (135:0.6) coordinate(down1\x);  
      \path(down\x) --++ (45:0.6) coordinate(down2\x);
            \path(down\x) --++ (-45:0.6) coordinate(down3\x);  
                        \path(down\x) --++ (-135:0.6) coordinate(down4\x);  
   }

  \path(2+20,3+1.5) coordinate   (midway);
    \path(2+20,3+.75) coordinate   (lowermidway);
    
   \draw [fill=white] plot [smooth cycle]
  coordinates {(up10) 
  (midway)
  (up22) (up32) (up42)
  (lowermidway)
     (up30)     (up40)  };

  \path(4+20,1.5) coordinate   (midway);
    \path(4+20,2.25) coordinate   (lowermidway);

   \draw [fill=white] plot [smooth cycle]
  coordinates {(up41) 
  (midway)
  (up33) (up23) (up13)
  (lowermidway)
     (up21)     (up11)  };

       \foreach \x in {0,2,4,6}
     {   \fill[white](\x+20,3) circle (4pt);   
                        \draw (\x+20,3) node {$\bullet$}; 
                  }
         \end{tikzpicture}          $$
         \caption{
       An {\color{red}\color{black}example} of an element of $V^0_{14}(2^3)$.  There are no inner southern arcs and there are no outer southern singletons.
} 
\label{a-concatenate-for-the-hard-thin2}         \end{figure}
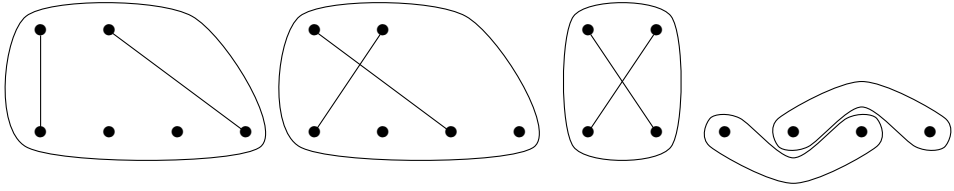

%; this dramatically reduces the complexity of our problem.  
%
%\begin{defn}
%We define $V^0_r{(a^b)}\subseteq V_r{(a^b)}$ to be the submodule spanned by all diagrams
%$d_{(\Lambda,\Lambda')} $ such that 
%$(i)$   $\Lambda'$ contains no singletons 
%$(ii)$ $\Lambda$ contains no northern arcs.  
%\end{defn}
%

\begin{eg}
An example element of $V^0_{14}(2^3)$ is depicted in \cref{a-concatenate-for-the-hard-thin2}.  
\end{eg}

 The next proposition is  an elementary application of idempotent truncation (see for example 
 \cite[Section 6.2]{MR2349209}) which will allow us to decompose the restriction of $\Fdelta$ to the partition algebra. 
Recall that for $\delta_{\rm in} \delta_{\rm out}  \neq 0$ 
we have defined the idempotent $\egen _{r-1}=\frac1{\delta_{\rm in} \delta_{\rm out} }\pgen _r \in \Pdelta\subseteq R_r(\delta_{\rm in} ,\delta_{\rm out} )$, and, that in \eqref{bob} and 
\eqref{bob2}, we saw that $\egen_{r-1}P_r(\delta_{\rm in} \delta_{\rm out}) \egen_{r-1} \cong  P_{r-1}(\delta_{\rm in} \delta_{\rm out})$  and $P_r(\delta_{\rm in} \delta_{\rm out}) /P_r(\delta_{\rm in} \delta_{\rm out}) \egen_{r-1}P_r(\delta_{\rm in} \delta_{\rm out}) \cong \mathbb{C}\W_r$.
% and we obtain \cref{bob} and \cref{bob2}. 
% By the general theory of idempotent truncation  and  
%\cref{bob,bob2} we obtain the following.  

\begin{prop}\label{whynamed}
 For $r \ge 2$,
$$
{\sf DR}\bigl(\Fdelta\bigr) \egen _{r-1} \Pdelta
={\sf DR}\bigl(\Fdelta\bigr),\qquad \ 
{\sf DQ}\bigl(\Fdelta\bigr)\egen _{r-1} =0,$$
and moreover  
\[ {\sf DR}\bigl(\Fdelta \bigr) \egen _{r-1} \cong  
  \Delta_{r-1}(\alpha^\beta) \]
as an $ \egen _{r-1} \Pdelta  \egen _{r-1} \cong P_{r-1}(\delta_{\rm in} \delta_{\rm out} )$-module
if the right-hand side is defined, and otherwise ${\sf DR}(\Fdelta ) \egen _{r-1}  = 0$.
  When $r=1$, 
  \[
   {\sf DR}  \bigl(   \Delta%^0
   _1  ((1)^{(1)}) \bigr)
%   \cong 
%\Delta_1  ((1)^{(1)})  
=0,  \qquad 
    {\sf DR}  \bigl(   \Delta%^0
    _1  (\varnothing ^{(1)}) \bigr)=0,% \cong 
%\Delta_1  (\varnothing ^{(1)}) 
  \qquad 
   {\sf DR} \bigl( \Delta_1(\varnothing ^{\varnothing}) \bigr) =
\Delta_1  (\varnothing ^{\varnothing} \bigr) .
\]
 
\end{prop}

\begin{proof}
 We consider the first  statement. We let $d_{(\Lambda,\Lambda')} $ be a ramified diagram basis element of $W_r(a^b)$ and $x\in {\rm S}(\alpha^\beta)$.  We shall write $x\otimes _{\W_a \wr \W_b} d_{(\Lambda,\Lambda')} $ in the form
\[
x\otimes _{\W_a \wr \W_b} d_{(\Lambda,\Lambda')}  = 
 x\otimes_{\W_a \wr \W_b}
  d_{(\bar{\Lambda},\bar{\Lambda'}) }\egen _{r-1} d
 \]
 %MkW: overline usually looks better but I think since these are all subscripts we should use bar
 for some ramified diagram $ d_{(\bar{\Lambda},\bar{\Lambda'})} \in W_r(a^b)$ and some partition diagram $d\in \Pdelta$
  and hence deduce the result.  First, suppose that $ \Lambda'  $ contains a singleton block  
 $ \{i\} $ for $1\leq i \leq r$.  In this case we set 
\[
 d_{(\bar{\Lambda},\bar{\Lambda'})}
 =
  d_{(\Lambda,\Lambda')}  \sgen _{i,r},
  \]
  where $\sgen _{i,r}=\sgen _{i ,i+1}\cdots \sgen _{r-1,r-2} \sgen _{r-1,r} 
   \sgen _{r-1,r-2} \cdots \sgen _{i,i+1}$.  
We set 
\[
d_{(\Lambda,\Lambda')}  = d_{(\bar{\Lambda},\bar{\Lambda'})} \egen _{r-1} \sgen _{i,r}
\]
as required.   Now suppose that $  {\Lambda'} $ contains a block $J$ with two southern vertices 
${j_1} < {j_2}  $.
We suppose that ${j_2}$ is maximal  with respect to this property.  
 In this case we set 
\[
 d_{(\bar{\Lambda},\bar{\Lambda'})}
 =
 d_{(\Lambda,\Lambda')}  \sgen _{{j_2},r}  \sgen _{ {j_1},r-1} . \]
%where $\sgen _{ {j_2}}=\sgen _{1,2}\sgen _{1,{j_2}}\sgen _{1,2}$.
We easily observe that 
\[
d_{(\Lambda,\Lambda')}  = d_{(\bar{\Lambda},\bar{\Lambda'})} 
 \egen _{r-1,r} (\pgen _{r-1,r}   \sgen _{ {j_1},r-1} \sgen _{{j_2},r} ).
\]  
We further observe that 
\[
x\otimes_{\W_a \wr \W_b}  d_{(\Lambda,\Lambda')}  = 
(x\otimes_{\W_a \wr \W_b} d_{(\bar{\Lambda},\bar{\Lambda'}  )})\egen _{r-1} d
\]
because of our maximality assumption on ${j_2}$;
as there is no inner block  to the right of $J$ containing a pair of southern vertices, thus 
$   \sgen _{ {j_1},r-1} \sgen _{{j_2},r}$
 does not swap the order of  blocks which can be permuted by the left action of $\W_a \wr \W_b$. 
 The first statement follows.

 We now consider the second and third statements.  Again, consider  $x\otimes_{\W_a \wr \W_b}
  d_{(\Lambda,\Lambda')} $ a  basis element of $\Fdelta$ and consider $d_{(\Lambda,\Lambda')}   \egen _{r-1}$  using \cref{p1} and conjugation.
 In all three cases the resulting outer partition contains a singleton block and therefore $d_{(\Lambda,\Lambda')}   \egen _{r-1} \in {\sf DR}(\Fdelta)$.
 Therefore the second statement holds.
%Considering only $x\otimes d_{(\Lambda,\Lambda')}  \in{\sf DR}(\Fdelta)$, we
Finally, we  see that all possible $(\Pi,\Pi') \in V_r(a^b) $ with a singleton part $\{r\}$ in both $ \Pi$ and $\Pi'$ can occur as  $d_{(\Lambda,\Lambda')}   \egen _{r-1}$, thus the third statement holds.
For $r=1$ modules are all 1-dimensional and the statement is easily verified.
   \end{proof}

The calculations of \cref{whynamed}, together with \cref{bob} and \cref{bob2} provide the following corollary.

\begin{cor}\label{explicitlystated}
There is a short exact sequence of $ \Pdelta $-modules  
\[
0
\to
{\sf DR}\bigl(\Fdelta\bigr)
\to 
\Delta_r(\alpha^\beta)\Res_{\Pdelta} \to {\sf DQ}\bigl(\Fdelta\bigr)  \to 0,\]
where 
\begin{align}\label{bycor}
{\sf DR}\bigl( \Fdelta \bigr) \cong 
 \Delta_{r-1}(\alpha^\beta)  \otimes _{\color{red}\color{black}P_{r-1}(\delta_{\rm in} \delta_{\rm out})}
e_{r-1} P_{r}(\delta_{\rm in} \delta_{\rm out} )
\end{align}
and 
$  {\sf DQ}\bigl( \Fdelta\bigr) $ decomposes as a direct sum of inflated
simple $\mathbb C\W_r$-modules.
% under  
%the identification 
% $$ P_{r}(\delta_{\rm in} \delta_{\rm out} )
%  /  P_{r}(\delta_{\rm in} \delta_{\rm out} ) e_{r-1}  P_{r}(\delta_{\rm in} \delta_{\rm out} )  \cong \mathbb C \W_r.$$
\end{cor}

\begin{cor}\label{cor:compfactors+DR}
For arbitrary parameters $\delta_{\rm in} ,\delta_{\rm out} \neq 0$,  the restriction $\Fdelta \res_{P_r(\delta_{\rm in} \delta_{\rm out} )} ^{\ram _r(\delta_{\rm in} ,\delta_{\rm out})}$ has a
{\color{red}\color{black} well-defined} 
 standard filtration  with   the following equality of filtration  multiplicities:
\begin{align*}
&\left[
\Fdelta\Res_{P_r(\delta_{\rm in} \delta_{\rm out} )} ^{\ram _r(\delta_{\rm in} ,\delta_{\rm out})} 
 \, : \, \Delta_{r }(\lambda)  \right]\raisebox{-6pt}{$\scriptstyle\Pdelta$} \\
&\hspace*{1.2in} =
\begin{cases} \left[  {\sf DQ}(\Delta_r(\alpha^\beta)) \, : \, \Delta_{r }(\lambda)  
\right]\raisebox{-6pt}{$\scriptstyle\Pdelta$}  & \textrm{if } |\lambda|=r,\\
\left[ \Delta_{r-1}(\alpha^\beta)
\Res_{P_r(\delta_{{\rm in}} \delta_{\rm out} )} ^{\ram _r(\delta_{\rm in}\delta_{\rm out} )}
{\color{red}\color{black}\Res^{P_r(\delta_{{\rm in}} \delta_{\rm out} )} _{P_{r-1}(\delta_{{\rm in}} \delta_{\rm out} )} }
 \, : \, \Delta_{r-1 }(\lambda)  \right]
\raisebox{-6pt}{$\scriptstyle P_{r-1}(\delta_{\rm in} \delta_{\rm out} ) $}  & \textrm{if } |\lambda|<r.
%\left[ {\sf DR}(\Delta_{r-1}(\alpha^\beta)) \, : \, \Delta_{r-1 }(\lambda)  \right]_{P_{r-1}
%(\delta_{\rm in} \delta_{\rm out} ) }  & \textrm{if } |\lambda|<r.
\end{cases}
\end{align*}
 \end{cor}
 
 %The corollary describes these well-defined multiplicities.
 
\begin{proof}{\color{red}\color{black}Since $\Pdelta$ is quasi-hereditary \cite{mar1} when $\delta_{\rm in} ,\delta_{\rm out} \neq 0$, one need only show that 
$
\color{red}\color{black}\Fdelta\Res_{P_r(\delta_{\rm in} \delta_{\rm out} )} ^{\ram _r(\delta_{\rm in} ,\delta_{\rm out})} $
 possesses a  standard filtration  in order to deduce that this filtration is well-defined 
(see, for example \cite[Appendix, A1(8)]{Donkin}).} 
The existence of a standard module filtration  of $\color{red}\color{black}\Fdelta\Res_{P_r(\delta_{\rm in} \delta_{\rm out} )} ^{\ram _r(\delta_{\rm in} ,\delta_{\rm out})} $ is proved by induction. The base case $r=1$ is clear from \cref{whynamed}: the three (1-dimensional) standard modules for the ramified partition algebra 
%on 1 dot  %MkW: can we just say r=1?
restrict to  standard modules for the partition algebra. 
  For  \hbox{$r>1$}, we use the short exact sequence above. The quotient is a direct sum of simple $\mathbb C\W_r$-modules, which are $ P_{r}(\delta_{\rm in} \delta_{\rm out} )$-standard modules by inflation.
  If the submodule  ${\sf DR}(\Fdelta)$ is non-zero then a standard filtration of 
  ${\sf DR}(\Fdelta) \cong   \Delta_{r-1}(\alpha^\beta) 
   \otimes e_{r-1} P_{r}(\delta_{\rm in} \delta_{\rm out} )$ is 
   obtained from the $P_{r-1}(\delta_{\rm in} \delta_{\rm out} )$-standard
    filtration of $\Delta_{r-1}(\alpha^\beta)\res_{P_{r-1}(\delta_{\rm in} \delta_{\rm out} )}^{\ram _{r-1}(\delta_{\rm in} ,\delta_{\rm out})} $ by globalisation
    using the isomorphism  $\Delta_{r-1}(\kappa) \otimes 
\egen_{r-1} P_{r}(\delta_{\rm in} \delta_{\rm out} ) \cong \Delta_{r}(\kappa)$.
\end{proof}

\section{General formula for ramified branching coefficients}\label{sec9}

We now consider the decomposition of the module
%Delta_r^0(\alpha^\beta) = 
\[
%{\sf DQ}\bigl(\Delta_{r}(\alpha^\beta)\bigr)
\begin{minipage}{2.04cm}
${\sf DQ}\bigl( \Delta_{r}(\alpha^\beta) \bigr)$
\end{minipage}
=\rightspecht {   \alpha ^{  \beta} }\otimes_{\W_a \wr \W_b} V^0_r{(a^b)}.\]
 %, which we regard as a $\C (\W_a\wr \W_b, \W_r)$-module via the canonical quotient map $P_r(mn)\to \C \W_r$.  Thus, for the remainder of this paper we need not  consider the ramified partition algebra structure, we can simply discuss  bimodules for wreath products of symmetric groups.
  %(although we will, of course, continue to draw pictures for these group algebra modules).  
Using the canonical quotient map $P_r(mn)\to \C \W_r$ arising from  \cref{bob2}, we regard $V^0_r{(a^b)}$ as a $(\C \W_a\wr \W_b, \C \W_r)$-bimodule. Thus, for the remainder of this paper, we need not  consider the %ramified 
partition algebra structure, we can simply discuss  bimodules for wreath products of symmetric groups.

 \subsection{Types of diagrams }\label{generatorsofthetransitive}
 %Fix $a,b,  r\geq 1$.   
Fix $a,b,r$ with $ab\leq r$ or $b \leq r$ if $a=0$.
We   wish to understand the left action  of $\W_a\wr \W_b$ 
and the right action $\W_r$ on $V^0_r(a^b)$. 
Recall that $V^0_r(a^b) \subset V_r(a^b)$ has a basis  given by the ramified diagrams 
$	d_	{(\Lambda,\Lambda')}$ 
 of propagating index $(a^b)$, as defined in Section~\ref{subsec:propagatingIndex},
 such that  
 \begin{itemize}[leftmargin=18pt] 
 \item the outer set-partition $\Lambda$ has no singleton blocks; 
 \item the inner set-partition $\Lambda'$ consists of propagating pairs (that is, pairs $\{i, \overline{j} \}$ for $i$ a northern vertex and $ \overline{j}$ a southern vertex) and southern singletons.
 \end{itemize}
Examples are depicted in 
\cref{a-concatenate-for-the-hard-thin2,plenty1,plenty2,plenty3,aplenty1,aplenty2,Phereiszero}.  
The purpose of this section is to define the propagating type and the non-propagating type of such  a ramified diagram 
 $ d_	{(\Lambda,\Lambda')}\in V^0_r(a^b)$   in such a way as to decompose
  this module and to determine a direct sum decomposition of 
  ${\sf DQ}(\Delta_{r}(\alpha^\beta))$ as a $\mathbb C \W_r$-module. %a direct sum of cyclically generated   $(\W_a\wr \W_b,\W_r)$-bimodules.     
  
 \begin{figure}[ht!]
 $$\begin{minipage}{4cm}\begin{tikzpicture}[scale=0.45] 
  
       \foreach \x in {0,2,4,6,8,9,10,12,14,16,18,20}
     {   \path(\x,3) coordinate (up\x);  
      \path(\x,0) coordinate (down\x);  
   }

  {   \path(down10) --++ (135:0.4) coordinate(down1ten);  
      \path(down10) --++ (45:0.4) coordinate(down2ten);
            \path(down10) --++ (-45:0.4) coordinate(down3ten);  
                        \path(down10) --++ (-135:0.4) coordinate(down4ten);  
   }

     {   \path(up10) --++ (135:0.4) coordinate(up1ten);  
      \path(up10) --++ (45:0.4) coordinate(up2ten);
            \path(up10) --++ (-45:0.4) coordinate(up3ten);  
                        \path(up10) --++ (-135:0.4) coordinate(up4ten);  
   }

  {   \path(down12) --++ (135:0.4) coordinate(down1twelve);  
      \path(down12) --++ (45:0.4) coordinate(down2twelve);
            \path(down12) --++ (-45:0.4) coordinate(down3twelve);  
                        \path(down12) --++ (-135:0.4) coordinate(down4twelve);  
   }

     {   \path(up12) --++ (135:0.4) coordinate(up1twelve);  
      \path(up12) --++ (45:0.4) coordinate(up2twelve);
            \path(up12) --++ (-45:0.4) coordinate(up3twelve);  
                        \path(up12) --++ (-135:0.4) coordinate(up4twelve);  
   }

  {   \path(down14) --++ (135:0.4) coordinate(down1fourteen);  
      \path(down14) --++ (45:0.4) coordinate(down2fourteen);
            \path(down14) --++ (-45:0.4) coordinate(down3fourteen);  
                        \path(down14) --++ (-135:0.4) coordinate(down4fourteen);  
   }

     {   \path(up14) --++ (135:0.4) coordinate(up1fourteen);  
      \path(up14) --++ (45:0.4) coordinate(up2fourteen);
            \path(up14) --++ (-45:0.4) coordinate(up3fourteen);  
                        \path(up14) --++ (-135:0.4) coordinate(up4fourteen);  
   }

  {   \path(down16) --++ (135:0.4) coordinate(down1sixteen);  
      \path(down16) --++ (45:0.4) coordinate(down2sixteen);
            \path(down16) --++ (-45:0.4) coordinate(down3sixteen);  
                        \path(down16) --++ (-135:0.4) coordinate(down4sixteen);  
   }

     {   \path(up16) --++ (135:0.4) coordinate(up1sixteen);  
      \path(up16) --++ (45:0.4) coordinate(up2sixteen);
            \path(up16) --++ (-45:0.4) coordinate(up3sixteen);  
                        \path(up16) --++ (-135:0.4) coordinate(up4sixteen);  
   }

  {   \path(down18) --++ (135:0.4) coordinate(down1eighteen);  
      \path(down18) --++ (45:0.4) coordinate(down2eighteen);
            \path(down18) --++ (-45:0.4) coordinate(down3eighteen);  
                        \path(down18) --++ (-135:0.4) coordinate(down4eighteen);  
   }

     {   \path(up18) --++ (135:0.4) coordinate(up1eighteen);  
      \path(up18) --++ (45:0.4) coordinate(up2eighteen);
            \path(up18) --++ (-45:0.4) coordinate(up3eighteen);  
                        \path(up18) --++ (-135:0.4) coordinate(up4eighteen);  
   }

  {   \path(down16) --++ (135:0.4) coordinate(down1twenty);  
      \path(down16) --++ (45:0.4) coordinate(down2twenty);
            \path(down16) --++ (-45:0.4) coordinate(down3twenty);  
                        \path(down16) --++ (-135:0.4) coordinate(down4twenty);  
   }

     {   \path(up16) --++ (135:0.4) coordinate(up1twenty);  
      \path(up16) --++ (45:0.4) coordinate(up2twenty);
            \path(up16) --++ (-45:0.4) coordinate(up3twenty);  
                        \path(up16) --++ (-135:0.4) coordinate(up4twenty);  
   }

     \foreach \x in {0,2,4,6,8,9,10,12,14,16,18,20}
     {   \path(up\x) --++ (135:0.4) coordinate(up1\x);  
      \path(up\x) --++ (45:0.4) coordinate(up2\x);
            \path(up\x) --++ (-45:0.4) coordinate(up3\x);  
                        \path(up\x) --++ (-135:0.4) coordinate(up4\x);  
   }

    \foreach \x in {0,2,4,6,8,9,10,12,14,16,18,20}
     {   \path(down\x) --++ (135:0.4) coordinate(down1\x);  
      \path(down\x) --++ (45:0.4) coordinate(down2\x);
            \path(down\x) --++ (-45:0.4) coordinate(down3\x);  
                        \path(down\x) --++ (-135:0.4) coordinate(down4\x);  
   }

        \draw [fill=white] plot [smooth cycle]
  coordinates {(up14) (up26) (down38)(down46) (4 ,0.6) (3 ,0.6)  (down32) (down12)  };

 \draw(6,3)--(2,0);

    \draw(8,0)--(4,3);

   \draw [fill=white] plot [smooth cycle]
  coordinates {(up10) (up22) 
 (down24) 
  (down34) 
  (down44) 
  (2,0.6)  (down40)  };

 \draw(0,3)--(0,0);                       
 \draw(2,3)--(4,0);

       \foreach \x in {0,2,4,6,8, }
     {   \path(\x+10,3) coordinate (aup\x);  
      \path(\x+10,0) coordinate (adown\x);  
   }

     \foreach \x in {0,2,4,6,8, }
     {   \path(aup\x) --++ (135:0.4) coordinate(aup1\x);  
      \path(aup\x) --++ (45:0.4) coordinate(aup2\x);
            \path(aup\x) --++ (-45:0.4) coordinate(aup3\x);  
                        \path(aup\x) --++ (-135:0.4) coordinate(aup4\x);  
   }

    \foreach \x in {0,2,4,6,8, }
     {   \path(adown\x) --++ (135:0.4) coordinate(adown1\x);  
      \path(adown\x) --++ (45:0.4) coordinate(adown2\x);
            \path(adown\x) --++ (-45:0.4) coordinate(adown3\x);  
                        \path(adown\x) --++ (-135:0.4) coordinate(adown4\x);  
   }

   \draw [fill=white] plot [smooth cycle]
  coordinates {(adown10) (adown22) (adown32)  (adown40)  };

     \foreach \x in {0,2,4,6,6}
     {   \fill[white](\x,3) circle (4pt);   
                \fill[white](\x,0) circle (4pt);    
                        \draw (\x,3) node {$\bullet$}; 
                           \draw (\x,0) node {${\bullet}$};  }

     \foreach \x in {0,2,4,6,8,10,12 }
     {   
                           \draw (\x,0) node {${\bullet}$};  }
      
         \end{tikzpicture}   \end{minipage}
          $$
\caption{A ramified diagram  $d_	{(\Lambda,\Lambda')} \in 
V_7^0(2^2)$. Here   the outer set-partition is  $\Lambda=\{	\{	1,2,\overline{1},\overline{3}	\}	,
\{	3,4, \overline{2},\overline{4},\overline{5}	\},
\{	\overline{6}, \overline{7}	\}
	\}$ and 
	the inner set-partition is $\Lambda'=
	\{	
\{1,\overline{1}\}
,	
\{2,\overline{3}\}
,\{3,\overline{5}\}
,\{4,\overline{2}\}
,\{ \overline{4}\}
,\{ \overline{6}\}	,\{ \overline{7}\}	\}$.  Compare with \cref{plenty2}.}
\label{plenty1} 
\end{figure}
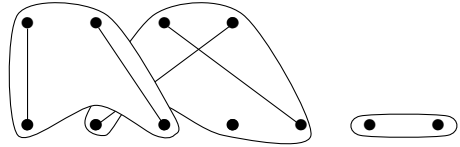

     \begin{figure}[ht!]
 $$\begin{minipage}{4cm}\begin{tikzpicture}[scale=0.45]
  
       \foreach \x in {0,2,4,6,8,9,10,12,14,16,18,20}
     {   \path(\x,3) coordinate (up\x);  
      \path(\x,0) coordinate (down\x);  
   }

  {   \path(down10) --++ (135:0.4) coordinate(down1ten);  
      \path(down10) --++ (45:0.4) coordinate(down2ten);
            \path(down10) --++ (-45:0.4) coordinate(down3ten);  
                        \path(down10) --++ (-135:0.4) coordinate(down4ten);  
   }

     {   \path(up10) --++ (135:0.4) coordinate(up1ten);  
      \path(up10) --++ (45:0.4) coordinate(up2ten);
            \path(up10) --++ (-45:0.4) coordinate(up3ten);  
                        \path(up10) --++ (-135:0.4) coordinate(up4ten);  
   }

  {   \path(down12) --++ (135:0.4) coordinate(down1twelve);  
      \path(down12) --++ (45:0.4) coordinate(down2twelve);
            \path(down12) --++ (-45:0.4) coordinate(down3twelve);  
                        \path(down12) --++ (-135:0.4) coordinate(down4twelve);  
   }

     {   \path(up12) --++ (135:0.4) coordinate(up1twelve);  
      \path(up12) --++ (45:0.4) coordinate(up2twelve);
            \path(up12) --++ (-45:0.4) coordinate(up3twelve);  
                        \path(up12) --++ (-135:0.4) coordinate(up4twelve);  
   }

  {   \path(down14) --++ (135:0.4) coordinate(down1fourteen);  
      \path(down14) --++ (45:0.4) coordinate(down2fourteen);
            \path(down14) --++ (-45:0.4) coordinate(down3fourteen);  
                        \path(down14) --++ (-135:0.4) coordinate(down4fourteen);  
   }

     {   \path(up14) --++ (135:0.4) coordinate(up1fourteen);  
      \path(up14) --++ (45:0.4) coordinate(up2fourteen);
            \path(up14) --++ (-45:0.4) coordinate(up3fourteen);  
                        \path(up14) --++ (-135:0.4) coordinate(up4fourteen);  
   }

  {   \path(down16) --++ (135:0.4) coordinate(down1sixteen);  
      \path(down16) --++ (45:0.4) coordinate(down2sixteen);
            \path(down16) --++ (-45:0.4) coordinate(down3sixteen);  
                        \path(down16) --++ (-135:0.4) coordinate(down4sixteen);  
   }

     {   \path(up16) --++ (135:0.4) coordinate(up1sixteen);  
      \path(up16) --++ (45:0.4) coordinate(up2sixteen);
            \path(up16) --++ (-45:0.4) coordinate(up3sixteen);  
                        \path(up16) --++ (-135:0.4) coordinate(up4sixteen);  
   }

  {   \path(down18) --++ (135:0.4) coordinate(down1eighteen);  
      \path(down18) --++ (45:0.4) coordinate(down2eighteen);
            \path(down18) --++ (-45:0.4) coordinate(down3eighteen);  
                        \path(down18) --++ (-135:0.4) coordinate(down4eighteen);  
   }

     {   \path(up18) --++ (135:0.4) coordinate(up1eighteen);  
      \path(up18) --++ (45:0.4) coordinate(up2eighteen);
            \path(up18) --++ (-45:0.4) coordinate(up3eighteen);  
                        \path(up18) --++ (-135:0.4) coordinate(up4eighteen);  
   }

  {   \path(down16) --++ (135:0.4) coordinate(down1twenty);  
      \path(down16) --++ (45:0.4) coordinate(down2twenty);
            \path(down16) --++ (-45:0.4) coordinate(down3twenty);  
                        \path(down16) --++ (-135:0.4) coordinate(down4twenty);  
   }

     {   \path(up16) --++ (135:0.4) coordinate(up1twenty);  
      \path(up16) --++ (45:0.4) coordinate(up2twenty);
            \path(up16) --++ (-45:0.4) coordinate(up3twenty);  
                        \path(up16) --++ (-135:0.4) coordinate(up4twenty);  
   }

     \foreach \x in {0,2,4,6,8,9,10,12,14,16,18,20}
     {   \path(up\x) --++ (135:0.4) coordinate(up1\x);  
      \path(up\x) --++ (45:0.4) coordinate(up2\x);
            \path(up\x) --++ (-45:0.4) coordinate(up3\x);  
                        \path(up\x) --++ (-135:0.4) coordinate(up4\x);  
   }

    \foreach \x in {0,2,4,6,8,9,10,12,14,16,18,20}
     {   \path(down\x) --++ (135:0.4) coordinate(down1\x);  
      \path(down\x) --++ (45:0.4) coordinate(down2\x);
            \path(down\x) --++ (-45:0.4) coordinate(down3\x);  
                        \path(down\x) --++ (-135:0.4) coordinate(down4\x);  
   }

       \foreach \x in {0,2,4,6,8, }
     {   \path(\x+10,3) coordinate (aup\x);  
      \path(\x+10,0) coordinate (adown\x);  
   }

     \foreach \x in {0,2,4,6,8, }
     {   \path(aup\x) --++ (135:0.4) coordinate(aup1\x);  
      \path(aup\x) --++ (45:0.4) coordinate(aup2\x);
            \path(aup\x) --++ (-45:0.4) coordinate(aup3\x);  
                        \path(aup\x) --++ (-135:0.4) coordinate(aup4\x);  
   }

    \foreach \x in {0,2,4,6,8, }
     {   \path(adown\x) --++ (135:0.4) coordinate(adown1\x);  
      \path(adown\x) --++ (45:0.4) coordinate(adown2\x);
            \path(adown\x) --++ (-45:0.4) coordinate(adown3\x);  
                        \path(adown\x) --++ (-135:0.4) coordinate(adown4\x);  
   }

   \draw [fill=white] plot [smooth cycle]
  coordinates {(up10) (up22) (down26)  (down36) (3.7,0.6)  (down42)  };

 \draw(0,3)--(6,0);                       
 \draw(2,3)--(2,0);

  \draw [fill=white] plot [smooth cycle]
  coordinates {(up14) (up26) (adown22)(adown32)(adown40)    (down48)  (up44) };

%  \draw [fill=white] plot [smooth cycle]
%  coordinates {(down10)  (2,1.2)  (down24)  (down34)
%   (down44)  (2,0.5)  (down30) (down40)     };

  \draw [fill=white] plot [smooth cycle]
  coordinates {(down10)  (2,1.2)  (down24)  (down34)
   (down44)  (2,0.5)  (down30) (down40)     };

%   
%   \draw [fill=white] plot [smooth cycle]
%  coordinates {(adown10) (adown22) (adown32)  (adown40)  };
%

 \draw(4,3)--(12,0);

    \draw(8,0)--(6,3);

     \foreach \x in {0,2,4,6,6}
     {   \fill[white](\x,3) circle (4pt);   
                \fill[white](\x,0) circle (4pt);    
                        \draw (\x,3) node {$\bullet$}; 
                           \draw (\x,0) node {${\bullet}$};  }

     \foreach \x in {0,2,4,6,8,10,12 }
     {   
                           \draw (\x,0) node {${\bullet}$};  }
      
         \end{tikzpicture}   \end{minipage}
          $$
\caption{Another ramified diagram in $V_7^0(2^2)$.  This diagram can be obtained from that of  \cref{plenty1} by acting on the right by the permutation 
$(1,4,6)(2,5,7,3) \in \W_7$.
%(although we remark that this permutation is not the unique such permutation).
}
 \label{plenty2} 
\end{figure}
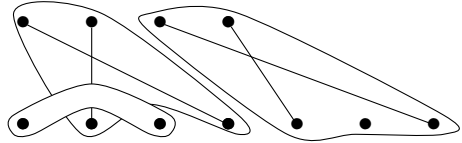

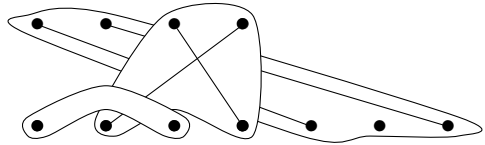
\begin{figure}[ht!]

 $$\begin{minipage}{4cm}\begin{tikzpicture}[scale=0.45]
  
       \foreach \x in {0,2,4,6,8,9,10,12,14,16,18,20}
     {   \path(\x,3) coordinate (up\x);  
      \path(\x,0) coordinate (down\x);  
   }

  {   \path(down10) --++ (135:0.4) coordinate(down1ten);  
      \path(down10) --++ (45:0.4) coordinate(down2ten);
            \path(down10) --++ (-45:0.4) coordinate(down3ten);  
                        \path(down10) --++ (-135:0.4) coordinate(down4ten);  
   }

     {   \path(up10) --++ (135:0.4) coordinate(up1ten);  
      \path(up10) --++ (45:0.4) coordinate(up2ten);
            \path(up10) --++ (-45:0.4) coordinate(up3ten);  
                        \path(up10) --++ (-135:0.4) coordinate(up4ten);  
   }

  {   \path(down12) --++ (135:0.4) coordinate(down1twelve);  
      \path(down12) --++ (45:0.4) coordinate(down2twelve);
            \path(down12) --++ (-45:0.4) coordinate(down3twelve);  
                        \path(down12) --++ (-135:0.4) coordinate(down4twelve);  
   }

     {   \path(up12) --++ (135:0.4) coordinate(up1twelve);  
      \path(up12) --++ (45:0.4) coordinate(up2twelve);
            \path(up12) --++ (-45:0.4) coordinate(up3twelve);  
                        \path(up12) --++ (-135:0.4) coordinate(up4twelve);  
   }

  {   \path(down14) --++ (135:0.4) coordinate(down1fourteen);  
      \path(down14) --++ (45:0.4) coordinate(down2fourteen);
            \path(down14) --++ (-45:0.4) coordinate(down3fourteen);  
                        \path(down14) --++ (-135:0.4) coordinate(down4fourteen);  
   }

     {   \path(up14) --++ (135:0.4) coordinate(up1fourteen);  
      \path(up14) --++ (45:0.4) coordinate(up2fourteen);
            \path(up14) --++ (-45:0.4) coordinate(up3fourteen);  
                        \path(up14) --++ (-135:0.4) coordinate(up4fourteen);  
   }

  {   \path(down16) --++ (135:0.4) coordinate(down1sixteen);  
      \path(down16) --++ (45:0.4) coordinate(down2sixteen);
            \path(down16) --++ (-45:0.4) coordinate(down3sixteen);  
                        \path(down16) --++ (-135:0.4) coordinate(down4sixteen);  
   }

     {   \path(up16) --++ (135:0.4) coordinate(up1sixteen);  
      \path(up16) --++ (45:0.4) coordinate(up2sixteen);
            \path(up16) --++ (-45:0.4) coordinate(up3sixteen);  
                        \path(up16) --++ (-135:0.4) coordinate(up4sixteen);  
   }

  {   \path(down18) --++ (135:0.4) coordinate(down1eighteen);  
      \path(down18) --++ (45:0.4) coordinate(down2eighteen);
            \path(down18) --++ (-45:0.4) coordinate(down3eighteen);  
                        \path(down18) --++ (-135:0.4) coordinate(down4eighteen);  
   }

     {   \path(up18) --++ (135:0.4) coordinate(up1eighteen);  
      \path(up18) --++ (45:0.4) coordinate(up2eighteen);
            \path(up18) --++ (-45:0.4) coordinate(up3eighteen);  
                        \path(up18) --++ (-135:0.4) coordinate(up4eighteen);  
   }

  {   \path(down16) --++ (135:0.4) coordinate(down1twenty);  
      \path(down16) --++ (45:0.4) coordinate(down2twenty);
            \path(down16) --++ (-45:0.4) coordinate(down3twenty);  
                        \path(down16) --++ (-135:0.4) coordinate(down4twenty);  
   }

     {   \path(up16) --++ (135:0.4) coordinate(up1twenty);  
      \path(up16) --++ (45:0.4) coordinate(up2twenty);
            \path(up16) --++ (-45:0.4) coordinate(up3twenty);  
                        \path(up16) --++ (-135:0.4) coordinate(up4twenty);  
   }

     \foreach \x in {0,2,4,6,8,9,10,12,14,16,18,20}
     {   \path(up\x) --++ (135:0.4) coordinate(up1\x);  
      \path(up\x) --++ (45:0.4) coordinate(up2\x);
            \path(up\x) --++ (-45:0.4) coordinate(up3\x);  
                        \path(up\x) --++ (-135:0.4) coordinate(up4\x);  
   }

    \foreach \x in {0,2,4,6,8,9,10,12,14,16,18,20}
     {   \path(down\x) --++ (135:0.4) coordinate(down1\x);  
      \path(down\x) --++ (45:0.4) coordinate(down2\x);
            \path(down\x) --++ (-45:0.4) coordinate(down3\x);  
                        \path(down\x) --++ (-135:0.4) coordinate(down4\x);  
   }

       \foreach \x in {0,2,4,6,8, }
     {   \path(\x+10,3) coordinate (aup\x);  
      \path(\x+10,0) coordinate (adown\x);  
   }

     \foreach \x in {0,2,4,6,8, }
     {   \path(aup\x) --++ (135:0.4) coordinate(aup1\x);  
      \path(aup\x) --++ (45:0.4) coordinate(aup2\x);
            \path(aup\x) --++ (-45:0.4) coordinate(aup3\x);  
                        \path(aup\x) --++ (-135:0.4) coordinate(aup4\x);  
   }

    \foreach \x in {0,2,4,6,8, }
     {   \path(adown\x) --++ (135:0.4) coordinate(adown1\x);  
      \path(adown\x) --++ (45:0.4) coordinate(adown2\x);
            \path(adown\x) --++ (-45:0.4) coordinate(adown3\x);  
                        \path(adown\x) --++ (-135:0.4) coordinate(adown4\x);  
   }

  \draw [fill=white] plot [smooth cycle]
  coordinates {(up10) (up22) (adown22)(adown32)(adown40)    (down48)  (up40) };

 \draw(2,3)--(12,0);

    \draw(8,0)--(0,3);

   \draw [fill=white] plot [smooth cycle]
  coordinates {(up14) (up26) (down36) (4,0.5)   (down32)   (down12)  };

 \draw(4,3)--(6,0);                       
 \draw(6,3)--(2,0);

%  \draw [fill=white] plot [smooth cycle]
%  coordinates {(down10)  (2,1.2)  (down24)  (down34)  (2,0.5)  (down40)     };

  \draw [fill=white] plot [smooth cycle]
  coordinates {(down10)  (2,1.2)  (down24)  (down34)
   (down44)  (2,0.5)  (down30) (down40)     };

%   
%   \draw [fill=white] plot [smooth cycle]
%  coordinates {(adown10) (adown22) (adown32)  (adown40)  };
%

     \foreach \x in {0,2,4,6,6}
     {   \fill[white](\x,3) circle (4pt);   
                \fill[white](\x,0) circle (4pt);    
                        \draw (\x,3) node {$\bullet$}; 
                           \draw (\x,0) node {${\bullet}$};  }

     \foreach \x in {0,2,4,6,8,10,12 }
     {   
                           \draw (\x,0) node {${\bullet}$};  }
      
         \end{tikzpicture}   \end{minipage}
          $$
\caption{Another diagram in $V_7^0(2^2)$.  This ramified diagram can be obtained from that of  \cref{plenty2} by acting on the left by $ ( (1,2), 1_{\W_2}; (1,2)) \in \W_2\wr  \W_2$.  Observe that it cannot be obtained from the ramified diagram of \cref{plenty2} via a right action.}
%applying the permutation 
%$ (1,4,2,3) \in \W_2\wr  \W_2 \leq  \W_4$ to the left of the diagram
%(although we remark that this permutation is not the unique such permutation).}
 \label{plenty3} 
\end{figure}

 Our aim is to decompose $V^0_r(a^b)$ and so we shall first 
  focus on the properties of the ramified diagram basis elements $d_{(\Lambda,\Lambda')}$ which are %`permutation invariant'. 
  preserved by the left and right actions.
   Roughly speaking, this means we shall need to describe
   the number of singletons and their  positioning   within the blocks of ${(\Lambda,\Lambda')}$.  
   For example, the diagrams in \cref{plenty1,plenty2,plenty3} are all obtained from one another by left action by $\W_a\wr \W_b$ and/or right action  by $\W_r$. On the other hand, those in \cref{plenty1,aplenty1,aplenty2} {\em cannot} be obtained from one another by these actions.  
 In order to discuss this in more detail, we first note that within the ramified diagram basis element $d_{(\Lambda,\Lambda')} \in V^0_r(a^b)$  we have that:
\begin{itemize}[leftmargin=18pt] 
\item there are precisely  $b$ propagating outer-blocks $P_1,\dots, P_b$. The refinement  
 of each block,   $\Lambda\cap P_i$ for $1\leq i \leq b$, consists of precisely $a$ propagating 
 pairs  and some number, $\gamma_i$ say, of southern singletons;

\item there are some  non-propagating southern outer-blocks,  say $Q_1,\dots , Q_l$ for some $l\geq0$. 
The refinement  
 of each block,   $\Lambda\cap Q_i $ for $1\leq i \leq \ell$, consists of some number, $\deltaP_i \geq 2$ say, of 
 singleton southern blocks.
 \end{itemize}
 
 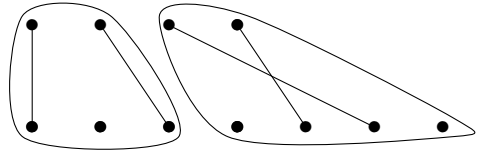
\begin{figure}[ht!]

 $$\begin{minipage}{4cm}\begin{tikzpicture}[scale=0.45]
  
       \foreach \x in {0,2,4,6,8,9,10,12,14,16,18,20}
     {   \path(\x,3) coordinate (up\x);  
      \path(\x,0) coordinate (down\x);  
   }

  {   \path(down10) --++ (135:0.4) coordinate(down1ten);  
      \path(down10) --++ (45:0.4) coordinate(down2ten);
            \path(down10) --++ (-45:0.4) coordinate(down3ten);  
                        \path(down10) --++ (-135:0.4) coordinate(down4ten);  
   }

     {   \path(up10) --++ (135:0.4) coordinate(up1ten);  
      \path(up10) --++ (45:0.4) coordinate(up2ten);
            \path(up10) --++ (-45:0.4) coordinate(up3ten);  
                        \path(up10) --++ (-135:0.4) coordinate(up4ten);  
   }

  {   \path(down12) --++ (135:0.4) coordinate(down1twelve);  
      \path(down12) --++ (45:0.4) coordinate(down2twelve);
            \path(down12) --++ (-45:0.4) coordinate(down3twelve);  
                        \path(down12) --++ (-135:0.4) coordinate(down4twelve);  
   }

     {   \path(up12) --++ (135:0.4) coordinate(up1twelve);  
      \path(up12) --++ (45:0.4) coordinate(up2twelve);
            \path(up12) --++ (-45:0.4) coordinate(up3twelve);  
                        \path(up12) --++ (-135:0.4) coordinate(up4twelve);  
   }

  {   \path(down14) --++ (135:0.4) coordinate(down1fourteen);  
      \path(down14) --++ (45:0.4) coordinate(down2fourteen);
            \path(down14) --++ (-45:0.4) coordinate(down3fourteen);  
                        \path(down14) --++ (-135:0.4) coordinate(down4fourteen);  
   }

     {   \path(up14) --++ (135:0.4) coordinate(up1fourteen);  
      \path(up14) --++ (45:0.4) coordinate(up2fourteen);
            \path(up14) --++ (-45:0.4) coordinate(up3fourteen);  
                        \path(up14) --++ (-135:0.4) coordinate(up4fourteen);  
   }

  {   \path(down16) --++ (135:0.4) coordinate(down1sixteen);  
      \path(down16) --++ (45:0.4) coordinate(down2sixteen);
            \path(down16) --++ (-45:0.4) coordinate(down3sixteen);  
                        \path(down16) --++ (-135:0.4) coordinate(down4sixteen);  
   }

     {   \path(up16) --++ (135:0.4) coordinate(up1sixteen);  
      \path(up16) --++ (45:0.4) coordinate(up2sixteen);
            \path(up16) --++ (-45:0.4) coordinate(up3sixteen);  
                        \path(up16) --++ (-135:0.4) coordinate(up4sixteen);  
   }

  {   \path(down18) --++ (135:0.4) coordinate(down1eighteen);  
      \path(down18) --++ (45:0.4) coordinate(down2eighteen);
            \path(down18) --++ (-45:0.4) coordinate(down3eighteen);  
                        \path(down18) --++ (-135:0.4) coordinate(down4eighteen);  
   }

     {   \path(up18) --++ (135:0.4) coordinate(up1eighteen);  
      \path(up18) --++ (45:0.4) coordinate(up2eighteen);
            \path(up18) --++ (-45:0.4) coordinate(up3eighteen);  
                        \path(up18) --++ (-135:0.4) coordinate(up4eighteen);  
   }

  {   \path(down16) --++ (135:0.4) coordinate(down1twenty);  
      \path(down16) --++ (45:0.4) coordinate(down2twenty);
            \path(down16) --++ (-45:0.4) coordinate(down3twenty);  
                        \path(down16) --++ (-135:0.4) coordinate(down4twenty);  
   }

     {   \path(up16) --++ (135:0.4) coordinate(up1twenty);  
      \path(up16) --++ (45:0.4) coordinate(up2twenty);
            \path(up16) --++ (-45:0.4) coordinate(up3twenty);  
                        \path(up16) --++ (-135:0.4) coordinate(up4twenty);  
   }

     \foreach \x in {0,2,4,6,8,9,10,12,14,16,18,20}
     {   \path(up\x) --++ (135:0.4) coordinate(up1\x);  
      \path(up\x) --++ (45:0.4) coordinate(up2\x);
            \path(up\x) --++ (-45:0.4) coordinate(up3\x);  
                        \path(up\x) --++ (-135:0.4) coordinate(up4\x);  
   }

    \foreach \x in {0,2,4,6,8,9,10,12,14,16,18,20}
     {   \path(down\x) --++ (135:0.4) coordinate(down1\x);  
      \path(down\x) --++ (45:0.4) coordinate(down2\x);
            \path(down\x) --++ (-45:0.4) coordinate(down3\x);  
                        \path(down\x) --++ (-135:0.4) coordinate(down4\x);  
   }

   \draw [fill=white] plot [smooth cycle]
  coordinates {(up10) (up22) (down34)  (down40)  };

 \draw(0,3)--(0,0);                       
 \draw(2,3)--(4,0);

%   
%   \draw [fill=white] plot [smooth cycle]
%  coordinates {(adown10) (adown22) (adown32)  (adown40)  };

       \foreach \x in {0,2,4,6,8, }
     {   \path(\x+10,3) coordinate (aup\x);  
      \path(\x+10,0) coordinate (adown\x);  
   }

     \foreach \x in {0,2,4,6,8, }
     {   \path(aup\x) --++ (135:0.4) coordinate(aup1\x);  
      \path(aup\x) --++ (45:0.4) coordinate(aup2\x);
            \path(aup\x) --++ (-45:0.4) coordinate(aup3\x);  
                        \path(aup\x) --++ (-135:0.4) coordinate(aup4\x);  
   }

    \foreach \x in {0,2,4,6,8, }
     {   \path(adown\x) --++ (135:0.4) coordinate(adown1\x);  
      \path(adown\x) --++ (45:0.4) coordinate(adown2\x);
            \path(adown\x) --++ (-45:0.4) coordinate(adown3\x);  
                        \path(adown\x) --++ (-135:0.4) coordinate(adown4\x);  
   }

        \draw [fill=white] plot [smooth cycle]
  coordinates {(up14) (up26) (adown22)  (adown32)    (down46)  };

 \draw(6,3)--(8,0);

    \draw(10,0)--(4,3);

     \foreach \x in {0,2,4,6,6}
     {   \fill[white](\x,3) circle (4pt);   
                \fill[white](\x,0) circle (4pt);    
                        \draw (\x,3) node {$\bullet$}; 
                           \draw (\x,0) node {${\bullet}$};  }

     \foreach \x in {0,2,4,6,8,10,12 }
     {   
                           \draw (\x,0) node {${\bullet}$};  }
      
         \end{tikzpicture}   \end{minipage}
          $$
          \caption{A diagram in $V_7^0(2^2)$.  This diagram cannot be obtained from any 
          of \cref{plenty1,plenty2,plenty3} or \cref{aplenty2} via the $\W_2\wr \W_2$-left-action or $\W_7$-right-action (although it can be obtained from any of these diagrams by right multiplication by a ramified partition diagram).}
          \label{aplenty1}
\end{figure}

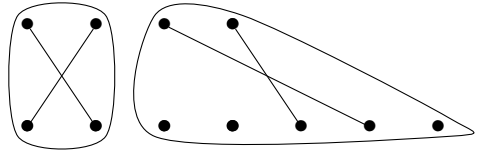
\begin{figure}[ht!]
    
 $$\begin{minipage}{4cm}\begin{tikzpicture}[scale=0.45]
  
       \foreach \x in {0,2,4,6,8,9,10,12,14,16,18,20}
     {   \path(\x,3) coordinate (up\x);  
      \path(\x,0) coordinate (down\x);  
   }

  {   \path(down10) --++ (135:0.4) coordinate(down1ten);  
      \path(down10) --++ (45:0.4) coordinate(down2ten);
            \path(down10) --++ (-45:0.4) coordinate(down3ten);  
                        \path(down10) --++ (-135:0.4) coordinate(down4ten);  
   }

     {   \path(up10) --++ (135:0.4) coordinate(up1ten);  
      \path(up10) --++ (45:0.4) coordinate(up2ten);
            \path(up10) --++ (-45:0.4) coordinate(up3ten);  
                        \path(up10) --++ (-135:0.4) coordinate(up4ten);  
   }

  {   \path(down12) --++ (135:0.4) coordinate(down1twelve);  
      \path(down12) --++ (45:0.4) coordinate(down2twelve);
            \path(down12) --++ (-45:0.4) coordinate(down3twelve);  
                        \path(down12) --++ (-135:0.4) coordinate(down4twelve);  
   }

     {   \path(up12) --++ (135:0.4) coordinate(up1twelve);  
      \path(up12) --++ (45:0.4) coordinate(up2twelve);
            \path(up12) --++ (-45:0.4) coordinate(up3twelve);  
                        \path(up12) --++ (-135:0.4) coordinate(up4twelve);  
   }

  {   \path(down14) --++ (135:0.4) coordinate(down1fourteen);  
      \path(down14) --++ (45:0.4) coordinate(down2fourteen);
            \path(down14) --++ (-45:0.4) coordinate(down3fourteen);  
                        \path(down14) --++ (-135:0.4) coordinate(down4fourteen);  
   }

     {   \path(up14) --++ (135:0.4) coordinate(up1fourteen);  
      \path(up14) --++ (45:0.4) coordinate(up2fourteen);
            \path(up14) --++ (-45:0.4) coordinate(up3fourteen);  
                        \path(up14) --++ (-135:0.4) coordinate(up4fourteen);  
   }

  {   \path(down16) --++ (135:0.4) coordinate(down1sixteen);  
      \path(down16) --++ (45:0.4) coordinate(down2sixteen);
            \path(down16) --++ (-45:0.4) coordinate(down3sixteen);  
                        \path(down16) --++ (-135:0.4) coordinate(down4sixteen);  
   }

     {   \path(up16) --++ (135:0.4) coordinate(up1sixteen);  
      \path(up16) --++ (45:0.4) coordinate(up2sixteen);
            \path(up16) --++ (-45:0.4) coordinate(up3sixteen);  
                        \path(up16) --++ (-135:0.4) coordinate(up4sixteen);  
   }

  {   \path(down18) --++ (135:0.4) coordinate(down1eighteen);  
      \path(down18) --++ (45:0.4) coordinate(down2eighteen);
            \path(down18) --++ (-45:0.4) coordinate(down3eighteen);  
                        \path(down18) --++ (-135:0.4) coordinate(down4eighteen);  
   }

     {   \path(up18) --++ (135:0.4) coordinate(up1eighteen);  
      \path(up18) --++ (45:0.4) coordinate(up2eighteen);
            \path(up18) --++ (-45:0.4) coordinate(up3eighteen);  
                        \path(up18) --++ (-135:0.4) coordinate(up4eighteen);  
   }

  {   \path(down16) --++ (135:0.4) coordinate(down1twenty);  
      \path(down16) --++ (45:0.4) coordinate(down2twenty);
            \path(down16) --++ (-45:0.4) coordinate(down3twenty);  
                        \path(down16) --++ (-135:0.4) coordinate(down4twenty);  
   }

     {   \path(up16) --++ (135:0.4) coordinate(up1twenty);  
      \path(up16) --++ (45:0.4) coordinate(up2twenty);
            \path(up16) --++ (-45:0.4) coordinate(up3twenty);  
                        \path(up16) --++ (-135:0.4) coordinate(up4twenty);  
   }

     \foreach \x in {0,2,4,6,8,9,10,12,14,16,18,20}
     {   \path(up\x) --++ (135:0.4) coordinate(up1\x);  
      \path(up\x) --++ (45:0.4) coordinate(up2\x);
            \path(up\x) --++ (-45:0.4) coordinate(up3\x);  
                        \path(up\x) --++ (-135:0.4) coordinate(up4\x);  
   }

    \foreach \x in {0,2,4,6,8,9,10,12,14,16,18,20}
     {   \path(down\x) --++ (135:0.4) coordinate(down1\x);  
      \path(down\x) --++ (45:0.4) coordinate(down2\x);
            \path(down\x) --++ (-45:0.4) coordinate(down3\x);  
                        \path(down\x) --++ (-135:0.4) coordinate(down4\x);  
   }

   \draw [fill=white] plot [smooth cycle]
  coordinates {(up10) (up22) (down32)  (down40)  };

 \draw(2,3)--(0,0);                       
 \draw(0,3)--(2,0);

%   
%   \draw [fill=white] plot [smooth cycle]
%  coordinates {(adown10) (adown22) (adown32)  (adown40)  };

       \foreach \x in {0,2,4,6,8, }
     {   \path(\x+10,3) coordinate (aup\x);  
      \path(\x+10,0) coordinate (adown\x);  
   }

     \foreach \x in {0,2,4,6,8, }
     {   \path(aup\x) --++ (135:0.4) coordinate(aup1\x);  
      \path(aup\x) --++ (45:0.4) coordinate(aup2\x);
            \path(aup\x) --++ (-45:0.4) coordinate(aup3\x);  
                        \path(aup\x) --++ (-135:0.4) coordinate(aup4\x);  
   }

    \foreach \x in {0,2,4,6,8, }
     {   \path(adown\x) --++ (135:0.4) coordinate(adown1\x);  
      \path(adown\x) --++ (45:0.4) coordinate(adown2\x);
            \path(adown\x) --++ (-45:0.4) coordinate(adown3\x);  
                        \path(adown\x) --++ (-135:0.4) coordinate(adown4\x);  
   }

        \draw [fill=white] plot [smooth cycle]
  coordinates {(up14) (up26) (adown22)  (adown32)    (down44)  };

 \draw(6,3)--(8,0);

    \draw(10,0)--(4,3);

     \foreach \x in {0,2,4,6,6}
     {   \fill[white](\x,3) circle (4pt);   
                \fill[white](\x,0) circle (4pt);    
                        \draw (\x,3) node {$\bullet$}; 
                           \draw (\x,0) node {${\bullet}$};  }

     \foreach \x in {0,2,4,6,8,10,12 }
     {   
                           \draw (\x,0) node {${\bullet}$};  }
      
         \end{tikzpicture}   \end{minipage}
          $$

     \caption{A diagram in $V_7^0(2^2)$.  This diagram cannot be obtained from any 
          of \cref{plenty1,plenty2,plenty3,aplenty1} via the $\W_2\wr \W_2$-left-action or $\W_7$-right-action.}
          \label{aplenty2}

\end{figure}

\begin{eg}
The diagrams in \cref{plenty1,plenty2,plenty3} each have $a=b=2$ and $r=7$.  
Each has two outer propagating blocks, $P_1$ and $P_2$, one of which  contains a singleton and the other contains no singletons.  
Each has a unique non-propagating outer block, $Q_1$, containing 2 singletons. 
\end{eg}

 \begin{eg}
Consider the ramified diagram in \cref{a-concatenate-for-the-hard-thin2}. Here  $a=2$, $b=3$ and $r= 14$. We call the propagating outer-blocks $P_1, P_2, P_3$ reading from left to right, %(we note that this is not an ambiguous statement in this case).  
and observe  that $P_1$ and~$P_2$ each contain  2 inner  singleton southern blocks and $P_3$  contains no inner singleton southern blocks 
(in addition to the $a=2$ inner propagating pairs).   
There are also two  non-propagating southern outer-blocks,  $Q_1, Q_2$, each of which contains   precisely two  singleton southern blocks.
 \end{eg}
 
 Before continuing further, we now introduce some notation for recording the 
  inner-singleton blocks. 
   We  let $\Cee$ denote the total number of southern singletons
 belonging to an outer {\em propagating} block
and  $ \Dee $ denote the total number of southern singletons
 belonging to an outer {\em non-propagating} block. 
We note that  every southern vertex belongs to either a propagating pair  inner-block (of which there are $ab$ in total) or  a  singleton  inner-block and therefore 
  $\Cee+\Dee ={r-ab}$.

\begin{eg}
The ramified diagrams in \cref{plenty1,plenty2,plenty3} each have  $\Cee=1$ and $\Dee=2$. \end{eg}

%\begin{eg}
%The diagram in \cref{a-concatenate-for-the-hard-thin2}, has $\Cee=2+2+0=4$ and $\Dee=2+2=4$.  
%\end{eg}

 We next consider how these singleton vertices are partitioned into propagating and non-propagating blocks. We remind the reader that the non-propagating outer-blocks each contain  {\em at least 2} inner-singleton vertices. 
Recall that   $ \ParSet_{>1}(\Dee )$
denotes the set of all integer partitions of $\Dee$ whose parts are all strictly greater than 1.
%, that is $$ \ParSet_{>1}(\Dee )=\{	(q_1,\dots, q_\ell) \mid q_1\geq q_2\geq \dots \geq q_\ell >1 \text { and }	q_1+\dots+q_\ell=\Dee\}  .$$  

\begin{defn} \label{defn:t:1}
We suppose that  $d_{(\Lambda,\Lambda')}$ has non-propagating outer blocks $Q_1,\dots , Q_\ell$ such that, for $1\leq i \leq \ell$,   $\Lambda\cap Q_i $  consists of some number $\deltaP_i \geq 2$ of 
 singleton southern blocks.
We define the {\sf non-propagating type} of  $d_{(\Lambda,\Lambda')}$ to be the partition
$\deltaP = (\deltaP_{1} ,\deltaP_{2} ,\deltaP_3,\dots,\deltaP_\ell )  \in  \ParSet_{>1}(\Dee )$ if $\ell\neq 0$ and to be $\varnothing$  if $\ell=0$.
\end{defn}

\begin{figure}[ht!]
$$  \begin{minipage}{3cm}\begin{tikzpicture} [xscale=0.45,yscale=-0.45]
  
       \foreach \x in {0,1,2,3,4,5}
     {  
      \path(\x*2,0) coordinate (down\x);  
   }

       \foreach \x in {0,1,2,3,4,5,6}
    {   \path(\x*2,3) coordinate (up\x);  
    }

     \foreach \x in {0,1,2,3,4,5,6}
     {   \path(up\x) --++ (135:0.6) coordinate(up1\x);  
      \path(up\x) --++ (45:0.6) coordinate(up2\x);
            \path(up\x) --++ (-45:0.6) coordinate(up3\x);  
                        \path(up\x) --++ (-135:0.6) coordinate(up4\x);  
   }

    \foreach \x in {0,1,2,3,4,5}
     {   \path(down\x) --++ (135:0.6) coordinate(down1\x);  
      \path(down\x) --++ (45:0.6) coordinate(down2\x);
            \path(down\x) --++ (-45:0.6) coordinate(down3\x);  
                        \path(down\x) --++ (-135:0.6) coordinate(down4\x);  
   }

   \draw [fill=white] plot [smooth cycle]
  coordinates {(up10) (up21) (up31) (down30)   (down40)  };

%   \draw [fill=white] plot [smooth cycle]
%  coordinates {(up12)   (up32) (down31)   (down11)  };
%

       \foreach \x in {0,2 }
     {   \fill[white](\x,3) circle (4pt);   
                        \draw (\x,3) node {$\bullet$};             
                                    \draw (0,0) node {$\bullet$}; 
%                                                \draw (4,0) node {$\bullet$}; 
                  }

       \foreach \x in {0,1,2,3,4,5}
     {  
      \path(\x*2+4,0) coordinate (down\x);  
   }

       \foreach \x in {0,1,2,3,4,5,6}
    {   \path(\x*2+4,3) coordinate (up\x);  
    }

     \foreach \x in {0,1,2,3,4,5,6}
     {   \path(up\x) --++ (135:0.6) coordinate(up1\x);  
      \path(up\x) --++ (45:0.6) coordinate(up2\x);
            \path(up\x) --++ (-45:0.6) coordinate(up3\x);  
                        \path(up\x) --++ (-135:0.6) coordinate(up4\x);  
   }

    \foreach \x in {0,1,2,3,4,5}
     {   \path(down\x) --++ (135:0.6) coordinate(down1\x);  
      \path(down\x) --++ (45:0.6) coordinate(down2\x);
            \path(down\x) --++ (-45:0.6) coordinate(down3\x);  
                        \path(down\x) --++ (-135:0.6) coordinate(down4\x);  
   }

   \draw [fill=white] plot [smooth cycle]
  coordinates {(up10) (up21) (up31) (down30)   (down40)  };

%   \draw [fill=white] plot [smooth cycle]
%  coordinates {(up12)   (up32) (down31)   (down11)  };
%

       \foreach \x in {0,2 }
     {   \fill[white](\x+4,3) circle (4pt);   
                        \draw (\x+4,3) node {$\bullet$};             
                                    \draw (0+4,0 ) node {$\bullet$}; 
%                                                \draw (4,0) node {$\bullet$}; 
                  }

       \foreach \x in {0,1,2,3,4,5}
     {  
      \path(\x*2+4+4,0) coordinate (down\x);  
   }

       \foreach \x in {0,1,2,3,4,5,6}
    {   \path(\x*2+4+4,3) coordinate (up\x);  
    }

     \foreach \x in {0,1,2,3,4,5,6}
     {   \path(up\x) --++ (135:0.6) coordinate(up1\x);  
      \path(up\x) --++ (45:0.6) coordinate(up2\x);
            \path(up\x) --++ (-45:0.6) coordinate(up3\x);  
                        \path(up\x) --++ (-135:0.6) coordinate(up4\x);  
   }

    \foreach \x in {0,1,2,3,4,5}
     {   \path(down\x) --++ (135:0.6) coordinate(down1\x);  
      \path(down\x) --++ (45:0.6) coordinate(down2\x);
            \path(down\x) --++ (-45:0.6) coordinate(down3\x);  
                        \path(down\x) --++ (-135:0.6) coordinate(down4\x);  
   }

   \draw [fill=white] plot [smooth cycle]
  coordinates {(up10) (up20)  (down30)   (down40)  };

%   \draw [fill=white] plot [smooth cycle]
%  coordinates {(up12)   (up32) (down31)   (down11)  };
%

       \foreach \x in {0  }
     {   \fill[white](\x+4+4,3) circle (4pt);   
                        \draw (\x+4+4,3) node {$\bullet$};             
                                    \draw (0+4+4,0 ) node {$\bullet$}; 
                   }

         \end{tikzpicture}  \end{minipage}  $$
         
         \caption{A diagram  in $V_5^0(0^3)$    with propagating type $(2,2,1)$ and non-propagating type $\varnothing$.}
         \label{Phereiszero}
         \end{figure}
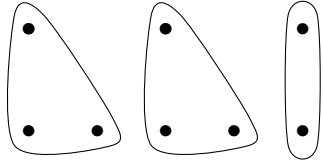
   
    We remind the reader that the  propagating outer-blocks   must each contain at least one southern vertex.  When $a=0$ this implies that every  propagating outer-block must contain at least one southern
    inner-singleton (see for example, \cref{Phereiszero}).  When $a>0$ a propagating outer-block is allowed to contain  zero southern inner-singletons (as already seen in \cref{plenty1,plenty2,plenty3}). 
    %This was written inconsistently with southern singleton then inner-singleton. I've added both
    %adjectives in each case
    In what follows, we use $\gamma$ to denote the numbers of southern inner-singletons in the propagating
    outer-blocks, ordering these numbers so that they are weakly decreasing.
    Thus $\gamma \in \ParSet _{(a^b)}(\Cee)$ where $\ParSet _{(a^b)}$ is specified as follows.
    
\begin{defn}\label{defn:ParSet}
Let $b \in \NN_0$.
\begin{itemize}[leftmargin=18pt]
\item[(i)]
We define
$\ParSet _{(0^b)}(\Cee)$ to be the set of all weakly %increasing
decreasing
 sequences of positive integers~$\gamma$ of length $b$ which 
 sum to $\Cee$.  
\item[(ii)] For $a \in \NN$, we define $\ParSet _{(a^b)}(\Cee)$ to be the set of all
weakly decreasing sequences of non-negative integers $\gamma$ of length $b$
which sum to $\Cee - ab$.
\end{itemize}
\end{defn}

\begin{defn}\label{defn:t:2}
We suppose that  $d_{(\Lambda,\Lambda')}$ has  propagating outer-blocks $P_1,\dots, P_b$ such that, for $1\leq i \leq b$,    $\Lambda\cap P_i$ consists of precisely $a$ propagating 
 pairs  and some number, $\gamma_i$ say, of southern singletons.
 We define the {\sf propagating type} of   $d_{(\Lambda,\Lambda')}$ to be the sequence 
 %$\ParSet _b(\Cee)$ denote the set of all weakly increasing sequences of non-negative integers 
   $\gamma=  
   ( \Cee^{c_\Cee}, \dots ,  1^{c_1}, 0^{c_0} )
\in  \ParSet _{(a^b)}(\Cee) $. 
%    We let  
%    such that $(a^b)+\gamma$ has  exactly $b$ non-zero parts.  
 \end{defn}
 
Note that by the remarks before Definition~\ref{defn:ParSet}, $\gamma \in \ParSet _{(a^b)}(p)$.
% as claimed in the definition above.
 In particular, when $a=0$, since  each propagating outer-block must contain
at least one southern inner-singleton, the sequence~$\gamma$ has no zero terms, and the zero multiplicity
$c_0$ is $0$. In case (ii) when $a > 0$, elements of $\ParSet_{(a^b)}(\gamma)$ may have zero parts:
these are counted by $c_0$ in the sums in \TheoremD.

\begin{defn}\label{defn:type}
We say that a ramified diagram $d_{(\Lambda,\Lambda')}$ as 
in \cref{defn:t:1,defn:t:2}  has {\sf type} $(\gamma,\varepsilon)$.  
\end{defn}

We shall see that the  $(\C\W_a\wr \W_b,\C\W_r)$-bimodule $V^0_r(a^b)$ decomposes as a direct sum 
in which each summand is spanned by the diagrams of a fixed type.
%consists of -> is spanned by
  
   \begin{eg}
   The diagrams in \cref{plenty1,plenty2,plenty3} all have propagating-type  $(1,0)$ and 
   non-propagating-type  $(2)$. 
   \end{eg}

   \begin{eg}
%   The diagram  in \cref{aplenty1} has  propagating-type  $(2,1)$ and 
%   non-propagating-type  $\varnothing$. 
%%   \end{eg} 
%%   \begin{eg}
%   The diagram  in \cref{aplenty2} has  propagating-type  $(0,3)$ and 
%   non-propagating-type  $\varnothing$. 
The diagrams  in \cref{aplenty1,aplenty2} have  propagating-type  $(2,1)$ and 
   $(3,0)$ respectively.  Both  diagrams  in \cref{aplenty1,aplenty2} have  
   non-propagating-type  $\varnothing$. 

   \end{eg}

\subsection{Elementary diagrams }\label{subsec:elementaryDiagrams}

%\subsection{A decomposition of the depth quotient}
Let $x, y \in \ZZ_{\geq0}$.
 For each pair $(0,0)\neq (x, y) \in  \ZZ_{\geq0}^2$ we define a ramified $( \max\{1,x\},x+y)$-set-partition diagram 
$ 
v_{x,y}
$     
 by setting the first $x$ inner blocks to be of the form $\{k,\overline{k}\}$ for $1\leq k \leq x$ and the remaining 
inner  blocks to be singletons; %  $\{x+y\}$ for $0\leq j \leq y$;
  there is a single outer block that is the union of all $\max\{1,x\}+x+y$ vertices.   
Examples are  depicted in \cref{diagexy}.  
 We  set $(\varnothing, y)$ to be $( 0,y)$-set-partition consisting of precisely one outer block, and whose inner blocks are all singletons (see \cref{thank-you-next} for an example).
   
 \begin{figure}[ht!]
 $$
  \begin{minipage}{3cm}\begin{tikzpicture}[xscale=0.45,yscale=-0.45]
  
       \foreach \x in {0,1,2,3,4,5}
     {  
      \path(\x*2,0) coordinate (down\x);  
   }

       \foreach \x in {0,1,2,3,4,5,6}
    {   \path(\x*2,3) coordinate (up\x);  
    }

     \foreach \x in {0,1,2,3,4,5,6}
     {   \path(up\x) --++ (135:0.6) coordinate(up1\x);  
      \path(up\x) --++ (45:0.6) coordinate(up2\x);
            \path(up\x) --++ (-45:0.6) coordinate(up3\x);  
                        \path(up\x) --++ (-135:0.6) coordinate(up4\x);  
   }

    \foreach \x in {0,1,2,3,4,5}
     {   \path(down\x) --++ (135:0.6) coordinate(down1\x);  
      \path(down\x) --++ (45:0.6) coordinate(down2\x);
            \path(down\x) --++ (-45:0.6) coordinate(down3\x);  
                        \path(down\x) --++ (-135:0.6) coordinate(down4\x);  
   }

   \draw [fill=white] plot [smooth cycle]
  coordinates {(up10) (up22) (down31)   (down40)  };

       \foreach \x in {0,2,4}
     {   \fill[white](\x,3) circle (4pt);   
                        \draw (\x,3) node {$\bullet$};                         \draw (0,0) node {$\bullet$}; 
                  }

         \end{tikzpicture}  \end{minipage} \qquad 
          \begin{minipage}{4cm}\begin{tikzpicture}[xscale=0.45,yscale=-0.45]
  
       \foreach \x in {0,1,2,3,4,5}
     {  
      \path(\x*2,0) coordinate (down\x);  
   }

       \foreach \x in {0,1,2,3,4,5,6}
    {   \path(\x*2,3) coordinate (up\x);  
    }

     \foreach \x in {0,1,2,3,4,5,6}
     {   \path(up\x) --++ (135:0.6) coordinate(up1\x);  
      \path(up\x) --++ (45:0.6) coordinate(up2\x);
            \path(up\x) --++ (-45:0.6) coordinate(up3\x);  
                        \path(up\x) --++ (-135:0.6) coordinate(up4\x);  
   }

    \foreach \x in {0,1,2,3,4,5}
     {   \path(down\x) --++ (135:0.6) coordinate(down1\x);  
      \path(down\x) --++ (45:0.6) coordinate(down2\x);
            \path(down\x) --++ (-45:0.6) coordinate(down3\x);  
                        \path(down\x) --++ (-135:0.6) coordinate(down4\x);  
   }

   \draw [fill=white] plot [smooth cycle]
  coordinates {(up10) (up23) (down32)   (down40)  };

 \draw(0,3)--(0,0);                       
    \draw(2,0)--(2,3);       
%    \draw(6,0)--++(90:3);
%    \draw(4,0)--++(90:3);    
    
     \foreach \x in {0,2,4,6}
     {   \fill[white](\x,3) circle (4pt);   
                        \draw (\x,3) node {$\bullet$}; 
                  }
                
      \foreach \x in {0,2}
     {   \fill[white](\x,0) circle (4pt);   
                        \draw (\x,0) node {$\bullet$}; 
                  }

         \end{tikzpicture}  \end{minipage} \qquad 
  \begin{minipage}{6cm}\begin{tikzpicture}[xscale=0.45,yscale=-0.45]
  
       \foreach \x in {0,1,2,3,4,5}
     {  
      \path(\x*2,0) coordinate (down\x);  
   }

       \foreach \x in {0,1,2,3,4,5,6}
    {   \path(\x*2,3) coordinate (up\x);  
    }

     \foreach \x in {0,1,2,3,4,5,6}
     {   \path(up\x) --++ (135:0.6) coordinate(up1\x);  
      \path(up\x) --++ (45:0.6) coordinate(up2\x);
            \path(up\x) --++ (-45:0.6) coordinate(up3\x);  
                        \path(up\x) --++ (-135:0.6) coordinate(up4\x);  
   }

    \foreach \x in {0,1,2,3,4,5}
     {   \path(down\x) --++ (135:0.6) coordinate(down1\x);  
      \path(down\x) --++ (45:0.6) coordinate(down2\x);
            \path(down\x) --++ (-45:0.6) coordinate(down3\x);  
                        \path(down\x) --++ (-135:0.6) coordinate(down4\x);  
   }

   \draw [fill=white] plot [smooth cycle]
  coordinates {(up10) (up25) (down33)   (down40)  };

 \draw(0,3)--(0,0);                       
    \draw(2,0)--(2,3);       
    \draw(6,0)--++(90:3);
    \draw(4,0)--++(90:3);    
    
     \foreach \x in {0,2,4,6,8,10}
     {   \fill[white](\x,3) circle (4pt);   
                        \draw (\x,3) node {$\bullet$}; 
                  }
                
      \foreach \x in {0,2,4,6}
     {   \fill[white](\x,0) circle (4pt);   
                        \draw (\x,0) node {$\bullet$}; 
                  }

         \end{tikzpicture}  \end{minipage} 
$$
\caption{Diagrams $v_{x,y}=v_{0,3}$, $v_{2,2}$, and  $v_{4,2}$ respectively. 
Note that $x$ is the number of propagating strands and $y$ is the number of   southern inner-singletons.}
\label{diagexy}
\end{figure}
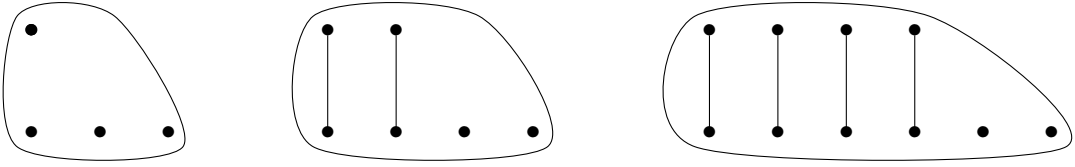

         \begin{figure}[ht!]
         \vspace*{-18pt}
 $$
  \begin{minipage}{4cm}\begin{tikzpicture}[xscale=0.45,yscale=-0.45]
  \clip(-2,2) rectangle (8,4);
       \foreach \x in {0,1,2,3,4,5}
     {  
      \path(\x*2,0) coordinate (down\x);  
   }

       \foreach \x in {0,1,2,3,4,5,6}
    {   \path(\x*2,3) coordinate (up\x);  
    }

     \foreach \x in {0,1,2,3,4,5,6}
     {   \path(up\x) --++ (135:0.6) coordinate(up1\x);  
      \path(up\x) --++ (45:0.6) coordinate(up2\x);
            \path(up\x) --++ (-45:0.6) coordinate(up3\x);  
                        \path(up\x) --++ (-135:0.6) coordinate(up4\x);  
   }

    \foreach \x in {0,1,2,3,4,5}
     {   \path(down\x) --++ (135:0.6) coordinate(down1\x);  
      \path(down\x) --++ (45:0.6) coordinate(down2\x);
            \path(down\x) --++ (-45:0.6) coordinate(down3\x);  
                        \path(down\x) --++ (-135:0.6) coordinate(down4\x);  
   }

   \draw [fill=white] plot [smooth cycle]
  coordinates {(up10) (up23) (up33)  (up40)  };

       \foreach \x in {0,2,4,6}
     {   \fill[white](\x,3) circle (4pt);   
                        \draw (\x,3) node {$\bullet$}; 
                  }

         \end{tikzpicture}  \end{minipage} $$
         \vspace*{-6pt}
         \caption{The diagram $v_{(\varnothing,4)}$.  Here $\varnothing$ records that this a   diagram consisting purely of southern vertices.}
         \label{thank-you-next}
         \end{figure}

\begin{defn}\label{defn:vGammaEpsilonDiagram} For $\gamma =  
   ( \Cee^{c_\Cee},  \dots ,1^{c_1}, 0^{c_0})
\in  \ParSet _{(a^b)}(\Cee)$ and $\deltaP \in  \ParSet_{>1}(\Dee ) $, the diagram 
 $
v_{\gamma,\deltaP}
 $, is defined by horizontal concatenation   as follows:
  $$
v_{  \gamma,\deltaP }=
\left( 
  (v_{a,\Cee})^{\circledast c_\Cee} \circledast \dots \circledast (v_{a,1})^{\circledast c_1} \circledast (v_{a,0})^{\circledast c_0} 
\right)
\circledast 
\left( v_{\varnothing,\deltaP_{1} }   \circledast  v_{\varnothing,\deltaP_{2} }   \circledast \dots \circledast  v_{\varnothing,\deltaP_\ell}  \right).
$$  
Let  %$V^0_r(\gamma,\deltaP) \subseteq V^0_r{(a^b)}$
$V^0_r(a^b: \gamma,\deltaP) \subseteq V^0_r{(a^b)}$ denote the  $(\C\W_a\wr \W_b,\C\W_r)$-bimodule generated by  
$v_{\gamma,\deltaP}\in V^0_r(a^b)$.  
\end{defn}

\begin{figure}[ht!]
$$  \begin{tikzpicture}[xscale=0.45,yscale=-0.45]
  
       \foreach \x in {0,1,2,3,4,5,6,7,8,9,10,11,12,13,14}
     {  
      \path(\x*2,0) coordinate (down\x);  
     \path(\x*2,3) coordinate (up\x);  
    }

     \foreach \x in {0,1,2,3,4,5,6,7,8,9,10,11,12,13,14}
     {   \path(up\x) --++ (135:0.6) coordinate(up1\x);  
      \path(up\x) --++ (45:0.6) coordinate(up2\x);
            \path(up\x) --++ (-45:0.6) coordinate(up3\x);  
                        \path(up\x) --++ (-135:0.6) coordinate(up4\x);  
   }

    \foreach \x in {0,1,2,3,4,5,6,7,8,9,10}
     {   \path(down\x) --++ (135:0.6) coordinate(down1\x);  
      \path(down\x) --++ (45:0.6) coordinate(down2\x);
            \path(down\x) --++ (-45:0.6) coordinate(down3\x);  
                        \path(down\x) --++ (-135:0.6) coordinate(down4\x);  
   }

   \draw [fill=white] plot [smooth cycle]
  coordinates {(up10) (up23) (down32)   (down40)  };

   \draw [fill=white] plot [smooth cycle]
  coordinates {(up14) (up27) (down36)   (down44)  };

   \draw [fill=white] plot [smooth cycle]
  coordinates {(up18) (up29) (down39)   (down48)  };

     \path(10*2,3) coordinate (up10);  
          \path(12*2,3) coordinate (up12);  

      {   \path(up10) --++ (135:0.6) coordinate(Xup1);  
      \path(up10) --++ (45:0.6) coordinate(Xup2);
            \path(up10) --++ (-45:0.6) coordinate(Xup3);  
                        \path(up10) --++ (-135:0.6) coordinate(Xup4);  
   }

   {   \path(up12) --++ (135:0.6) coordinate(XupA1);  
      \path(up12) --++ (45:0.6) coordinate(XupA2);
            \path(up12) --++ (-45:0.6) coordinate(XupA3);  
                        \path(up12) --++ (-135:0.6) coordinate(XupA4);  
   }

  \draw [fill=white] plot [smooth cycle]
  coordinates {(Xup1) (XupA2) (XupA3)   (Xup4)  };

    \path(13*2,3) coordinate (up10);  
          \path(14*2,3) coordinate (up12);  

      {   \path(up10) --++ (135:0.6) coordinate(Xup1);  
      \path(up10) --++ (45:0.6) coordinate(Xup2);
            \path(up10) --++ (-45:0.6) coordinate(Xup3);  
                        \path(up10) --++ (-135:0.6) coordinate(Xup4);  
   }

   {   \path(up12) --++ (135:0.6) coordinate(XupA1);  
      \path(up12) --++ (45:0.6) coordinate(XupA2);
            \path(up12) --++ (-45:0.6) coordinate(XupA3);  
                        \path(up12) --++ (-135:0.6) coordinate(XupA4);  
   }

  \draw [fill=white] plot [smooth cycle]
  coordinates {(Xup1) (XupA2) (XupA3)   (Xup4)  };

 \draw(8,0)--++(90:3);
    \draw(10,0)--++(90:3);

 \draw(16,0)--++(90:3);
    \draw(18,0)--++(90:3);
 
 \draw(0,3)--(0,0);                       
    \draw(2,0)--(2,3);

      \foreach \x in {0,2,4,6,8,10,12,14,16,18,20,22,24,26,28}
     {   \fill[white](\x,3) circle (4pt);   
                        \draw (\x,3) node {$\bullet$}; 
                  }
                
      \foreach \x in {0,2,8,10,16,18}
     {   \fill[white](\x,0) circle (4pt);   
                        \draw (\x,0) node {$\bullet$}; 
                  }

         \end{tikzpicture}          $$
         \caption{
         The diagram $v_{\gamma,\deltaP}$ for $\gamma=(2^2,0)$ and $\deltaP=(3,2)$ with 
$a=2$ and $b=3$,  obtained via horizontal concatenation  as 
$(v_{(2,2)})^{\circledast 2} \circledast v_{(2,0)} \circledast v_{(\varnothing,3)} \circledast v_{(\varnothing,2)} $.   This diagram is   discussed in           \cref{a-concatenate-for-the-hard-thing2}.
} 
\label{a-concatenate-for-the-hard-thing}         \end{figure}

Since the propagating type and non-propagating type of a ramified diagram are invariant under both the left-$\W_a \wr \W_b$ and the right-$\W_r$ actions, the following proposition follows.

\begin{prop}\label{Stepp1}
There is a  direct sum decomposition of the  $(\W_a \wr \W_b, \W_r)$-module $V^0_r{(a^b)}$ as follows:
$$V^0_r{(a^b)}= \bigoplus _{
\begin{subarray}c
\Cee+\Dee=r-ab
\\
\gamma \in \ParSet _{(a^b)}(\Cee)
\\
\deltaP \in \ParSet_{>1}(\Dee)
\end{subarray}
}V^0_r\bigl((a^b): \gamma,\deltaP\bigr) .$$
\end{prop}

\begin{eg}\label{a-concatenate-for-the-hard-thing2}
Suppose that $r=15$, $a=2$, $b=3$ and $r-ab=9$.  
We let $\Cee=4$ and $\Dee=5$. 
There are 4 possible choices of $\gamma \in \ParSet_{(2^3)}(4)$, 
namely  $(4,0^2)$, $(3,1,0)$, $(2^2,0)$ and $(2,1^2)$. 
 There are two choices of $\deltaP\in \ParSet_{>1}(5)$, namely $(5)$ and $(3,2)$.  
The diagram $v_{\gamma,\deltaP}$ for $\gamma=(2^2,0)$ and $\deltaP=(3,2)$ is depicted in \cref{a-concatenate-for-the-hard-thing}.
\end{eg}

 \subsection{The direct sum  decomposition of the depth quotient}
We are now ready to provide the complete decomposition of 
%$\Delta^0_r(\alpha^\beta)$
${\sf DQ}(\Delta_{r}(\alpha^\beta))=\rightspecht{  \alpha ^{  \beta} } \otimes_{\W_a \wr \W_b} V^0_r{(a^b)}$
 into irreducible summands. Examples~\ref{eg:alphaEmpty} and~\ref{eg:alphaNonEmpty}
and the discussion before Definition~\ref{defn:IndInfRes} illustrate the key ideas in the proof.
 In light of \cref{Stepp1}, we  focus our attention on a fixed summand $V^0_r(a^b: \gamma,\deltaP) \subseteq V^0_r(a^b)$.  We shall consider two extremal cases first.

\begin{lem}\label{DQdecomp1}
For any $\deltaP \in \ParSet_{>1}(\Dee )$, there is an isomorphism of right %$\Sym_r$-modules 
$\Sym_\Dee$-modules 
\[ 
% V^0_\Dee\bigl((a^b): \varnothing ,\deltaP\bigr)
V^0_\Dee\bigl(\varnothing : \varnothing ,\deltaP\bigr)
 \cong  \C  \Ind_{\Stab(\deltaP)}^{\W_\Dee}.
\]
\end{lem}

\begin{proof}
The diagrammatic module  is a transitive permutation  $\W_q $-module with the required stabiliser,  and so the result follows. 
\end{proof}

We remark that this lemma is also proved in \cite[Theorem~7.11]{MR4756467}.
 
     \begin{figure}[ht!]
     \vspace*{-24pt}
 $$
  \begin{minipage}{4cm}\begin{tikzpicture}[xscale=0.45,yscale=-0.45]
  
       \foreach \x in {0,1,2,3,4,5}
     {  
      \path(\x*2,0) coordinate (down\x);  
   }

       \foreach \x in {0,1,2,3,4,5,6}
    {   \path(\x*2,3) coordinate (up\x);  
    }

     \foreach \x in {0,1,2,3,4,5,6}
     {   \path(up\x) --++ (135:0.6) coordinate(up1\x);  
      \path(up\x) --++ (45:0.6) coordinate(up2\x);
            \path(up\x) --++ (-45:0.6) coordinate(up3\x);  
                        \path(up\x) --++ (-135:0.6) coordinate(up4\x);  
   }

    \foreach \x in {0,1,2,3,4,5}
     {   \path(down\x) --++ (135:0.6) coordinate(down1\x);  
      \path(down\x) --++ (45:0.6) coordinate(down2\x);
            \path(down\x) --++ (-45:0.6) coordinate(down3\x);  
                        \path(down\x) --++ (-135:0.6) coordinate(down4\x);  
   }

   \draw [fill=white] plot [smooth cycle]
  coordinates {(up10) (up21) (up31)   (up40)  };
 
   \draw [fill=white] plot [smooth cycle]
  coordinates {(up12) (up23) (up33)   (up42)  };

       \foreach \x in {0,2,4,6}
     {   \fill[white](\x,3) circle (4pt);   
                        \draw (\x,3) node {$\bullet$}; 
                  }

         \end{tikzpicture}  \end{minipage} \qquad 
           \begin{minipage}{4cm}\begin{tikzpicture}[xscale=0.45,yscale=-0.45]
  
       \foreach \x in {0,1,2,3,4,5}
     {  
      \path(\x*2,0) coordinate (down\x);  
   }

       \foreach \x in {0,1,2,3,4,5,6}
    {   \path(\x*2,3) coordinate (up\x);  
    }

     \foreach \x in {0,1,2,3,4,5,6}
     {   \path(up\x) --++ (135:0.6) coordinate(up1\x);  
      \path(up\x) --++ (45:0.6) coordinate(up2\x);
            \path(up\x) --++ (-45:0.6) coordinate(up3\x);  
                        \path(up\x) --++ (-135:0.6) coordinate(up4\x);  
   }

    \foreach \x in {0,1,2,3,4,5}
     {   \path(down\x) --++ (135:0.6) coordinate(down1\x);  
      \path(down\x) --++ (45:0.6) coordinate(down2\x);
            \path(down\x) --++ (-45:0.6) coordinate(down3\x);  
                        \path(down\x) --++ (-135:0.6) coordinate(down4\x);  
   }

  \path(2,3+1.5) coordinate   (midway);
    \path(2,3+.75) coordinate   (lowermidway);
    
   \draw [fill=white] plot [smooth cycle]
  coordinates {(up10) 
  (midway)
  (up22) (up32) (up42)
  (lowermidway)
     (up30)     (up40)  };

  \path(4,1.5) coordinate   (midway);
    \path(4,2.25) coordinate   (lowermidway);

   \draw [fill=white] plot [smooth cycle]
  coordinates {(up41) 
  (midway)
  (up33) (up23) (up13)
  (lowermidway)
     (up21)     (up11)  };

     \foreach \x in {0,2,4,6}
     {   \fill[white](\x,3) circle (4pt);   
                        \draw (\x,3) node {$\bullet$}; 
                  }

         \end{tikzpicture}  \end{minipage} 
\qquad 
           \begin{minipage}{4cm}\begin{tikzpicture}[xscale=0.45,yscale=-0.45]
  
       \foreach \x in {0,1,2,3,4,5}
     {  
      \path(\x*2,0) coordinate (down\x);  
   }

       \foreach \x in {0,1,2,3,4,5,6}
    {   \path(\x*2,3) coordinate (up\x);  
    }

     \foreach \x in {0,1,2,3,4,5,6}
     {   \path(up\x) --++ (135:0.6) coordinate(up1\x);  
      \path(up\x) --++ (45:0.6) coordinate(up2\x);
            \path(up\x) --++ (-45:0.6) coordinate(up3\x);  
                        \path(up\x) --++ (-135:0.6) coordinate(up4\x);  
   }

    \foreach \x in {0,1,2,3,4,5}
     {   \path(down\x) --++ (135:0.6) coordinate(down1\x);  
      \path(down\x) --++ (45:0.6) coordinate(down2\x);
            \path(down\x) --++ (-45:0.6) coordinate(down3\x);  
                        \path(down\x) --++ (-135:0.6) coordinate(down4\x);  
   }

   \draw [fill=white] plot [smooth cycle]
  coordinates {(up11) (up22) (up32)   (up41)  };

  \path(3,1.5) coordinate   (midway);
    \path(3,2) coordinate   (lowermidway);
    
         \path(up3) --++ (-155:1.5) coordinate(up43);  
         \path(up0) --++ (-25:1.5) coordinate(Xup43);

   \draw [fill=white] plot [smooth cycle]
  coordinates {(up40) 
  (midway)
  (up33) (up23) (up13) (up43)
  (lowermidway)  (Xup43)
     (up20)     (up10)  };

       \foreach \x in {0,2,4,6}
     {   \fill[white](\x,3) circle (4pt);   
                        \draw (\x,3) node {$\bullet$}; 
                  }

         \end{tikzpicture}  \end{minipage}          
$$
\caption{The diagrammatic basis of $V_4^0(\varnothing,(2^2))$. 
This right $\W_4$-module is the transitive permutation module isomorphic to $\C \ind_{\W_2\wr \W_2}^{\W_4}$. %To see this, simply note that the stabiliser of any diagram is the copy of ${\W_2\wr \W_2}$ whose inner components cannot distinguish between the inner pairs of dots and whose outer components cannot distinguish between between the outer  circled pairs.
Observe that the stabiliser of the first diagram shown is the usual copy of ${\W_2\wr \W_2} \leq \W_4$.}
\label{StabD2}
\end{figure}

\begin{eg}\label{StabD}
Let $\Dee=4$ and $\deltaP=(2^2)$.  The module $V^0_4(\varnothing,\deltaP)$ is 3-dimensional with basis as depicted in \cref{StabD2}. This module is isomorphic to $\C {\uparrow}_{\W_2\wr \W_2}^{\W_4}$. 
 \end{eg}

We next consider another extreme  case, where $\Dee=0$ and $\gamma=(s^b)$ so $r = (a+s)b$.

\begin{lem}\label{DQdecomp2}
Let %$b = sc$ and let
 $r = (a+s)b$. There is an isomorphism of right $\Sym_r$-modules
\[ \bigl(
\rightspecht { \alpha} \oslash \rightspecht {\beta} \bigr) \otimes_{\Sym_a \wr \Sym_b} V^0_r\bigl( (a^b) : (s^b), \varnothing \bigr) \cong 
\bigl(
\rightspecht {\alpha} \otimes \mathbb{C}_{\Sym_s} \Ind_{\Sym_a \times \Sym_s}^{\Sym_{a+s}} \,\oslash\,
\rightspecht {\beta} \bigr) 
	\Ind_{\Sym_{a+s} \wr \Sym_b}^{\Sym_{(a+s)b}}
\]
\end{lem}
%\begin{proof}
%The diagrammatic module  is a transitive permutation $ ( \W_{a_i} \wr  \W_{\gamma_i},\W_r)$-module with the required stabilisers (on left and right),  and so the result follows.  See \cref{StabD,StabD2,StabD3,StabD4} for   examples. [{\bf Red correction
%after Kent visit.}]
%\end{proof}
\begin{proof}
By \cite[page 56, Corollary 3]{Alperin}, it suffices to find a right $\mathbb{C}\Sym_{a+s} \wr \Sym_b$-submodule $X$ of
$\bigl( 
\rightspecht {\alpha} \oslash\rightspecht {\beta} \bigr) \otimes_{\Sym_a \wr \Sym_b} V^0_r\bigl((a^b): (s^b), \varnothing\bigr)$ that
is isomorphic to $\bigl( 
\rightspecht { \alpha} \otimes \mathbb{C}_{\Sym_s} \bigr) \Ind_{\Sym_a \times \Sym_s}^{\Sym_{a+s}}\, \oslash\,
\rightspecht {\beta}$
and such that
\[ \dim  \, (
\rightspecht {\alpha} \oslash\rightspecht {\beta} ) \otimes_{\Sym_a \wr \Sym_b} V^0_r\bigl( (a^b): (s^b), \varnothing \bigr) 
 = |\Sym_{(a+s)b} : \Sym_{a+s} \wr \Sym_b|\dim X.\]
Let $k = [\Sym_{a+s} : \Sym_a \times \Sym_s]$ and let $\vartheta _1, \ldots, \vartheta _k$ be right-coset representatives
for $\Sym_a \times \Sym_s$ in $\Sym_{a+s}$; thus $\Sym_{a+s} = \bigsqcup_{j=1}^k (\Sym_a \times \Sym_s)\vartheta _j$.
Recall that $v_{(a,s)}$ denotes the diagram
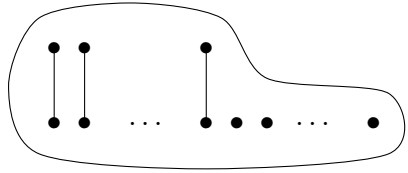
\begin{center}
\begin{tikzpicture}[x=1cm,y=1cm]
\renewcommand{\dot}[2]{\node at (#1, #2) {$\bullet$};}
\draw (0,0)--(0,1);
\draw (0.4,0)--(0.4,1);
\draw (2,0)--(2,1);
\node at (1.2,0) {$\ldots$};
\dot{0}{1}\dot{0.4}{1}\dot{2}{1}
\dot{0}{0}\dot{0.4}{0}\dot{2}{0}\dot{2.4}{0}\dot{2.8}{0}\dot{4.2}{0}
\node at (3.4,0) {$\ldots$};
\draw plot [smooth cycle] coordinates {(-0.2,-0.4) (-0.6,0.5) (-0.2,1.4) (1, 1.6) (2.2,1.4) (2.8,0.6) (4.4,0.4) (4.4,-0.4) (2,-0.6)};
\end{tikzpicture}
\end{center}
where there are $a$  northern and $a+s$ southern vertices.
%dots at the top and $a+s$ dots in the bottom.
As a vector space, we define~$X$ by
\[ X = \bigl\langle (x \otimes y) \otimes v_{a,s} \vartheta _{i_1} \concat \cdots \concat v_{a,s} \vartheta _{i_b} \mid 
 i_1, \ldots, i_b \in \{1,\ldots, k\}
, x \in 
(\rightspecht {\alpha})^{\otimes b}, y \in \rightspecht {\beta}  
 \bigr\rangle  \]
Note that the ramified $\bigl(ab, (a+s)b\bigr)$ diagrams
appearing in the final tensor factor in each chosen basis element of $X$
have outer blocks
\[ \bigl\{ 1,\ldots, a, \overline{1}, \ldots, \overline{a+s} \bigr\},
\bigl\{a+1, \ldots, 2a, \overline{(a+s)+1}, \ldots, \overline{2(a+s)} \bigr\},\ldots  \]
Thus there are no outer crossings.
%, and since $v_{(a,s)}$ has no inner crossings, there are also no inner crossings. {\bf Query.}
Observe that any $\pi \in \Sym_{a} \leq \Sym_{a+s}$  satisfies
$v_{(a,s)}\pi=\pi v_{(a,s)}$; this is the $(a, a+s)$-diagram  with the permutation $\pi$ on the first $a$ strings, followed by $s$ %bottom row isolated dots.
isolated souther vertices.

We now give $X$ the structure of a right $\mathbb{C}\Sym_{a+s} \wr \Sym_b$-module. 
For the action of the base group we must first understand how $\Sym_{a+s}$ acts
on each $v_{(a,s)}\vartheta _{i_c}$. 
Let $\sigma \in \Sym_{a+s}$ and suppose
that $\vartheta _{i_c} \sigma \in (\Sym_a \times \Sym_s)\vartheta _j$,
so that $\vartheta _{i_c} \sigma = \pi_1 \pi_2 \vartheta _j$ where $\pi_1 \in \Sym_a$ and $\pi_2 \in \Sym_s$. 
We then have
\begin{equation}
\label{eq:singleFactorAction} 
v_{(a,s)}\vartheta _{i_c} \cdot \sigma = v_{(a,s)} \cdot \pi_1 \pi_2 \vartheta _j = \pi_1  v_{(a,s)} \vartheta _j 
\end{equation}
since $\pi_2$ acts trivially on the isolated southern dots.
% Note that $\pi_1 \in \Sym_s$ acts by permuting the $s$ top dots in the diagram $v_{(a,s)}$.
We now define the action of an arbitrary element 
$(\sigma_1, \ldots, \sigma_b)$ of the base group $\Sym_{a+s} \times \cdots \times \Sym_{a+s}$ of
$\Sym_{a+s} \wr \Sym_b$ in the obvious way by
\[ (x \otimes y) \otimes v_{(a,s)}\vartheta _{i_1} \concat \cdots \concat v_{(a,s)}\vartheta _{i_s} 
\cdot (\sigma_1, \ldots, \sigma_b) =
(x \otimes y) \otimes (v_{(a,s)}\vartheta _{i_1} \cdot \sigma_1) \concat \cdots \concat (v_{(a,s)}\vartheta _{i_s} \cdot \sigma_s) \]
where each factor in the concatenation of the right-hand side is defined by~\eqref{eq:singleFactorAction}.
Now let $\tau \in \Sym_b$. We define
\[ v_{(a,s)} \vartheta _{i_1} \concat \cdots \concat v_{(a,s)} \vartheta _{i_b} (1_{\Sym_a},\ldots, 1_{\Sym_a} ; \tau)
= (1_{\Sym_a}, \ldots, 1_{\Sym_a} ; \tau) (v_{(a,s)} \vartheta _{i_{\tau(1)}} \concat \cdots \concat v_{(a,s)} \vartheta _{i_{\tau(b)}}).\]
Let $w$ denote a basis vector spanning the trivial $\Sym_{\Sym_s}$-module.
We now claim that as right-$\C(\Sym_{a+s} \times \Sym_b)$-modules there is an isomorphism
\[ X \cong \bigl( \rightspecht {\alpha} \otimes \mathbb{C}_{\Sym_{s}} \bigr)\Ind_{\Sym_a \times \Sym_s}^{\Sym_{a+s}} \;\oslash\; \rightspecht {\beta} \]
defined by
\begin{equation}
\label{eq:Xisomorphism} 
(x_1 \otimes \cdots \otimes x_b \otimes y) \otimes (v_{(a,s)}\vartheta _{i_1} \concat \cdots \concat v_{(a,s)} \vartheta _{i_b})
\mapsto \bigl( (x_1 \otimes w) \otimes \vartheta _{i_1} \bigr) \otimes \cdots \otimes \bigl( (x_b \otimes w) \otimes \vartheta _{i_b}
\bigr) \otimes y. \end{equation}
To see that this commutes with the action of the base group it suffices to check this for 
\[ (\sigma, 1_{\Sym_{a+s}}, \ldots, 1_{\Sym_{a+s}}) \]
where $\sigma \in \Sym_{a+s}$. Suppose that $\vartheta _{i_1} \cdot \sigma = \pi_1 \pi_2 \vartheta _j$ where, as before
$\pi_1 \in \Sym_a$, $\pi_2 \in \Sym_s$. Acting on the left-hand side of~\eqref{eq:Xisomorphism} we obtain
\[ (x_1 \pi_1 \otimes x_2 \otimes \cdots \otimes x_b \otimes y) \otimes
(v_{(a,s)}\vartheta _{i_1} \concat \cdots \concat v_{(a,s)} \vartheta _{i_b}).\]
Acting on the right-hand side of~\eqref{eq:Xisomorphism} we obtain
\begin{align*} 
 (x_1 \otimes w)  \otimes \vartheta _1 &\sigma \otimes (x_2 \otimes w) \otimes \vartheta _{i_2} \cdots 
\otimes (x_b \otimes w) \otimes \vartheta _{i_b} \otimes y \\
 &= (x_1 \otimes w) \otimes \pi_1\pi_2 \vartheta _j \otimes (x_2 \otimes w) \otimes \vartheta _{i_2} \cdots 
\otimes (x_b \otimes w) \otimes \vartheta _{i_b} \otimes y \\
 &= (x_1 \pi_1 \otimes w \pi_2) \otimes \vartheta _j \otimes (x_2 \otimes w) \otimes \vartheta _{i_2} \cdots 
\otimes (x_b \otimes w) \otimes \vartheta _{i_b} \otimes y 
\end{align*}
and we see that the actions are compatible with~\eqref{eq:Xisomorphism}.
For the top group, again let $\tau \in \Sym_b$. Acting on the left-hand side of~\eqref{eq:Xisomorphism} we obtain
\[ (x_{\tau(1)} \otimes \cdots \otimes x_{\tau(b)} \otimes \tau y) \otimes 
(v_{(a,s)} \vartheta _{i_{\tau(1)}} \concat \cdots \concat v_{(a,s)} \vartheta _{i_{\tau(b)}} ) \]
and on the right-hand side
\[ (x_{\tau(1)} \otimes w \otimes \vartheta _{i_{\tau(1)}} ) \otimes \cdots \otimes (x_{\tau(b)} \otimes w \otimes \vartheta _{i_{\tau(b)}})
\otimes \tau y \]
and again the actions agree.

Finally we check the dimensions. We have
\begin{align*} 
\dim \bigl( (&
\rightspecht {\alpha} \otimes %&
\C_{\Sym_a} )\Ind_{\Sym_a \times \Sym_s}^{\Sym_{a+s}} \:\oslash\; \rightspecht{\beta} \bigr)
\times |\Sym_{(a+s)b} : \Sym_{a+s} \wr \Sym_b| 
\\ 
&= 
\bigl( \dim \rightspecht {\alpha} \times |\Sym_{a+s} : \Sym_a \times \Sym_s| \bigr)^b \dim
\rightspecht {\beta} \times
|\Sym_{(a+s)b} : \Sym_{a+s} \wr \Sym_b|  
\\ 
&= 
\bigl(\dim \rightspecht{\alpha} \bigr)^b \frac{(a+s)!^b}{a!^b s!^b} 
\dim \rightspecht{\beta} \frac{\bigl((a+s)b\bigr)!}{(a+s)!^b b!} 
\\ 
&=  
\bigl(\dim 
\rightspecht {\alpha} \bigr)^b \dim
\rightspecht {\beta} \frac{ \bigl(((a+s)b\bigr)!}{a!^b s!^b b!}. 
\end{align*}
The number of $\Sym_a \times \Sym_b$-coset representatives required for the 
diagrams in $V^0_r((a^b): (s^b), \varnothing)$ is 
\[% \textstyle 
\frac{((a+s)b)!}{(a+s)!^b b!} \frac{(a+s)!^b}{a!^b s!^b} 
= \frac{(a+s)b)!}{a!^b s!^b b!}.
 \]
To see this, note that we can choose the $b$ blocks of $a+s$ bottom row dots that will be the blocks in 
$$ %\textstyle 
\binom{(a+s)b}{a+s,\ldots,a+s}\frac{1}{b!}$$
ways and we  then choose the $s$ dots within each block to be singletons in $\binom{a+s}{s}^b$ ways. 
Therefore
\[ \dim \bigl( 
\rightspecht {\alpha} \oslash 
\rightspecht {\beta} \otimes V^0_r(s^b, \varnothing) \bigr)
= ( \dim
\rightspecht {\alpha})^b \dim 
\rightspecht {\beta} \times \frac{\bigl((a+s)b\bigr)!}{a!^b s!^b b!} \] 
%MkW Oct missing (, found while going through for \bigl( \bigr)
and the dimensions agree.
\end{proof}

 \begin{figure}[ht!]
 \vspace*{-12pt}

 $$ \qquad \qquad \begin{tikzpicture}[xscale=0.45,yscale=-0.45]
  
       \foreach \x in {0,1,2,3,4,5,6,7,8,9,10,11,12,13,14}
     {  
      \path(\x*2,0) coordinate (down\x);  
     \path(\x*2,3) coordinate (up\x);  
    }

     \foreach \x in {0,1,2,3,4,5,6,7,8,9,10,11,12,13,14}
     {   \path(up\x) --++ (135:0.6) coordinate(up1\x);  
      \path(up\x) --++ (45:0.6) coordinate(up2\x);
            \path(up\x) --++ (-45:0.6) coordinate(up3\x);  
                        \path(up\x) --++ (-135:0.6) coordinate(up4\x);  
   }

    \foreach \x in {0,1,2,3,4,5,6,7,8,9,10}
     {   \path(down\x) --++ (135:0.6) coordinate(down1\x);  
      \path(down\x) --++ (45:0.6) coordinate(down2\x);
            \path(down\x) --++ (-45:0.6) coordinate(down3\x);  
                        \path(down\x) --++ (-135:0.6) coordinate(down4\x);  
   }

   \draw [fill=white] plot [smooth cycle]
  coordinates {(up10) (up23) (down32)   (down40)  };

   \draw [fill=white] plot [smooth cycle]
  coordinates {(up14) (up27) (down36)   (down44)  };

 \draw(8,0) --(10,3);
    \draw(10,0) --(14,3);

 \draw(16,0)--++(90:3);
    \draw(18,0)--++(90:3);
 
 \draw(0,3)--(0,0);                       
    \draw(2,0)--(6,3);

      \foreach \x in {0,2,4,6,8,10,12,14,16,18,20,22}
     {   \fill[white](\x,3) circle (4pt);   
                        \draw (\x,3) node {$\bullet$}; 
                  }
                
      \foreach \x in {0,2,8,10,16,18}
     {   \fill[white](\x,0) circle (4pt);   
                        \draw (\x,0) node {$\bullet$}; 
                  }

     \foreach \x in {0,1,2,3,4 }
     {  
      \path(\x*2+8+8,0) coordinate (down\x);  
     \path(\x*2+8+8,3) coordinate (up\x);  
    }

     \foreach \x in {0,1,2,3,4 }
     {   \path(up\x) --++ (135:0.6) coordinate(up1\x);  
      \path(up\x) --++ (45:0.6) coordinate(up2\x);
            \path(up\x) --++ (-45:0.6) coordinate(up3\x);  
                        \path(up\x) --++ (-135:0.6) coordinate(up4\x);  
   }

    \foreach \x in {0,1,2,3,4,5,6,7,8,9,10}
     {   \path(down\x) --++ (135:0.6) coordinate(down1\x);  
      \path(down\x) --++ (45:0.6) coordinate(down2\x);
            \path(down\x) --++ (-45:0.6) coordinate(down3\x);  
                        \path(down\x) --++ (-135:0.6) coordinate(down4\x);  
   }

   \draw  plot [smooth cycle]
  coordinates {(up10) (up23) (down32)   (down40)  };

         \end{tikzpicture}          $$
         
         \caption{A prototypical diagram appearing in an element of $X$. Here $a=2$ and $b=3$ and $s=2$.  
         The cosets (of minimal length) are $\vartheta _1=(2,4)$,
         $\vartheta _2=(1,2,4)$  and $\vartheta _3={\rm id}_{\W_4}$.
         }
         \label{proto1}
    \end{figure}
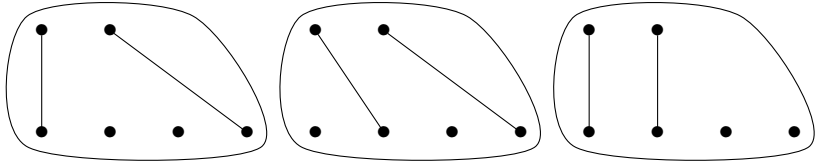

We may now use the previous two lemmas to describe the decomposition in the general case.
We remind the reader that the set $\ParSet_{(a^b)}(\Cee)$ was defined in Definition~\ref{defn:ParSet};
splitting into the two cases $a > 0$, in which $c_0 > 0$ is permitted, and $a=0$ in which case,
as we observed after Definition~\ref{defn:t:2}, $c_0 = 0$.
We remind the reader of our standing convention
that $\beta^i \vdash c_i$ in a sum indicates that the sum is over all
relevant sequences of partitions indexed by $i$.

\begin{thm}\label{thm:DQdecomp} Suppose that $r-ab=\Cee+\Dee$, $\gamma=  
   ( \Cee^{c_\Cee}, \dots ,  1^{c_1}, 0^{c_0} )
\in  \ParSet _{(a^b)}(\Cee) $ and $\deltaP  \in  \ParSet_{>1}(\Dee )$. Then the $P_r(\delta_{\rm in}\delta_{\rm out})$-submodule 
\[ \rightspecht {\alpha^\beta}\otimes_{_{\W_a \wr \W_b} } V^0_r\bigl( (a^b): \gamma,\deltaP \bigr) 
 \subseteq {\sf DQ}\bigl(\Delta_{r}(\alpha^\beta)\bigr)
\]
  decomposes as follows 
\[
 \bigoplus _{\beta^i \vdash c_i} 
  c^\beta_{\beta^\Cee, \ldots, \beta^1, \beta^0} \Bigl(
  \bigotimes  _{i=0}^\Cee
 \Bigl(
 (
\rightspecht {\alpha} \boxtimes \C  )
\Ind
_{\W_a\times \W_{ i }}
^{\W_{a+i}}
)
\oslash \rightspecht{\beta^i}
 \Bigr)
 \boxtimes \C _{{\rm Stab}(\deltaP)}\Bigr)
 \Ind
 _{{\rm Stab}((a^b)+\gamma)\times {\rm Stab}(\deltaP)} ^{\W_r}
 \]
via the canonical quotient map $P_r(\delta_{\rm in}\delta_{\rm out})\to \C  \W_r$ 
{\color{red}\color{black}   from  \cref{bob2}.}
\end{thm}

\begin{proof}
We first note that as an $\C\W_r$-module, 
\[ V^0_r\bigl( (a^b): \gamma,\deltaP \bigr) \cong  \bigl(V^0_{\Cee +ab}\bigl((a^b): \gamma,\varnothing\bigr) \boxtimes  V^0_\Dee \bigl( \varnothing: \varnothing,\deltaP)\bigr) \Ind_{\W_{\Cee +ab} \times \W_\Dee } ^{\W_r},
\]
simply because the action of $\W_r$
permutes propagating outer blocks amongst themselves  and  permutes non-propagating outer blocks amongst themselves.  Thus, by \cref{DQdecomp1} and the transitivity of induction, it suffices  to show that  
  $\rightspecht{\alpha^\beta}\otimes V^0_{ab+\Cee}\bigl((a^b): \gamma,\varnothing \bigr) $  
 decomposes as follows: 
\begin{align}\label{ajkglhjdflkhgdsl} %unused
\bigoplus _{\beta^i \vdash c_i} 
 c^\beta_{\beta^\Cee, \ldots, \beta^1, \beta^0}
 \Bigl( \bigotimes  _{i=0}^\Cee
( \rightspecht {\alpha} \boxtimes \C ) \Ind_{\W_a\times \W_{ i }}^{\W_{a+i}} 
\:\oslash\; 
\rightspecht {\beta^i}  \Bigr)\Ind_{{\rm Stab}((a^b)+\gamma)}^{\W_{ab+\Cee}}.
\end{align}
We prove this statement. Firstly, for any right $\C\W_a\wr \W_{b}$-module $X$, 
\[ X\otimes_{\W_a\wr \W_{b}} V^0_{ab+\Cee}\bigl((a^b): \gamma,\varnothing \bigr)
\cong 
X \Res_{\W_a\wr \W_{\underline{\bf c}} }
\otimes_{\W_a\wr \W_{\underline{\bf c}}}  
\Bigl( \bigotimes  _{i=0}^{\Cee}
V^0_{(a+i)c_i}\bigl((a^{c_i}): (i^{c_i}),\varnothing \bigr)
\Bigr)\Ind_{\prod_i \W_{(a+i)c_i} }^{\W_{ab+\Cee}},
\]
where $\underline{\bf c}=(c_p, \dots, c_1, c_0)$.
To see this, observe that a ramified diagram $v \in V^0_{ab+\Cee}((a^b): \gamma,\varnothing )$ may be written as 
\[v= \vartheta  \bigl(  (v_{a,\Cee})^{\circledast c_\Cee} \circledast \dots \circledast (v_{a,1})^{\circledast c_1} \circledast (v_{a,0})^{\circledast c_0}   \bigr) \sigma,$$
for $ \vartheta  \in \W_a\wr \W_b$ and $\sigma \in \W_{ab+\Cee}$.
(Recall that $c_i$ is the multiplicity of $i$ as a part of $\gamma$.)
Then, for $x \in X$, the isomorphism sends
  $$x \otimes v\mapsto 
 x\vartheta  \otimes  \bigl((v_{a,\Cee})^{\circledast c_\Cee} \circledast \dots \circledast (v_{a,1})^{\circledast c_1} \circledast (v_{a,0})^{\circledast c_0}\bigr) \sigma.
 \] 
Therefore 
 $\rightspecht {\alpha^\beta}\otimes V^0_{ab+\Cee}\bigl((a^b): \gamma,\varnothing ) $  is isomorphic to
\[ 
(\rightspecht {\alpha}\oslash \rightspecht {\beta}) \Res_{\W_a\wr \W_{\underline{\bf c}} }
\otimes_{\W_a\wr \W_{\underline{\bf c}}}  
\Bigl( \bigotimes  _{i=0}^\Cee
V^0_{(a+i)c_i}\bigl((a^{c_i}): (i^{c_i}),\varnothing \bigr)
%V^0_{(a+s)c_s}(a^b: (s^{c_s}),\varnothing )
\Bigr)\Ind_{\prod_{i=0}^\Cee \W_{(a+i)c_i}}^{\W_{ab+\Cee}}.
\]
Using \cite[Lemma 3.3(2)]{ChuangTan}, this is isomorphic to
\[ \bigoplus _{\beta^i \vdash c_i} 
  c^\beta_{\beta^\Cee, \ldots, \beta^1, \beta^0}
\Bigl(\bigotimes_{i=0}^\Cee 
\rightspecht {\alpha} \oslash \rightspecht {\beta^i} \Bigr) 
\otimes_{\W_a\wr \W_{\underline{\bf c}}}  
\Bigl( \bigotimes  _{i=0}^\Cee
V^0_{(a+s)c_i}\bigl((a^{c_i}): (i^{c_i}),\varnothing \bigr)
\Bigr)\Ind_{\prod_{i=0}^\Cee \W_{(a+i)c_i}}^{\W_{ab+\Cee}}.
\]
where $c^\beta_{\beta^\Cee, \ldots, \beta^1, \beta^0}$ is the generalized Littlewood--Richardson coefficient
defined in~\eqref{eq:genLR}.
Regrouping terms, this becomes
\[ \bigoplus _{\beta^i \vdash c_i} 
  c^\beta_{\beta^\Cee, \ldots, \beta^1, \beta^0} \Bigl(
  \bigotimes_{i=0}^\Cee
  \Bigl(  ( \rightspecht{\alpha} \oslash 
  \rightspecht{\beta^i}) \otimes_{\W_a\wr \W_{c_i}}
V^0_{(a+i)c_i}\bigl((a^{c_i}): (i^{c_i}),\varnothing \bigr)  \Bigr)  
   \Bigr)\Ind_{\prod_{i=0}^\Cee \W_{(a+i)c_i}}^{\W_{ab+\Cee}}.
\]
\Cref{DQdecomp2} provides the isomorphism to
\[ \bigoplus _{\beta^i \vdash c_i} 
  c^\beta_{\beta^\Cee, \ldots, \beta^1, \beta^0} \Bigl(
  \bigotimes_{i=0}^\Cee
 \bigl( (\rightspecht {\alpha} \otimes \mathbb{C}_{\Sym_i} )\Ind_{\Sym_a \times \Sym_i}^{\Sym_{a+i}}\, \oslash\,
 \rightspecht {\beta^i} \bigr)
 \Bigr)\Ind_{\prod_{i=0}^\Cee \W_{(a+i)c_i}}^{\W_{ab+\Cee}},
\]
 and then transitivity of induction yields the desired statement.  
\end{proof}

%The full decomposition of the depth quotient is then obtained by summing over all partitions 
%$\gamma \in \ParSet _{(a^b)}(\Cee)$ and $\deltaP \vdash  \Dee$ such that $\Cee + \Dee = r - ab$
%as noted in Proposition~\ref{Stepp1}.

\subsection{Proofs of Theorems A and D} We are now ready to prove \TheoremD and as a corollary, \TheoremA.

\begin{proof}[Proof of \TheoremD]
The full decomposition
of the depth quotient is 
obtained from Proposition~\ref{Stepp1}  and Theorem~\ref{thm:DQdecomp} by summing over all partitions 
$\gamma \in \ParSet _{(a^b)}(\Cee)$ and $\deltaP \vdash  \Dee$ such that $\Cee + \Dee = r - ab$,
using Corollary~\ref{cor:stable1} and the first case 
of Corollary~\ref{cor:compfactors+DR}, applied with $\lambda = \kappa$. \TheoremD then 
follows from Corollary~\ref{cor:compfactors+DR}.
\end{proof}

%\bf Mark: we claim to prove \TheoremD (was C) in  Section 9 in the introduction,
%and I think it is leaving far too much to the reader to expect them to see it at once from the result above.

\begin{proof}[Proof of \TheoremA]
This follows immediately from the case $\alpha = \varnothing$ of \TheoremC and \TheoremD.
\end{proof}

\section{Examples and applications}\label{egs}
We shall write $\rc$ for the ramified branching coefficients formally defined
in Definition~\ref{defn:ramifiedBranchingCoefficient} and determined
in \TheoremD. Thus
\begin{equation}
\label{eq:rc}
\rc( \alpha^\beta, \kappa) =  \bigl[\Delta_r(\alpha^\beta)\Res^{\ram_r(m,n)}_{P_r(mn)} : L_r(\kappa)\bigr]\ForPrmn. 
\end{equation}
The values of $m$, $n$ and $r$ will be clear from context.
By the final part of \TheoremC, when $\alpha = \varnothing$ we have
$p(\beta[n],(m), \kappa[mn]) = \rc(\varnothing^\beta,\kappa)$
%MkW Oct: \alpha -> \varnothing on right-hand-side
provided $m \ge r - |\beta| + [\beta \not= \varnothing]$ and $n \ge r + \beta_1$.

\subsection{Examples of Theorems~A, C and D}
%It's not possible to get hyperlinks here even using \protect
We consider
\[ \Delta_5 (\varnothing^{(2,1)})\Res^{\ram_5(m,n)}_{P_5(mn)}, 
\quad 
 \Delta_5 \bigl( (1)^{(2,1)}\bigr)\Res^{\ram_5(m,n)}_{P_5(mn)}      
\]
and find all the  composition factors $L_5(\kappa)$ for $\kappa\vdash 5$
 %MkW Oct: \vdash 5 rather than \ParSet for consistency
of these modules
by decomposing the depth quotient.
We have chosen to change only the partition $\alpha$ (from $\varnothing$ to $(1)$) as this minor change results in big changes in the ramified branching coefficient
$[\Delta_r(\alpha^\beta)\res^{\ram_r(m,n)}_{P_r(mn)} : L_r(\kappa)]_{P_5(mn)}$
 and even bigger changes in the stable 
plethysm coefficients. Indeed, as we discussed in Section~\ref{subsec:alphaEmptyVersusNonEmpty},
in the first case we obtain the stable values of
$p\bigl( (n-3,2,1), (m), \kappa[mn]\bigr)$ for $\kappa \vdash 5$
and arbitrary $m$ \emph{and} $n$, whereas in the second case 
we obtain the stable values of
$p\bigl( (2,1), (m-1,1), \kappa[mn]\bigr)$ for $\kappa \vdash 5$ and
arbitrary $m$; the outer partition $(2,1)$ is now fixed.

\begin{eg}\label{eg:alphaEmpty} 
We take $\alpha=\varnothing$ and $\beta=(2,1)$ and $\kappa \vdash 5$. %R: add \kappa \vdash 5
By \TheoremC, provided $n \ge 7$ and $m \ge 3$ 
we have $p\bigl((n-3,2,1), (m), \kappa[mn]\bigr) = \rc\bigl( \varnothing^{(2,1)}, \kappa 
\bigr)$. %Rcs
We shall derive below the stable plethysm and ramified branching coefficients 
\begin{align*}
p\bigl((n-3,2,1),(	m), (mn-5,5)	\bigr)&=2=	\rc \bigl( \varnothing^{(2,1)}, (5)\bigr) , \\ 
p\bigl((n-3,2,1),(	m), (mn-5,4,1)	\bigr)&=5=\rc \bigl(  \varnothing^{(2,1)}, (4,1)\bigr) , \\ 
p\bigl((n-3,2,1),(	m), (mn-5,3,2)	\bigr)&=4=\rc \bigl(  \varnothing^{(2,1)}, (3,2)\bigr) , \\
p\bigl((n-3,2,1),(	m), (mn-5,3,1^2)	\bigr)&=3=\rc \bigl(  \varnothing^{(2,1)}, (3,1^2)\bigr) , \\
p\bigl((n-3,2,1),(	m), (mn-5,2^2,1)	\bigr)&=2=\rc \bigl(  \varnothing^{(2,1)}, (2^2,1)\bigr), \\
p\bigl((n-3,2,1),(	m), (mn-5,2 ,1^3)	\bigr)&=0=\rc \bigl(  \varnothing^{(2,1)}, (2,1^3)\bigr) ,\\
p\bigl((n-3,2,1),(	m), (mn-5,1^5)	\bigr)&=0=\rc\bigl(  \varnothing^{(2,1)}, (1^5)\bigr) 
\end{align*}%   if and only if $n\geq k+\kappa_1$ and $m\geq r-k+1$.  
for $m$ and $n$ satisfying these bounds. We decompose the depth quotient 
${\sf DQ} \bigl(\Delta_5(\varnothing^{(2,1)})\bigr)$
as in Theorem~\ref{thm:DQdecomp} thereby computing the coefficients above
by the formula in Theorems~\hyperlink{thmA}{A} and~\hyperlink{thmD}{D}.
 There are three summands of ${\sf DQ} \bigl(\Delta_5(\varnothing^{(2,1)})\bigr)$ which are of interest.  
 These are generated by the diagrams $v_{(\gamma,\deltaP)}$ depicted in \cref{M1M2M3empty}.
 To see this from the formulae, note that $|\gamma|+|\deltaP|=5-0\times 3=5$ and 
 since $\alpha = \varnothing$ and $|\beta| = 3$,
the partition $\gamma$ has three non-zero parts. As always,
 $\deltaP$ has no parts of size 1 since non-propagating blocks may not be singletons
 (such basis elements lie in the depth radical). Each outer block has at least one southern dot,
 so we have two further dots to place. Our options are as follows:
\begin{itemize}[leftmargin=18pt] 
\item place both extra dots in the {\em same}   propagating block (set $ \gamma=(3,1^2)$) leaving no extra dots to place in a non-propagating block (that is, $\deltaP=\varnothing$).
 \item place  each extra dot in a {\em separate}   propagating block (set $ \gamma=(2^2,1)$) leaving no extra dots to place in a non-propagating block (that is, $\deltaP=\varnothing$).
 \item place both extra dots in the {\em same} non-propagating block (set $\deltaP = (2)$ and $\gamma = (1^3)$) 
 %(set $\deltaP=\varnothing$) MkW: wrong
 %leaving no extra dots to place in a non-propagating block (that is, $\gamma=(1^3)$).
 \end{itemize}
% Note that non-propagating blocks cannot be singletons (such basis elements belong to the depth radical) and so the cases $\deltaP=(1)$ or $(1,1)$ do not occur. This corresponds to the condition
% that $\deltaP \in \ParSet_{>1}$ in the formula in Theorems~A and~D.

\begin{figure}[ht!]
 $$\scalefont{0.8} \begin{tikzpicture}[xscale=0.45,yscale=-0.45]
  
       \foreach \x in {0,1,2,3,4,5}
     {  
      \path(\x*2,0) coordinate (down\x);  
   }

       \foreach \x in {0,1,2,3,4,5,6}
    {   \path(\x*2,3) coordinate (up\x);  
    }

     \foreach \x in {0,1,2,3,4,5,6}
     {   \path(up\x) --++ (135:0.6) coordinate(up1\x);  
      \path(up\x) --++ (45:0.6) coordinate(up2\x);
            \path(up\x) --++ (-45:0.6) coordinate(up3\x);  
                        \path(up\x) --++ (-135:0.6) coordinate(up4\x);  
   }

    \foreach \x in {0,1,2,3,4,5}
     {   \path(down\x) --++ (135:0.6) coordinate(down1\x);  
      \path(down\x) --++ (45:0.6) coordinate(down2\x);
            \path(down\x) --++ (-45:0.6) coordinate(down3\x);  
                        \path(down\x) --++ (-135:0.6) coordinate(down4\x);  
   }
    
   \draw [fill=white] plot [smooth cycle]
  coordinates {(up10) (up22) (down31)   (down40)  };
 
     \foreach \x in {3}
     {   \path(up\x) --++ (135:0.6) coordinate(up14);  
      \path(up4) --++ (45:0.6) coordinate(up24);
            \path(up4) --++ (-45:0.6) coordinate(up34);  
                        \path(up4) --++ (-135:0.6) coordinate(up44);  
   }

    \foreach \x in {3}
     {   \path(down4) --++ (135:0.6) coordinate(down14);  
      \path(down4) --++ (45:0.6) coordinate(down24);
            \path(down4) --++ (-45:0.6) coordinate(down34);  
                        \path(down4) --++ (-135:0.6) coordinate(down44);  
   }

   \draw [fill=white] plot [smooth cycle]
  coordinates {(up13) (up23)  (down33)   (down43)  };

     \foreach \x in {4}
     {   \path(up\x) --++ (135:0.6) coordinate(up14);  
      \path(up4) --++ (45:0.6) coordinate(up24);
            \path(up4) --++ (-45:0.6) coordinate(up34);  
                        \path(up4) --++ (-135:0.6) coordinate(up44);  
   }

    \foreach \x in {4 }
     {   \path(down4) --++ (135:0.6) coordinate(down14);  
      \path(down4) --++ (45:0.6) coordinate(down24);
            \path(down4) --++ (-45:0.6) coordinate(down34);  
                        \path(down4) --++ (-135:0.6) coordinate(down44);  
   }
  
   \draw [fill=white] plot [smooth cycle]
  coordinates {(up14) (up24)  (down34)   (down44)  };

        \foreach \x in {0,2,4,6,8}
     {   \fill[white](\x,3) circle (4pt);   
                        \draw (\x,3) node {$\bullet$}; 
                  }
                
      \foreach \x in {0,6,8}
     {   \fill[white](\x,0) circle (4pt);   
                        \draw (\x,0) node {$\bullet$}; 
                  }
                  
\draw[rounded corners ,densely dotted ] (-0.25,-1)  rectangle( 8.25,-2.5)   node[midway] 
%{$\bigl(1+(1,2)\bigr)\bigl(1-(1,3)\bigr)$}; %
{$\color{red}\color{black} \bigl(1 - (1,3) \bigr)\bigl( 1 + (1,2) \bigr)$};
         \end{tikzpicture} \!\!\!\!\!\!
 \begin{tikzpicture}[xscale=0.45,yscale=-0.45]
  
       \foreach \x in {0,1,2,3,4,5}
     {  
      \path(\x*2,0) coordinate (down\x);  
   }

       \foreach \x in {0,1,2,3,4,5,6}
    {   \path(\x*2,3) coordinate (up\x);  
    }

     \foreach \x in {0,1,2,3,4,5,6}
     {   \path(up\x) --++ (135:0.6) coordinate(up1\x);  
      \path(up\x) --++ (45:0.6) coordinate(up2\x);
            \path(up\x) --++ (-45:0.6) coordinate(up3\x);  
                        \path(up\x) --++ (-135:0.6) coordinate(up4\x);  
   }

    \foreach \x in {0,1,2,3,4,5}
     {   \path(down\x) --++ (135:0.6) coordinate(down1\x);  
      \path(down\x) --++ (45:0.6) coordinate(down2\x);
            \path(down\x) --++ (-45:0.6) coordinate(down3\x);  
                        \path(down\x) --++ (-135:0.6) coordinate(down4\x);  
   }
   
   \draw [fill=white] plot [smooth cycle]
  coordinates {(up10) (up21) (down31)   (down40)  };

     \foreach \x in {3}
     {   \path(up\x) --++ (135:0.6) coordinate(up14);  
      \path(up4) --++ (45:0.6) coordinate(up24);
            \path(up4) --++ (-45:0.6) coordinate(up34);  
                        \path(up4) --++ (-135:0.6) coordinate(up44);  
   }

    \foreach \x in {3}
     {   \path(down4) --++ (135:0.6) coordinate(down14);  
      \path(down4) --++ (45:0.6) coordinate(down24);
            \path(down4) --++ (-45:0.6) coordinate(down34);  
                        \path(down4) --++ (-135:0.6) coordinate(down44);  
   }

   \draw [fill=white] plot [smooth cycle]
  coordinates {(up12) (up23)  (down33)   (down42)  };

     \foreach \x in {4}
     {   \path(up\x) --++ (135:0.6) coordinate(up14);  
      \path(up4) --++ (45:0.6) coordinate(up24);
            \path(up4) --++ (-45:0.6) coordinate(up34);  
                        \path(up4) --++ (-135:0.6) coordinate(up44);  
   }

    \foreach \x in {4 }
     {   \path(down4) --++ (135:0.6) coordinate(down14);  
      \path(down4) --++ (45:0.6) coordinate(down24);
            \path(down4) --++ (-45:0.6) coordinate(down34);  
                        \path(down4) --++ (-135:0.6) coordinate(down44);  
   }

   \draw [fill=white] plot [smooth cycle]
  coordinates {(up14) (up24)  (down34)   (down44)  };

        \foreach \x in {0,2,4,6,8}
     {   \fill[white](\x,3) circle (4pt);   
                        \draw (\x,3) node {$\bullet$}; 
                  }
                
      \foreach \x in {0,4,8}
     {   \fill[white](\x,0) circle (4pt);   
                        \draw (\x,0) node {$\bullet$}; 
                  }

\draw[rounded corners ,densely dotted ] (-0.25,-1)  rectangle( 8.25,-2.5)   node[midway] 
%{$\bigl(1+(1,2)\bigr)\bigl(1-(1,3)\bigr)$}; %$\otimes _{\W_3}$\specht (2,1)} ; %--  (-0.25,8) -- (-0.25,-)
                        {$\color{red}\color{black} \bigl(1 - (1,3) \bigr)\bigl( 1 + (1,2) \bigr)$};

         \end{tikzpicture} \!\!\!\!\!\!
 \begin{tikzpicture}[xscale=0.45,yscale=-0.45]
  
       \foreach \x in {0,1,2,3,4,5}
     {  
      \path(\x*2,0) coordinate (down\x);  
   }

       \foreach \x in {0,1,2,3,4,5,6}
    {   \path(\x*2,3) coordinate (up\x);  
    }

     \foreach \x in {0,1,2,3,4,5,6}
     {   \path(up\x) --++ (135:0.6) coordinate(up1\x);  
      \path(up\x) --++ (45:0.6) coordinate(up2\x);
            \path(up\x) --++ (-45:0.6) coordinate(up3\x);  
                        \path(up\x) --++ (-135:0.6) coordinate(up4\x);  
   }

    \foreach \x in {0,1,2,3,4,5}
     {   \path(down\x) --++ (135:0.6) coordinate(down1\x);  
      \path(down\x) --++ (45:0.6) coordinate(down2\x);
            \path(down\x) --++ (-45:0.6) coordinate(down3\x);  
                        \path(down\x) --++ (-135:0.6) coordinate(down4\x);  
   }

   \draw [fill=white] plot [smooth cycle]
  coordinates {(up10) (up20) (down30)   (down40)  };
   \draw [fill=white] plot [smooth cycle]
  coordinates {(up11) (up21) (down31)   (down41)  };
   \draw [fill=white] plot [smooth cycle]
  coordinates {(up12) (up22) (down32)   (down42)  };

   \draw [fill=white] plot [smooth cycle]
  coordinates {(up13) (up24) (up34)   (up43)  };

        \foreach \x in {0,2,4,6,8}
     {   \fill[white](\x,3) circle (4pt);   
                        \draw (\x,3) node {$\bullet$}; 
                  }
                
      \foreach \x in {0,2,4}
     {   \fill[white](\x,0) circle (4pt);   
                        \draw (\x,0) node {$\bullet$}; 
                  }

\draw[rounded corners ,densely dotted ] (-0.25,-1)  rectangle( 8.25,-2.5)   node[midway] 
%{%$\otimes _{\W_3}$\specht (2,1)
%$\bigl(1+(1,2)\bigr)\bigl(1-(1,3)\bigr)$} ; %--  (-0.25,8) -- (-0.25,-)
{$\color{red}\color{black} \bigl(1 - (1,3) \bigr)\bigl( 1 + (1,2) \bigr)$};                        

         \end{tikzpicture} 
$$
\caption{ The generators $c_{(2,1)}^\ast	\otimes v_{(\gamma,\deltaP)}$ for 
$ {(\gamma,\deltaP)}$ equal to $ \bigl((3,1^2),\varnothing\bigr)$ and  
$\bigl((2^2,1),\varnothing\bigr)$ and 
$\bigl((1^3),(2)\bigr)$ respectively. 
These generate  the direct summands, which we denote by $M_1$, $M_2$, and $M_3$
of ${\sf DQ} \bigl(\Delta_5(\varnothing^{(2,1)})\bigr)$. }
\label{M1M2M3empty}
\end{figure}

 For $M_1$ we have $\gamma=(3,1^2)$ so the multiplicities of the  parts (read as throughout in decreasing order)
 are $1$ and $2$ and we restrict the Specht module $\rightspecht{(2,1)}$ to $\W_1 \times \W_2$, obtaining
\begin{align}\label{M1}
\rightspecht {(2,1)} \Res_{\W_1 \times \W_2}^{\W_{3}} =   \rightspecht {(1)} \otimes 
%\W_{(1,2)}$ is not defined
\rightspecht {(2)}
 \hskip1pt \oplus\hskip1pt     \rightspecht {(1)} \otimes \rightspecht {(1^2) }.
  \end{align}
Since   $\varepsilon=\varnothing $ we have 
 \begin{align*}
M_1 &\cong 
 \bigl(  
 (\rightspecht {(1) }\otimes 
\rightspecht { (2) }
  \oplus     \rightspecht { (1)} \otimes \rightspecht {(1^2)}\bigr)
  \Ind_{\W_1 \times \W_2}^{\W_{3}\times \W_2}  
  \Ind_{\W_{3}\times \W_2}^{\W_5}
 %\\
%&
\,\cong \,
\rightspecht {(5)} \oplus 2 \rightspecht {  (4,1) }\oplus  \rightspecht {(3,2)} \oplus \rightspecht {(3,1^2)} .\end{align*}
 
 For $M_2$ we have  $\gamma=(2^2,1)$  
so the multiplicities of the parts are now $2$ and $1$ and we take a similar restriction
\begin{align}\label{M2}
\rightspecht {(2,1)} {\Res}_{\W_2 \times \W_1}^{\W_{3}} = \rightspecht { (2)} \otimes \rightspecht { (1) }
\hskip1pt\oplus \hskip1pt \rightspecht { (1^2)} \boxtimes  \rightspecht { (1) }.\end{align}
Since   $\varepsilon=\varnothing $ we have
 \begin{align*}
M_2 &\cong 
\bigl((\rightspecht {(2) }\oslash  \rightspecht {(1^2) })\boxtimes \rightspecht { (1)} \bigr)
\Ind_{\W_2\wr \W_2\times \W_1}^{\W_5}
%changed from (1^2) oslash (2)
\hskip1pt\oplus \hskip1pt
\bigl( (\rightspecht { (2)} \oslash\rightspecht { (2)} )\boxtimes \rightspecht { (1)}
\bigr) \Ind_{\W_2\wr \W_2\times \W_1}^{\W_5}
\\
&\cong
\rightspecht {  (5) }\oplus 
2
\rightspecht {  (4,1) }\oplus 
2
\rightspecht {  (3,2)} \oplus 
\rightspecht { (3,1^2)} \oplus
\rightspecht {  (2^2,1)}\end{align*}

 For $M_3$ we have  $\gamma=(1^3)$ %still wrong and so $\underline{c}=(1^3)$ 
and so  we do not restrict the Specht module $\rightspecht {(2,1)}$. 
Since $\epsilon = (2)$ we have
\[ M_3\cong (
\rightspecht { (2,1) }\otimes \C _{\W_2}	)\Ind_{\W_3\times \W_2}^{\W_5}\cong  
\rightspecht { (4,1)} \oplus  \rightspecht { (3,2)} \oplus \rightspecht {  (3,1^2)}  
\oplus \rightspecht { (2^2,1)}. \]

Summing over the coefficients appearing in the decompositions of $M_1,$ $M_2$ and $M_3$ we obtain the ramified branching coefficients as stated at the beginning of this example. 
\end{eg}%MkW: checked on Magma with symmetric functions

\begin{eg}\label{eg:alphaNonEmpty}
We now take $\alpha=(1)$ and keep $\beta=(2,1)$ and $\kappa \vdash 5$. Now $n = |\beta| = 3$ is fixed.
%Rc, added 'is fixed'
By \TheoremC we have
\[ p\bigl((2,1), (m-1,1), \kappa{[3m]}\bigr) \le \rc\bigl( (1)^{(2,1)}, \kappa
\bigr). \] 
(Note that the hypothesis $r \ge n|\alpha|$ holds because $5 \ge 3 \times 1$.)
By Theorem~1.2 in \cite{deBoeckPagetWildon} or the theorem proved in \S 2.1 of \cite{BrionStability}, 
given any partition $\kappa \vdash 5$ %MkW Oct: \vdash 5 rather than \in \ParSet(5)
the plethysm coefficient
$p\bigl( (2,1), (m-1,1), \kappa[12] + (3m-12) \bigr)$ %Rc (use p notation rather than inner product)
%$\langle s_{(2,1)} \circ s_{(m-1,1)}, s_{\kappa[12] + (3m-12)} %we set m = 4+N, thus N = m - 4
%\rangle$ 
is constant for all $m \ge 4$. 
%$\langle s_{(2,1)} \circ s_{(3+N,1)}, s_{\kappa[12] + (3N)}
%\rangle$ is constant for all $N \in \N_0$. 

%Let's get this right: in the setup of Theorem 1.2 
%we have fixed \nu = (2,1) and r = 1 and we can take \mu = (3,1) so that mn = 12 and any partition of 5 is okay
%to put into \kappa[12]. Therefore the stability property is
%		<s_{(2,1)} • s_{(3,1) + M}, s_{(7 ; \kappa) + 3M}> is ultimately constant
%where \kappa is a partition of 5. By the bound in Theorem 1.2, which is also the bound from Brion, we have
%stability for M >= (n-1)\mu_1 + \mu_2 - \lambda_1 = (3-1)3 + 1 - 7 = 0. I.e. stability is immediate.

%MkW: 'we shall derive' for consistency with earlier
We shall derive the following stable plethysm and ramified branching coefficients:
\begin{align*}
p\bigl(( 2,1),(	m-1,1), (3m-5,5)	\bigr)&= 1\leq 2 = 		\rc\bigl((1)^{(2,1)},(5)\bigr)	, \\ 
p\bigl(( 2,1),(	m-1,1), (3m-5,4,1)		\bigr)&= 3
\leq 6 = 		\rc\bigl((1)^{(2,1)},(4,1)\bigr)
, \\ 
p\bigl(( 2,1),(	m-1,1), (3m-5,3,2)		)&= 4
\leq 7 = 		\rc\bigl((1)^{(2,1)},(3,2)\bigr)
, \\ 
p\bigl(( 2,1),(	m-1,1), (3m-5,3,1^2)	\bigr)&= 4\leq 6  = 		\rc\bigl((1)^{(2,1)},(3,1^2)\bigr), \\ 
p\bigl(( 2,1),(	m-1,1), (3m-5,2^2,1)		\bigr)&=  3
\leq 6 = 		\rc((1)^{(2,1)},(2^2,1)\bigr),
\\
p\bigl(( 2,1),(	m-1,1), (3m-5,2 ,1^3)		\bigr)&=  2\leq 3 = 		\rc\bigl((1)^{(2,1)},(2,1^3)\bigr),
\\
p\bigl(( 2,1),(	m-1,1), (3m-5, 1^5)		\bigr)&=0\leq1 = 		\rc\bigl((1)^{(2,1)},(1^5)\bigr)
\end{align*}%   if and only if $n\geq k+\kappa_1$ and $m\geq r-k+1$.  
for   $m\geq 4$. Note that in contrast to the previous case, the  
outer partition in the plethysm is fixed as $(2,1)$ and only the inner partition $(m-1,1)$ varies.
We notice that none of the bounds are sharp  in this case. 
The stable values of the plethysm coefficients can easily be 
calculated using computer algebra.  %MkW Oct: moved this sentence for consistency with 'we shall derive'
We now calculate the ramified branching coefficients. 
Again using Theorem~\ref{thm:DQdecomp}, 
there are three summands of $\DQ \bigl(\Delta_5((1)^{(2,1)})\bigr)$ which are of interest.  
 These are generated by the diagrams $ v_{(\gamma,\deltaP)}$ depicted in \cref{N1N2N3}.

\begin{figure}[ht!]
 $$\scalefont{0.8} \begin{tikzpicture}[xscale=0.45,yscale=-0.45]
  
       \foreach \x in {0,1,2,3,4,5}
     {  
      \path(\x*2,0) coordinate (down\x);  
   }

       \foreach \x in {0,1,2,3,4,5,6}
    {   \path(\x*2,3) coordinate (up\x);  
    }

     \foreach \x in {0,1,2,3,4,5,6}
     {   \path(up\x) --++ (135:0.6) coordinate(up1\x);  
      \path(up\x) --++ (45:0.6) coordinate(up2\x);
            \path(up\x) --++ (-45:0.6) coordinate(up3\x);  
                        \path(up\x) --++ (-135:0.6) coordinate(up4\x);  
   }

    \foreach \x in {0,1,2,3,4,5}
     {   \path(down\x) --++ (135:0.6) coordinate(down1\x);  
      \path(down\x) --++ (45:0.6) coordinate(down2\x);
            \path(down\x) --++ (-45:0.6) coordinate(down3\x);  
                        \path(down\x) --++ (-135:0.6) coordinate(down4\x);  
   }

   \draw [fill=white] plot [smooth cycle]
  coordinates {(up10) (up22) (down31)   (down40)  };
 
     \foreach \x in {3}
     {   \path(up\x) --++ (135:0.6) coordinate(up14);  
      \path(up4) --++ (45:0.6) coordinate(up24);
            \path(up4) --++ (-45:0.6) coordinate(up34);  
                        \path(up4) --++ (-135:0.6) coordinate(up44);  
   }

    \foreach \x in {3}
     {   \path(down4) --++ (135:0.6) coordinate(down14);  
      \path(down4) --++ (45:0.6) coordinate(down24);
            \path(down4) --++ (-45:0.6) coordinate(down34);  
                        \path(down4) --++ (-135:0.6) coordinate(down44);  
   }

   \draw [fill=white] plot [smooth cycle]
  coordinates {(up13) (up23)  (down33)   (down43)  };

     \foreach \x in {4}
     {   \path(up\x) --++ (135:0.6) coordinate(up14);  
      \path(up4) --++ (45:0.6) coordinate(up24);
            \path(up4) --++ (-45:0.6) coordinate(up34);  
                        \path(up4) --++ (-135:0.6) coordinate(up44);  
   }

    \foreach \x in {4 }
     {   \path(down4) --++ (135:0.6) coordinate(down14);  
      \path(down4) --++ (45:0.6) coordinate(down24);
            \path(down4) --++ (-45:0.6) coordinate(down34);  
                        \path(down4) --++ (-135:0.6) coordinate(down44);  
   }

   \draw [fill=white] plot [smooth cycle]
  coordinates {(up14) (up24)  (down34)   (down44)  };

        \foreach \x in {0,2,4,6,8}
     {   \fill[white](\x,3) circle (4pt);   
                        \draw (\x,3) node {$\bullet$}; 
                  }
      
      \draw(0,0)--(0,3);
                
\draw(6,0)--(6,3);                
\draw(8,0)--(8,3);                
                
      \foreach \x in {0,6,8}
     {   \fill[white](\x,0) circle (4pt);   
                        \draw (\x,0) node {$\bullet$}; 
                  }

\draw[rounded corners ,densely dotted ] (-0.25,-1)  rectangle( 8.25,-2.5)   node[midway] 
%{$\bigl(1+(1,2)\bigr)\bigl(1-(1,3)\bigr)$}; %\otimes _{\W_3}$\specht (2,1)} ; %--  (-0.25,8) -- (-0.25,-)
   {$\color{red}\color{black} \bigl(1 - (1,3) \bigr)\bigl( 1 + (1,2) \bigr)$};                     

         \end{tikzpicture} \!\!\!\!\!\!
 \begin{tikzpicture}[xscale=0.45,yscale=-0.45]
  
       \foreach \x in {0,1,2,3,4,5}
     {  
      \path(\x*2,0) coordinate (down\x);  
   }

       \foreach \x in {0,1,2,3,4,5,6}
    {   \path(\x*2,3) coordinate (up\x);  
    }

     \foreach \x in {0,1,2,3,4,5,6}
     {   \path(up\x) --++ (135:0.6) coordinate(up1\x);  
      \path(up\x) --++ (45:0.6) coordinate(up2\x);
            \path(up\x) --++ (-45:0.6) coordinate(up3\x);  
                        \path(up\x) --++ (-135:0.6) coordinate(up4\x);  
   }

    \foreach \x in {0,1,2,3,4,5}
     {   \path(down\x) --++ (135:0.6) coordinate(down1\x);  
      \path(down\x) --++ (45:0.6) coordinate(down2\x);
            \path(down\x) --++ (-45:0.6) coordinate(down3\x);  
                        \path(down\x) --++ (-135:0.6) coordinate(down4\x);  
   }

   \draw [fill=white] plot [smooth cycle]
  coordinates {(up10) (up21) (down31)   (down40)  };

% 
% \draw(0,3)--(0,0);                       
 %    \draw(6,0)--++(90:3);
%    \draw(4,0)--++(90:3);    

     \foreach \x in {3}
     {   \path(up\x) --++ (135:0.6) coordinate(up14);  
      \path(up4) --++ (45:0.6) coordinate(up24);
            \path(up4) --++ (-45:0.6) coordinate(up34);  
                        \path(up4) --++ (-135:0.6) coordinate(up44);  
   }

    \foreach \x in {3}
     {   \path(down4) --++ (135:0.6) coordinate(down14);  
      \path(down4) --++ (45:0.6) coordinate(down24);
            \path(down4) --++ (-45:0.6) coordinate(down34);  
                        \path(down4) --++ (-135:0.6) coordinate(down44);  
   }
    
   \draw [fill=white] plot [smooth cycle]
  coordinates {(up12) (up23)  (down33)   (down42)  };
   
     \foreach \x in {4}
     {   \path(up\x) --++ (135:0.6) coordinate(up14);  
      \path(up4) --++ (45:0.6) coordinate(up24);
            \path(up4) --++ (-45:0.6) coordinate(up34);  
                        \path(up4) --++ (-135:0.6) coordinate(up44);  
   }

    \foreach \x in {4 }
     {   \path(down4) --++ (135:0.6) coordinate(down14);  
      \path(down4) --++ (45:0.6) coordinate(down24);
            \path(down4) --++ (-45:0.6) coordinate(down34);  
                        \path(down4) --++ (-135:0.6) coordinate(down44);  
   }

   \draw [fill=white] plot [smooth cycle]
  coordinates {(up14) (up24)  (down34)   (down44)  };

        \foreach \x in {0,2,4,6,8}
     {   \fill[white](\x,3) circle (4pt);   
                        \draw (\x,3) node {$\bullet$}; 
                  }

      \draw(0,0)--++(90:3);      \draw(4,0)--++(90:3);
                      \draw(8,0)--++(90:3);
                
      \foreach \x in {0,4,8}
     {   \fill[white](\x,0) circle (4pt);   
                        \draw (\x,0) node {$\bullet$}; 
                  }

\draw[rounded corners ,densely dotted ] (-0.25,-1)  rectangle( 8.25,-2.5)   node[midway]
% {$\bigl(1+(1,2)\bigr)\bigl(1-(1,3)\bigr)$}; %$\otimes _{\W_3}$\specht (2,1)} ; %--  (-0.25,8) -- (-0.25,-)
   {$\color{red}\color{black} \bigl(1 - (1,3) \bigr)\bigl( 1 + (1,2) \bigr)$};                     

         \end{tikzpicture} \!\!\!\!\!\!
 \begin{tikzpicture}[xscale=0.45,yscale=-0.45]
  
       \foreach \x in {0,1,2,3,4,5}
     {  
      \path(\x*2,0) coordinate (down\x);  
   }

       \foreach \x in {0,1,2,3,4,5,6}
    {   \path(\x*2,3) coordinate (up\x);  
    }

     \foreach \x in {0,1,2,3,4,5,6}
     {   \path(up\x) --++ (135:0.6) coordinate(up1\x);  
      \path(up\x) --++ (45:0.6) coordinate(up2\x);
            \path(up\x) --++ (-45:0.6) coordinate(up3\x);  
                        \path(up\x) --++ (-135:0.6) coordinate(up4\x);  
   }

    \foreach \x in {0,1,2,3,4,5}
     {   \path(down\x) --++ (135:0.6) coordinate(down1\x);  
      \path(down\x) --++ (45:0.6) coordinate(down2\x);
            \path(down\x) --++ (-45:0.6) coordinate(down3\x);  
                        \path(down\x) --++ (-135:0.6) coordinate(down4\x);  
   }

   \draw [fill=white] plot [smooth cycle]
  coordinates {(up10) (up20) (down30)   (down40)  };
   \draw [fill=white] plot [smooth cycle]
  coordinates {(up11) (up21) (down31)   (down41)  };
   \draw [fill=white] plot [smooth cycle]
  coordinates {(up12) (up22) (down32)   (down42)  };

   \draw [fill=white] plot [smooth cycle]
  coordinates {(up13) (up24) (up34)   (up43)  };

        \foreach \x in {0,2,4,6,8}
     {   \fill[white](\x,3) circle (4pt);   
                        \draw (\x,3) node {$\bullet$}; 
                  }
                
      \foreach \x in {0,2,4}
     {   \fill[white](\x,0) circle (4pt);   
                        \draw (\x,0) node {$\bullet$}; 
                  }
                        
                              \draw(0,0)--++(90:3);      \draw(2,0)--++(90:3);      \draw(4,0)--++(90:3);
                        
\draw[rounded corners ,densely dotted ] (-0.25,-1)  rectangle( 8.25,-2.5)   node[midway]
{$\color{red}\color{black} \bigl(1 - (1,3) \bigr)\bigl( 1 + (1,2) \bigr)$};
% {%$\otimes _{\W_3}$\specht (2,1)
%$\bigl(1+(1,2)\bigr)\bigl(1-(1,3)\bigr)$} ; %--  (-0.25,8) -- (-0.25,-)

         \end{tikzpicture} 
$$
\caption{ The generators $c_{(2,1)}^\ast	\otimes  v_{(\gamma,\deltaP)}$ for 
$ {(\gamma,\deltaP)}$ equalling  $ ((1^2,0),\varnothing)$ and  $((2,0^2)	,\varnothing)$ and $(\varnothing	,(2))$ respectively. The distinguished zero parts for each $\gamma \in \ParSet_{(1^3)}$ are indicated.
These diagrams generate  the direct summands, which we denote by $N_1$, $N_2$, and $N_3$
of  $\DQ (\Delta_5((1)^{(2,1)}))$. }
\label{N1N2N3} %duplicate label but wasn't used
\end{figure}

Arguing as in the previous example, we have that 
$$N_3\cong (\rightspecht{(2,1)} \otimes \C _{\W_2}	)\Ind_{\W_3\times \W_2}^{\W_5}\cong  
 \rightspecht{(4,1)} \oplus  \rightspecht{(3,2)} \oplus \rightspecht{(3,1^2)}  \oplus \rightspecht{(2^2,1)}.$$ 
However, the other two direct summands behave very differently.

 For $N_1$ we have  $\gamma=(2,0^2)$  and so 
the multiplicities of the parts are again $1$ and $2$ and the restriction
of $\rightspecht{(2,1)}$ is given by~\eqref{M1}. Following Theorem~\ref{thm:DQdecomp},
from the first summand $\rightspecht{(1)} \otimes \rightspecht{(2)}$ we obtain
\smash{$\rightspecht{(1)} \otimes \C \ind_{\W_1 \times \W_2}^{\W_3} \oslash \rightspecht{(2)}$}
which we may write as $M^{(2,1)} \oslash \rightspecht{(2)}$,
where the first tensor factor 
is the Young permutation module $M^{(2,1)}$ for $\C \W_3$; similarly 
from the second summand $\rightspecht{(1)} \otimes \rightspecht{(1^2)}$ we obtain
$M^{(2,1)} \oslash \rightspecht{(1^2)}$. Thus
%However, the induction functor that we apply to this module is quite different.
\begin{align*}
N_1&\cong 
(M^{(2,1)} \otimes 
\rightspecht {(2)} )\Ind_{\W_3 \times \W_2}^{\W_5}
\hskip1pt\oplus \hskip1pt
(
M^{(2,1)} \otimes
\rightspecht {(1^2) })\Ind_{\W_3 \times \W_2}^{\W_5} 
 \\
&\cong
\bigl( \rightspecht { (5)} \oplus 2\rightspecht { (4,1) }\oplus 2 \rightspecht { (3,2)}
 \oplus\rightspecht {  (3,1^2)} \oplus 
 \rightspecht {  (2^2,1)} \bigr)
\\[-4pt]
&\hspace*{1.5in} \oplus 
\bigl(    
 \rightspecht {
   (4,1)} 
   \oplus 
   \rightspecht {  (3,2)}
 \oplus 2
 \rightspecht {  (3,1^2) }\oplus 
 \rightspecht {  (2^2,1)})
 \oplus
 \rightspecht { (2,1^3)} \bigr)
 \\[-1pt]
&\cong
\rightspecht {  (5) 
}
\oplus
3
\rightspecht {  (4,1) }
\oplus3
\rightspecht {  (3,2) }
\oplus3
\rightspecht { (3,1^2) }
\oplus2
\rightspecht {  (2^2,1) }
\oplus
\rightspecht { (2 ,1^3) }.
\end{align*}
 
  For $N_2$ we have  $\gamma=(1^2,0)$  and so the restriction is as in \cref{M2}.
  Again the induction function differs, and we have
 \begin{align*}
N_2&\cong 
\bigl( (\rightspecht { (1^2)}\oslash \rightspecht {(2)} 
\oplus 
 \rightspecht { (2)}\oslash \rightspecht {(2)}  
\oplus 
 \rightspecht { (1^2)}\oslash \rightspecht {(1^2)} 
\oplus 
 \rightspecht { (2)}\oslash \rightspecht {(1^2)} 
 )  \otimes 
\rightspecht {(  1)}\bigr)
\Ind_{{ \W_{2} \wr \W_2}\times \W_1}^{\W_5}
\\
&=
%{\rm ind}_{\W_4 \times \W_1}^{\W_5}
\bigl( (\rightspecht {(3,1)} + \rightspecht {(2^2)} +  \rightspecht {(4)}  \oplus  
\rightspecht {(2,1^2)} + \rightspecht {(2^2)} +  \rightspecht {(1^4)} )  \otimes \rightspecht {(1)}  \bigr) 
\Ind_{\W_4 \times \W_1}^{\W_5}
\\
&=
\rightspecht {(5)}
 \oplus 
2\rightspecht {(4, 1)} 
 \oplus 
 3\rightspecht {(3, 2)}  
 \oplus 2\rightspecht {( 3, 1^2)} \oplus 3  \rightspecht {(2^ 2, 1)} 
 \oplus 2 \rightspecht {(2, 1^3)}
  \oplus  \rightspecht {(1^5)} .
\end{align*}
Summing over the coefficients appearing in the decompositions of $N_1,$ $N_2$ and $N_3$ we obtain the ramified branching coefficients as stated at the beginning of this example.
%MkW: checked 
\end{eg}

\subsection{IndInfRes}
As motivation for the following definition, we return to
Example~\ref{eg:alphaEmpty}.
In this example we computed the summand of \smash{$\DQ\bigl( \Delta_r(\varnothing^\beta)\res^{\ram_r(m,n)}_{P_r(mn)} \bigr)$}
corresponding to a diagram $v_{(\gamma, \epsilon)}$ in four steps.
In the first step
we restricted $\rightspecht{\beta}$ to $S_\mathbf{c}$
where $\mathbf{c}$ is the 
composition recording the number of propagating outer parts in the
diagram $v_{(\gamma, \epsilon)}$ (see Definition~\ref{defn:vGammaEpsilonDiagram}) having 
each possible number of southern dots between $1$ and $r$.
(Thus $c_i$ is the number of parts of $\gamma$ equal to $i$ and
since there are $|\beta| = b$ outer propagating parts,
$\mathbf{c}$ is a composition of $b$.) 
This gives us a sum of tensor products $\bigotimes_{i=1}^r \rightspecht{\beta^i}$
where each $\beta^i$ is a partition of $c_i$.
For instance, for the first diagram in the example, reproduced below,
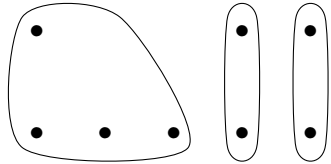
\begin{center} \begin{tikzpicture}[xscale=0.45,yscale=-0.45]
  
       \foreach \x in {0,1,2,3,4,5}
     {  
      \path(\x*2,0) coordinate (down\x);  
   }

       \foreach \x in {0,1,2,3,4,5,6}
    {   \path(\x*2,3) coordinate (up\x);  
    }

     \foreach \x in {0,1,2,3,4,5,6}
     {   \path(up\x) --++ (135:0.6) coordinate(up1\x);  
      \path(up\x) --++ (45:0.6) coordinate(up2\x);
            \path(up\x) --++ (-45:0.6) coordinate(up3\x);  
                        \path(up\x) --++ (-135:0.6) coordinate(up4\x);  
   }

    \foreach \x in {0,1,2,3,4,5}
     {   \path(down\x) --++ (135:0.6) coordinate(down1\x);  
      \path(down\x) --++ (45:0.6) coordinate(down2\x);
            \path(down\x) --++ (-45:0.6) coordinate(down3\x);  
                        \path(down\x) --++ (-135:0.6) coordinate(down4\x);  
   }
    
   \draw [fill=white] plot [smooth cycle]
  coordinates {(up10) (up22) (down31)   (down40)  };

% 
% \draw(0,3)--(0,0);                       
 %    \draw(6,0)--++(90:3);
%    \draw(4,0)--++(90:3);    

     \foreach \x in {3}
     {   \path(up\x) --++ (135:0.6) coordinate(up14);  
      \path(up4) --++ (45:0.6) coordinate(up24);
            \path(up4) --++ (-45:0.6) coordinate(up34);  
                        \path(up4) --++ (-135:0.6) coordinate(up44);  
   }

    \foreach \x in {3}
     {   \path(down4) --++ (135:0.6) coordinate(down14);  
      \path(down4) --++ (45:0.6) coordinate(down24);
            \path(down4) --++ (-45:0.6) coordinate(down34);  
                        \path(down4) --++ (-135:0.6) coordinate(down44);  
   }

   \draw [fill=white] plot [smooth cycle]
  coordinates {(up13) (up23)  (down33)   (down43)  };

     \foreach \x in {4}
     {   \path(up\x) --++ (135:0.6) coordinate(up14);  
      \path(up4) --++ (45:0.6) coordinate(up24);
            \path(up4) --++ (-45:0.6) coordinate(up34);  
                        \path(up4) --++ (-135:0.6) coordinate(up44);  
   }

    \foreach \x in {4 }
     {   \path(down4) --++ (135:0.6) coordinate(down14);  
      \path(down4) --++ (45:0.6) coordinate(down24);
            \path(down4) --++ (-45:0.6) coordinate(down34);  
                        \path(down4) --++ (-135:0.6) coordinate(down44);  
   }

   \draw [fill=white] plot [smooth cycle]
  coordinates {(up14) (up24)  (down34)   (down44)  };

        \foreach \x in {0,2,4,6,8}
     {   \fill[white](\x,3) circle (4pt);   
                        \draw (\x,3) node {$\bullet$}; 
                  }
                
      \foreach \x in {0,6,8}
     {   \fill[white](\x,0) circle (4pt);   
                        \draw (\x,0) node {$\bullet$}; 
                  }

         \end{tikzpicture} 
\end{center}
we have $c= (c_5, c_4, c_3, c_2, c_1) = 
(0,0,1,0,2)$ %with $c_3 = 1$ and  $c_1 = 2$ and $c_i = 0$ for all other $i$,
recording the multiplicities of the parts in $\gamma = (3,1,1)$. %MkW Oct and R
We saw in~\eqref{M1} that the restricted module satisfies
$\rightspecht {(2,1)}\res_{\W_1 \times \W_2}^{\W_{3}} =   \rightspecht {(1)} \otimes 
\rightspecht {(2)}
  \hskip1pt\oplus \hskip1pt    \rightspecht {(1)} \otimes \rightspecht {(1^2) }$.
 As seen in Lemma~\ref{DQdecomp2},
the outer propagating parts having $i$ southern dots are permuted amongst themselves by the wreath
product $\W_i \wr \W_{c_i}$. Since $\alpha = \varnothing$
and so there are no inner propagating parts, the base group acts trivially.
In the second step we obtain the action of the base group
by inflating each tensor factor $\rightspecht{\beta^i}$ from $\W_{c_i}$ to $\W_i \wr \W_{c_i}$, obtaining
\[ \Inf_{\W_1}^{\W_3 \wr \W_1} \rightspecht{(1)} \otimes \Inf_{\W_2}^{\W_1 \wr \W_2} \rightspecht{(2)}
\hskip1pt\oplus\hskip1pt \Inf_{\W_1}^{\W_3 \wr \W_1} \rightspecht{(1)} \otimes \Inf_{\W_2}^{\W_1 \wr \W_2} \rightspecht{(1,1)}. \]
Setting $\Cee = \sum_i i c_i$, 
the full action of the symmetric group $\W_\Cee$ is 
then given, as seen in the proof of Theorem~\ref{thm:DQdecomp} by a third step in which we induce
from $\prod_i \W_i \wr \W_{c_i}$ to $\W_\Cee$.
(Note that here $\alpha = \varnothing$ so $a=0$ and the module
$(\rightspecht{\alpha} \otimes \mathbb{C})\ind_{\W_a \times \W_i}^{\W_{a+i}}$
is simply the trivial $\W_{i}$-module.) Using transitivity of induction we then
finish by tensoring with the trivial module $\C_{\mathrm{Stab}(\epsilon)}$ and inducing
from $\W_\Cee \times \mathrm{Stab}(\deltaP)$ to $\W_r$.
The following functor performs the first three steps.

\begin{defn}\label{defn:IndInfRes}
Let $\gamma = (\Cee^{c_\Cee}, \ldots, 1^{c_1})$ be a partition of $\Cee$
having exactly $b$ parts.
We define
$\IndInfRes_{\gamma} : {\rm mod-}\C \W_b 
\rightarrow {\rm mod-}\C \W_\Cee$
on each right $\C \W_b$-module $W$ by
\[ \IndInfRes_\gamma  W = 
\Bigl( \hskip0.5pt \prod_{i=1}^\Cee \Inf_{\W_{c_i}}^{\W_i \wr \W_{c_i}}
\bigl( W\Res_{\W_{c_\Cee} \times \cdots \times \W_{c_1}} \bigr) \Bigr) \Ind_G^{\W_\Cee} \]
where the subgroup $G$ in the induction functor
is $\W_\Cee \wr \W_{c_\Cee} \times \cdots \times \W_1 \wr \W_{c_1}$. Given $b$, $\Cee \in \NN$ we define
\smash{$\IndInfRes^{\W_\Cee}_{\W_b} : {\rm mod-}\C \W_b 
\rightarrow {\rm mod-}\C \W_\Cee$ }
by \smash{$\IndInfRes^{\W_\Cee}_{\W_b}  = \sum_\gamma \IndInfRes_\gamma$}, 
where the sum is over all $\gamma \vdash \Cee$ having exactly $b$ parts.
\end{defn}
% MkW: query to Rowena, I don't think it matters that the natural order in prod_{i=1}^\Cee does not agree
% with $\W_\Cee \wr \W_{c_\Cee} \times \cdots \times \W_1 \wr \W_{c_1}$ as definition of subgroup
% It is probably much worse to comment on this and risk confusing a reader who can't see the issue at all.

\begin{prop}\label{prop:IndInfRes}
Let $\kappa \vdash r$. For any partition $\beta$ of $b$,
%Let $\beta \vdash k$ and let $\kappa \vdash r$.
the ramified branching coefficient
$\rc(\varnothing^\beta, \kappa)$ satisfies
\[ \rc(\varnothing^\beta, \kappa) =
\sum_{\Cee, \Dee : \Cee + \Dee = r \atop \deltaP \in \ParSet_{>1}(\Dee)}
\bigl[\bigl( \IndInfRes^{\W_\Cee}_{\W_b} \rightspecht{\beta} \otimes \C_{\mathrm{Stab}(\deltaP)}
\Ind^{\W_\Dee} \bigr)
 \Ind_{\W_{\Cee} \times \W_{\Dee}}^{\W_r} : 
 \rightspecht{\kappa} \bigr]_{\W_r}. \]
%where the sum is over all $\Cee$ and $\Dee$ such that $\Cee + \Dee = r$,
%all partitions $1^{c_1}2^{c_2} \ldots \Cee^{c_\Cee}$ of $\Cee$,
%all partitions $\deltaP \in \ParSet_{>1}(\Dee)$, and all partitions
%$\beta^i$ of $c_i$ for each $i \in \{1,\ldots, \Cee\}$.
Moreover if $m \ge r - |\beta| + [\beta \not= \varnothing]$ and $n \ge r + \beta_1$
then either side is equal to 
the plethysm coefficient 
 %$\rc(\varnothing^\beta, \kappa) =
 $p(\beta[n], (m), \kappa[mn])$.
\end{prop}

\begin{proof}
Fix a partition $\gamma = (\Cee^{c_\Cee},\ldots,1^{c_1}) \vdash \Cee$ having exactly $b$ parts.
We have 
\[ \rightspecht{\beta}\Res_{\W_{c_\Cee} \times \cdots \times \W_{c_1}} =
\bigoplus_{\beta^i \vdash c_i}
c^\beta_{\beta^\Cee, \ldots, \beta^1} \rightspecht{\beta^\Cee} \otimes \cdots \otimes \rightspecht{\beta^1} \]
where \smash{$c^\beta_{\beta^\Cee, \ldots, \beta^1}$} is a generalized Littlewood--Richardson coefficient
as defined in~\eqref{eq:genLR}.
%, and the
%direct sum is over all sequences of partitions $\beta^\Cee, \ldots, \beta^1$
%such that $\beta^i \vdash c_i$ for $ 1\le i \le \Cee$. 
%MkW Oct kept generalized LR as first use this section and put \beta^i \vdash c_i in sum;
%is covered by standing convention
By transitivity of induction it follows that
for each $\epsilon \in \ParSet_{>1}(\Dee)$, the composition multiplicity
\[ \bigl[ \bigl(
\IndInfRes_\gamma \leftspecht{\beta} \otimes \C_{\Stab(\epsilon)}\Ind^{\W_\Dee} \bigr)
\Ind_{\W_\Cee \times \W_\Dee}^{\W_r} : \rightspecht{\kappa} \bigr]_{\W_r} \]
is precisely the contribution to the sum in \TheoremD coming from the partitions $\gamma$ and $\epsilon$.
The result now follows from the definition of $\IndInfRes_\gamma$
and $\IndInfRes^{\W_\Cee}_{\W_b}$ by summing over all partitions $\gamma$ and $\epsilon$.
The result on $p( \beta[n], (m), \kappa[mn] )$ follows similarly from \TheoremA. 
%Rc / MkW: omitted this, it wasn't a good idea to offer two different proofs
%or alternatively (and avoiding the need to consider the sums again)
%by repeating the proof of \TheoremA as an immediate corollary of Theorems~\hyperlink{thmC}{C} 
%and~\hyperlink{thmD}{D}.
\end{proof}

\subsection{Marked partitions and plethysm coefficients
when $\beta$ has one row and $\alpha = \varnothing$}\label{subsec:markedPartitions}
%R: added \alpha=\varnothing
In this subsection we apply Proposition~\ref{prop:IndInfRes} to give an elegant \emph{and clearly positive} formula
for the ramified branching coefficients when $\beta$ has a single part.
We require the following definition.

\begin{defn}\label{defn:markedPartition}
Let $b \in \NN$. A {\sf {$b$}-marked partition} of $r \in \NN$ is a pair
of partitions $(\gamma, \deltaP)$ such that $\ell(\gamma) = b$, 
$\deltaP \in \ParSet_{>1}(|\deltaP|)$ and
$|\gamma| + |\deltaP| = r$.
%We set $\Stab (\gamma, \deltaP) = \Stab \gamma \times \Stab(\deltaP)$.
Let $\MParSet_b(r)$ denote the set of $b$-marked partitions of $r$.
\end{defn}

Thus a $b$-marked partition of $r$ may be regarded as an ordinary partition
of $r$ having $b$ distinguished parts, such that only the distinguished parts
may have size $1$. Marked partitions $(\gamma, \deltaP)$ are the types,
in the sense of Definition~\ref{defn:type}, of ramified diagrams when $a=0$.

\begin{prop}\label{prop:nuOneRow}
Let $\kappa$ be a partition of $r$ and let $b \in \NN_0$. Then
\[ \rc(\varnothing^{(b)}, \kappa ) = \sum_{(\gamma,\deltaP) \in \MParSet_b(r)}
\bigl[\mathbb{C}_{\Stab(\gamma) \times \Stab(\deltaP)} \Ind^{S_r} : \rightspecht{\kappa}\bigr]_{\W_r}. \]
%R: \Stab(\gamma,\deltaP) not defined: changed to product
Moreover 
if 
$m \ge r-b + [b \not= 0]$ and $n \ge r+b$ then either side is the plethysm coefficient
$p\bigl((n-b,b), (m), \kappa[mn]\bigr)$.
\end{prop}

\begin{proof}
It is easy to see that
$\IndInfRes^{\W_b}_{\W_\Cee} \rightspecht{(b)} = \sum_\gamma \mathbb{C}_{\Stab(\gamma)} 
\ind^{\W_\Cee}$ where the sum is over all $\gamma \vdash \Cee$ such that $\ell(\gamma) = b$.
The result now follows from Proposition~\ref{prop:IndInfRes} using transitivity of induction.
\end{proof}

The special case $\kappa = (r)$ is worth noting. For any partitions $\gamma$ and
$\epsilon$, it follows from Frobenius reciprocity that
$\bigl[\mathbb{C}_{\Stab(\gamma) \times \Stab(\deltaP)} \Ind^{S_r} : \rightspecht{(r)}\bigr]_{\W_r} = 1$
and so 
\begin{equation}\label{eq:twoRowMarkedPartitions}
 \rc\bigl( \varnothing^{(b)}, (r)\bigr) =  \bigl|\MParSet_b(r)\bigr|. \end{equation}
This leads to a simple closed form for the generating function of the stable limit of the 
corresponding plethysm
coefficients. 
Let $P(z) = \prod_{i=1}^\infty (1-z^i)^{-1}$ be the generating function for the sequence of partition numbers.

\begin{prop}\label{prop:nuOneRowGF}
Let $b \in \NN_0$. We
have
\[ \sum_{r=0}^\infty 
%\Bigl\langle 
\lim_{m,n \rightarrow \infty} 
p\bigl( (n-b,b), (m), (mn-r,r) \bigr)
%s_{(n-b,b)} \circ s_{(m)}, s_{(mn-r,r)}
% \Bigr\rangle 
%\overline{p}((b), \varnothing, (r))
z^r = \frac{z^b}{(1-z^2) \ldots (1-z^b)} P(z). \]
\end{prop}

\begin{proof}
By~\eqref{eq:twoRowMarkedPartitions} it is equivalent to show that 
the generating function for $b$-marked partitions 
is the right-hand side in the proposition. In turn
this follows because partitions with exactly~$b$ parts
are enumerated by $z^b / (1-z) \ldots (1-z^b)$ and partitions
with no singleton parts are enumerated by $(1-z) P(z)$.
\end{proof}

One reason for the interest in Proposition~\ref{prop:nuOneRowGF} is that, via Euler's Pentagonal
Number Theorem (see for instance \cite[Corollary~1.7]{AndrewsPartitions}), it gives an efficient
recurrence relation for the stable limits of the plethysm coefficients $p\bigl( (n-b,b), (m), (mn-r,r) \bigr)$.
When $b=0$ the generating function in the proposition enumerates partitions of $r$ into non-singleton parts;
this is OEIS \cite{OEIS} sequence A002865. When $b=1$ the generating function is $zP(z)$
enumerating partitions, with a shift by $1$. This  is OEIS sequence A000041. When $b=2$ the generating
function is 
\[ \frac{z^2 P(z)}{1-z^2} = \frac{z^2}{(1-z^2)^2} \frac{1}{(1-z)(1-z^3) \ldots }. \]
Since $z^2/(1-z^2)^2 =
\sum_{k=1}^\infty kz^{2k}$ and the remaining part of the right-hand side enumerates partitions
into parts not of size $2$, the coefficient of $z^r$ in the right-hand side
is the total number of parts of size $2$ in all partitions of $r$. 
%MkW Oct: rewrote. I don't want to add a lengthy explanation of this convolution argument, but maybe
% this is now clearer?
The coefficients of $z^2P(z)/(1-z^2)$ form
sequence A024786 in OEIS.
The sequences for greater $b$ do not, at the time of writing, appear in OEIS.

\subsection{Symmetric functions}\label{subsec:symFuncs}
We finish this section by restating Theorems~\hyperlink{thmA}{A} and~\hyperlink{thmD}{D} 
in the language of symmetric functions
and using this restatement to prove three new stability
results. We remind the reader of our standing convention
that $\beta^i \vdash i$ in a sum indicates that the sum is over all relevant sequences of partitions.

\begin{defn}\label{defn:GsymmetricFunction}
Let  $\beta \vdash b$ and let $\gamma = (\Cee^{c_\Cee}, \ldots, 1^{c_1}) \vdash \Cee$ be a partition
such that $\ell(\gamma) \le b$.
Given a non-empty partition $\alpha$ we set $c_0 = |\beta| - \ell(\gamma)$ and define
\[ G^{\alpha}_{\beta,\gamma} = 
\sum_{\beta^i \vdash c_i} 
c^\beta_{\beta^\Cee,\ldots,\beta^1,\beta^0}\prod_{i=0}^\Cee s_{\beta^i} \circ (s_\alpha s_{(i)}).\]
\vspace*{-1pt}
If $\ell(\gamma) = b$ we define
\vspace*{-1pt}
\[ G^\varnothing_{\beta,\gamma} =
\sum_{\beta^i \vdash c_i} 
c^\beta_{\beta^\Cee,\ldots,\beta^1}\prod_{i=1}^\Cee s_{\beta^i} \circ s_{(i)}.\]
We set $G^\alpha_{\beta,\gamma} = 0$ in all other cases.
\end{defn}

Note that in the first product $s_{(0)}$ should be interpreted as $s_\varnothing = 1$.
Thus whenever
$G^{\alpha}_{\beta,\gamma}$ is non-zero 
its degree is $|\gamma| + |\alpha||\beta|$.
For example $G^\alpha_{\beta, \varnothing} = s_\beta \circ s_\alpha$ for any partitions~$\alpha$ and $\beta$
and $G^\varnothing_{\beta,(1^b)} = s_\beta \circ s_{(1)} = s_\beta$ for any partition $\beta$.
The coefficients $G^{\alpha}_{\beta,\gamma}$ have not appeared before in the literature.

\begin{defn}\label{defn:HsymmetricFunction}
Given a partition $\deltaP = (\Dee^{e_\Dee}, \ldots, 2^{e_2}, 1^{e_1})$, we define
$H_\deltaP = \prod_{j=1}^\Dee s_{(e_j)} \circ s_{(j)}$.
\end{defn}

In our application we have $\deltaP \in \ParSet_{>1}(\Dee)$ for some $\Dee$ and so $e_1 = 0$.
It is worth noting that 
if~$\deltaP$ has at most one part of any given size then $H_\deltaP$ is the
complete homogeneous symmetric function denoted $h_\deltaP$ in the standard notation.
By the following lemma,
$H_\deltaP$ corresponds to the permutation module of $\W_\Dee$ acting on the set-partitions
into parts of size specified by $\deltaP$.

\begin{lem}\label{lemma:stab} %I'm not going to save two letters and use lem
For each partition $\deltaP \vdash \Dee$, the symmetric function corresponding under the characteristic isometry
to the module $\C_{\Stab(\deltaP)} \ind^{\W_\Dee}$ is $H_\deltaP$.
\end{lem}

\begin{proof}
By Lemma~\ref{lemma:characteristicIsometry}(b)
the plethysm
$s_{(e_j)} \circ s_{(j)}$ corresponds under the
characteristic isometry to the induced module 
\smash{$\C \ind_{\W_j \wr \W_{e_j}}^{\W_{je_j}}$}.
Using that $\Stab(\epsilon) = \W_\Dee \wr \W_{e_\Dee} \times \cdots \times \W_1 \wr \W_{e_1}$
where~$e_j$ is the multiplicity of $j$ as a part of $\deltaP$, the lemma now follows from
Lemma~\ref{lemma:characteristicIsometry}(a), that the induced product
of modules corresponds to the ordinary product of symmetric functions.
\end{proof}

\begin{prop}\label{prop:symmetricFunctions}
Let $\alpha \vdash a$, $\beta \vdash b$ and $\kappa \vdash r$ be partitions.
The ramified
branching coefficient $\rc(\alpha^\beta, \kappa)$ satisfies
\[ \rc(\alpha^\beta, \kappa) = 
\sum_{\begin{subarray}c\Cee,\,\Dee\, : \, \Cee+\Dee \,=\, r - ab \\
\gamma \vdash \Cee,\,   \deltaP \in \ParSet_{>1}(\Dee)\end{subarray}}
\bigl\langle   G^\alpha_{\beta,\gamma} H_\deltaP, s_\kappa \bigr\rangle. \]
Moreover if $\alpha = \varnothing$, 
$m \ge r-b + [b \not= 0]$ and $n \ge r+b$ 
 then either side is the plethysm coefficient
$p(\beta[n], (m), \kappa[mn])$. 
\end{prop}

\begin{proof}
By Lemma~\ref{lemma:stab}, $H_\deltaP$ is the symmetric function corresponding to 
$\mathbb{C}_{\Stab(\deltaP)}\ind^{\W_\Dee}$.
Therefore, by Lemma~\ref{lemma:characteristicIsometry}(b) and \TheoremD,
to prove the claim on the ramified branching coefficient
it suffices to show that $G^\alpha_{\beta,\gamma}$ is the symmetric function
corresponding under the characteristic isometry to
\[
\bigoplus
_{
\begin{subarray}c
\gamma=  
   ( \Cee^{c_\Cee}, \dots ,  1^{c_1}, 0^{c_0} )
\in  \ParSet _{(a^b)}(\Cee) 
%  \\
%\ell( (a^b) + \gamma) )= b
\\[-3pt]
 \beta^i \vdash c_i \text{ for }  1\leq i \leq \Cee 
\end{subarray}
} 
\!\!\!\!\!\!\!\!\!\!\! c^\beta_{\beta^\Cee, \ldots, \beta^1,\beta^0} \Bigl(
\bigotimes  _{i=0}^\Cee 
\Bigl(
 (
\rightspecht {\alpha} \, \boxtimes\, \C  )
\Ind
_{\W_a\times \W_{ i }}
^{\W_{a+i}}
)
\,\oslash\, \rightspecht{\beta^i}
 \Bigr) \Ind_{\Stab((a^b)+\gamma)}^{\W_\Cee} \Bigr).
\]
(The set $\ParSet_{(a^b)}(\Cee)$ was defined in Definition~\ref{defn:ParSet}.)
By Lemma~\ref{lemma:characteristicIsometry}(b), 
\smash{$\rightspecht{\alpha} \otimes \C \ind_{\W_a \times \W_i}^{\W_{a+i}}$}
corresponds to $s_\alpha s_{(i)}$ and hence, using both parts of this lemma,
the tensor product corresponds to the symmetric function
$\prod_{i=1}^\Cee s_{\beta^i} \circ \bigl( s_{\alpha}s_{(i)} \bigr)$.
The proposition now follows from the definition
of $G^\alpha_{\beta,\gamma}$, noting that if
 $\alpha \not= \varnothing$ then $\gamma \in \ParSet_{(a^b)}(\Cee)$
and $\gamma$ has $c_0 = b-\ell(\gamma)$ distinguished zero parts, while
if $\alpha = \varnothing$ then $\gamma \in \ParSet_{(0^b)}(\Cee)$
and so $\ell(\gamma) = b$.
The result on $p( \beta[n], (m), \kappa[mn] )$ follows 
as in the proof of Proposition~\ref{prop:IndInfRes}.
\end{proof}

Note that each 
$G_\gamma$ and $H_\deltaP$ 
can be expressed as a linear combination of Schur functions using the Littlewood--Richardson
rule and  plethysm coefficients $p(\beta^k, (k), \lambda)$ for varying partitions~$\lambda$.
A further application of the Littlewood--Richardson rule then expresses each $G_\gamma H_{\deltaP}$
as a linear combination of Schur functions. This makes precise the claim in the introduction that
Corollary~D allows stable plethysm coefficients to be computed using \emph{much smaller} Littlewood--Richardson
and plethysm coefficients.

\subsection{Applications of Proposition~\ref{prop:symmetricFunctions}}
As a warm up we give the symmetric functions proof of Proposition~\ref{prop:nuOneRow}.
It is well known that the Littlewood--Richardson coefficient
$c^\lambda_{\mu \nu}$ is non-zero only if $\ell(\lambda) \ge \ell(\mu)$. 
It follows that
the generalized Littlewood--Richardson coefficient 
 \smash{$c^{(b)}_{\beta^\Cee,\ldots,\beta^1}$}
 in the sum defining $G^\varnothing_{(b),\gamma}$
is non-zero
if and only if $\beta^i = (c_i)$ for each $i$, and in this case its value is $1$.
Therefore 
\[ G^\varnothing_{(b), (\Cee^{c_\Cee}, \ldots, 1^{c_1})} 
= \prod_{i=1}^\Cee s_{(c_i)} \circ s_{(i)}. \]
Observe that this is the symmetric function $H_\gamma$ in Definition~\ref{defn:HsymmetricFunction}.
Therefore, by Proposition~\ref{prop:symmetricFunctions}, provided
$m \ge r-b + [b \not= 0]$ and $n \ge r+b$, we have
\[ p\bigl((n-b,b), (m), \kappa \bigr) = \sum_{\substack{
\Cee, \Dee : \Cee + \Dee = r \\
\gamma \vdash \Cee, \ell(\gamma) =b \\  \epsilon \in \ParSet_{>1}(\Dee)}}  %MkW Oct: missed >1 condition
\langle H_\gamma H_\epsilon, s_\kappa \rangle \]
in which we may simplify the condition defining the summation to $(\gamma, \epsilon) \in \MParSet_b(r)$.
This expresses the plethysm coefficient $p\bigl((n-b,b), (m), \kappa[mn]\bigr)$ for $m \ge r$ and $n \ge r$
 as a \emph{clearly positive}
sum of generalized plethysm coefficients of smaller degrees.

\begin{rmk}\label{rmk:MR4756467}
Taking the special case $b = 0$ of this result and substituting the definition of~$H_\epsilon$, we obtain
\[ p\bigl( (n), (m), \kappa[mn]\bigr) = \sum_{\deltaP \in \ParSet_{>1}(r)}
\bigl\langle \hskip1pt \prod_{j \ge 2} s_{(e_j)} \circ s_{(j)}, s_\kappa\bigr\rangle \] 
where $e_j$ is the multiplicity of $j$ as a part of the partition $\deltaP$.
This recovers the main result,
\TheoremB, of \cite{MR4756467}, originally proved as the main theorem in \cite{MR1651092}.
%R: add reference to Manivel
\end{rmk}

We now give three further applications of Proposition~\ref{prop:symmetricFunctions}
that prove new results in the case when $\alpha = \varnothing$.

\subsubsection*{The case $\beta = (1^b)$} %R: change \ell -> b
We require the following basic results on symmetric functions.

\begin{lem}\label{lemma:surelyEveryoneKnowsThis}Let $\Cee \in \NN$ be divisible by $i \in \N$
and let $\pi$ be a partition
of $\Cee/i$.
\begin{thmlist}
\item[\emph{(a)}] 
The plethysm $s_\pi \circ s_{(i)}$ 
has $s_{(\Cee)}$ as a constituent if and only if $\pi$ has exactly one part;
in this case the multiplicity is $1$.
\item[\emph{(b)}] The plethysm $s_\pi \circ s_{(i)}$ has $s_{(1^\Cee)}$ as a constituent
if and only if $i=1$ and $\pi = (1^\Cee)$.
\end{thmlist}
\end{lem}

\begin{proof}
Part (a) follows from Corollary 9.1 of \cite{PagetWildonGeneralized},
which implies
as a special case that the lexicographically greatest constituent
of $s_\pi \circ s_{(i)}$ is $(|\pi|(i-1) + \pi_1, \pi_2, \ldots, \pi_{\ell(\pi)})$.
For (b), we use that every constituent of $s_\pi \circ s_{(i)}$ appears
in $s_{(i)} \times \cdots \times s_{(i)}$ where there are $|\pi|$ factors
in the product. By the Littlewood--Richardson rule (or its simpler special case,
Young's rule), if~$s_\rho$ appears in this product then $\ell(\rho) \le |\pi|$.
Therefore $s_\pi \circ s_{(i)}$ has $s_{(1^\Cee)}$ as a constituent
only if $i=1$, and then since $s_\pi \circ s_{(1)} = s_\pi$, we have $\pi = (1^\Cee)$.
\end{proof}

Let
$\DMParSet_b(r)$ be the set of $b$-marked partitions $(\gamma, \deltaP)$ of $r$
such that the parts of $\gamma$ are distinct.

\begin{prop}\label{prop:hooks}
Let $m, n \in \NN$ and let $b < n$. Suppose that $m \ge r - b+1$ and $n \ge r+1$.
\begin{thmlist} %R: change b -> l
\item We have $p\bigl( (n-b,1^b), (m), (mn-r,r)\bigr) = |\DMParSet_b(r)|$.
\item We have $p\bigl((n-b,1^b), (m), (mn-r,1^r)\bigr) = [r= b]$.
\end{thmlist}
\end{prop}

\begin{proof}
Observe that, dually to the case $\beta = (b)$, 
the generalized Littlewood--Richardson coefficient 
 \smash{$c^{(1^b)}_{\beta^\Cee,\ldots,\beta^1}$}
 in the sum defining $G^\varnothing_{(1^b),\gamma}$
%R: add \gamma (suggestion but I think it's a good one)
 where $\gamma = (\Cee^{c_\Cee}, \ldots, 1^{c_1})$
is non-zero
if and only if $\beta^i = (1^{c_i})$ for each $i$, and in this case its value is $1$.
Therefore 
\[ G^\varnothing_{(1^b),\gamma} = %R: remove sum, just wrong
\prod_{i=1}^\Cee s_{(1^{c_i})} \circ s_{(i)}. \]
For (i), it follows easily from Lemma~\ref{lemma:surelyEveryoneKnowsThis}(a) 
that \smash{$G^\varnothing_{(1^b),\gamma}H_\epsilon$} has $s_{(r)}$ as a constituent
if and only if $\gamma$ has distinct parts, %R: no restriction on \epsilon
and in this case 
the multiplicity is $1$. Therefore by Proposition~\ref{prop:symmetricFunctions} we
have $p\bigl( (n-b,1^b), (m), (r) \bigr) = |\DMParSet_b(r)|$, as required.
Similarly for~(ii), 
it follows easily from Lemma~\ref{lemma:surelyEveryoneKnowsThis}(b) that
\smash{$G^\varnothing_{(1^b),\gamma}H_\epsilon$ has $s_{(1^r)}$}
as a constituent only if $\gamma = (1^\Cee)$ and then, since $\epsilon \in \ParSet(>1)$,
$\epsilon = \varnothing$. Since we require $\ell(\gamma) = |(1^b)|$ and $|\gamma| + |\epsilon| = r$, 
%R: add total size $r$
it then
follows that $r= b$. Hence by Proposition~\ref{prop:symmetricFunctions},
$p\bigl( (n-b,1^b), (m), (mn-r,1^r) \bigr) \not= 0$ 
%R: correct (1^r) to (mn-r,1^r)
%R: \ell -> b
if and only if $r=b$, and in this
case, the coefficient is $1$.
\end{proof}

We remark that (ii) 
is a special case of Theorem 3.1(2) in \cite{LangleyRemmel}; the short
proof given here and the explicit
positive formula in (i) are both new.

\subsubsection*{The case $\kappa = (b)$}
We now generalize the argument for Proposition~\ref{prop:hooks}(i).
Given a partition~$\beta$, let $\mathcal{S}_\beta(p)$ be the set of semistandard $\beta$-tableaux 
having entries from $\NN$ whose sum of entries is~$p$. For example,
\[ \raisebox{10pt}{$\young(1123,23,3)$} \in \mathcal{S}_{(4,2,1)}(15) \]

\begin{prop}\label{prop:generalBetaOneRowKappa}
Let $\beta \vdash b$ be a partition.
Then
\[ \rc\bigl(\varnothing^\beta, (r) \bigr) = %R: (mn-r,r) -> (r), as this is a rc coefficient
\sum_{\Cee = b}^r
\bigl| \mathcal{S}_\beta(p) \bigr|\hskip0.5pt \bigl| \ParSet_{>1}(r-\Cee) \bigr|. \]
Moreover if
$m \ge r - b + [b \not= 0]$ and $n \ge r+b$ 
 then either side is the plethysm coefficient
$p(\beta[n], (m), \kappa[mn])$. 
\end{prop}

\begin{proof}
Let $\gamma = (\Cee^{c_\Cee}, \ldots, 1^{c_1})$ be a partition
with $\ell(\gamma) = b$.
By  Lemma~\ref{lemma:surelyEveryoneKnowsThis}(a), the contribution
to \smash{$G^\varnothing_{\beta,\gamma}$} from
partitions $\beta^\Cee, \ldots, \beta^1$ in the sum in Definition~\ref{defn:GsymmetricFunction}
is non-zero
if and only if each $\beta^i$ has at most one part, and in this case the contribution is $1$.
Since \smash{$c^\beta_{(c_1), (c_2), \ldots, (c_\Cee)}$} is the number
of semistandard tableaux of shape $\beta$ having exactly $c_i$ entries equal
to $i$ for each $1 \le i \le \Cee$, and by the hypothesis $|\gamma| = \Cee$
we have $\sum_{i=1}^\Cee ic_i = p$, it follows that
\[ \sum_{\gamma \vdash \Cee \atop \ell(\gamma) = b}
  \langle G^\varnothing_{\beta,\gamma}, s_{(p)} \rangle = |\mathcal{S}_\beta(p)|. \]
Note that this set is empty unless $p \ge |\beta| = b$. %Used in final line
Again by Lemma~\ref{lemma:surelyEveryoneKnowsThis}(a),
 $\langle H_\deltaP, s_{(\Dee)} \rangle = 1$ for each partition $\deltaP$.
Now by Proposition~\ref{prop:symmetricFunctions}  we have
\begin{align*} %\bigl( p\bigl( \beta[n], (m), \kappa[mn]) 
%R: rc not p and \kappa should be (r)
\rc\bigl(\varnothing^\beta, (r)\bigr)
&= \sum_{
\begin{subarray}c\Cee,\,\Dee\, : \, \Cee+\Dee \,=\, r  \\
\gamma \vdash \Cee,\,   \deltaP \in \ParSet_{>1}(\Dee)\end{subarray}}
\langle G^\varnothing_{\beta,\gamma}H_\deltaP, s_{(r)} \rangle \\
&= \sum_{\Cee,,\Dee\, : \, \Cee+\Dee = r}\,
\sum_{\gamma \vdash \Cee : \ell(\gamma) = b} \langle G^\varnothing_{\beta,\gamma},
s_{(\Cee)} \rangle \sum_{\deltaP \in \ParSet_{>1}(\Dee)}
 \langle H_\deltaP, s_{(\Dee)} \rangle \\
&= \sum_{\Cee,\,\Dee\, : \, \Cee + \Dee = r \atop \Cee \ge b}
\bigl| \mathcal{S}_\beta(p)\bigr|\hskip0.5pt\bigl|\ParSet_{>1}(\Dee) \bigr|\end{align*}
where the second line is a final application of Lemma~\ref{lemma:surelyEveryoneKnowsThis}(a),
%R: (a) -> (i): no, the lemma deliberately uses (a) and (b)
and that $G^\varnothing_{\beta,\gamma} = 0$ unless $\ell(\gamma) = b$, 
%R: clarify explanation of lines
the third line substitutes the results on $G^\varnothing_{\beta,\gamma}$ and $H_\deltaP$
just obtained.
\end{proof}

For instance, to deduce Proposition~\ref{prop:hooks}(i) from this proposition, observe
that each $t \in \mathcal{S}_{(1^b)}(\Cee)$ has $b$ distinct entries
summing to $\Cee$, and so there is an obvious bijection between
the set $\mathcal{S}_{(1^b)}(\Cee) \times \ParSet_{>1}(\Dee)$ and
the subset of $\MParSet_b(r)$ of those marked partitions $(\gamma, \epsilon)$
such that $|\gamma| = p$, $|\deltaP| = q$ and~$\gamma$ has distinct parts.
Again it is a notable feature of Proposition~\ref{prop:generalBetaOneRowKappa}
that the formula is explicit
and clearly positive.

\subsubsection*{Cases where $|\kappa| \le |\beta| + 2$}
We end by showing how \TheoremA determines the 
plethysm coefficients $p(\beta[n], (m), \kappa[mn])$ when $|\kappa| \le |\beta|+1$
and giving an illustrative example of how it can be used to compute this plethysm
coefficient when $|\kappa| = |\beta| + 2$.
We require the following lemma; in (iii), $\beta - \Box$ denotes a partition 
obtained from $\beta$ by removing a single removable box from its Young diagram.

\begin{lem}\label{lemma:GsymmetricFunctionSmallCee}
Let $\beta \vdash b$ and let $\kappa \vdash \Cee$. %R: \gamma -> \kappa.
\begin{thmlist}
\item Let $\Cee < b$. Then $G^\varnothing_{\beta,\gamma} = 0$.
\item Let $\Cee = b$. Then $G^\varnothing_{\beta,\gamma} \not= 0$ if and only
if $\gamma = (1^b)$ and $G^\varnothing_{\beta,(1^b)} = s_\beta$.
\item Let $\Cee = b+1$. Then $G^\varnothing_{\beta,\gamma} \not= 0 $ if and only
if $\gamma \hskip-1.2pt=\hskip-1.2pt (2,1^{b-1})$ and 
$G^\varnothing_{\beta,(2,1^{b-1})} \hskip-1.2pt=\hskip-1.2pt  \sum_{\pi = \beta - \Box} s_\pi s_{(2)}$.
\end{thmlist}
\end{lem}

\begin{proof}
That $G^\varnothing_{\beta,\gamma} = 0$ in each case
 follows easily from the remark after
Definition~\ref{defn:GsymmetricFunction} that \smash{$G^\varnothing_{\beta,\gamma}$} has degree $p$
whenever it is non-zero. For (ii) we use that
a generalized Littlewood--Richardson coefficient with a single bottom factor
\smash{$c^\beta_\pi$} is non-zero if and only if $\pi = \beta$ and for~(iii)
that if $\pi \vdash b-1$ then $c^\beta_{\pi,(1)}$ is non-zero if and only
if $\pi$ is obtained from $\beta$ by removing a single box; the
product in \smash{$G^\varnothing_{\beta,\gamma}$} is then $(s_\pi \circ s_{(1)})(s_{(1)} \circ s_{(2)})
= s_\pi s_{(2)}$, as required.
\end{proof}

\begin{prop}\label{lemma:plethysmCoefficientsSmallC}
Let $\beta \vdash b$ and let $\kappa \vdash r$. %MkW: \gamma -> \kappa 
Suppose that $m \ge r - |\beta| + [\beta \not= \varnothing]$
 and $n \ge r + \beta_1$.
\begin{thmlist}
\item If $r < b$ then $p( \beta[n], (m), \kappa[mn] ) = 0$.
\item If $r = b$ then $p( \beta[n], (m), \kappa[mn] )  \not= 0$ if and only if
$\kappa = \beta$ and then the coefficient is $1$.
\item If $r = b+1$ then $p( \beta[n], (m), \kappa[mn] ) \not= 0$ if and only if
$\kappa$ is obtained from $\beta$ by first removing a box and then adding two boxes, not
both in the same column.%Not bothering with multiplicity as its fiddly
\end{thmlist}
\end{prop}

\begin{proof}
Each part follows easily from the corresponding part in Lemma~\ref{lemma:GsymmetricFunctionSmallCee},
using Proposition~\ref{prop:symmetricFunctions},
noting that in each case since $\epsilon \in \ParSet_{>1}$
and $|\epsilon| \le 1$, we have $\epsilon = \varnothing$ in the sum in this proposition; 
for (iii) we use  Young's rule to multiply $s_\pi$ by $s_{(2)}$.
\end{proof}

We remark that the multiplicity in case (iii) can be arbitrarily high in the case when~$\kappa$ 
is obtained from $\beta$ by adding a single box: for example
if $\beta$ is the staircase partition $(\ell,\ell-1,\ldots, 1)$ and 
$\kappa = (\ell+1,\ell-1,\ldots, 1)$, 
then $\kappa$ can be obtained by removing \emph{any} of the $b$ removable boxes from $\beta$,
and then adding two boxes, not both in the same column. 
Therefore
$p\bigl( (\ell,\ell-1,\ldots, 1)[n], (m), (\ell+1,\ell-1,\ldots,1)[mn] \bigr) = \ell$
whenever $m$ and $n$ satisfy
$m \ge 2$ and $n \ge \binom{\ell+1}{2} + \ell+1 = \binom{\ell+2}{2}$.
%In particular $p\bigl( (\ell+1,\ell,\ell-1,\ldots, 1), (m), (\ell+1,\ell-1,\ldots,1)[mn]\bigr) = \ell$.
%MkW: added details: MAGMA check
% > InnerProduct(R.[3,2,1]~R.[2],R.[8,3,1]);
% 2
% > InnerProduct(R.[4,3,2,1]~R.[2],R.[13,4,2,1]);
% 3
% > InnerProduct(R.[5,4,3,2,1]~R.[2],R.[19,5,3,2,1]);
% 4
% > 
% Above stability:
% > InnerProduct(R.[5,3,2,1]~R.[2],R.[15,4,2,1]);    
% 3
% > InnerProduct(R.[6,3,2,1]~R.[2],R.[17,4,2,1]);
% 3
% > InnerProduct(R.[4,3,2,1]~R.[3],R.[23,4,2,1]);    
% 3

% Below stability: one off, so this seems to be another case where the n bound is tight; the m bound is
% of course tight since when m = 1, just get s_\beta[n].
% > InnerProduct(R.[3,3,2,1]~R.[2],R.[11,4,2,1]);    
% 2
% > InnerProduct(R.[4,4,3,2,1]~R.[2],R.[17,5,3,2,1]);
% 3

We conclude with an example
illustrative of the case when $r = b+2$.

\begin{eg}
We take $\beta = (3,3,3)$ and $r = 11$. There are three non-zero products
$G^\varnothing_{\beta,\gamma} H_\epsilon$ in the right-hand side of Proposition~\ref{prop:symmetricFunctions}.
\begin{itemize}[leftmargin=18pt]
\item $\gamma = (3,1^8)$ and $\epsilon = \varnothing$: the multiplicities of the parts of $\gamma$
are $c_3 = 1$, $c_2 = 0$ and $c_1 = 8$, and since $(3,3,3)$ has a unique removable box
we must then take $\beta^3 = (1)$ and $\beta^1 = (3,3,2)$ to obtain a non-zero Littlewood--Richardson
coefficient \smash{$c^{(3,3,3)}_{\beta^3, \varnothing, \beta^1}$}. Thus in this case the product is 
\smash{$(s^{(1)} \circ s^{(3)})(s_{(3,3,2)} \circ s_{(1)}) =s_{(3)} s_{(3,3,2)} $}.
\item $\gamma = (2,2,1^7)$ and $\epsilon = \varnothing$: the multiplicities of the parts of $\gamma$
are $c_2 = 2$ and $c_1 = 7$, and 
we may take either $\beta^2 = (2)$ and $\beta^1 = (3,3,1)$ or $\beta^2 = (1,1)$ and $\beta^1 = (3,2,2)$
to obtain a non-zero Littlewood--Richardson coefficient. The product is
\begin{align*} \hspace*{0.25in}(s^{(2)} \circ  s_{(2)})(s_{(3,3,1)} \circ s_{(1)}) + (s_{(1,1)}  &{}\circ s_{(2)})(s_{(3,2,2)} \circ s_{(1)}) \\
&= s_{(4)}s_{(3,3,1)} + s_{(2,2)}s_{(3,3,1)} + s_{(3,1)}s_{(3,2,2)}. \end{align*}
\item $\gamma = (1^9)$ and $\epsilon = (2)$: the product is $s_{(3,3,3)} s_{(2)}$.
\end{itemize}
Therefore
\begin{align*} p\bigl( (n-9&{},3,3,3), (m), \kappa[mn] \bigr) \\ &= \langle
s_{(3)}s_{(3,3,2)}  + s_{(4)}s_{(3,3,1)} + s_{(2,2)}s_{(3,3,1)} + s_{(3,1)}s_{(3,2,2)}
+ s_{(3,3,3)} s_{(2)}, s_\kappa \rangle. \end{align*}
It is now routine to use the Littlewood--Richardson rule to
calculate the plethysm coefficients $p\bigl( (n-9,3,3,3), (m), \kappa[mn] \bigr)$
for $n \ge 14$ %r + \beta_1 = 11 + 3 = 4
 and $m \ge 3$ %r - |\beta| + [\beta\not= \varnothing] = 11 - 9 + 1 = 3
for each $\kappa \vdash 11$. For instance when $\kappa = (3,3,3,2)$, the plethysm
coefficient is $4$, with contributions of $1$, $2$ and $1$ from the three products.
In fact it suffices to take $n = 13$ and $m=3$ to get the stable value.
\end{eg}

% > R.[3,3,2]*R.[3];
% R.[3,3,3,2] + R.[4,3,2,2] + R.[4,3,3,1] + R.[5,3,2,1] + R.[5,3,3] + R.[6,3,2]
% > 
% > 
% > R.[4]*R.[3,3,1];
% R.[4,3,3,1] + R.[5,3,2,1] + R.[5,3,3] + R.[6,3,1,1] + R.[6,3,2] + R.[7,3,1]
% >  
% > R.[2,2]*R.[3,3,1];
% R.[3,3,2,2,1] + R.[3,3,3,2] + R.[4,3,2,1,1] + R.[4,3,2,2] + R.[4,3,3,1] + 
% R.[4,4,1,1,1] + R.[4,4,2,1] + R.[5,3,2,1] + R.[5,3,3] + R.[5,4,1,1] + R.[5,4,2] 
% + R.[5,5,1]
% > 
% > R.[3,1]*R.[3,2,2];
% R.[3,3,2,2,1] + R.[3,3,3,2] + R.[4,2,2,2,1] + R.[4,3,2,1,1] + 2*R.[4,3,2,2] + 
% R.[4,3,3,1] + R.[4,4,2,1] + R.[5,2,2,1,1] + R.[5,2,2,2] + 2*R.[5,3,2,1] + 
% R.[5,3,3] + R.[5,4,2] + R.[6,2,2,1] + R.[6,3,2]
% > 
% > 
% > R.[3,3,3]*R.[2];
% R.[3,3,3,2] + R.[4,3,3,1] + R.[5,3,3]

% > InnerProduct(R.[4,3,3,3]~R.[3],R.[39-11,3,3,3,2]);
% 4

%> GH([],[3,3,3],11);
% 2*R.[3,3,2,2,1] + 4*R.[3,3,3,2] 										<- used this
% + R.[4,2,2,2,1] + 2*R.[4,3,2,1,1] + 4*R.[4,3,2,2] + 5*R.[4,3,3,1] + 
% R.[4,4,1,1,1] + 2*R.[4,4,2,1] + R.[5,2,2,1,1] + R.[5,2,2,2] + 5*R.[5,3,2,1] + 5*R.[5,3,3] + 
% R.[5,4,1,1] + 2*R.[5,4,2] + R.[5,5,1] + R.[6,2,2,1] + R.[6,3,1,1] + 3*R.[6,3,2] + R.[7,3,1]

\section{The bounds in Theorems~A and~D %No chance of getting hyperlinks in section heading
when $\alpha = \varnothing$ cannot be weakened}
\label{sec:tight}

In this section
we show that when 
$\beta = (b)$ the bound  on $n$ in Theorems~\hyperlink{thmA}{A} and~\hyperlink{thmC}{C}
cannot be weakened for infinitely many~$\kappa$, and when $\beta = \varnothing$,
similarly the bound on $m$ cannot be weakened when $\kappa = (r)$. %R 'can' would be better read as 'cannot'

We require the following proposition and lemma, which generalize the Cayley--Sylvester formula
(see \cite[\S 20]{CayleyQuantics} or, for an elegant modern proof using the symmetric group, 
\cite[Corollary~2.12]{GiannelliArchMath})
that the plethysm coefficient $p\bigl((n), (m), (mn-r,r)\bigr)$ is the difference 
between the number of partitions of $r$
and $r-1$ contained in the $n \times m$ box. To prove Proposition~\ref{prop:plethysmToT},
we shall use the combinatorial model for a general plethysm $s_\nu \circ s_\mu$ from
\cite{deBoeckPagetWildon};
we give enough details to make this section self-contained provided the reader
takes one basic lemma from \cite{deBoeckPagetWildon} on trust.

\begin{prop}\label{prop:plethysmToT}
For $k$, $r \in \NN$ satisfying $n-k \ge k$ and $mn - r \ge r$, 
we define $\Tmnkr{m}{n}{k}{r}$ to be the set of semistandard Young tableaux of shape $(n-k,k)$ with entries from $\{0,1,\ldots, m\}$ whose
sum of all $n$ entries is $r$.
Then
\[
p\bigl((n-b,b), (m), (mn -r,r)\bigr) 
= \bigl|\Tmnkr{m}{n}{b}{r}\bigr| - \bigl|\Tmnkr{m}{n}{b}{r-1}\bigr|
\]
\end{prop}

\begin{proof}
The set of $(m)$-tableaux with entries from $\mathbb{N}_0$ is totally ordered by 
comparing their entries, read left-to-right, by the lexicographic order. Let $<$ denote this total order.
For example when $m=3$ we have
$\young(012) < \young(013) < \young(033)$\hskip0.5pt. 
Define a \textsf{plethystic semistandard tableaux} of \textsf{shape} $\bigl((n-b,b), (m)\bigr)$ to be an $(n-b,b)$-tableau $T$ whose
entries are $(m)$-tableaux, arranged in $T$ so that they are weakly increasing under $<$ read along each row,
and strictly increasing under $<$ read down each column. If for each $i \in \NN
$,
the plethystic semistandard tableau 
$T$ has exactly $\omega_i$ entries of $i$ in its $(m)$-tableau entries, then
we say~$T$ has \textsf{weight} $\omega$ and write
$x^T = x_1^{\omega_1}x_2^{\omega_2} \ldots $. (Note that, by this definition, 
zero entries are not considered when computing the weight.) 
Let $\PSSYT\bigl( (n-b,b), (m) \bigr)_\omega$ denote the set 
of plethystic semistandard tableau of shape $\bigl((n-b,b), (m)\bigr)$ and weight $\omega$.
For example 
\[ \pyoung{1.6cm}{0.7cm}{ {{\young(000), \young(001)}, {\young(011), \young(011)}} } \raisebox{18pt}{ ,}
\quad \pyoung{1.6cm}{0.7cm}{ {{\young(000), \young(001)}, {\young(011), \young(012)}} }\]
are in $\PSSYT\bigl( (2,2), (3) \bigr)_{(5)}$ and $\PSSYT\bigl( (2,2), (3) \bigr)_{(4,1)}$, respectively.
 By \cite[Lemma~3.1]{deBoeckPagetWildon}, we have
\[ s_{(n-b,b)} \circ s_{(m)} = \sum_{T} x^T \]
where the sum is over all   plethystic semistandard tableaux $T$ of shape $\bigl((n-b,b), (m)\bigr)$.
Hence the coefficient of the monomial symmetric function labelled by the partition $(mn-r,r)$ 
in $s_{(n-b,b)} \circ s_{(m)}$ is $\bigl| \PSSYT\bigl( (n-b,b), (m) \bigr)_{(r)} \bigr|$. 
By the duality 
between the complete homogeneous and monomial symmetric functions (see \cite[(7.30)]{StanleyII}),
this coefficient is $\langle s_{(n-b,b)} \circ s_{(m)}, h_{(mn-r,r)} \rangle$. Hence, by the relation
$s_{(mn-r,r)} = h_{(mn-r,r)} - h_{(mn-r+1,r-1)}$, we get
%\begin{align*} 
\[ \langle s_{(n-b,b)} \circ s_{(m)}, s_{(mn-r,r)} \rangle = % &=
\bigl| \PSSYT\bigl( (n-b,b), (m) \bigr)_{(r)} \bigr| \\
%& \qquad 
- \bigl| \PSSYT\bigl( (n-b,b), (m) \bigr)_{(r-1)} \bigr|. \] % \end{align*}
Finally observe that each $(m)$-tableau entry in a plethystic semistandard tableaux of  weight
$(r)$ has entries from $0$ and $1$, with exactly $r$ entries of $1$, and so 
the set \smash{$\PSSYT\bigl( (n-b,b), (m) \bigr)_{(r)}$} and the set $\Tmnkr{m}{n}{b}{r}$
are in bijection by the map that replaces each $(m)$-tableau entry in a 
plethystic semistandard tableau by its number of $1$s.
Note that the map preserves the semistandard condition
thanks to our choice of the order $<$ on $(m)$-tableaux.
For example, the image of the plethystic semistandard tableaux shown left above is 
\raisebox{4pt}{$\young(01,22)$}\hskip1pt. 
\end{proof}

Let $\PP{b}{m}{r}$ denote
the set of pairs of partitions $(\lambdaP,\muP)$ such that $\lambdaP_1, \muP_1 \le m$,
%$\ell(\lambdaP) \le b$
$\ell(\lambdaP) = b$
 and
$|\lambdaP| + |\muP| = r$. Note there is no  restriction on $\ell(\muP)$.

\begin{lem}\label{lemma:TtoR}{\ }
\begin{thmlist}
\item If $n \ge r + b$ and $n \ge 2b$ then
$\bigl|\Tmnkr{m}{n}{b}{r}\bigr| = \bigl|\PP{b}{m}{r}\bigr|$.
\item We have $\bigl|T\bigl( (r-1,b), m \bigr)_r \bigr|  = \bigl|\PP{b}{m}{r}\bigr|-1$.
\end{thmlist}
\end{lem}

\begin{proof}
For (i) we suppose that $n \ge r\,+\,b$.
Let $t_{(i,j)}$ be the entry in box $(i,j)$ of the tableau $t \in \Tmnkr{m}{n}{b}{r}$. 
Suppose, for a contradiction, that $t_{(1,c)} \ge 1$ for some $c \le b$; that is, there is an entry
of $1$ above a box in the second row of $t$.
Since $t_{(2,j)} > t_{(1,j)}$ for each $j$, the sum of the entries of $t$, namely $r$, is at least 
\[ (n-b-c+1) + (c-1) + 2(b-c+1) =
n+b-2(c-1) \] 
where $n-b-c+1$ counts entries in the boxes $(1,c), \ldots, (1,n-b)$, $c-1$ counts entries in the boxes
$(2,1), \ldots, (2,c-1)$ and $2(b-c+1)$ counts entries in the boxes
$(2,c),\ldots, (2,b)$.
Therefore $r \ge (n+b)-2(c-1)$, which given the hypothesis $n \ge r+b$,
%given the hypothesis $n \ge r+b-1$,
 implies that $c > b$, a contradiction.
Therefore $t_{(1,1)} = \ldots = t_{(1,b)} = 0$ and~$t$ is determined by the pair 
of partitions $\lambdaP = (t_{(2,b)}, \ldots t_{(2,1)})$
and $\muP = (t_{(1,n-b)}, \ldots, t_{(1,b+1)})$. Note that since $t_{(2,1)} \ge 1$, we have $\ell(\lambdaP) = b$.
Therefore $(\lambdaP, \muP) \in \PP{b}{m}{r}$. For example, taking $n=8$, $b=3$, $r=5$ and any $m \ge 2$,
\[ \young(00011,111)\raisebox{-3pt}{\hskip2pt ,} \ \ \young(00002,111)\raisebox{-3pt}{\hskip2pt ,} \ \
\young(00001,112) \]
are in bijection with $\bigl((1,1,1),(1,1)\bigr), \bigl((1,1,1),(2)\bigr), \bigl((2,1,1),(1)\bigr) \in
\PP{3}{m}{5}$, respectively.
Conversely,
given $(\lambdaP, \muP) \in \PP{b}{m}{r}$, 
since  $\ell(\muP) \le |\muP| = r - |\lambdaP| \le r - \ell(\lambdaP) = r -b \le (n-b)-b = n-2b$,
we may reverse the process just to described to define a tableau $t \in \Tmnkr{m}{n}{b}{r}$
whose image is $(\lambda, \muP)$. This gives 
a bijection proving (i). 

If instead $n = r+b-1$ then 
the inverse map fails to be well-defined when $\ell(\muP) = r-b$ and $\ell(\lambdaP) = b$ 
and $\muP_1 \le \lambdaP_{\ell(\lambdaP)}$.
Thus the partition pair $(\lambdaP, \muP) = ((1^b), (1^{r-b})) \in \PP{b}{m}{r}$
is not the image of a tableau $t \in \Tmnkr{m}{r+b-1}{b}{r}$.
Since this is the only case where $\muP_1 \le \lambdaP_{\ell(\lambdaP)}$,~(ii) now follows from (i).
\end{proof}

The subset
of $\PP{b}{m}{r}$ of those partition pairs $(\lambdaP, \muP)$ in which $\muP$ has a singleton part
is in bijection with $\PP{b}{m}{r-1}$ by removing the final part of $\muP$.
Recall from Definition~\ref{defn:markedPartition} that a $b$-marked partition of $r$
is a pair $(\gamma, \deltaP)$
such that $|\gamma| + |\deltaP| = r$, $\ell(\gamma) = b$ and $\epsilon$ has no singleton parts.
Let $\MParSet^m_b(r)$ be the subset of $\MParSet_b(r)$ consisting of
 those marked partitions $(\gamma, \deltaP)$ for which $\gamma_1 \le m$ and $\deltaP_1 \le m$.
By this bijection we have
\begin{equation} \label{eq:PPtoMP} \bigl|\PP{b}{m}{r}\bigr| - 
\bigl|\PP{b}{m}{r-1}\bigr| = \bigl|\MParSet^m_b(r) \bigr| .\end{equation}
Recall that, by \TheoremC, the ramified branching coefficient $\rc(\varnothing^{(b)},(r) \bigr)$ 
is the stable limit of the plethysm coefficients $p\bigl( (n-b,b), (m), (mn-r,r) \bigr)$ for large $m$ and $n$.
%By~\eqref{eq:twoRowMarkedPartitions} we have
%$\rc(\varnothing^{(b), (r) \bigr) = \MParSet_b(r) \end{equation}
%provided $n \ge r + \beta_1$ and $m \ge r - |\beta| + [\beta\not=\varnothing]$.

\begin{cor}\label{cor:mnsharp}
Let $m$, $n \in \NN$. Let $b \in \NN_0$ and let $r \in \NN$ with $r > b$. 
\begin{thmlist}
\item If $n \ge r + b$ then 
$p\bigl((n-b,b), (m), (mn-r,r)\bigr) 
%$\langle s_{(n-b,b)} \circ s_{(m)}, s_{(mn-r,r)}\rangle 
= \bigl|\MParSet^m_b(r) \bigr|$
and if $m \ge r-b + [b \not=0]$ then each side is
$\rc\bigl(\varnothing^{(b)}, (r)\bigr)$. %\overline{p}((b), \varnothing, (r))$.
\item If $b > 0$ then
$p\bigl((r,b), (r-b), ((r-b)(r+b)-r,r)\bigr)
%$\langle s_{(r,b)} \circ s_{(r-b)}, s_{(m(r+b)-r,r)} \rangle 
= \rc\bigl(\varnothing^{(b)}, (r) \bigr) -1$.
\item  We have 
$p\bigl((r), (r-1), (r^2-r,r)\bigr)
%$\langle s_{(r)} \circ s_{(r-1)}, s_{(r^2-r,r)} \rangle 
= \rc\bigl(\varnothing^\varnothing, (r)\bigr)-1$.
\item  We have 
$p\bigl((r-1,b), (m), m(r+b-1)-r,r\bigr)
%$\langle s_{(r-1,b)} \circ s_{(m)}, s_{(m(r+b-1)-r,r)}\rangle 
= \bigl|\MParSet^m_b(r)\bigr| - 1 $
and if~$m \ge  r - b + [b\not=0]$, then each side is
$\rc\bigl(\varnothing^{(b)}, (r)\bigr) - 1$.
\end{thmlist}
\end{cor}

% Original statement using p-bar rather than rc. Better for unity of paper not to introduce yet another notation.
%\begin{cor}\label{cor:mnsharp}
%Let $m$, $n \in \N$. Let $b \in \N_0$ and let $r \in \N$ with $r > b$. 
%\begin{thmlist}
%\item If $n \ge r + b$ then 
%$p\bigl((n-b,b), (m), (mn-r,r)\bigr) 
%%$\langle s_{(n-b,b)} \circ s_{(m)}, s_{(mn-r,r)}\rangle 
%= |\PP^\star(b \times m)_r|$
%and if $b=0$ or $m \ge r-(b-1)$ each side is the stable symmetric  plethysm coefficient 
%$\overline{p}((b), \varnothing, (r))$.
%\item For $b \in \N$ we have 
%$p\bigl((r,b), (r-b), (m(r+b)-r,r)\bigr)
%%$\langle s_{(r,b)} \circ s_{(r-b)}, s_{(m(r+b)-r,r)} \rangle 
%= \overline{p}((b), \varnothing, (r)) - 1$.
%\item  We have 
%$p\bigl((r), (r-1), (r^2-r,r)\bigr)
%%$\langle s_{(r)} \circ s_{(r-1)}, s_{(r^2-r,r)} \rangle 
%= \pl{\varnothing}{(r)}-1$.
%\item  We have 
%$p\bigl((r-1,b), (m), m(r+b-1)-r,r\bigr)
%%$\langle s_{(r-1,b)} \circ s_{(m)}, s_{(m(r+b-1)-r,r)}\rangle 
%= |\PP^\star(b \times m)_r| - 1$
%and when~$m \ge r-(b-1)$ the common value is
%$\overline{p}((b), \varnothing, (r))$
%\end{thmlist}
%\end{cor}

\bigskip
\begin{proof}{\ }
%By~\eqref{eq:twoRowMarkedPartitions}, we have
%$\rc\bigl(\varnothing^{(b)},(r)\bigr) = \bigl|\MParSet_b(r)\bigr|$.

\begin{itemize}[leftmargin=18pt] 
\item[(i)] 
By Proposition~\ref{prop:plethysmToT} and Lemma~\ref{lemma:TtoR}(i), when $n \ge r+b$ we have
\begin{align*}
%\langle s_{(n-b,b)} \circ s_{(m)}, s_{(mn-b,b)}\rangle 
\hspace*{0.5in}p\bigl((n-b,b), (m), (mn-b,b)\bigr)&= \bigl|\Tmnkr{m}{n}{b}{r}\bigr| 
- \bigl|\Tmnkr{m}{n}{b}{r-1}\bigr| \\ &=
 \bigl|\PP{b}{m}{r}\bigr| - \bigl|\PP{b}{m}{r-1} \bigr|.\end{align*}
The first claim now follows using~\eqref{eq:PPtoMP}.
Suppose first of all that $b=0$. If $(\lambdaP, \muP) \in \MParSet^m_b(r)$ then
$\lambdaP = \varnothing$ and the largest possible part in $\muP$ is $r$.
Since by hypothesis, $m \ge r$, this restriction has no force, and so
$\MParSet^m_0(r) = \MParSet_0(r)$. Thus the second claim follows from
$\rc\bigl(\varnothing^{\varnothing},(r)\bigr) = \bigl|\MParSet_0(r)\bigr|$.
Now suppose $b > 0$. If $(\lambdaP, \muP) \in \MParSet^m_b(r)$ then
since $\ell(\lambdaP) = b$,
the largest possible part in $\lambdaP$ is $r-(b-1)$ 
and the largest possible part in $\muP$ is $r-b$. 
Since $m \ge r-b + [b\not=0] = r-b+1$ by hypothesis, again
the restriction that $\lambdaP_1 \le m$ and $\muP_1 \le m$ has no force.
Arguing in the same way as the case $b=0$ we now have
$\MParSet^m_b(r) = \MParSet_b(r)$
and $\rc\bigl(\varnothing^{(b)},(r)\bigr) = \bigl|\MParSet_b(r)\bigr|$ and the second
claim follows in the same way.
%MkW: this is the edit to address problems in email Friday 27 October
%Observe that if $(\lambdaP, \muP) \in \MParSet^m_b(r)$ then
%since $\ell(\lambdaP) = b$,
%the largest possible part in $\lambdaP$ is $r$ if $b=0$ or $r-(b-1)$ if $b>0$, 
%and the largest possible part in $\muP$ is $r-b$. 
%Therefore, when $m \ge r-b + [b\not=0]$, 
%the restriction that $\lambdaP_1 \le m$ and $\muP_1 \le m$ has no force,
%and so $\MParSet^m_b(r) = \MParSet_b(r)$. Thus the second claim in (i) 
%follows from~$\rc\bigl(\varnothing^{(b)},(r)\bigr) = \bigl|\MParSet_b(r)\bigr|$

\item[(ii)] We have $n = r+b$.
The unique pair in $\MParSet^{r-b+1}_b(r)$ 
%$\PP^\star\bigl(b \times (r-b+1)\bigr)$ 
not present in
%$\PP^\star\bigl( b \times r-b\bigr)$ 
$\MParSet^{r-b}_b(r) = \MParSet_b(r)$ 
is $\bigl((r-b+1,1^{b-1}),\varnothing)$, therefore (ii) follows
from both parts of (i). 

\item[(iii)] Again we have $n=r+b$, since now $b=0$.
The unique pair in $\MParSet^{r-1}_0(r)$ not present in 
 %\PP^\star\bigl(0 \times (r-1)\bigr)_r$ not present in
$\MParSet_0^r(r) = \MParSet_0(r)$ is $\bigl((r),\varnothing\bigr)$, therefore (iii)
follows similarly from both parts of~(i).

\item[(iv)] We have $n = r+b-1$. Arguing similarly to (i) using Proposition~\ref{prop:plethysmToT} and
then Lemma~\ref{lemma:TtoR}(ii) for the first summand and Lemma~\ref{lemma:TtoR}(i) for the second, we have
\begin{align*}
\hspace*{0.25in}p\bigl((r-1,b), (m), (m(r+b-1)-r,r)\bigr)&= 
\bigl| T\bigl((r-1,b),m)_r \bigr| - 1  - \bigl| T \bigl((r-1,b),m \bigr)_{r-1} \bigr| \\
 &=
 \bigl|\PP{b}{m}{r}\bigr| - \bigl|\PP{b}{m}{r-1} \bigr| - 1.\end{align*}
The first claim now follows as in (i) using~\eqref{eq:PPtoMP}. As seen in (i),
when $m \ge r - b + [b\not= 0]$ we have 
$\MParSet^m_b(r) = \MParSet_b(r)$ and so the second claim
follows from~$\rc\bigl(\varnothing^{(b)},(r)\bigr) = \bigl|\MParSet_b(r)\bigr|$.\end{itemize}
\end{proof}

We summarise the results of this section in the following corollary.

\begin{cor}\label{cor:boundsAreTight}
Let $b \in \NN_0$. Let $r, b \in \NN$ with $r > b$. 
When $\beta = (b)$ and $\kappa = (r)$ the bounds
$m \ge r-b + [b \ge 1]$ and $n \ge r + \beta_1$ 
in Theorems~\hyperlink{thmA}{A} and~\hyperlink{thmC}{C} cannot be weakened.
\end{cor}

\begin{proof}
By Corollary~\ref{cor:mnsharp}(ii),
if $b > 0$, then $p\bigl((r,b), (r-b), (m(r+b)-r,r)\bigr)$ is 
one less than the stable value, which is attained when $m$ increases from $r-b$ to $r-b + [b \not=0] = r-b+1$.
By Corollary~\ref{cor:mnsharp}(iii)
$p\bigl((r), (r-1), (r^2-r,r)\bigr)$ is again one less than the stable value,
which is attained when $m$ increases from $r-1$ to $r-b + [b\not=0] = r$. (Note that here $b=0$.)
By Corollary~\ref{cor:mnsharp}(iv),
$p\bigl((r-1,b), (m), m(r+b-1)-r,r\bigr)$
is one less than the stable value, attained when $n$ increases from $r+b-1$ to $r+b$.
\end{proof}

We finish with a corollary for modules for the ramified partition algebra, showing that
the bounds $m \ge r-|\beta| + [\beta \not=0]$ and $n \ge r + \beta_1$
in \TheoremC cannot be weakened in infinitely many cases. 
We remark that an alternative proof is given by reading the outline in Section~\ref{subsec:structure},
noting that the only use of these bounds is in step (f): thus if we
assume, for a contradiction, that $L_r(\varnothing^\beta) = \Delta_r(\varnothing^\beta)$ then
we obtain that the plethysm coefficient $p\bigl( \beta[n], (m), \kappa[mn])$ is equal
to its stable value; this contradicts Corollary~\ref{cor:mnsharp} for appropriately chosen $m$, $n$
and  partitions $\beta$,~$\kappa$.

\begin{cor}\label{cor:onlyIfProperQuotient}
There exist infinitely many partitions $\beta$ 
such that if $n< r+\beta _1$ or $m< r-|\beta|+[\beta\not= \varnothing]$ then
$L_{r }(\varnothing^{\beta }) $ is a proper quotient of $\Delta_{r }(\varnothing^{\beta } )$
and the inequality in \TheoremC is strict.  
   \end{cor}
\begin{proof} 
By \TheoremB we have %MkW Oct: not Theorem C
$p(\beta[n],\varnothing,\kappa[mn]) = 
\bigl[ L_r(\varnothing^\beta)\Res^{R_r(m,n)}_{P_r(mn)} :
L_r(\kappa) \bigr]\ForPrmn$ for each $\kappa \vdash r$. Hence by \TheoremC, 
\[ p(\beta[n],\varnothing,\kappa[mn]) \le \bigl[ \Delta_r(\varnothing^\beta)\Res^{R_r(m,n)}_{P_r(mn)} :
L_r(\kappa) \bigr]\ForPrmn \] with equality
if every composition factor $L_r(\kappa)$ of $\Delta_r(\varnothing^\beta)\Res^{R_r(m,n)}_{P_r(mn)}$
appears in $L_r(\varnothing^\beta)\Res^{R_r(m,n)}_{P_r(mn)}$, with equal multiplicity.
Therefore, taking the contrapositive of the case for equality, if
\[ p(\beta[n],\varnothing,\kappa[mn])  < \bigl[\Delta_r(\varnothing^\beta)\Res^{R_r(m,n)}_{P_r(mn)} :
L_r(\kappa) \bigr]\ForPrmn \]
then \smash{$\Delta_r(\varnothing^\beta)\res^{R_r(m,n)}_{P_r(mn)}$} is not simple.
The right-hand side in the two displayed equation above is 
the ramified branching coefficient $\rc(\varnothing^\beta, \kappa)$ found in \TheoremC, or its
equivalent restatement, Theorem~\ref{prop:symmetricFunctions}.
Since by \TheoremC the ramified branching coefficient is the stable limit of the 
plethysm coefficients $p(\beta[n], (m), \kappa[mn])$ for large $m$ and $n$,
the result now follows from Corollary~\ref{cor:boundsAreTight}.
\end{proof}

\bibliographystyle{amsalpha}   
\bibliography{book3}

 \end{document}